\documentclass[12pt]{amsart}%
\usepackage{amsfonts}
\usepackage{amsmath}
\usepackage{amssymb}
\usepackage{graphicx}%
\providecommand{\U}[1]{\protect\rule{.1in}{.1in}}
\newtheorem{theorem}{Theorem}
\theoremstyle{plain}

\newtheorem{corollary}{Corollary}

\newtheorem{example}{Example}

\newtheorem{lemma}{Lemma}

\newtheorem{proposition}{Proposition}
\newtheorem{remark}{Remark}

\numberwithin{equation}{section}
\numberwithin{equation}{section}
\numberwithin{theorem}{section}
\numberwithin{lemma}{section}
\numberwithin{remark}{section}
\numberwithin{example}{section}
\numberwithin{proposition}{section}
\numberwithin{definition}{section}
\numberwithin{corollary}{section}
\begin{document}
\title[Separable morphisms of operator Hilbert systems]{Separable morphisms of operator Hilbert systems, Pietsch factorizations and
entanglement breaking maps}
\author{Anar Dosi}
\address[Dosi=Dosiev]{ Middle East Technical University NCC, Guzelyurt, KKTC, Mersin
10, Turkey }
\email{dosiev@yahoo.com, dosiev@metu.edu.tr}
\urladdr{http://www.mathnet.ru/php/person.phtml?option\_lang=eng\&personid=23380}
\date{March 28, 2019}
\subjclass[2000]{Primary 46L07; Secondary 46B40, 47L25}
\keywords{Quantum cone, multinormed $W^{\ast}$-algebra, quantum system, quantum order }

\begin{abstract}
In this paper we investigate operator Hilbert systems and their separable
morphisms. We prove that the operator Hilbert space of Pisier is an operator
system, which possesses the self-duality property. It is established a link
between unital positive maps and Pietch factorizations, which allows us to
describe all separable morphisms from an abelian $C^{\ast}$-algebra to an
operator Hilbert system. Finally, we prove a key property of entanglement
breaking maps that involves operator Hilbert systems.

\end{abstract}
\maketitle

\section{\textbf{Introduction\label{Sec1}}}

The separable morphisms between operator systems play a fundamental role in
many aspects of quantum information theory. A key result proven in
\cite{PaulTT} by Paulsen, Todorov and Tomforde asserts that a separability of
a linear mapping between finite dimensional matrix algebras is equivalent to
its property to be an entanglement breaking mapping. The latter in turn is
equivalent to max matrix (or min-max matrix) positive mapping of the related
operator system structures. Thus a separable channel can be thought as a max
matrix positive mapping between finite-dimensional matrix algebras preserving
the related traces. Whether the separable morphisms characterize the max
matrix positive maps of operator systems was formulated in \cite[Problem
6.16]{PaulTT} as an open problem. How to be with the min-max matrix positive
maps (see \cite[Problem 6.17]{PaulTT})? On this concern a possible
characterization of separable morphisms between some operator systems is of
great importance.

The operator systems are unital self-adjoint subspaces of the operator space
$\mathcal{B}\left(  H\right)  $ of all bounded linear operators on a Hilbert
space $H$. They critically occurred in Paulsen's approach \cite{Pa} to the
normed quantum functional analysis \cite{ER}, \cite{Pis}, \cite{Hel}. Abstract
characterization of operator systems was proposed by Choi and Effros in
\cite{ECh}. They are matrix-ordered $\ast$-vector spaces with their
Archimedian matrix order units. In the duality concept (see \cite{Dfaa},
\cite{DfaaR}) they are weakly closed, unital, separated, quantum cones on a
$\ast$-vector space $\mathcal{X}$ with a unit $e$. Recall that a quantum cone
$\mathfrak{C}$ on $\mathcal{X}$ is a quantum additive subset of the hermitian
matrix space $\mathcal{M}\left(  X\right)  _{h}$ over $\mathcal{X}$ such that
$a^{\ast}\mathfrak{C}a\subseteq\mathfrak{C}$ for all scalar matrices $a\in M$.
If $\mathfrak{C-e}$ is an absorbent quantum set in $\mathcal{M}\left(
X\right)  _{h}$, then we say that $\mathfrak{C}$ is unital, where
$\mathfrak{e=}\left\{  e^{\oplus n}:n\in\mathbb{N}\right\}  $. If
$\mathfrak{C\cap-C=}\left\{  0\right\}  $, the quantum cone is called a
separated one. The operator system structures of ordered spaces were
investigated in \cite{PT} and \cite{PaulTT}. They can be treated as
quantizations of unital cones in a unital $\ast$-vector space. Tensor products
of operator systems were considered in \cite{Kav}. For the quotients,
exactness and nuclearity in the operator system category see \cite{Kav2}. The
matrix duality and quantum polars of quantum cones were investigated in
\cite{Dfaa}, \cite{Dhjm} and \cite{DCcones}. Based on duality of quantum
cones, the classification of operator system structures among the operator
space structures on a unital $\ast$-vector space was obtained in \cite{DSbM}
(see also \cite{DAJM}). It is proved that the operator system structures on a
unital $\ast$-vector space $\mathcal{X}$ with their unital quantum cones
$\mathfrak{C}$ are in bijection relation with the operator space structures on
$\mathcal{X}$ with their hermitian unit balls $\mathfrak{B}$; the latter means
that $\mathfrak{B}^{\ast}=\mathfrak{B}$ and $e\in\mathfrak{B}$. Thus there are
no operator column and row Hilbert systems as well as Haagerup tensor product
of operator systems in their direct proper senses. Nonetheless the operator
Hilbert space $H_{o}$ of Pisier turns out to be an operator system whose
matrix norm is equivalent to the orignal matrix norm of $H_{o}$. That is a key
missing object of the theory of operator systems, which plays an important
role in the separability problem mentioned above. Notice that in the finite
dimensional case the operator Hilbert system was constructed in \cite{NP} by
Ng and Paulsen.

The present paper is devoted to operator Hilbert systems and their morphisms.
First we describe the $\min$ and $\max$ quantizations of the related unital
cone $\mathfrak{c}$ of a unital Hilbert $\ast$-space $H$, and the related
state space of the cone. To be precise, fix a unital Hilbert $\ast$-space $H$
with its unit hermitian vector $e$, and define the $\sigma\left(
H,\overline{H}\right)  $-closed, unital cone $\mathfrak{c}_{e}=\left\{
\zeta\in H_{h}:\left\Vert \zeta\right\Vert \leq\sqrt{2}\left(  \zeta,e\right)
\right\}  $, where $\overline{H}$ is the conjugate Hilbert space,
$\sigma\left(  H,\overline{H}\right)  $ is the weak topology obtained by means
of the canonical duality $\left\langle \cdot,\cdot\right\rangle $ of the pair
$\left(  H,\overline{H}\right)  $. By a quantization of $\mathfrak{c}_{e}$ we
mean a weakly closed, separated, unital, quantum cone $\mathfrak{C\subseteq
}M\left(  H\right)  _{h}$ such that $\mathfrak{C\cap}H=\mathfrak{c}_{e}$. So
are the quantizations $\min\mathfrak{c}_{e}$ and $\max\mathfrak{c}_{e}$, and
$\max\mathfrak{c}_{e}\subseteq\mathfrak{C\subseteq}\min\mathfrak{c}_{e}$ for
every quantization $\mathfrak{C}$ of $\mathfrak{c}_{e}$. Using the matrix
duality $\left\langle \left\langle \cdot,\cdot\right\rangle \right\rangle $ of
the pair $\left(  M\left(  H\right)  ,M\left(  \overline{H}\right)  \right)  $
associated with $\left(  H,\overline{H}\right)  $, one can define the quantum
polar $\mathfrak{C}^{\boxdot}=\left\{  \overline{\eta}\in M\left(
\overline{H}\right)  _{h}:\left\langle \left\langle \mathfrak{C}%
,\overline{\eta}\right\rangle \right\rangle \geq0\right\}  $ to be a quantum
cone on $\overline{H}$. We have also the conjugate cone $\overline
{\mathfrak{c}_{e}}$ and conjugate quantum cone $\overline{\mathfrak{C}}$ on
$\overline{H}$. In Section \ref{secQHS}, we prove that the operator Hilbert
space $H_{o}$ is an operator system whose quantum cone $\mathfrak{C}_{o}$ is a
quantization of $\mathfrak{c}_{e}$ and it is self-dual in the sense of
$\mathfrak{C}_{o}^{\boxdot}=\overline{\mathfrak{C}}_{o}$, that is, $H_{o}$ is
a self-dual operator system. Moreover, $\left(  \max\overline{\mathfrak{c}%
_{e}}\right)  ^{\boxdot}=\min\mathfrak{c}_{e}$ and $\left(  \min
\overline{\mathfrak{c}_{e}}\right)  ^{\boxdot}=\max\mathfrak{c}_{e}$.

In Section \ref{sectionPMOH}, we investigate the positive maps between
operator Hilbert systems. Since the $\min$-operator system structure on a
unital $\ast$-vector space is given by the standard cone\ $C\left(  X\right)
_{+}$ of the abelian $C^{\ast}$-algebra $C\left(  X\right)  $ of all complex
continuous functions on a compact Hausdorff topological space $X$ (see
\cite{PT}), the characterization of separable morphisms $C\left(  X\right)
\rightarrow H$ plays a key role in the solution of the $\min$-$\max$ matrix
positive mapping problem confirmed above. Fix a hermitian basis $F$ for $H$
containing $e$, and a probability measure $\mu$ on $X$. A family of real
valued Borel functions $k=\left\{  k_{f}:f\in F\right\}  \subseteq
\operatorname{ball}L^{\infty}\left(  X,\mu\right)  $ with $k_{e}=1$ is said to
be \textit{an }$H$-\textit{support on }$X$\textit{ }if $k_{f}\perp k_{e}$,
$f\neq e\ $and $\sum_{f\neq e}\left(  v,k_{f}\right)  ^{2}\leq\left(
v,k_{e}\right)  ^{2}\ $in $L^{2}\left(  X,\mu\right)  \ $for all $v\in
C\left(  X\right)  _{+}$. If $k$ is an $H$-support on $X$ then $T:C\left(
X\right)  \rightarrow H$, $Tv=\sum_{f}\left(  v,k_{f}\right)  f$ is a unital
positive mapping, that is, $T\left(  1\right)  =e$ and $T\left(  C\left(
X\right)  _{+}\right)  \subseteq\mathfrak{c}_{e}$. There is a bijection
between unital positive maps $T:C\left(  X\right)  \rightarrow H$ and
$H$-supports $k$ on $X$ (see below Theorem \ref{corCHe1}). In this case, $T$
is an absolutely summable mapping, which admits a unique bounded linear
extension $T_{k}:L^{2}\left(  X,\mu\right)  \rightarrow H$ being a unital
positive mapping of operator Hilbert systems. Moreover, $T_{k}$ coincides with
the Pietsch extension of an absolutely summable mapping $T$ \cite{AlP}. The
present theory can be treated as an ordered version of Pietsch factorizations
for absolutely summable maps.

Recall that a positive mapping $\phi:\mathcal{V\rightarrow W}$ of operator
systems is called\textit{ separable} if $\phi=\sum_{l}p_{l}\odot q_{l}$ is a
sum of $1$-rank operators made from positive functionals $q_{l}$ on
$\mathcal{V}$ and positive elements $p_{l}$ in $\mathcal{W}$ in the sense of
$\phi\left(  v\right)  =\lim_{k}\sum_{l=1}^{k}q_{l}\left(  v\right)  p_{l}$ in
$\mathcal{W}$ for every $v\in\mathcal{V}$. We obtain the following
characterization of the separable morphisms from $C\left(  X\right)  $ to $H$.
A morphism $C\left(  X\right)  \rightarrow H$ is separable iff its support $k$
on $X$ is maximal, in the sense of $\sum_{f\neq e}k_{f}^{2}\leq1$ in
$L^{\infty}\left(  X,\mu\right)  $. Thus the separable morphisms $C\left(
X\right)  \rightarrow H$ are in bijection relation with the maximal supports
on $X$. In this case, all extensions $T_{k}:L^{2}\left(  X,\mu\right)
\rightarrow H$ are Hilbert-Schmidt operators, whereas the original separable
morphisms $T:C\left(  X\right)  \rightarrow H$ are nuclear operators.

It is known \cite{PaulTT} that a linear mapping $\phi:M_{n}\rightarrow\left(
M_{m},\max M_{m}^{+}\right)  $ is matrix positive iff $\phi:M_{n}\rightarrow
M_{m}$ is separable. A problem of Paulsen-Todorov-Tomforde from \cite{PaulTT}
asks that whether the latter statement characterizes the matrix positive maps
$\phi:\mathcal{V\rightarrow}\left(  \mathcal{W},\max\mathcal{W}_{+}\right)  $
of operator systems. We provide an example of a morphism between operator
Hilbert systems which is not separable. Namely, let $H$ be an infinite
dimensional Hilbert space with its hermitian basis $F$. Fix $e,u\in F$, which
in turn define the unital cones $\mathfrak{c}_{e}$ and $\mathfrak{c}_{u}$ in
$H$, respectively. If $T\in\mathcal{B}\left(  H\right)  $ is a unitary given
by $T=u\odot\overline{e}+e\odot\overline{u}+\sum_{f\neq u,e}f\odot\overline
{f}$, then the matrix positive mapping $T:\left(  H,\max\mathfrak{c}%
_{u}\right)  \rightarrow\left(  H,\max\mathfrak{c}_{e}\right)  $ given by $T$
is not separable.

Finally, we consider the finite dimensional case, and prove that the unital
cone $\mathfrak{c}_{e}$ of the $2$-dimensional Hilbert space $\ell_{2}^{2}$
admits only one quantization, that is, $\min\mathfrak{c}_{e}=\max
\mathfrak{c}_{e}$. As an application to quantum information theory we prove
the following key property of the operator Hilbert systems. Let $H$ be an
operator Hilbert system, $\mathcal{M}$ either a finite-dimensional von Neumann
algebra or another operator Hilbert system, and let $\varphi:H\rightarrow
\mathcal{M}$ be a linear mapping. Then $\varphi$ is an entanglement breaking
mapping iff $\varphi^{\ast}:\mathcal{M}^{\ast}\rightarrow\left(  \overline
{H},\max\overline{\mathfrak{c}_{e}}\right)  $ is matrix positive. Similarly,
$\varphi^{\ast}:\overline{H}\rightarrow\mathcal{M}^{\ast}$ is an entanglement
breaking mapping iff $\varphi:\mathcal{M}\rightarrow\left(  H,\max
\mathfrak{c}_{e}\right)  $ is matrix positive.

\section{\textbf{Preliminaries}\label{secPr}}

In this section we introduce some preliminary notions and results. The vector
space of all $m\times n$-matrices $v=\left[  v_{ij}\right]  _{i,j}$ over a
complex vector space $V$ is denoted by $M_{m,n}\left(  V\right)  $, and we set
$M_{m}\left(  V\right)  =M_{m,m}\left(  V\right)  $ and $M_{m,n}%
=M_{m,n}\left(  \mathbb{C}\right)  $. Further, $M\left(  V\right)  $
(respectively, $M$) denotes the vector space of all infinite (respectively,
scalar) matrices over $V$ with only finitely many non-zero entries. A linear
mapping $\varphi:V\rightarrow W$ admits the canonical linear extensions
$\varphi^{\left(  n\right)  }:M_{n}\left(  V\right)  \rightarrow M_{n}\left(
W\right)  $ (respectively, $\varphi^{\left(  \infty\right)  }:M\left(
V\right)  \rightarrow M\left(  W\right)  $) over all matrix spaces defined as
$\varphi^{\left(  n\right)  }\left(  \left[  v_{ij}\right]  _{i,j}\right)
=\left[  \varphi\left(  v_{ij}\right)  \right]  _{i,j}$ (respectively,
$\varphi^{\left(  \infty\right)  }|M_{n}\left(  V\right)  =\varphi^{\left(
n\right)  }$). Notice that $\varphi^{\left(  \infty\right)  }$ preserves the
standard matrix operations.

\subsection{The quantum duality\label{subsecDD}}

By a quantum set $\mathfrak{B}$\textit{ on }$V$ we mean a collection
$\mathfrak{B=}\left(  \mathfrak{B}_{n}\right)  $ of subsets $\mathfrak{B}%
_{n}\subseteq M_{n}\left(  V\right)  $, $n\geq1$. Sometimes we write
$\mathfrak{B}\cap M_{n}\left(  V\right)  $ instead of $\mathfrak{B}_{n}$. If
$\mathfrak{B}$ and $\mathfrak{C}$ are quantum sets on $V$ then we put
$\mathfrak{B\subseteq C}$ whenever $\mathfrak{B}_{n}\subseteq\mathfrak{C}_{n}%
$, $n\geq1$. In a similar way, all set-theoretic operations and basic
algebraic operations can be defined over all quantum sets on $V$. The
Minkowski functional of an absorbent (in $M\left(  V\right)  $) absolutely
matrix convex set (see \cite{EW}) is called \textit{a matrix seminorm }on $V$.
A polynormed (or locally convex) topology defined by a separating family of
matrix seminorms is called \textit{a} \textit{quantum topology}, and the
vector space $V$ equipped with a quantum topology is called \textit{a quantum
space. }Thus a quantum topology $\mathfrak{t}$ on $V$ can be identified with a
filter base of absorbent, absolutely matrix convex sets on $V$ such that
$\left\{  \varepsilon\mathfrak{U:U\in t},\varepsilon>0\right\}  $ is a
neighborhood filter base of the origin with respect to the relevant polynormed
topology in $M\left(  V\right)  $. In particular, it inherits a polynormed
topology $\mathfrak{t}|M_{n}\left(  V\right)  $ in each $M_{n}\left(
V\right)  $. Note that $\mathfrak{t}|M_{n}\left(  V\right)  =\left(
\mathfrak{t}|V\right)  ^{n^{2}}$ \cite{DNew} (see also \cite{Djfa}), where
$\left(  \mathfrak{t}|V\right)  ^{n^{2}}$ indicates the direct product
topology in $V^{n^{2}}$ generated by $\mathfrak{t}|V$. Conversely, each
polynormed topology $t$ in $V$ is a trace of a certain quantum topology
$\mathfrak{t}$ in $M\left(  V\right)  $ called its \textit{quantization}, that
is, $t=\mathfrak{t}|V$. All these quantizations are running within $\min$ and
$\max$ quantizations \cite{EW}, that is, if $\mathfrak{t}$ is a quantum
topology on $V$ with $t=\mathfrak{t}|V$, then $\min t\subseteq
\mathfrak{t\subseteq}\max t$. A quantum space whose quantum topology is given
by a matrix norm is a called \textit{an operator }(or \textit{quantum normed})
\textit{space.} By a morphism between quantum spaces we mean a \textit{matrix
continuous linear mapping. }A linear mapping $\varphi:V\rightarrow W$ between
quantum spaces is matrix continuous iff $\varphi^{\left(  \infty\right)
}:M\left(  V\right)  \rightarrow M\left(  W\right)  $ is a continuous linear
mapping of the relevant polynormed spaces.

Let $\left(  V,W\right)  $ be a dual pair of vector spaces with the pairing
$\left\langle \cdot,\cdot\right\rangle :V\times W\rightarrow\mathbb{C}$. This
pairing defines a \textit{quantum }(or \textit{matrix})\textit{ pairing}
\[
\left\langle \left\langle \cdot,\cdot\right\rangle \right\rangle :M_{m}\left(
V\right)  \times M_{n}\left(  W\right)  \rightarrow M_{mn},\quad\left\langle
\left\langle v,w\right\rangle \right\rangle =\left[  \left\langle
v_{ij},w_{st}\right\rangle \right]  _{\left(  i,s\right)  ,\left(  j,t\right)
}=w^{\left(  m\right)  }\left(  v\right)  =v^{\left(  n\right)  }\left(
w\right)  ,
\]
where $v=\left[  v_{ij}\right]  _{i,j}\in M_{m}\left(  V\right)  $, $w=\left[
w_{st}\right]  _{s,t}\in M_{n}\left(  W\right)  $, which are identified with
the canonical linear maps $v:W\rightarrow M_{m}$, $v\left(  y\right)  =\left[
\left\langle v_{ij},y\right\rangle \right]  _{i,j}$, and $w:V\rightarrow
M_{n}$, $w\left(  x\right)  =\left[  \left\langle x,w_{st}\right\rangle
\right]  _{s,t}$, respectively. The same size matrix spaces $M_{n}\left(
V\right)  $ and $M_{n}\left(  W\right)  $ are also in the canonical duality
determined by \textit{the scalar pairing}%
\[
\left\langle \cdot,\cdot\right\rangle :M_{n}\left(  V\right)  \times
M_{n}\left(  W\right)  \rightarrow\mathbb{C},\quad\left\langle
v,w\right\rangle =\sum_{i,j}\left\langle v_{ij},w_{ij}\right\rangle .
\]
So, we have the weak and Mackey topologies $\sigma\left(  M_{n}\left(
V\right)  ,M_{n}\left(  W\right)  \right)  $ and $\varkappa\left(
M_{n}\left(  V\right)  ,M_{n}\left(  W\right)  \right)  $, respectively.
Actually, $\sigma\left(  M_{n}\left(  V\right)  ,M_{n}\left(  W\right)
\right)  =\sigma\left(  V,W\right)  ^{n^{2}}$ and $\varkappa\left(
M_{n}\left(  V\right)  ,M_{n}\left(  W\right)  \right)  =\varkappa\left(
V,W\right)  ^{n^{2}}$ (see \cite[4.4.2, 4.4.3]{Sh} and \cite{DTrans}). If
$V=W=\mathbb{C}$ then the scalar pairing $\left\langle \cdot,\cdot
\right\rangle :M_{n}\times M_{n}\rightarrow\mathbb{C}$ is given by
$\left\langle a,b\right\rangle =\sum_{i,j}a_{ij}b_{ji}=\tau\left(
ab^{\operatorname{t}}\right)  =\tau\left(  a^{\operatorname{t}}b\right)  $,
where $\tau$ is the trace on $M_{n}$ and $a^{\operatorname{t}}$ (or
$b^{\operatorname{t}}$) indicates to the transpose matrix. This duality
defines the trace class norm $\left\Vert a\right\Vert _{1}=\tau\left(
\left\vert a\right\vert \right)  =\sup\left\{  \left\vert \left\langle
a,b\right\rangle \right\vert :b\in\operatorname{ball}M_{n}\right\}  $, $a\in
M_{n}$. The space $M_{n}$ equipped with the norm $\left\Vert \cdot\right\Vert
_{1}$ is denoted by $T_{n}$, which is the predual of the von Neumann algebra
$M_{n}$. The following assertion was proved in \cite{DInt}.

\begin{theorem}
\label{tdos}Let $\left(  V,W\right)  $ be a dual pair. The weak topology
$\sigma\left(  V,W\right)  $ admits only one quantization $\mathfrak{s}\left(
V,W\right)  $ called \textit{the weak quantum topology of the dual pair
}$\left(  V,W\right)  $.
\end{theorem}

The quantum topology $\mathfrak{s}\left(  V,W\right)  $ has the defining
family $\left\{  p_{w}:w\in M\left(  W\right)  \right\}  $ of matrix
seminorms, where $p_{w}\left(  v\right)  =\left\Vert \left\langle \left\langle
v,w\right\rangle \right\rangle \right\Vert $ (see \cite{DInt}). Thus
$\min\sigma\left(  V,W\right)  =\max\sigma\left(  V,W\right)  =\mathfrak{s}%
\left(  V,W\right)  $. Moreover, $\mathfrak{s}\left(  V,W\right)  |V^{n^{2}%
}=\left(  \mathfrak{s}\left(  V,W\right)  |V\right)  ^{n^{2}}=\sigma\left(
V,W\right)  ^{n^{2}}=\sigma\left(  M_{n}\left(  V\right)  ,M_{n}\left(
W\right)  \right)  $ for all $n$ (see \cite{DNew}, \cite{DInt}). A quantum
topology $\mathfrak{t}$\ on $V$ is said to be compatible with the duality
$\left(  V,W\right)  $ if $\left(  V,\mathfrak{t}|V\right)  ^{\prime}=W$. In
this case, $\mathfrak{t}$ has a neighborhood filter base of the origin, which
consists of $\mathfrak{s}\left(  V,W\right)  $-closed, absorbent, absolutely
matrix convex sets in $M\left(  V\right)  $. Moreover, $\mathfrak{s}\left(
V,W\right)  \preceq\mathfrak{t\preceq r}\left(  V,W\right)  $ \cite[Lemma
5.1]{DNew}, where $\mathfrak{r}\left(  V,W\right)  =\max\varkappa\left(
V,W\right)  $.

Given a quantum set $\mathfrak{B}$ in $M\left(  V\right)  $ we have its weak
closure $\mathfrak{B}^{-}$ with respect to the weak quantum topology
$\mathfrak{s}\left(  V,W\right)  $, and the \textit{absolute matrix }(or
\textit{operator})\textit{ polar} $\mathfrak{B}^{\odot}$ \textit{in }$M\left(
W\right)  $ defined as the quantum set $\mathfrak{B}^{\odot}=\left\{  w\in
M\left(  W\right)  :\sup\left\Vert \left\langle \left\langle \mathfrak{B}%
,w\right\rangle \right\rangle \right\Vert \leq1\right\}  $. One can easily
verify that $\mathfrak{B}^{\odot}$ is $\mathfrak{s}\left(  W,V\right)
$-closed, absolutely matrix convex set in $M\left(  W\right)  $ (see
\cite{DNew}). Similarly, it is defined the absolute matrix polar
$\mathfrak{M}^{\odot}\subseteq M\left(  V\right)  $ of a quantum set
$\mathfrak{M}\subseteq M\left(  W\right)  $. If $\mathfrak{t}$ is a quantum
topology on $V$ compatible with the duality $\left(  V,W\right)  $ then
$\mathfrak{t}^{\odot}=\left\{  n\mathfrak{B}^{\odot}:\mathfrak{B\in t}%
,n\in\mathbb{N}\right\}  $ is a quantum bornology base, which consists of
$\mathfrak{s}\left(  W,V\right)  $-compact quantum sets on $W$ (see
\cite{DT}). The following quantum version of the classical bipolar theorem was
proved in \cite{EW} (see also \cite{EWin}) by Effros and Webster.

\begin{theorem}
\label{BiT}Let $\left(  V,W\right)  $ be a dual pair and let $\mathfrak{B}$ be
an absolutely matrix convex set in $M\left(  V\right)  $. Then $\mathfrak{B}%
^{\odot\odot}=\mathfrak{B}^{-}$, where $\mathfrak{B}^{-}$ is the
$\mathfrak{s}\left(  V,W\right)  $-closure of $\mathfrak{B}$.
\end{theorem}

In \cite{Dfaa} we found a new proof of the Bipolar Theorem \ref{BiT} based on
the duality theory of quantum cones. Thus the method of quantum cones is an
alternative tool to investigate quantum spaces.

\subsection{Involution and quantum cones\label{subsecInv}}

By \textit{an involution} on a vector space $\mathcal{X}$ we mean a conjugate
linear (or $\ast$-linear) mapping $x\mapsto x^{\ast}$ on $\mathcal{X}$ such
that $x^{\ast\ast}=x$ for all $x\in\mathcal{X}$. A vector space equipped with
an involution is called \textit{a }$\ast$\textit{-vector space}. An element
$x\in\mathcal{X}$ is called \textit{hermitian }if $x^{\ast}=x$. The set of all
hermitian elements is denoted by $\mathcal{X}_{h}$, which is a real linear
subspace in $\mathcal{X}$. It is easy to see that each $x\in\mathcal{X}$ has a
unique decomposition $x=\operatorname{Re}\left(  x\right)  +i\operatorname{Im}%
\left(  x\right)  $ with hermitians $\operatorname{Re}\left(  x\right)  $ and
$\operatorname{Im}\left(  x\right)  $. Now assume that $\mathcal{X}$ is a
$\ast$-vector space and $\left(  \mathcal{X},\mathcal{Y}\right)  $ is a dual
pair such that the involution on $\mathcal{X}$ is $\sigma\left(
\mathcal{X},\mathcal{Y}\right)  $-continuous. Then $\mathcal{Y}$ possesses the
canonical involution $y\mapsto y^{\ast}$, $\left\langle x,y^{\ast
}\right\rangle =\left\langle x^{\ast},y\right\rangle ^{\ast}$. Indeed, the
linear functional $y^{\ast}$ being a composition of weakly continuous mappings
$x\mapsto x^{\ast}$ and $x\mapsto\left\langle x,y\right\rangle $ turns out to
be weakly continuous. Hence $y^{\ast}\in\mathcal{Y}$. In this case $\left(
\mathcal{X},\mathcal{Y}\right)  $ is called \textit{a dual }$\ast
$\textit{-pair. }The involution on $\mathcal{X}$ is naturally extended to an
involution over the matrix space $M\left(  \mathcal{X}\right)  $. Namely, if
$x=\left[  x_{ij}\right]  _{i,j}\in M_{n}\left(  \mathcal{X}\right)  $ then we
set $x^{\ast}=\left[  x_{ji}^{\ast}\right]  _{i,j}\in M_{n}\left(
\mathcal{X}\right)  $, whereas $x^{\operatorname{t}}=\left[  x_{ji}\right]
_{i,j}$ indicates to the transpose of $x$. Thus $M\left(  \mathcal{X}\right)
$ turns out to be a $\ast$-vector space too. Note that $\left\langle
\left\langle x,y^{\ast}\right\rangle \right\rangle =\left\langle \left\langle
x^{\ast},y\right\rangle \right\rangle ^{\ast}$ for all $x\in M\left(
\mathcal{X}\right)  $ and $y\in M\left(  \mathcal{Y}\right)  $ (see
\cite[Lemma 4.1]{Dhjm}). If $a,b\in M$ and $y\in M\left(  \mathcal{Y}\right)
$ then $\left(  ayb\right)  ^{\ast}=b^{\ast}y^{\ast}a^{\ast}$. Indeed,%
\begin{align*}
\left\langle \left\langle x,\left(  ayb\right)  ^{\ast}\right\rangle
\right\rangle  &  =\left\langle \left\langle x^{\ast},ayb\right\rangle
\right\rangle ^{\ast}=\left(  \left(  I\otimes a\right)  \left\langle
\left\langle x^{\ast},y\right\rangle \right\rangle \left(  I\otimes b\right)
\right)  ^{\ast}=\left(  I\otimes b^{\ast}\right)  \left\langle \left\langle
x^{\ast},y\right\rangle \right\rangle ^{\ast}\left(  I\otimes a^{\ast}\right)
\\
&  =\left(  I\otimes b^{\ast}\right)  \left\langle \left\langle x,y^{\ast
}\right\rangle \right\rangle \left(  I\otimes a^{\ast}\right)  =\left\langle
\left\langle x,b^{\ast}y^{\ast}a^{\ast}\right\rangle \right\rangle
\end{align*}
for all $x\in M\left(  \mathcal{X}\right)  $. Further, if $x\in M_{n}\left(
\mathcal{X}\right)  $ and $y\in M_{n}\left(  \mathcal{Y}\right)  $ then
$\left\langle x,y^{\ast}\right\rangle =\left\langle x^{\ast},y\right\rangle
^{\ast}$ (see \cite{Dhjm} for the details), that is, $\left(  M_{n}\left(
\mathcal{X}\right)  ,M_{n}\left(  \mathcal{Y}\right)  \right)  $ equipped with
the scalar pairing is a $\ast$-dual pair as well.

Let $\mathcal{X}$ be a $\ast$-vector space. Then $M\left(  \mathcal{X}\right)
_{h}=\left\{  x\in M\left(  \mathcal{X}\right)  :x^{\ast}=x\right\}  $ is a
real subspace of $M\left(  \mathcal{X}\right)  $. If $\mathfrak{B\subseteq
}M\left(  \mathcal{X}\right)  _{h}$ is a quantum set then we say that
$\mathfrak{B}$ is \textit{a hermitian quantum set on }$\mathcal{X}$. A
hermitian quantum set $\mathfrak{C}$ over $\mathcal{X}$ is said to be
\textit{a quantum cone on} $\mathcal{X}$ if $\mathfrak{C+C\subseteq C}$ and
$a^{\ast}\mathfrak{C}a\subseteq\mathfrak{C}$ for all $a\in M$. A quantum cone
$\mathfrak{C}$ is said to be a \textit{quantum }$\ast$\textit{-cone} if
$M\left(  \mathcal{X}\right)  _{h}=\mathfrak{C-C}$. A quantum cone
$\mathfrak{C}$ on $\mathcal{X}$ is called \textit{a separated quantum cone on}
$\mathcal{X}$ if $\mathfrak{C\cap-C=}\left\{  0\right\}  $. Any $C^{\ast}%
$-algebra $A$ possesses the quantum $\ast$-cone $M\left(  A\right)
_{+}=\left(  M_{n}\left(  A\right)  _{+}\right)  $, which is the set of all
positive elements in $M\left(  A\right)  $. Obviously, any intersection of
quantum cones is a quantum cone. In particular, \textit{the quantum cone}
$\mathfrak{U}^{c}$ \textit{generated by a quantum set }$\mathfrak{U}$ is well defined.

Now let $\left(  \mathcal{X},\mathcal{Y}\right)  $ be a dual $\ast$-pair. If
$\mathfrak{C}$ is a quantum set on $\mathcal{X}$ then its \textit{quantum
polar} $\mathfrak{C}^{\boxdot}$ \textit{in }$M\left(  \mathcal{Y}\right)  $ is
defined as the quantum set $\mathfrak{C}^{\boxdot}=\left\{  y\in M\left(
\mathcal{Y}\right)  _{h}:\left\langle \left\langle \mathfrak{C},y\right\rangle
\right\rangle \geq0\right\}  $. The latter is $\mathfrak{s}\left(
\mathcal{Y},\mathcal{X}\right)  $-closed quantum cone on $\mathcal{Y}$. If
$\mathfrak{C}$ is a quantum $\ast$-cone on $\mathcal{X}$ then $\mathfrak{C}%
^{\boxdot}=\left\{  y\in M\left(  \mathcal{Y}\right)  :\left\langle
\left\langle \mathfrak{C},y\right\rangle \right\rangle \geq0\right\}  $ and it
is a separated quantum cone on $\mathcal{Y}$ (see \cite{Dhjm}). The following
bipolar theorem for quantum cones was proved in \cite{Dhjm}.

\begin{theorem}
\label{t1}Let $\left(  \mathcal{X},\mathcal{Y}\right)  $ be a dual $\ast$-pair
and let $\mathfrak{C}$ be a $\mathfrak{s}\left(  \mathcal{X},\mathcal{Y}%
\right)  $-closed, quantum cone on $\mathcal{X}$. Then $\mathfrak{C=C}%
^{\boxdot\boxdot}$. In this case, for $y\in M_{n}\left(  \mathcal{Y}\right)
_{h}$ we have $y\in\mathfrak{C}^{\boxdot}$ iff $\left\langle \mathfrak{C}%
_{n},y\right\rangle \geq0$.
\end{theorem}

Let $\left(  \mathcal{X}_{1},\mathcal{Y}_{1}\right)  $ and $\left(
\mathcal{X}_{2},\mathcal{Y}_{2}\right)  $ be dual $\ast$-pairs, and let
$\varphi:\mathcal{X}_{1}\rightarrow\mathcal{X}_{2}$ be a weakly continuous
$\ast$-linear mapping with its algebraic dual mapping $\varphi^{\ast}$, that
is, $\left\langle \varphi\left(  x_{1}\right)  ,y_{2}\right\rangle
=\left\langle x_{1},\varphi^{\ast}\left(  y_{2}\right)  \right\rangle $ for
all $x\in\mathcal{X}_{1}$ and $y_{2}\in\mathcal{Y}_{2}$. Then
\begin{equation}
\varphi^{\ast}\left(  \mathcal{Y}_{2}\right)  \subseteq\mathcal{Y}_{1}%
\quad\text{and\quad}\left\langle \left\langle \varphi^{\left(  \infty\right)
}\left(  x_{1}\right)  ,y_{2}\right\rangle \right\rangle =\left\langle
\left\langle x_{1},\left(  \varphi^{\ast}\right)  ^{\left(  \infty\right)
}\left(  y_{2}\right)  \right\rangle \right\rangle \label{conj1}%
\end{equation}
for all $x_{1}\in M\left(  \mathcal{X}_{1}\right)  $ and $y_{2}\in M\left(
\mathcal{Y}_{2}\right)  $. Indeed, $\varphi^{\ast}\left(  y_{2}\right)  $
being a composition of the weakly continuous mapping $\varphi$ and
$\sigma\left(  \mathcal{X}_{2},\mathcal{Y}_{2}\right)  $-continuous functional
$y_{2}$ turns out to be $\sigma\left(  \mathcal{X}_{1},\mathcal{Y}_{1}\right)
$-continuous. Therefore $\varphi^{\ast}\left(  y_{2}\right)  \in
\mathcal{Y}_{1}$. Moreover,
\[
\left\langle \left\langle \varphi^{\left(  \infty\right)  }\left(
x_{1}\right)  ,y_{2}\right\rangle \right\rangle =\left[  \left\langle
\varphi\left(  x_{1,i,j}\right)  ,y_{2,s,t}\right\rangle \right]  _{\left(
i,s\right)  ,\left(  j,t\right)  }=\left[  \left\langle x_{1,i,j}%
,\varphi^{\ast}\left(  y_{2,s,t}\right)  \right\rangle \right]  _{\left(
i,s\right)  ,\left(  j,t\right)  }=\left\langle \left\langle x_{1},\left(
\varphi^{\ast}\right)  ^{\left(  \infty\right)  }\left(  y_{2}\right)
\right\rangle \right\rangle .
\]
Notice that $\varphi^{\ast}:\mathcal{Y}_{2}\rightarrow\mathcal{Y}_{1}$ is a
(weakly continuous) $\ast$-linear mapping, for $\left\langle x_{1}%
,\varphi^{\ast}\left(  y_{2}^{\ast}\right)  \right\rangle =\left\langle
\varphi\left(  x_{1}\right)  ,y_{2}^{\ast}\right\rangle =\left\langle
\varphi\left(  x_{1}\right)  ^{\ast},y_{2}\right\rangle ^{\ast}=\left\langle
\varphi\left(  x_{1}^{\ast}\right)  ,y_{2}\right\rangle ^{\ast}=\left\langle
x_{1}^{\ast},\varphi^{\ast}\left(  y_{2}\right)  \right\rangle ^{\ast
}=\left\langle x_{1},\varphi^{\ast}\left(  y_{2}\right)  ^{\ast}\right\rangle
$, $x_{1}\in\mathcal{X}_{1}$. Note also that $\varphi^{\left(  \infty\right)
}:M\left(  \mathcal{X}_{1}\right)  \rightarrow M\left(  \mathcal{X}%
_{2}\right)  $ is a $\ast$-linear mapping. Indeed, $\varphi^{\left(
\infty\right)  }\left(  x^{\ast}\right)  =\varphi^{\left(  \infty\right)
}\left(  \left[  x_{ji}^{\ast}\right]  _{i,j}\right)  =\left[  \varphi\left(
x_{ji}^{\ast}\right)  \right]  _{i,j}=\left[  \varphi\left(  x_{ji}\right)
^{\ast}\right]  _{i,j}=\left[  \varphi\left(  x_{ij}\right)  \right]
_{i,j}^{\ast}=\varphi^{\left(  \infty\right)  }\left(  x\right)  ^{\ast}$ for
all $x\in M\left(  \mathcal{X}_{1}\right)  $.

\begin{lemma}
\label{lDualM}Let $\left(  \mathcal{X}_{1},\mathcal{Y}_{1}\right)  $ and
$\left(  \mathcal{X}_{2},\mathcal{Y}_{2}\right)  $ be dual $\ast$-pairs,
$\mathfrak{C}_{1}$ and $\mathfrak{C}_{2}$ quantum sets on $\mathcal{X}_{1}$
and $\mathcal{X}_{2}$, respectively, and let $\varphi:\mathcal{X}%
_{1}\rightarrow\mathcal{X}_{2}$ be a weakly continuous $\ast$-linear mapping
such that $\varphi^{\left(  \infty\right)  }\left(  \mathfrak{C}_{1}\right)
\subseteq\mathfrak{C}_{2}$. Then $\left(  \varphi^{\ast}\right)  ^{\left(
\infty\right)  }\left(  y^{\ast}\right)  =\left(  \varphi^{\ast}\right)
^{\left(  \infty\right)  }\left(  y\right)  ^{\ast}$ for all $y\in M\left(
\mathcal{Y}_{2}\right)  $, and $\left(  \varphi^{\ast}\right)  ^{\left(
\infty\right)  }\left(  \mathfrak{C}_{2}^{\boxdot}\right)  \subseteq
\mathfrak{C}_{1}^{\boxdot}$. Similarly, if $\left(  \varphi^{\ast}\right)
^{\left(  \infty\right)  }\left(  \mathfrak{K}_{2}\right)  \subseteq
\mathfrak{K}_{1}$ for quantum sets $\mathfrak{K}_{1}$ and $\mathfrak{K}_{2}$
on $\mathcal{Y}_{1}$ and $\mathcal{Y}_{2}$, respectively, then $\varphi
^{\left(  \infty\right)  }\left(  \mathfrak{K}_{1}^{\boxdot}\right)
\subseteq\mathfrak{K}_{2}^{\boxdot}$.
\end{lemma}

\begin{proof}
Take $y\in M\left(  \mathcal{Y}_{2}\right)  $. For every $x\in M\left(
\mathcal{X}_{1}\right)  $ we have
\begin{align*}
\left\langle \left\langle x,\left(  \varphi^{\ast}\right)  ^{\left(
\infty\right)  }\left(  y\right)  ^{\ast}\right\rangle \right\rangle  &
=\left\langle \left\langle x^{\ast},\left(  \varphi^{\ast}\right)  ^{\left(
\infty\right)  }\left(  y\right)  \right\rangle \right\rangle ^{\ast
}=\left\langle \left\langle \varphi^{\left(  \infty\right)  }\left(  x^{\ast
}\right)  ,y\right\rangle \right\rangle ^{\ast}=\left\langle \left\langle
\varphi^{\left(  \infty\right)  }\left(  x\right)  ^{\ast},y\right\rangle
\right\rangle ^{\ast}\\
&  =\left\langle \left\langle \varphi^{\left(  \infty\right)  }\left(
x\right)  ,y^{\ast}\right\rangle \right\rangle =\left\langle \left\langle
x,\left(  \varphi^{\ast}\right)  ^{\left(  \infty\right)  }\left(  y^{\ast
}\right)  \right\rangle \right\rangle
\end{align*}
thanks to (\ref{conj1}). Hence $\left(  \varphi^{\ast}\right)  ^{\left(
\infty\right)  }\left(  y\right)  ^{\ast}=\left(  \varphi^{\ast}\right)
^{\left(  \infty\right)  }\left(  y^{\ast}\right)  $. Finally, if
$y\in\mathfrak{C}_{2}^{\boxdot}$ then $\left(  \varphi^{\ast}\right)
^{\left(  \infty\right)  }\left(  y\right)  \in M\left(  \mathcal{Y}\right)
_{h}$ and $\left\langle \left\langle \mathfrak{C}_{1},\left(  \varphi^{\ast
}\right)  ^{\left(  \infty\right)  }\left(  y\right)  \right\rangle
\right\rangle =\left\langle \left\langle \varphi^{\left(  \infty\right)
}\left(  \mathfrak{C}_{1}\right)  ,y\right\rangle \right\rangle \subseteq
\left\langle \left\langle \mathfrak{C}_{2},y\right\rangle \right\rangle \geq
0$, which means that $\left(  \varphi^{\ast}\right)  ^{\left(  \infty\right)
}\left(  y\right)  \in\mathfrak{C}_{1}^{\boxdot}$, that is, $\left(
\varphi^{\ast}\right)  ^{\left(  \infty\right)  }\left(  \mathfrak{C}%
_{2}^{\boxdot}\right)  \subseteq\mathfrak{C}_{1}^{\boxdot}$. The rest follows
from the symmetry and (\ref{conj1}).
\end{proof}

\subsection{The \textbf{unital quantum cones\label{subsecUQC}}}

Let $\mathcal{X}$ be a $\ast$-vector space with its fixed hermitian element
$e$. We say that $\left(  \mathcal{X},e\right)  $ or just $\mathcal{X}$ is
\textit{a unital space. }The quantum set $\left(  \left\{  e_{n}\right\}
\right)  $ on $\mathcal{X}$ is denoted by $\mathfrak{e}$, where $e_{n}%
=e^{\oplus n}\in M_{n}\left(  \mathcal{X}\right)  _{h}$. A quantum cone
$\mathfrak{C}$ on the unital space $\left(  \mathcal{X},e\right)  $ is said to
be \textit{a unital quantum cone }if $\mathfrak{C-e}$ is absorbent in
$M\left(  \mathcal{X}\right)  _{h}$. Note that $\mathfrak{e\subseteq C}$ and
$\mathfrak{C}$ turns out to be a quantum $\ast$-cone if $\mathfrak{C}$ is a
unital quantum cone. Moreover, $\mathfrak{C-e}$ is a matrix convex set in
$M\left(  \mathcal{X}\right)  $ containing the origin (see \cite{Dhjm} for the
details). The quantum set $\cap_{r>0}r\left(  \mathfrak{C-}%
\mathcal{\mathfrak{e}}\right)  $ is called the algebraic closure of
$\mathfrak{C}$ and it is denoted by $\mathfrak{C}^{-}$. Note that
$\mathfrak{C\subseteq C}^{-}$ whenever $\mathfrak{e\subseteq C}$. We say that
$\mathfrak{C}$ is \textit{a closed }(\textit{or an Archimedian})\textit{
quantum cone }if it coincides with its algebraic closure, that is,
$\mathfrak{C=C}^{-}$. Notice that $\mathfrak{C}^{-}$ is smaller than any
(polynormed) topological closure of $\mathfrak{C}$. By analogy, a cone
$\mathfrak{c}$ in $\mathcal{X}$ is said to be unital if $\mathfrak{c-}e$ is
absorbent in $\mathcal{X}_{h}$, and it is closed if $\mathfrak{c}%
^{-}=\mathfrak{c}$, where $\mathfrak{c}^{-}=\cap_{r>0}r\left(  \mathfrak{c-}%
e\right)  $ is the algebraic closure of $\mathfrak{c}$. In particular,
$\mathcal{X}_{h}=\mathfrak{c-c}$ and $e\in\mathfrak{c}$.

\begin{lemma}
\label{lemPTT}Let $\mathcal{X}$ be a unital $\ast$-vector space with its unit
$e$, and let $\mathfrak{C}$ be a quantum cone on $\mathcal{X}$. If
$\mathfrak{C}_{m}$ is unital in the sense that $\mathfrak{C}_{m}-e^{\oplus m}$
is absorbent in $M_{m}\left(  \mathcal{X}\right)  _{h}$ for some $m$ then
$\mathfrak{C}$ is a unital quantum cone. In particular, if $\mathfrak{c}$ is a
unital cone in $\mathcal{X}$ then $\mathfrak{c}^{c}$ is a unital quantum cone
on $\mathcal{X}$.
\end{lemma}

\begin{proof}
Take $x\in\mathcal{X}_{h}$. Then $x^{\oplus m}\in M_{m}\left(  \mathcal{X}%
\right)  _{h}$ and $x^{\oplus m}+re^{\oplus m}\in\mathfrak{C}_{m}$ for some
$r>0$. Since $\mathfrak{C}$ is a quantum cone, we deduce that
$x+re=\varepsilon\left(  x^{\oplus m}+re^{\oplus m}\right)  \varepsilon^{\ast
}\in\mathfrak{c}$, where $\varepsilon=\left[
\begin{array}
[c]{ccc}%
1 & 0\ldots & 0
\end{array}
\right]  \in M_{1,m}$ and $\mathfrak{c=C}_{1}$. Hence $\mathfrak{c}$ is a
unital cone. In particular, $e\in\mathfrak{c}$ and $\mathcal{X}_{h}%
=\mathfrak{c-c}$.

Now take $x\in M_{n}\left(  \mathcal{X}\right)  _{h}$ and prove that
$x+re^{\oplus n}\in\mathfrak{C}$ for some $r>0$. If $x=av^{\oplus n}$ for some
$a\in M_{n}^{+}$ and $v\in\mathfrak{c}$ then $x=a^{1/2}v^{\oplus n}a^{1/2}%
\in\mathfrak{c}^{c}\subseteq\mathfrak{C}$. But if $x=-av^{\oplus n}$ then
$-a+rI_{n}\geq0$ and $-v+se\in\mathfrak{c}$ for some real $r,s\geq0$. It
follows that
\[
x+rse^{\oplus n}=\left(  -a+rI_{n}\right)  v^{\oplus n}+r\left(  -v\right)
^{\oplus n}+rse^{\oplus n}=\left(  -a+rI_{n}\right)  v^{\oplus n}+r\left(
-v+se\right)  ^{\oplus n}\in\mathfrak{c}^{c}\text{.}%
\]
Taking into account that $\mathcal{X}_{h}=\mathfrak{c-c}$, we conclude that
$x+re^{\oplus n}\in\mathfrak{c}^{c}$ for some $r>0$ whenever $x=av^{\oplus n}$
with $a\in\left(  M_{n}\right)  _{h}$ and $v\in\mathcal{X}_{h}$. Thus
$\mathfrak{c}^{c}-e^{\oplus n}$ absorbs all hermitians from $\left(
M_{n}\right)  _{h}\otimes\mathcal{X}_{h}$. But $M_{n}\left(  \mathcal{X}%
\right)  _{h}=\left(  M_{n}\right)  _{h}\otimes\mathcal{X}_{h}$ (see
\cite[Lemma 3.7]{PaulTT}). Hence $\mathfrak{c}^{c}$ is unital, which in turn
implies that $\mathfrak{C}$ is a unital quantum cone on $\mathcal{X}$.
\end{proof}

Now let $\mathcal{X}$ be a unital $\ast$-vector space with its unit $e$ and
let $\left(  \mathcal{X},\mathcal{Y}\right)  $ be a dual $\ast$-pair. Consider
the following quantum subset $M\left(  \mathcal{Y}\right)  _{e}=\left\{  y\in
M\left(  \mathcal{Y}\right)  :\left\langle \left\langle e,y\right\rangle
\right\rangle =I\right\}  \ $in $M\left(  \mathcal{Y}\right)  $, which is
$\mathfrak{s}\left(  \mathcal{Y},\mathcal{X}\right)  $-closed and matrix
additive set. The following unital bipolar theorem was proved in
\cite{DCcones}.

\begin{theorem}
\label{tUBP}Let $\left(  \mathcal{X},\mathcal{Y}\right)  $ be a dual $\ast
$-pair with the unital space $\mathcal{X}$, and let $\mathfrak{C}$ be a
$\mathfrak{s}\left(  \mathcal{X},\mathcal{Y}\right)  $-closed, unital quantum
cone on $\mathcal{X}$. Then $\mathfrak{C}=\left(  \mathfrak{C}^{\boxdot}\cap
M\left(  \mathcal{Y}\right)  _{e}\right)  ^{\boxdot}$.
\end{theorem}

The quantum set $\mathfrak{C}^{\boxdot}\cap M\left(  \mathcal{Y}\right)  _{e}$
from Theorem \ref{tUBP} is called \textit{a matricial state space of
}$\mathfrak{C}$, and it is denoted by $\mathcal{S}\left(  \mathfrak{C}\right)
$. Notice that $\mathcal{S}\left(  \mathfrak{C}\right)  $ is a matrix additive
subset in $M\left(  \mathcal{Y}\right)  _{h}$. In the case, of a $C^{\ast}%
$-algebra $\mathcal{A}$ we write $\mathcal{S}\left(  \mathcal{A}_{+}\right)  $
instead of $\mathcal{S}\left(  M\left(  \mathcal{A}\right)  _{+}\right)  $
keeping in mind the canonical quantum cone $M\left(  \mathcal{A}\right)  _{+}$
of positive elements in $M\left(  \mathcal{A}\right)  $.

Finally, let $\mathfrak{c}$ be a separated, (algebraically) closed, unital
cone in $\mathcal{X}$. Recall that a linear functional $\sigma:\mathcal{X}%
\rightarrow\mathbb{C}$ is said to be \textit{a state of the cone}
$\mathfrak{c}$ if $\sigma\left(  e\right)  =1$ and $\sigma\left(
\mathfrak{c}\right)  \geq0$ (that is, $\sigma$ is positive). If $S\left(
\mathfrak{c}\right)  $ is the set of all states of the cone $\mathfrak{c}$,
then $\left\Vert x\right\Vert _{e}=\sup\left\vert S\left(  \mathfrak{c}%
\right)  \left(  x\right)  \right\vert $, $x\in\mathcal{X}$ is an order $\ast
$-norm on $\mathcal{X}$ in the sense of $\left\Vert x^{\ast}\right\Vert
_{e}=\left\Vert x\right\Vert _{e}$, $x\in X$, and $\left\Vert x\right\Vert
_{e}=\inf\left\{  r>0:-re\leq x\leq re\right\}  $ for all $x\in\mathcal{X}%
_{h}$ (see \cite{PT}). Put $\mathcal{Y}$ to be the normed dual of
$\mathcal{X}$ equipped with the norm $\left\Vert \cdot\right\Vert _{e}$. Then
$\left(  \mathcal{X},\mathcal{Y}\right)  $ is a dual $\ast$-pair, $S\left(
\mathfrak{c}\right)  \subseteq\mathcal{Y}$, and $\mathfrak{c=}S\left(
\mathfrak{c}\right)  ^{\boxdot}\cap\mathcal{X}$ \cite{PT}. The unital quantum
cone $S\left(  \mathfrak{c}\right)  ^{\boxdot}$ (with respect to $\left(
\mathcal{X},\mathcal{Y}\right)  $) is called the minimal quantization
$\min\mathfrak{c}$ of the cone $\mathfrak{c}$, whereas $\mathfrak{c}%
^{\boxdot\boxdot}$ is the maximal quantization $\max\mathfrak{c}$ of the cone
$\mathfrak{c}$ (see \cite{DSbM}). Thus for every separated, closed, unital
quantum cone $\mathfrak{C}$ with $\mathfrak{c=C\cap}\mathcal{X}$ we have
$\max\mathfrak{c\subseteq C\subseteq}\min\mathfrak{c}$. Notice that
$\max\mathfrak{c}$ is the $\mathfrak{s}\left(  \mathcal{X},\mathcal{Y}\right)
$-closure of the unital quantum cone $\mathfrak{c}^{c}$ generated by
$\mathfrak{c}$ (see Lemma \ref{lemPTT}).

\subsection{The lattice ideal generated by a Radon measure\label{subsecARM}}

Now let $X$ be a compact Hausdorff topological space, $C\left(  X\right)  $ is
the abelian $C^{\ast}$-algebra of all complex continuous functions on $X$
equipped with the uniform norm $\left\Vert v\right\Vert _{\infty}%
=\sup\left\vert v\left(  X\right)  \right\vert $, $v\in C\left(  X\right)  $,
whose topological dual $C\left(  X\right)  ^{\ast}$ is reduced to the Banach
space $\mathcal{M}\left(  X\right)  $ of all Radon charges on $X$. Note that
$\mathcal{M}\left(  X\right)  $ is a $\ast$-vector space with the natural
involution $\mu\mapsto\mu^{\ast}$, $\left\langle v,\mu^{\ast}\right\rangle
=\left\langle v^{\ast},\mu\right\rangle ^{\ast}$ for all $v\in C\left(
X\right)  $. The real vector space of all hermitian charges is denoted by
$\mathcal{M}\left(  X\right)  _{h}$, which is equipped with the cone
$\mathcal{M}\left(  X\right)  _{+}$ of positive measures on $X$. It is well
known that $\mathcal{M}\left(  X\right)  _{h}$ is a complete vector lattice
with respect to the vector order induced by means of cone $\mathcal{M}\left(
X\right)  _{+}$. The related lattice operations are denoted by $\vee$ and
$\wedge$, respectively. A real vector subspace $V\subseteq\mathcal{M}\left(
X\right)  _{h}$ is said to be a closed subspace if it contains $\vee S$ (sup)
and $\wedge S$ (inf) whenever $S\subseteq V$. A vector subspace $I\subseteq
\mathcal{M}\left(  X\right)  _{h}$ is said to be \textit{an ideal of}
$\mathcal{M}\left(  X\right)  _{h}$ if $\left\vert \lambda\right\vert
\leq\left\vert \mu\right\vert $ for $\lambda\in\mathcal{M}\left(  X\right)
_{h}$ and $\mu\in I$ implies that $\lambda\in I$. In this case, $\left\vert
\mu\right\vert ,\mu_{+},\mu_{-}\in I$ whenever $\mu\in I$, and $I$ turns out
to be a vector sublattice. Any intersection of ideals turns out to be an ideal
automatically, therefore each subset $S\subseteq\mathcal{M}\left(  X\right)
_{h}$ generates an ideal to be the smallest ideal of $\mathcal{M}\left(
X\right)  _{h}$ containing $S$. \ An ideal which in turn is a closed subspace
is called \textit{a closed ideal. }Similarly, one can define \textit{the
closed ideal in }$\mathcal{M}\left(  X\right)  _{h}$\textit{ generated by}
$S$. The closed ideal in $\mathcal{M}\left(  X\right)  _{h}$ generated by a
singleton $\left\{  \mu\right\}  $ is denoted by $I_{\mu}\left(  X\right)  $.
One can prove that $I_{\mu}\left(  X\right)  =I_{\left\vert \mu\right\vert
}\left(  X\right)  $, and $\lambda\in\mathcal{M}\left(  X\right)  _{+}$
belongs to $I_{\mu}\left(  X\right)  $ iff $\lambda=\vee\left\{  \lambda\wedge
n\left\vert \mu\right\vert :n\in\mathbb{N}\right\}  $.

Let $\mu\in\mathcal{M}\left(  X\right)  _{+}$. The $L^{p}$-spaces
corresponding to a Radon measure $\mu\in\mathcal{M}\left(  X\right)  _{+}$ are
denoted by $L^{p}\left(  X,\mu\right)  $, $1\leq p\leq\infty$, which are
$\ast$-vector spaces. The Banach space $L^{1}\left(  X,\mu\right)  $ is
identified with a closed subspace of $\mathcal{M}\left(  X\right)  $ up to an
isometrical isomorphism such that $L^{1}\left(  X,\mu\right)  _{h}=I_{\mu
}\left(  X\right)  $. The identification is given by the mapping $L^{1}\left(
X,\mu\right)  \rightarrow\mathcal{M}\left(  X\right)  $, $\eta\mapsto\eta\mu$,
where $\left\langle h,\eta\mu\right\rangle =\int h\left(  t\right)
\eta\left(  t\right)  d\mu$ for all $h\in C\left(  X\right)  $. The present
result is well known as Lebesgue-Nikodym theorem \cite[Ch.V, 5.5, Theorem
2]{BourInt}.

Further, notice that $\mu\in\mathcal{M}\left(  X\right)  $ iff $\left\langle
v,\left\vert \mu\right\vert \right\rangle =\sup\left\{  \mu\left(  w\right)
:w\in C\left(  X\right)  ,\left\vert w\right\vert \leq v\right\}  <\infty$ for
every $v\in C\left(  X\right)  _{+}$. In this case, $\left\vert \mu\right\vert
\in\mathcal{M}\left(  X\right)  _{+}$ and $\mu=u\left\vert \mu\right\vert $
for a Borel function $u$ on $X$ such that $\left\vert u\right\vert =1$ almost
everywhere with respect to $\left\vert \mu\right\vert $. The space of all
probability measures on $X$ is denoted by $\mathcal{P}\left(  X\right)  $,
which is a $w^{\ast}$-compact subspace in the space $\mathcal{M}\left(
X\right)  $. Notice that $\mathcal{P}\left(  X\right)  $ is the $w^{\ast}%
$-closure of the convex hull of its extremal boundary $\partial\mathcal{P}%
\left(  X\right)  $ which consists of Dirac measures $\delta_{t}$, $t\in X$
thanks to Krein-Milman theorem.

Fix $\mu\in\mathcal{M}\left(  X\right)  _{+}$. Recall that a point $s\in X$ is
said to be\textit{ a }$\mu$\textit{-mass }if $\mu\left(  s\right)  >0$. Notice
that $s$ is a unique mass with respect to $\delta_{s}$.

\begin{lemma}
\label{lematom1}A point $s\in X$ is a $\mu$-mass iff $\delta_{s}\in I_{\mu
}\left(  X\right)  $. In this case $\delta_{s}=s^{\prime}\mu$ with $s^{\prime
}\in L^{2}\left(  X,\mu\right)  $. In this case, $s^{\prime}=\mu\left(
s\right)  ^{-1}\left[  s\right]  $.
\end{lemma}

\begin{proof}
First assume that $s$ is a $\mu$-mass. Take a Borel set $N\subseteq X$ such
that $\mu\left(  N\right)  =0$. Then $s\notin N$, which in turn implies that
$\delta_{s}\left(  N\right)  =0$. Hence $\delta_{s}$ is absolutely continuous
with respect to $\mu$. By Lebesgue-Nikodym theorem, $\delta_{s}=s^{\prime}%
\mu\in I_{\mu}\left(  X\right)  $, $s^{\prime}\left(  t\right)  \geq0$ for
$\mu$-almost all $t\in X$, and $s^{\prime}\in L^{1}\left(  X,\mu\right)  $.
But $s^{\prime}\left(  s\right)  =\left\langle s^{\prime},\delta
_{s}\right\rangle =\left\langle s^{\prime},s^{\prime}\mu\right\rangle
=\int\left(  s^{\prime}\right)  ^{2}d\mu$, that is, $s^{\prime}\in
L^{2}\left(  X,\mu\right)  $. Conversely, suppose that $\delta_{s}\in I_{\mu
}\left(  X\right)  $. Since $\left\{  s\right\}  $ is a Borel set, the
condition $\mu\left(  s\right)  =0$ would imply that $\delta_{s}\left(
s\right)  =0$, a contradiction.

Actually, $s^{\prime}=\mu\left(  s\right)  ^{-1}\left[  s\right]  $. Indeed,
$\mu\left(  s\right)  ^{-1}\left[  s\right]  $ is a Borel function from
$L^{2}\left(  X,\mu\right)  $ and
\[
\left\langle v,\left(  \mu\left(  s\right)  ^{-1}\left[  s\right]  \right)
\mu\right\rangle =\mu\left(  s\right)  ^{-1}\int v\left[  s\right]  d\mu
=\mu\left(  s\right)  ^{-1}\int v\left(  s\right)  d\mu=v\left(  s\right)
=\left\langle v,\delta_{s}\right\rangle
\]
for all $v\in C\left(  X\right)  $, which means that $s^{\prime}=\mu\left(
s\right)  ^{-1}\left[  s\right]  $.
\end{proof}

\begin{remark}
Notice that $s^{\prime}=\dfrac{d\delta_{s}}{d\mu}$ is the Radon-Nikodym
derivative of $\delta_{s}$ with respect to $\mu$.
\end{remark}

\subsection{Pietsch factorization\label{subsecPF}}

Let $V$ be a (Hausdorff) polynormed space. A family $\left(  v_{i}\right)
_{i\in I}$ in $V$ is said to be an absolutely summable if $\sum_{i\in
I}\left\Vert v_{i}\right\Vert <\infty$ for every continuous seminorm
$\left\Vert \cdot\right\Vert $ on $V$. A continuous mapping $T:V\rightarrow W$
of polynormed spaces is called an absolutely summable if $\left(
Tv_{i}\right)  _{i\in I}$ is absolutely summable in $W$ for every summable
family $\left(  v_{i}\right)  _{i\in I}$ in $V$, that is, $T$ transforms
summable families from $V$ to absolutely summable ones in $W$. A linear
mapping $T:V\rightarrow W$ between normed spaces $V$ and $W$ is absolutely
summable iff there exists a positive $\rho$ such that $\sum_{\mathfrak{n}%
}\left\Vert Tv_{n}\right\Vert \leq\rho\sup\left\{  \sum_{\mathfrak{n}%
}\left\vert \left\langle v_{n},a\right\rangle \right\vert :a\in
\operatorname{ball}V^{\ast}\right\}  $ for all finite families $\left(
v_{n}\right)  _{n\in\mathfrak{n}}$ in $V$ \cite[Proposition 2.2.1]{AlP}. Put
$\pi\left(  T\right)  =\inf\left\{  \rho\right\}  $, which is a norm in the
space $\mathcal{A}\left(  V,W\right)  $ of all absolutely summable maps
between $V$ and $W$. If $W$ is complete then $\mathcal{A}\left(  V,W\right)  $
equipped with the $\pi$-norm is a Banach space. Now let $T\in\mathcal{B}%
\left(  V,W\right)  $ (the space of all bounded linear operators from $V$ to
$W$) and let $X\subseteq\operatorname{ball}V^{\ast}$ be an essential subset in
the sense of $\left\Vert v\right\Vert =\sup\left\vert \left\langle
v,X\right\rangle \right\vert $ for all $v\in V$, that is, the canonical
representation $V\rightarrow C\left(  X\right)  $, $v\mapsto\left\langle
v,\cdot\right\rangle $ is an isometry. The known result of Pietsch
\cite[Theorem 2.3.3]{AlP} asserts that $T\in\mathcal{A}\left(  V,W\right)  $
iff there exists $\mu\in\mathcal{M}\left(  X\right)  _{+}$ such that
$\left\Vert Tv\right\Vert \leq\int_{X}\left\vert \left\langle v,t\right\rangle
\right\vert d\mu\left(  t\right)  $ for all $v\in V$. In this case,
$\pi\left(  T\right)  =\min\left\{  \mu\left(  X\right)  \right\}  $ over all
$\mu\in\mathcal{M}\left(  X\right)  _{+}$ with the just indicated property. In
particular, $T\in\mathcal{A}\left(  C\left(  X\right)  ,W\right)  $ iff
$\left\Vert Tv\right\Vert \leq\int_{X}\left\vert v\left(  t\right)
\right\vert d\mu\left(  t\right)  $, $v\in C\left(  X\right)  $ for a certain
$\mu\in\mathcal{M}\left(  X\right)  _{+}$, where $X$ is a compact Hausdorff
topological space. For the Hilbert spaces $K$ and $H$ we have $\mathcal{A}%
\left(  K,H\right)  =\mathcal{B}^{2}\left(  K,H\right)  $ and $\left\Vert
T\right\Vert _{2}\leq\pi\left(  T\right)  \leq\sqrt{3}\left\Vert T\right\Vert
_{2}$, where $\mathcal{B}^{2}\left(  K,H\right)  $ is the space of all
Hilbert-Schmidt operators from $K$ to $H$. The idea of the proof of the
following key lemma of Pietsch will be used later on. For the completeness we
provide its proof.

\begin{lemma}
\label{lemAlP1}Let $X$ be a compact Hausdorff topological space, $\mu
\in\mathcal{M}\left(  X\right)  _{+}$, $\iota:C\left(  X\right)  \rightarrow
L^{2}\left(  X,\mu\right)  $ the canonical representation, $H$ a Hilbert space
and let $T:L^{2}\left(  X,\mu\right)  \rightarrow H$, $T=\sum_{r=1}^{m}%
\zeta_{r}\odot\overline{\eta_{r}}$ be a finite-rank operator given by a finite
family $\left(  \zeta_{r}\right)  _{r}\subseteq H$ and $\mu$-step functions
$\left(  \eta_{r}\right)  _{r}\subseteq L^{2}\left(  X,\mu\right)  $. Then
$T\iota:C\left(  X\right)  \rightarrow H$ is a nuclear operator with
$\left\Vert T\iota\right\Vert _{1}\leq\left\Vert T\right\Vert _{2}$.
\end{lemma}

\begin{proof}
One can choose a partition $X=X_{1}\cup\ldots\cup X_{n}$ of $X$ into $\mu
$-measurable subsets $X_{r}\subseteq X$ such that $T=\sum_{r=1}^{n}\eta
_{r}\odot\overline{\chi_{r}}$ for a new family $\left(  \eta_{r}\right)
_{r}\subseteq H$ and an orthogonal family $\left(  \chi_{r}\right)
_{r}\subseteq L^{2}\left(  X,\mu\right)  $, where $\chi_{r}$ is the
characteristic function of $X_{r}$. Put $\widehat{\chi_{r}}=\mu\left(
X_{r}\right)  ^{-1/2}\chi_{r}$ and $\mu_{r}=\chi_{r}\mu\in I_{\mu}\left(
X\right)  _{+}$. Notice that $\left(  \widehat{\chi_{r}}\right)  _{r}$ is a
finite orthonormal family in $L^{2}\left(  X,\mu\right)  $, $T\left(
\widehat{\chi_{r}}\right)  =\mu\left(  X_{r}\right)  ^{-1/2}\left(  \chi
_{r},\chi_{r}\right)  \eta_{r}=\mu\left(  X_{r}\right)  ^{1/2}\eta_{r}$ and
$\sum_{r}\left\Vert T\left(  \widehat{\chi_{r}}\right)  \right\Vert ^{2}%
\leq\left\Vert T\right\Vert _{2}^{2}$. Moreover, $\left\Vert \mu
_{r}\right\Vert =\left\langle 1,\mu_{r}\right\rangle =\int\chi_{r}d\mu
=\mu\left(  X_{r}\right)  $ for all $r$, and $T\left(  \iota\left(  v\right)
\right)  =\sum_{r}\left(  v,\chi_{r}\right)  \eta_{r}=\sum_{r}\left(  \int
v\left(  t\right)  \chi_{r}\left(  t\right)  d\mu\right)  \eta_{r}=\sum
_{r}\left\langle v,\mu_{r}\right\rangle \eta_{r}$ for all $v\in C\left(
X\right)  $. Thus $T\iota=\sum_{r=1}^{n}\eta_{r}\odot\mu_{r}$ is a nuclear
operator and
\begin{align*}
\left\Vert T\iota\right\Vert _{1}  &  \leq\sum_{r}\left\Vert \eta
_{r}\right\Vert \left\Vert \mu_{r}\right\Vert =\sum_{r}\mu\left(
X_{r}\right)  ^{1/2}\mu\left(  X_{r}\right)  ^{1/2}\left\Vert \eta
_{r}\right\Vert \leq\left(  \sum_{r}\mu\left(  X_{r}\right)  \right)
^{1/2}\left(  \sum_{r}\mu\left(  X_{r}\right)  \left\Vert \eta_{r}\right\Vert
^{2}\right)  ^{1/2}\\
&  =\left(  \sum_{r}\mu\left(  X_{r}\right)  \left(  \eta_{r},\eta_{r}\right)
\right)  ^{1/2}=\left(  \sum_{r}\left(  \mu\left(  X_{r}\right)  ^{1/2}%
\eta_{r},\mu\left(  X_{r}\right)  ^{1/2}\eta_{r}\right)  \right)  ^{1/2}\\
&  =\left(  \sum_{r}\left(  T\left(  \widehat{\chi_{r}}\right)  ,T\left(
\widehat{\chi_{r}}\right)  \right)  \right)  ^{1/2}=\left(  \sum_{r}\left\Vert
T\left(  \widehat{\chi_{r}}\right)  \right\Vert ^{2}\right)  ^{1/2}%
\leq\left\Vert T\right\Vert _{2},
\end{align*}
that is, $\left\Vert T\iota\right\Vert _{1}\leq\left\Vert T\right\Vert _{2}$.
\end{proof}

Actually, the assertion proven in Lemma \ref{lemAlP1} is true for every
$T\in\mathcal{B}^{2}\left(  L^{2}\left(  X,\mu\right)  ,H\right)  $
\cite[3.3.3 Proposition 2]{AlP}. As a result we obtain the following
factorization \cite[3.3.4]{AlP} of an absolutely summable mapping.

\begin{proposition}
\label{propAlP2}Let $T\in\mathcal{A}\left(  V,W\right)  $ and let
$X\subseteq\operatorname{ball}V^{\ast}$ be an essential subset. There exists a
$\mu\in\mathcal{M}\left(  X\right)  _{+}$ such that $T$ can be factorized as
$T=T_{2}\iota T_{1}$, that is, the following diagram
\[%
\begin{array}
[c]{ccc}%
C\left(  X\right)  & \overset{\iota}{\longrightarrow} & L^{2}\left(
X,\mu\right) \\
^{T_{1}}\uparrow &  & \downarrow^{T_{2}}\\
V & \overset{T}{\longrightarrow} & W
\end{array}
\]
commutes with $\left\Vert T_{1}\right\Vert \leq1$ and $\left\Vert
T_{2}\right\Vert \leq\pi\left(  T\right)  $.
\end{proposition}

The factorization from Proposition \ref{propAlP2} is known as the Pietsch factorization.

\begin{remark}
\label{remIT11}If $T\in\mathcal{B}\left(  H,C\left(  X\right)  \right)  $ then
$\iota T\in\mathcal{B}^{2}\left(  H,L^{2}\left(  X,\mu\right)  \right)  $ for
every $\mu\in\mathcal{M}\left(  X\right)  _{+}$. Indeed, for a Hilbert basis
$F$ for $H$ we have $\sum_{f\in F}\left\vert \left(  Tf\right)  \left(
t\right)  \right\vert ^{2}=\sum_{f\in F}\left\vert \left\langle Tf,\delta
_{t}\right\rangle \right\vert ^{2}=\sum_{f\in F}\left\vert \left\langle
f,T^{\ast}\delta_{t}\right\rangle \right\vert ^{2}=\sum_{f\in F}\left\vert
\left(  f,T^{\ast}\delta_{t}\right)  \right\vert ^{2}\leq\left\Vert T^{\ast
}\delta_{t}\right\Vert ^{2}\leq\left\Vert T^{\ast}\right\Vert ^{2}=\left\Vert
T\right\Vert ^{2}$, $t\in X$. It follows that $\left\Vert \iota T\right\Vert
_{2}^{2}=\sum_{f}\left\Vert \left(  \iota T\right)  f\right\Vert ^{2}=\sum
_{f}\int\left\vert \left(  Tf\right)  \left(  t\right)  \right\vert ^{2}%
d\mu\leq\left\Vert T\right\Vert ^{2}\int1<\infty$.
\end{remark}

Based on these results one can prove that a superposition of two absolutely
summable maps turns out to be a nuclear operator (see \cite[3.3.5]{AlP}).

\section{Quantum cones on a Hilbert space\label{secQHS}}

In this section we introduce unital cones in a Hilbert space and classify
their quantizations.

\subsection{Hilbert $\ast$-space}

Let $H$ be a Hilbert space. By an involution on $H$ we mean a $\ast$-linear
mapping $H\rightarrow H$, $\zeta\mapsto\zeta^{\ast}$ such that $\zeta
^{\ast\ast}=\zeta$ and $\left(  \zeta^{\ast},\eta^{\ast}\right)  =\left(
\zeta,\eta\right)  ^{\ast}$ for all $\zeta,\eta\in H$. In the case of
$H=\ell^{2}\left(  F\right)  $ the mapping $\zeta\mapsto\zeta^{\ast}$ with
$\zeta^{\ast}=\sum_{f\in F}\zeta_{f}^{\ast}f$ for $\zeta=\sum_{f\in F}%
\zeta_{f}f$ is a natural involution on $H$, where $\zeta_{f}=\left(
\zeta,f\right)  $, $f\in F$. The set of all hermitian vectors from a Hilbert
$\ast$-space $H$ is denoted by $H_{h}$. Notice that $H_{h}$ is a real Hilbert
space, for $\left(  \zeta_{,}\eta\right)  \in\mathbb{R}$ whenever $\zeta
,\eta\in H_{h}$. For every $\zeta\in H$ we have a unique expansion
$\zeta=\operatorname{Re}\zeta+i\operatorname{Im}\zeta$ into its hermitian
parts, $\left(  i\operatorname{Im}\zeta,\operatorname{Re}\zeta\right)
=\left(  \operatorname{Re}\zeta,i\operatorname{Im}\zeta\right)  ^{\ast
}=\left(  \left(  \operatorname{Re}\zeta\right)  ^{\ast},\left(
i\operatorname{Im}\zeta\right)  ^{\ast}\right)  =-\left(  \operatorname{Re}%
\zeta,i\operatorname{Im}\zeta\right)  $, and $\left\Vert \zeta\right\Vert
^{2}=\left\Vert \operatorname{Re}\zeta\right\Vert ^{2}+\left\Vert
\operatorname{Im}\zeta\right\Vert ^{2}+\left(  i\operatorname{Im}%
\zeta,\operatorname{Re}\zeta\right)  +\left(  \operatorname{Re}\zeta
,i\operatorname{Im}\zeta\right)  =\left\Vert \operatorname{Re}\zeta\right\Vert
^{2}+\left\Vert \operatorname{Im}\zeta\right\Vert ^{2}$. Take a (real) Hilbert
basis $F$ for $H_{h}$, which turns out to be a (complex) basis for $H$. For
every $\zeta\in H$ with $\zeta=\sum_{f\in F}\zeta_{f}f$ we have
$\operatorname{Re}\zeta=\sum_{f}\left(  \operatorname{Re}\zeta_{f}\right)  f$
and $\operatorname{Im}\zeta=\sum_{f}\left(  \operatorname{Im}\zeta_{f}\right)
f$, which in turn implies that $\zeta^{\ast}=\operatorname{Re}\zeta
-i\operatorname{Im}\zeta=\sum_{f}\zeta_{f}^{\ast}f$. Thus every involution on
$H$ is reduced to the above considered example of $\ell^{2}\left(  F\right)  $
with respect to a suitable basis for $H$.

The conjugate Hilbert space to $H$ is denoted by $\overline{H}$, whose vectors
are denoted by $\overline{\zeta}$, $\zeta\in H$. Thus $\lambda\overline{\zeta
}=\overline{\lambda^{\ast}\zeta}$ and $\left(  \overline{\zeta},\overline
{\eta}\right)  =\left(  \zeta,\eta\right)  ^{\ast}=\left(  \zeta^{\ast}%
,\eta^{\ast}\right)  $ for all $\zeta,\eta\in H$ and $\lambda\in\mathbb{C}$.
Notice that the canonical mapping $\psi:\overline{H}\rightarrow H^{\ast}$,
$\psi\left(  \overline{\eta}\right)  =\left(  \cdot,\eta\right)  $ is an
isometric isomorphism. Thus $\left(  H,\overline{H}\right)  $ is a dual pair
with the canonical duality $\left\langle \zeta,\overline{\eta}\right\rangle
=\left(  \zeta,\eta\right)  $, $\zeta,\eta\in H$. Moreover, it is a dual
$\ast$-pair, for $\left\langle \zeta^{\ast},\overline{\eta}\right\rangle
=\left(  \zeta^{\ast},\eta\right)  =\left(  \zeta,\eta^{\ast}\right)  ^{\ast
}=\left\langle \zeta,\overline{\eta^{\ast}}\right\rangle ^{\ast}$, $\zeta
,\eta\in H$, which means that the involution is weakly continuous. In
particular, $\overline{H}$ possesses the involution $\overline{\eta}%
\mapsto\overline{\eta}^{\ast}$, $\left\langle \zeta,\overline{\eta}^{\ast
}\right\rangle =\left\langle \zeta^{\ast},\overline{\eta}\right\rangle ^{\ast
}$ (see Subsection \ref{subsecInv}). Thus $\left\langle \zeta,\overline{\eta
}^{\ast}\right\rangle =\left\langle \zeta,\overline{\eta^{\ast}}\right\rangle
$ for all $\zeta\in H$, which in turn implies that $\overline{\eta}^{\ast
}=\overline{\eta^{\ast}}$. In particular, $\left(  \overline{\zeta}^{\ast
},\overline{\eta}^{\ast}\right)  =\left(  \overline{\zeta^{\ast}}%
,\overline{\eta^{\ast}}\right)  =\left(  \zeta^{\ast},\eta^{\ast}\right)
^{\ast}=\left(  \zeta,\eta\right)  =\left(  \overline{\zeta},\overline{\eta
}\right)  ^{\ast}$, $\zeta,\eta\in H$, which means that $\overline{H}$ is a
Hilbert $\ast$-space as well.

Later on we fix a hermitian unit vector $e$ from $H$, which can be extended up
to a basis $F$ for $H_{h}$. Thus $\left(  H,e\right)  $ is a unital space.
Since $\overline{e}^{\ast}=\overline{e^{\ast}}=\overline{e}$, it follows that
$\left(  \overline{H},\overline{e}\right)  $ is a unital Hilbert $\ast$-space
either. As above in Subsection \ref{subsecDD}, the duality $\left\langle
\cdot,\cdot\right\rangle $ of the dual $\ast$-pair $\left(  H,\overline
{H}\right)  $ can be extended up to a matrix duality $\left\langle
\left\langle \cdot,\cdot\right\rangle \right\rangle :M\left(  H\right)  \times
M\left(  \overline{H}\right)  \rightarrow M$ by $\left\langle \left\langle
\zeta,\overline{\eta}\right\rangle \right\rangle =\left[  \left\langle
\zeta_{ik},\overline{\eta_{jl}}\right\rangle \right]  _{\left(  i,j\right)
,\left(  k,l\right)  }=\left[  \left(  \zeta_{ik},\eta_{jl}\right)  \right]
_{\left(  i,j\right)  ,\left(  k,l\right)  }$, and each $\left(  M_{n}\left(
H\right)  ,M_{n}\left(  \overline{H}\right)  \right)  $ is a dual $\ast$-pair
(see Subsection \ref{subsecInv}). In this case, for $a\in M_{n,m}$,
$\overline{\eta}\in M_{m}\left(  \overline{H}\right)  $ and $b\in M_{m,n}$ we
have
\begin{equation}
\overline{a\eta b}=\left[  \sum_{k,l=1}^{m}\overline{a_{ik}\eta_{kl}b_{lj}%
}\right]  _{i,j}=\left[  \sum_{k,l=1}^{m}a_{ik}^{\ast}\overline{\eta_{kl}%
}b_{lj}^{\ast}\right]  _{i,j}=\left(  a^{\ast}\right)  ^{\operatorname{t}%
}\overline{\eta}\left(  b^{\ast}\right)  ^{\operatorname{t}}. \label{faeb}%
\end{equation}
Note also that $\overline{\eta}^{\ast}=\left[  \overline{\eta_{ij}}\right]
_{i,j}^{\ast}=\left[  \overline{\eta_{ji}}^{\ast}\right]  _{i,j}=\left[
\overline{\eta_{ji}^{\ast}}\right]  _{i,j}=\overline{\eta^{\ast}}$ and
\begin{align*}
\left\langle \left\langle \zeta^{\ast},\overline{\eta}^{\ast}\right\rangle
\right\rangle  &  =\left[  \left\langle \zeta_{ki}^{\ast},\overline{\eta_{lj}%
}^{\ast}\right\rangle \right]  _{\left(  i,j\right)  ,\left(  k,l\right)
}=\left[  \left(  \zeta_{ki}^{\ast},\eta_{lj}^{\ast}\right)  \right]
_{\left(  i,j\right)  ,\left(  k,l\right)  }=\left[  \left(  \zeta_{ki}%
,\eta_{lj}\right)  ^{\ast}\right]  _{\left(  i,j\right)  ,\left(  k,l\right)
}\\
&  =\left[  \left\langle \zeta_{ki},\overline{\eta_{lj}}\right\rangle ^{\ast
}\right]  _{\left(  i,j\right)  ,\left(  k,l\right)  }=\left[  \left\langle
\zeta_{ik},\overline{\eta_{jl}}\right\rangle \right]  _{\left(  i,j\right)
,\left(  k,l\right)  }^{\ast}=\left\langle \left\langle \zeta,\overline{\eta
}\right\rangle \right\rangle ^{\ast}%
\end{align*}
for all $\zeta\in M\left(  H\right)  $ and $\eta\in M\left(  \overline
{H}\right)  $. Recall that the matrix norm $\left\Vert \cdot\right\Vert _{o}$
of an operator Hilbert space $H_{o}$ is given by $\left\Vert \zeta\right\Vert
_{o}=\left\Vert \left\langle \left\langle \zeta,\overline{\zeta}\right\rangle
\right\rangle \right\Vert ^{1/2}$, $\zeta\in M\left(  H_{o}\right)  $. Thus
$\left\Vert \zeta^{\ast}\right\Vert _{o}=\left\Vert \left\langle \left\langle
\zeta^{\ast},\overline{\zeta^{\ast}}\right\rangle \right\rangle \right\Vert
^{1/2}=\left\Vert \left\langle \left\langle \zeta^{\ast},\overline{\zeta
}^{\ast}\right\rangle \right\rangle \right\Vert ^{1/2}=\left\Vert \left\langle
\left\langle \zeta,\overline{\zeta}\right\rangle \right\rangle ^{\ast
}\right\Vert ^{1/2}=\left\Vert \left\langle \left\langle \zeta,\overline
{\zeta}\right\rangle \right\rangle \right\Vert ^{1/2}=\left\Vert
\zeta\right\Vert _{o}$ for all $\zeta\in M\left(  H_{o}\right)  $, and
$\left\Vert e\right\Vert _{o}=\left\Vert e\right\Vert =1$.

\subsection{Hilbert space norm on $M\left(  H\right)  $}

As above let $\left(  H,e\right)  $ be a unital Hilbert $\ast$-space and let
$F$ be a basis for $H_{h}$ which contains $e$. Along with the matrix pairing
$\left\langle \left\langle \cdot,\cdot\right\rangle \right\rangle $ we have
the scalar pairing $\left\langle \cdot,\cdot\right\rangle :M_{n}\left(
H\right)  \times M_{n}\left(  \overline{H}\right)  \rightarrow\mathbb{C}$
given by $\left\langle \zeta,\overline{\eta}\right\rangle =\sum_{i,j}%
\left\langle \zeta_{ij},\overline{\eta_{ij}}\right\rangle $ (see Subsection
\ref{subsecDD}). Note that $M_{n}\left(  H\right)  $ is a Hilbert space with
$\left(  \zeta,\eta\right)  =\sum_{i,j}\left(  \zeta_{ij},\eta_{ij}\right)
=\sum_{i,j}\left\langle \zeta_{ij},\overline{\eta_{ij}}\right\rangle
=\left\langle \zeta,\overline{\eta}\right\rangle $ for all $\zeta,\eta\in
M_{n}\left(  H\right)  $. Moreover, every $\zeta\in M_{n}\left(  H\right)  $
admits a unique expansion $\zeta=\sum_{f\in F}\left\langle \left\langle
\zeta,\overline{f}\right\rangle \right\rangle f^{\oplus n}$. Indeed,
$\zeta=\left[  \zeta_{ij}\right]  _{i,j}=\left[  \sum_{f}\left(  \zeta
_{ij},f\right)  f\right]  _{i,j}=\sum_{f}\left[  \left(  \zeta_{ij},f\right)
\right]  _{i,j}\otimes f=\sum_{f\in F}\left\langle \left\langle \zeta
,\overline{f}\right\rangle \right\rangle f^{\oplus n}$. Note that%
\begin{align}
\left\langle \zeta,a\otimes\overline{f}\right\rangle  &  =\left\langle
\zeta,a\overline{f}^{\oplus n}\right\rangle =\sum\left(  \zeta_{ij}%
,a_{ij}^{\ast}f\right)  =\sum a_{ij}\left(  \zeta_{ij},f\right)  =\tau\left(
a\left\langle \left\langle \zeta,\overline{f}\right\rangle \right\rangle
^{\operatorname{t}}\right) \label{fND}\\
&  =\tau\left(  a^{\operatorname{t}}\left\langle \left\langle \zeta
,\overline{f}\right\rangle \right\rangle \right) \nonumber
\end{align}
for all $\zeta\in M_{n}\left(  H\right)  $ and $a\in M_{n}$, where $\tau$
indicates to the standard trace of a matrix. On the matrix space $M_{n}\left(
H\right)  $ we have the Hilbert space norm $\left\Vert \zeta\right\Vert
_{2}=\left\langle \zeta,\overline{\zeta}\right\rangle ^{1/2}$, $\zeta\in
M_{n}\left(  H\right)  $. The family of unit balls $\operatorname{ball}%
\left\Vert \cdot\right\Vert _{2}$ is an absolutely convex quantum set
$\mathfrak{H}$ in $M\left(  H\right)  $ whereas $\mathfrak{B=}%
\operatorname{ball}\left\Vert \cdot\right\Vert _{o}$ is an absolutely matrix
convex set in $M\left(  H\right)  $. The self-dual property of $H_{o}$ asserts
\cite[3.5.2]{ER} that $\mathfrak{B}^{\odot}=\overline{\mathfrak{B}}$ with
respect to the duality $\left(  H,\overline{H}\right)  $. In particular,
$\left\Vert \zeta\right\Vert _{o}=\sup\left\Vert \left\langle \left\langle
\zeta,\overline{\mathfrak{B}}\right\rangle \right\rangle \right\Vert $ for all
$\zeta\in M\left(  H_{o}\right)  $. For the hermitian parts $\mathfrak{H\cap
}M\left(  H\right)  _{h}$ and $\mathfrak{B}\cap M\left(  H\right)  _{h}$ we
use the notations $\mathfrak{H}_{h}$ and $\mathfrak{B}_{h}$, respectively.

\begin{lemma}
\label{lHS11}If $\zeta,\eta\in M_{n}\left(  H\right)  $ then $\left\langle
\zeta,\overline{\eta}\right\rangle =\sum_{f}\tau\left(  \left\langle
\left\langle \zeta,\overline{f}\right\rangle \right\rangle \left\langle
\left\langle \eta,\overline{f}\right\rangle \right\rangle ^{\ast}\right)  $
and $\left\Vert \zeta\right\Vert _{o}\leq\left\Vert \zeta\right\Vert _{2}%
\leq\sqrt{n}\left\Vert \zeta\right\Vert _{o}$. In particular, $\left\Vert
\zeta\right\Vert _{2}=\left(  \sum_{f}\tau\left(  \left\vert \left\langle
\left\langle \zeta,\overline{f}\right\rangle \right\rangle \right\vert
^{2}\right)  \right)  ^{1/2}$ and $\mathfrak{H}\cap M_{n}\left(  H\right)
\subseteq\mathfrak{B}\cap M_{n}\left(  H\right)  \subseteq\sqrt{n}%
\mathfrak{H}\cap M_{n}\left(  H\right)  $ for all $n\in\mathbb{N}$.
\end{lemma}

\begin{proof}
Take $\zeta\in M_{n}\left(  H\right)  $. Using (\ref{faeb}) and (\ref{fND}),
we derive that%
\begin{align*}
\left\langle \zeta,\overline{\eta}\right\rangle  &  =\sum_{f}\left\langle
\zeta,\left(  \left\langle \left\langle \eta,\overline{f}\right\rangle
\right\rangle ^{\ast}\right)  ^{\operatorname{t}}\overline{f}^{\oplus
n}\right\rangle =\sum_{f}\left\langle \zeta,\left(  \left\langle \left\langle
\eta,\overline{f}\right\rangle \right\rangle ^{\ast}\right)
^{\operatorname{t}}\otimes\overline{f}\right\rangle =\sum_{f}\tau\left(
\left\langle \left\langle \eta,\overline{f}\right\rangle \right\rangle ^{\ast
}\left\langle \left\langle \zeta,\overline{f}\right\rangle \right\rangle
\right) \\
&  =\sum_{f}\tau\left(  \left\langle \left\langle \zeta,\overline
{f}\right\rangle \right\rangle \left\langle \left\langle \eta,\overline
{f}\right\rangle \right\rangle ^{\ast}\right)  .
\end{align*}
In particular,
\[
\left\Vert \zeta\right\Vert _{2}=\left\langle \zeta,\overline{\zeta
}\right\rangle ^{1/2}=\left(  \sum_{f}\tau\left(  \left\langle \left\langle
\zeta,\overline{f}\right\rangle \right\rangle ^{\ast}\left\langle \left\langle
\zeta,\overline{f}\right\rangle \right\rangle \right)  \right)  ^{1/2}=\left(
\sum_{f}\tau\left(  \left\vert \left\langle \left\langle \zeta,\overline
{f}\right\rangle \right\rangle \right\vert ^{2}\right)  \right)  ^{1/2},
\]
that is, $\left\Vert \zeta\right\Vert _{2}=\left(  \sum_{f}\tau\left(
\left\vert \left\langle \left\langle \zeta,\overline{f}\right\rangle
\right\rangle \right\vert ^{2}\right)  \right)  ^{1/2}$. It follows that
\begin{align*}
\left\Vert \zeta\right\Vert _{o}^{2}  &  =\left\Vert \left\langle \left\langle
\zeta,\overline{\zeta}\right\rangle \right\rangle \right\Vert =\left\Vert
\sum_{f,g}\left(  \left\langle \left\langle \zeta,\overline{f}\right\rangle
\right\rangle \left(  \left\langle \left\langle \zeta,\overline{g}%
\right\rangle \right\rangle ^{\ast}\right)  ^{\operatorname{t}}\otimes
I_{n}\right)  \left\langle \left\langle f^{\oplus n},\overline{g}^{\oplus
n}\right\rangle \right\rangle \right\Vert \\
&  \leq\sum_{f}\left\Vert \left\langle \left\langle \zeta,\overline
{f}\right\rangle \right\rangle \left(  \left\langle \left\langle
\zeta,\overline{f}\right\rangle \right\rangle ^{\ast}\right)
^{\operatorname{t}}\otimes I_{n}\right\Vert \leq\sum_{f}\left\Vert
\left\langle \left\langle \zeta,\overline{f}\right\rangle \right\rangle
\right\Vert ^{2}\\
&  =\sum_{f}\left\Vert \left\vert \left\langle \left\langle \zeta,\overline
{f}\right\rangle \right\rangle \right\vert ^{2}\right\Vert \leq\sum_{f}%
\tau\left(  \left\vert \left\langle \left\langle \zeta,\overline
{f}\right\rangle \right\rangle \right\vert ^{2}\right)  =\left\Vert
\zeta\right\Vert _{2}^{2},
\end{align*}
that is, $\left\Vert \zeta\right\Vert _{o}\leq\left\Vert \zeta\right\Vert
_{2}$. Furthermore, $\left\Vert \zeta\right\Vert _{2}^{2}=\left\langle
\zeta,\overline{\zeta}\right\rangle =\left(  \left\langle \left\langle
\zeta,\overline{\zeta}\right\rangle \right\rangle I_{n},I_{n}\right)
\leq\left\Vert \left\langle \left\langle \zeta,\overline{\zeta}\right\rangle
\right\rangle \right\Vert \left\Vert I_{n}\right\Vert _{2}^{2}=n\left\Vert
\zeta\right\Vert _{o}^{2}$, that is, $\left\Vert \zeta\right\Vert _{2}%
\leq\sqrt{n}\left\Vert \zeta\right\Vert _{o}$. The rest is clear.
\end{proof}

\begin{remark}
\label{remlHS}Notice that Schwarz inequality for the scalar pairing on
$\left(  H,\overline{H}\right)  $ follows from the matrix and then classical
Schwarz inequalities in the following way
\[
\left\langle \zeta,\overline{\eta}\right\rangle =\sum_{f}\tau\left(
\left\langle \left\langle \zeta,\overline{f}\right\rangle \right\rangle
\left\langle \left\langle \eta,\overline{f}\right\rangle \right\rangle ^{\ast
}\right)  \leq\left(  \sum_{f\neq e}\tau\left(  \left\vert \left\langle
\left\langle \zeta,\overline{f}\right\rangle \right\rangle \right\vert
^{2}\right)  \right)  ^{1/2}\left(  \sum_{f\neq e}\tau\left(  \left\vert
\left\langle \left\langle \eta,\overline{f}\right\rangle \right\rangle
\right\vert ^{2}\right)  \right)  ^{1/2}=\left\Vert \zeta\right\Vert
_{2}\left\Vert \eta\right\Vert _{2}%
\]
for all $\zeta,\eta\in M_{n}\left(  H\right)  $ (see Lemma \ref{lHS11}).
\end{remark}

Using the matrix ball $\mathfrak{B}$ and the scalar pairing $\left\langle
\cdot,\cdot\right\rangle $, we can define the norm (not a matrix one)
$\left\Vert \zeta\right\Vert _{so}=\sup\left\vert \left\langle \zeta
,\overline{\mathfrak{B}}\right\rangle \right\vert $ on $M\left(  H\right)  $.

\begin{corollary}
\label{corHS1}If $\zeta\in M_{n}\left(  H\right)  $ then $\left\Vert
\zeta\right\Vert _{so}\leq\sqrt{n}\left\Vert \zeta\right\Vert _{2}$ and
$\left\Vert \zeta\right\Vert _{o}\leq\left\Vert \zeta\right\Vert _{so}\leq
n\left\Vert \zeta\right\Vert _{o}$.
\end{corollary}

\begin{proof}
For all $\eta\in\mathfrak{B}$ and $a,b\in\operatorname{ball}HS_{n}$
(Hilbert-Schmidt operators) we have $\left(  \left\langle \left\langle
\zeta,\overline{\eta}\right\rangle \right\rangle a,b\right)  =\left\langle
\zeta,b^{\ast}\overline{\eta}a\right\rangle $ and $b^{\ast}\overline{\eta}%
a\in\overline{\mathfrak{B}}$. Then $\left\Vert \left\langle \left\langle
\zeta,\overline{\eta}\right\rangle \right\rangle \right\Vert =\sup\left\vert
\left(  \left\langle \left\langle \zeta,\overline{\eta}\right\rangle
\right\rangle \operatorname{ball}HS_{n},\operatorname{ball}HS_{n}\right)
\right\vert \leq\sup\left\vert \left\langle \zeta,\overline{\mathfrak{B}%
}\right\rangle \right\vert =\left\Vert \zeta\right\Vert _{so}$, which in turn
implies that $\left\Vert \zeta\right\Vert _{o}=\sup\left\Vert \left\langle
\left\langle \zeta,\overline{\mathfrak{B}}\right\rangle \right\rangle
\right\Vert \leq\left\Vert \zeta\right\Vert _{so}$. Further, using Lemma
\ref{lHS11}, we derive that $\left\Vert \zeta\right\Vert _{so}\leq
\sup\left\vert \left\langle \zeta,\sqrt{n}\overline{\mathfrak{H}}\right\rangle
\right\vert =\sqrt{n}\sup\left\vert \left\langle \zeta,\overline{\mathfrak{H}%
}\right\rangle \right\vert $. Take $\eta\in\mathfrak{H}$. But (see Remark
\ref{remlHS}) $\left\langle \zeta,\overline{\eta}\right\rangle \leq\left\Vert
\zeta\right\Vert _{2}\left\Vert \eta\right\Vert _{2}\leq\left\Vert
\zeta\right\Vert _{2}$, that is, $\left\Vert \zeta\right\Vert _{so}\leq
\sqrt{n}\left\Vert \zeta\right\Vert _{2}$. Finally, $\left\Vert \zeta
\right\Vert _{so}\leq n\left\Vert \zeta\right\Vert _{o}$ by to Lemma
\ref{lHS11}.
\end{proof}

Notice that $\left\Vert \zeta\right\Vert =\left\Vert \zeta\right\Vert
_{2}=\left\Vert \zeta\right\Vert _{o}=\left\Vert \zeta\right\Vert _{so}$ for
all $\zeta\in H=M_{1}\left(  H\right)  $.

\subsection{The unital cone $\mathfrak{c}$ in $\left(  H,e\right)  $}

As above let $\left(  H,e\right)  $ be a unital Hilbert $\ast$-space and let
$\left(  H,\overline{H}\right)  $ be the related dual $\ast$-pair. We define
the following closed (or $\sigma\left(  H,\overline{H}\right)  $-closed) cone
\[
\mathfrak{c=}\left\{  \zeta\in H_{h}:\left\Vert \zeta\right\Vert \leq\sqrt
{2}\left(  \zeta,e\right)  \right\}
\]
in $H$. Note that $e\in\mathfrak{c}$, and $\left(  \zeta,e\right)  \geq0$
whenever $\zeta\in\mathfrak{c}$. Take $\zeta\in H_{h}$ with $\zeta=\zeta
_{0}+\left(  \zeta,e\right)  e$, where $\zeta_{0}=\sum_{f\neq e}\left(
\zeta,f\right)  f\in H_{h}^{e}$, $H_{h}^{e}=H_{h}\cap H^{e}$ and
$H^{e}=\left\{  e\right\}  ^{\bot}$. Since $\left\Vert \zeta\right\Vert
^{2}=\left\Vert \zeta_{0}\right\Vert ^{2}+\left(  \zeta,e\right)  ^{2}$, we
conclude that $\zeta\in\mathfrak{c}$ iff $\left\Vert \zeta_{0}\right\Vert
\leq\left(  \zeta,e\right)  $. Thus $\zeta=\zeta_{0}+\lambda e\in\mathfrak{c}$
whenever $\lambda\geq\left\Vert \zeta_{0}\right\Vert $. The set of all states
of the cone $\mathfrak{c}$ is denoted by $S\left(  \mathfrak{c}\right)  $.
Since $\left\langle \zeta,\overline{e}\right\rangle \geq0$ for all $\zeta
\in\mathfrak{c}$, and $\left\langle e,\overline{e}\right\rangle =1$, we obtain
that $\overline{e}\in S\left(  \mathfrak{c}\right)  $. We write $\zeta\leq
\eta$ for $\zeta,\eta\in H$ whenever $\eta-\zeta\in\mathfrak{c}$.

\begin{lemma}
\label{lHSc11}The cone $\mathfrak{c}$ is a separated, unital cone such that
$-e\leq\operatorname{ball}\left(  H_{h}^{e}\right)  \leq e$ and
$\mathfrak{c\cap}H_{h}^{e}=\left\{  0\right\}  $. Thus $\mathfrak{c\cap
-c=}\left\{  0\right\}  $ and $\mathfrak{c-}e$ is an absorbent set in $H_{h}$.
In particular, $-e\leq F\leq e$, $\left(  F\backslash\left\{  e\right\}
\right)  \cap\mathfrak{c=\varnothing}$ and $H_{h}=\mathfrak{c-c}$. Moreover,
$S\left(  \mathfrak{c}\right)  \subseteq\overline{H}$ and $S\left(
\mathfrak{c}\right)  =\operatorname{ball}\left(  \overline{H_{h}^{e}}\right)
+\overline{e}$.
\end{lemma}

\begin{proof}
Take $\zeta\in\mathfrak{c\cap-c}$. Since $\left(  \pm\zeta,e\right)  \geq0$
and $\left\Vert \zeta\right\Vert \leq\sqrt{2}\left(  \zeta,e\right)  $, it
follows that $\left\Vert \zeta\right\Vert =0$ or $\zeta=0$. Note that
$\left\Vert e-\zeta_{0}\right\Vert =\sqrt{1+\left\Vert \zeta_{0}\right\Vert
^{2}}\leq\sqrt{2}=\sqrt{2}\left(  e-\zeta_{0},e\right)  $ for all $\zeta
_{0}\in\operatorname{ball}\left(  H_{h}^{e}\right)  $, which means that
$e\geq\operatorname{ball}\left(  H_{h}^{e}\right)  $. But $\operatorname{ball}%
\left(  H_{h}^{e}\right)  +e\subseteq\mathfrak{c}$ as well. Hence
$-e\leq\operatorname{ball}\left(  H_{h}^{e}\right)  \leq e$. Taking into
account that $\left(  H_{h}^{e},e\right)  =\left\{  0\right\}  $, we deduce
that $\mathfrak{c\cap}H_{h}^{e}=\left\{  0\right\}  $. Further, $\mathfrak{c-}%
e$ is an absorbent set in $H_{h}$. Indeed, for $\zeta\in H_{h}$ choose a real
$r$ with $\left\Vert \zeta_{0}\right\Vert -\left(  \zeta,e\right)  \leq r$.
Then $\zeta+re=\zeta_{0}+\left(  \left(  \zeta,e\right)  +\left(  re,e\right)
\right)  e=\zeta_{0}+\left(  \zeta+re,e\right)  e$ and $\left\Vert \zeta
_{0}\right\Vert \leq\left(  \zeta+re,e\right)  $, which means that
$\zeta+re\in\mathfrak{c}$. In particular, $\zeta=\zeta+re-re$ and
$\zeta+re,re\in\mathfrak{c}$, thereby $H_{h}=\mathfrak{c-c}$.

Further, prove that $S\left(  \mathfrak{c}\right)  =\operatorname{ball}\left(
\overline{H_{h}^{e}}\right)  +\overline{e}$. Take $\eta=\eta_{0}+e$ with
$\eta_{0}\in\operatorname{ball}\left(  H_{h}^{e}\right)  $. Then $\left\langle
e,\overline{\eta}\right\rangle =\left(  e,\eta\right)  =1$. If $\zeta
\in\mathfrak{c}$ then $\zeta=\zeta_{0}+\left(  \zeta,e\right)  e$ with
$\left\Vert \zeta_{0}\right\Vert \leq\left(  \zeta,e\right)  $. Note that
$\left\langle \zeta,\overline{\eta}\right\rangle =\left(  \zeta_{0},\eta
_{0}\right)  +\left(  \zeta,e\right)  $ and $\left\vert \left(  \zeta_{0}%
,\eta_{0}\right)  \right\vert \leq\left\Vert \zeta_{0}\right\Vert \left\Vert
\eta_{0}\right\Vert \leq\left\Vert \zeta_{0}\right\Vert \leq\left(
\zeta,e\right)  $. But $\left(  \zeta_{0},\eta_{0}\right)  $ is real,
therefore $\left\langle \zeta,\overline{\eta}\right\rangle \geq0$.
Consequently, $\overline{\eta}\in S\left(  \mathfrak{c}\right)  $. Conversely,
take $\sigma\in S\left(  \mathfrak{c}\right)  $. Using the fact $H_{h}%
=\mathfrak{c-c}$, we deduce that $\sigma$ is a $\ast$-linear functional. Take
$\zeta\in H_{h}$ with $\zeta=\zeta_{0}+\left(  \zeta,e\right)  e$, where
$\zeta_{0}\in H_{h}^{e}$. Since $-e\leq\left\Vert \zeta_{0}\right\Vert
^{-1}\zeta_{0}\leq e$, we derive that $\left\vert \sigma\left(  \zeta
_{0}\right)  \right\vert \leq\left\Vert \zeta_{0}\right\Vert $, which in turn
implies that $\left\vert \sigma\left(  \zeta\right)  \right\vert =\left\vert
\sigma\left(  \zeta_{0}\right)  +\left(  \zeta,e\right)  \right\vert
\leq\left\Vert \zeta_{0}\right\Vert +\left\vert \left(  \zeta,e\right)
\right\vert \leq2\left\Vert \zeta\right\Vert $. In the general case of
$\zeta\in H$ we derive that $\left\vert \sigma\left(  \zeta\right)
\right\vert \leq\left\vert \sigma\left(  \operatorname{Re}\zeta\right)
+if\left(  \operatorname{Im}\zeta\right)  \right\vert =\left(  \sigma\left(
\operatorname{Re}\zeta\right)  ^{2}+\sigma\left(  \operatorname{Im}%
\zeta\right)  ^{2}\right)  ^{1/2}\leq2\left(  \left\Vert \operatorname{Re}%
\zeta\right\Vert ^{2}+\left\Vert \operatorname{Im}\zeta\right\Vert
^{2}\right)  ^{1/2}\leq2\sqrt{2}\left\Vert \zeta\right\Vert $, which means
that $\sigma$ is a bounded linear functional on $H$, that is, $\sigma
=\overline{\eta}$ for a certain $\eta\in H$. But $\eta=\eta_{0}+\left(
\eta,e\right)  e=\eta_{0}+\sigma\left(  e\right)  e=\eta_{0}+e$, where
$\eta_{0}\in H^{e}$. Prove that $\eta_{0}\in\operatorname{ball}\left(
H_{h}^{e}\right)  $. Take any $\zeta_{0}\in H_{h}^{e}$, and put $\zeta
=\zeta_{0}+\left\Vert \zeta_{0}\right\Vert e\in\mathfrak{c}$. Then
$\sigma\left(  \zeta\right)  =\left\langle \zeta,\overline{\eta}\right\rangle
=\left(  \zeta_{0},\eta_{0}\right)  +\left\Vert \zeta_{0}\right\Vert \geq0$,
which in turn implies that $\left(  \zeta_{0},\eta_{0}\right)  \in\mathbb{R}$.
But $\left(  \zeta_{0},\eta_{0}\right)  =\left(  \zeta_{0},\operatorname{Re}%
\eta_{0}\right)  -i\left(  \zeta_{0},\operatorname{Im}\eta_{0}\right)  $,
therefore $\operatorname{Im}\eta_{0}\perp H_{h}^{e}$. Since $H_{h}^{e}$ is a
real Hilbert space and $\operatorname{Im}\eta_{0}\in H_{h}^{e}$, we conclude
that $\operatorname{Im}\eta_{0}=0$. Thus $\eta_{0}\in H_{h}^{e}$ and $\left(
\zeta_{0},\eta_{0}\right)  \geq-\left\Vert \zeta_{0}\right\Vert $ for all
$\zeta_{0}\in H_{h}^{e}$. In particular, $\left(  -\zeta_{0},\eta_{0}\right)
\geq-\left\Vert \zeta_{0}\right\Vert $ or $\left\vert \left(  \zeta_{0}%
,\eta_{0}\right)  \right\vert \leq\left\Vert \zeta_{0}\right\Vert $.
Consequently, $\left\Vert \eta_{0}\right\Vert =\sup\left\vert \left(
\operatorname{ball}\left(  H_{h}^{e}\right)  ,\eta_{0}\right)  \right\vert
\leq1$, which means that $\eta_{0}\in\operatorname{ball}\left(  H_{h}%
^{e}\right)  $. Thus $S\left(  \mathfrak{c}\right)  =\operatorname{ball}%
\left(  \overline{H_{h}^{e}}\right)  +\overline{e}$.
\end{proof}

By symmetry we have the cone $\overline{\mathfrak{c}}=\left\{  \overline
{\zeta}\in\overline{H}_{h}:\left\Vert \overline{\zeta}\right\Vert \leq\sqrt
{2}\left(  \overline{\zeta},\overline{e}\right)  \right\}  $ in $\left(
\overline{H},\overline{e}\right)  $, and $S\left(  \overline{\mathfrak{c}%
}\right)  =\operatorname{ball}\left(  H_{h}^{e}\right)  +e$ thanks to Lemma
\ref{lHSc11}. Note that $S\left(  \overline{\mathfrak{c}}\right)
\subseteq\mathfrak{c}$ and $S\left(  \mathfrak{c}\right)  \subseteq
\overline{\mathfrak{c}}$.

\begin{remark}
\label{remDPC}Note that $\mathfrak{c=}\left\{  \zeta\in H_{h}:\left\langle
\zeta,S\left(  \mathfrak{c}\right)  \right\rangle \geq0\right\}  $. Indeed,
take $\zeta\in H_{h}$ with $\left\langle \zeta,S\left(  \mathfrak{c}\right)
\right\rangle \geq0$. Then $\left(  \zeta,e\right)  =\left\langle
\zeta,\overline{e}\right\rangle \geq0$ and $\left(  \zeta,-\left\Vert
\zeta_{0}\right\Vert \zeta_{0}+e\right)  \geq0$, which in turn implies that
$\left\Vert \zeta_{0}\right\Vert =\left(  \zeta_{0},\left\Vert \zeta
_{0}\right\Vert ^{-1}\zeta_{0}\right)  \leq\left(  \zeta,e\right)  $. Hence
$\zeta\in\mathfrak{c}$. The inclusion $\mathfrak{c\subseteq}\left\{  \zeta\in
H_{h}:\left\langle \zeta,S\left(  \mathfrak{c}\right)  \right\rangle
\geq0\right\}  $ is obvious.
\end{remark}

Consider the norm $\left\Vert \zeta\right\Vert _{e}=\sup\left\vert
\left\langle \zeta,S\left(  \mathfrak{c}\right)  \right\rangle \right\vert $,
$\zeta\in H$ associated with the unital cone $\mathfrak{c}$ (see Subsection
\ref{subsecUQC}).

\begin{proposition}
\label{propOKE21}The norm $\left\Vert \cdot\right\Vert _{e}$ on $H$ is a
unital $\ast$-norm, which is equivalent to the original norm of $H$. Moreover,
$\left\Vert \cdot\right\Vert _{e}$ is an order norm in the sense of
$\left\Vert \zeta\right\Vert _{e}=\inf\left\{  r>0:-re\leq\zeta\leq
re\right\}  $ for all $\zeta\in H_{h}$. In particular, $\min\mathfrak{c}%
=S\left(  \mathfrak{c}\right)  ^{\boxdot}$ and $\max\mathfrak{c}%
=\mathfrak{c}^{\boxdot\boxdot}$ with respect to the dual $\ast$-pair $\left(
H,\overline{H}\right)  $.
\end{proposition}

\begin{proof}
Using Lemma \ref{lHSc11}, we obtain that $\left\Vert \zeta\right\Vert _{e}%
\leq\left\Vert \zeta\right\Vert \sup\left\Vert S\left(  \mathfrak{c}\right)
\right\Vert =\left\Vert \zeta\right\Vert \sup\left\Vert \operatorname{ball}%
\left(  \overline{H_{h}^{e}}\right)  +\overline{e}\right\Vert =2\left\Vert
\zeta\right\Vert $, $\zeta\in H$ and $\left\Vert e\right\Vert _{e}%
=\sup\left\vert \left\langle e,S\left(  \mathfrak{c}\right)  \right\rangle
\right\vert =1$. Moreover, $\left\Vert \zeta^{\ast}\right\Vert _{e}%
=\sup\left\vert \left\langle \zeta^{\ast},S\left(  \mathfrak{c}\right)
\right\rangle \right\vert =\sup\left\vert \left\langle \zeta,S\left(
\mathfrak{c}\right)  \right\rangle ^{\ast}\right\vert =\left\Vert
\zeta\right\Vert _{e}$, $\left\Vert \operatorname{Re}\zeta\right\Vert
_{e}=\sup\left\vert \left\langle \operatorname{Re}\zeta,S\left(
\mathfrak{c}\right)  \right\rangle \right\vert =\sup\left\vert
\operatorname{Re}\left\langle \zeta,S\left(  \mathfrak{c}\right)
\right\rangle \right\vert \leq\left\Vert \zeta\right\Vert _{e}$. Similarly,
$\left\Vert \operatorname{Im}\zeta\right\Vert _{e}\leq\left\Vert
\zeta\right\Vert _{e}$. Now take a nonzero $\zeta\in H_{h}$ with its expansion
$\zeta=\zeta_{0}+\lambda e$, where $\lambda=\left(  \zeta,e\right)  $. Then
\begin{align*}
\left\Vert \zeta\right\Vert  &  =\left\Vert \zeta\right\Vert ^{-1}\left(
\zeta,\zeta_{0}+\lambda e\right)  =\left\Vert \zeta\right\Vert ^{-1}\left(
\zeta,\left\Vert \zeta_{0}\right\Vert \left(  \left\Vert \zeta_{0}\right\Vert
^{-1}\zeta_{0}+e\right)  +\left(  \lambda-\left\Vert \zeta_{0}\right\Vert
\right)  e\right) \\
&  =\left\Vert \zeta_{0}\right\Vert \left\Vert \zeta\right\Vert ^{-1}\left(
\zeta,\left\Vert \zeta_{0}\right\Vert ^{-1}\zeta_{0}+e\right)  +\left(
\lambda-\left\Vert \zeta_{0}\right\Vert \right)  \left\Vert \zeta\right\Vert
^{-1}\left(  \zeta,e\right) \\
&  \leq\left\Vert \zeta_{0}\right\Vert \left\Vert \zeta\right\Vert ^{-1}%
\sup\left\langle \zeta,\operatorname{ball}\left(  \overline{H_{h}^{e}}\right)
+\overline{e}\right\rangle +\left\vert \lambda-\left\Vert \zeta_{0}\right\Vert
\right\vert \left\Vert \zeta\right\Vert ^{-1}\left\langle \zeta,\overline
{e}\right\rangle \\
&  \leq\left\Vert \zeta\right\Vert _{e}+\left(  \left\vert \lambda\right\vert
+\left\Vert \zeta_{0}\right\Vert \right)  \left\Vert \zeta\right\Vert
^{-1}\left\langle \zeta,\overline{e}\right\rangle \leq\left\Vert
\zeta\right\Vert _{e}+2\left\langle \zeta,\overline{e}\right\rangle
\leq3\left\Vert \zeta\right\Vert _{e}\text{,}%
\end{align*}
which in turn implies that $\left\Vert \zeta\right\Vert =\left(  \left\Vert
\operatorname{Re}\zeta\right\Vert ^{2}+\left\Vert \operatorname{Im}%
\zeta\right\Vert ^{2}\right)  ^{1/2}\leq5\left\Vert \zeta\right\Vert _{e}$.
Notice that if $\zeta_{0}\in H_{h}^{e}$ then
\begin{equation}
\left\Vert \zeta_{0}\right\Vert =\left(  \zeta_{0},\left\Vert \zeta
_{0}\right\Vert ^{-1}\zeta_{0}\right)  =\left(  \zeta_{0},\left\Vert \zeta
_{0}\right\Vert ^{-1}\zeta_{0}+e\right)  =\sup\left\vert \left\langle
\zeta_{0},\operatorname{ball}\left(  \overline{H_{h}^{e}}\right)
+\overline{e}\right\rangle \right\vert =\left\Vert \zeta_{0}\right\Vert _{e}.
\label{X0}%
\end{equation}
Now put $\alpha=\inf\left\{  r>0:-re\leq\zeta\leq re\right\}  $ for $\zeta\in
H_{h}$. Then (see Remark \ref{remDPC})
\begin{align*}
\alpha &  =\inf\left\{  r>0:re\pm\zeta\geq0\right\}  =\inf\left\{
r>0:\left\langle re\pm\zeta,\overline{\eta}\right\rangle \geq0,\overline{\eta
}\in S\left(  \mathfrak{c}\right)  \right\} \\
&  =\inf\left\{  r>0:r\pm\left\langle \zeta,\overline{\eta}\right\rangle
\geq0,\overline{\eta}\in S\left(  \mathfrak{c}\right)  \right\}  =\inf\left\{
r>0:\sup\left\vert \left\langle \zeta,S\left(  \mathfrak{c}\right)
\right\rangle \right\vert \leq r\right\} \\
&  =\sup\left\vert \left\langle \zeta,S\left(  \mathfrak{c}\right)
\right\rangle \right\vert =\left\Vert \zeta\right\Vert _{e}\text{,}%
\end{align*}
that is, $\left\Vert \cdot\right\Vert _{e}$ is an order norm on $H$.
Consequently, $\overline{H}$ is the normed dual of $H$ equipped with the norm
$\left\Vert \cdot\right\Vert _{e}$. It follows that $\min\mathfrak{c}=S\left(
\mathfrak{c}\right)  ^{\boxdot}$ and $\max\mathfrak{c}=\mathfrak{c}%
^{\boxdot\boxdot}$ with respect to the dual $\ast$-pair $\left(
H,\overline{H}\right)  $ (see Subsection \ref{subsecUQC}).
\end{proof}

\begin{remark}
As we have seen from the proof of Proposition \ref{propOKE21} that
$5^{-1}\left\Vert \zeta\right\Vert \leq\left\Vert \zeta\right\Vert _{e}%
\leq2\left\Vert \zeta\right\Vert $ for all $\zeta\in H$. A similar estimations
are obtained below in Lemma \ref{lBP2} for the related matrix norms.
\end{remark}

\begin{corollary}
\label{cormincbar}If $S\left(  \overline{\mathfrak{c}}\right)  ^{c}\cap H$ is
the cone in $H$ generated by $S\left(  \overline{\mathfrak{c}}\right)  $ then
$\mathfrak{c=}S\left(  \overline{\mathfrak{c}}\right)  ^{c}\cap H=\mathbb{R}%
_{+}S\left(  \overline{\mathfrak{c}}\right)  $. Thus $\min\mathfrak{c}%
=\overline{\mathfrak{c}}^{\boxdot}$ in $M\left(  H\right)  _{h}$,
$\min\overline{\mathfrak{c}}=\mathfrak{c}^{\boxdot}$ in $M\left(  \overline
{H}\right)  _{h}$, and both $S\left(  \overline{\mathfrak{c}}\right)  $ and
$\mathfrak{c}$ generate the same closed quantum cone $\mathfrak{c}%
^{\boxdot\boxdot}$ in $M\left(  H\right)  _{h}$. Thus $\max\mathfrak{c=}%
S\left(  \overline{\mathfrak{c}}\right)  ^{\boxdot\boxdot}$ in $M\left(
H\right)  _{h}$, and $\max\overline{\mathfrak{c}}=S\left(  \mathfrak{c}%
\right)  ^{\boxdot\boxdot}$ in $M\left(  \overline{H}\right)  _{h}$. In
particular, every $\zeta\in M_{n}\left(  H\right)  _{h}$ with $\sum_{f\neq
e}\left\vert \left\langle \left\langle \zeta,\overline{f}\right\rangle
\right\rangle \right\vert \leq\left\langle \left\langle \zeta,\overline
{e}\right\rangle \right\rangle $ in $M_{n}$ belongs to $\max\mathfrak{c}$.
\end{corollary}

\begin{proof}
If $\zeta=\zeta_{0}+\lambda e\in\mathfrak{c}$ with $\left\Vert \zeta
_{0}\right\Vert \leq\lambda$ then
\[
\zeta=\lambda\left(  \lambda^{-1}\zeta_{0}+e\right)  \in\lambda\left(
\operatorname{ball}\left(  H_{h}^{e}\right)  +e\right)  \subseteq\lambda
S\left(  \overline{\mathfrak{c}}\right)  \subseteq\mathbb{R}_{+}S\left(
\overline{\mathfrak{c}}\right)  \subseteq S\left(  \overline{\mathfrak{c}%
}\right)  ^{c}\cap H
\]
by Lemma \ref{lHSc11}. Since $S\left(  \overline{\mathfrak{c}}\right)
^{c}\cap H\subseteq\mathfrak{c}$, the equalities $S\left(  \overline
{\mathfrak{c}}\right)  ^{c}\cap H=\mathbb{R}_{+}S\left(  \overline
{\mathfrak{c}}\right)  =\mathfrak{c}$ follow. Using Proposition
\ref{propOKE21}, we deduce that $\min\overline{\mathfrak{c}}=S\left(
\overline{\mathfrak{c}}\right)  ^{\boxdot}=\left(  \mathbb{R}_{+}S\left(
\overline{\mathfrak{c}}\right)  \right)  ^{\boxdot}=\mathfrak{c}^{\boxdot}$,
and $\max\mathfrak{c}=\mathfrak{c}^{\boxdot\boxdot}=S\left(  \overline
{\mathfrak{c}}\right)  ^{\boxdot\boxdot}$, which is the closed quantum cone in
$M\left(  H\right)  _{h}$ generated by $S\left(  \overline{\mathfrak{c}%
}\right)  $ or $\mathfrak{c}$. By symmetry, we have $\min\mathfrak{c}%
=\overline{\mathfrak{c}}^{\boxdot}$ and $\max\overline{\mathfrak{c}}=S\left(
\mathfrak{c}\right)  ^{\boxdot\boxdot}$.

Finally, take $\zeta\in M_{n}\left(  H\right)  _{h}$ with $\sum_{f\neq
e}\left\vert \left\langle \left\langle \zeta,\overline{f}\right\rangle
\right\rangle \right\vert \leq\left\langle \left\langle \zeta,\overline
{e}\right\rangle \right\rangle $ in $M_{n}$. Each $\left\langle \left\langle
\zeta,\overline{f}\right\rangle \right\rangle $ is diagonalizable being a
hermitian matrix from $M_{n}$, that is, $\left\langle \left\langle
\zeta,\overline{f}\right\rangle \right\rangle =\mu_{f}^{\ast}v_{f}\left\vert
d_{f}\right\vert \mu_{f}$, where $d_{f}$ is a diagonal real matrix with its
polar decomposition $d_{f}=v_{f}\left\vert d_{f}\right\vert $, $\left\Vert
v_{f}\right\Vert \leq1$ ($v_{f}$ is a diagonal matrix) and a unitary $\mu_{f}%
$. Then
\begin{align*}
\left\langle \left\langle \zeta,\overline{f}\right\rangle \right\rangle
f^{\oplus n}  &  =\mu_{f}^{\ast}\left\vert d_{f}\right\vert \left(
v_{f}f^{\oplus n}+e^{\oplus n}\right)  \mu_{f}-\mu_{f}^{\ast}\left\vert
d_{f}\right\vert e^{\oplus n}\mu_{f}\\
&  =\mu_{f}^{\ast}\left\vert d_{f}\right\vert ^{1/2}\left(  v_{f}f^{\oplus
n}+e^{\oplus n}\right)  \left\vert d_{f}\right\vert ^{1/2}\mu_{f}-\left\vert
\left\langle \left\langle \zeta,\overline{f}\right\rangle \right\rangle
\right\vert e^{\oplus n}%
\end{align*}
and $v_{f}f^{\oplus n}+e^{\oplus n}\in\left(  \operatorname{ball}\left(
H_{h}^{e}\right)  +e\right)  ^{\oplus n}=S\left(  \overline{\mathfrak{c}%
}\right)  ^{\oplus n}$. It follows that
\begin{align*}
\zeta &  =\sum_{f\neq e}\left\langle \left\langle \zeta,\overline
{f}\right\rangle \right\rangle f^{\oplus n}+\left\langle \left\langle
\zeta,\overline{e}\right\rangle \right\rangle e^{\oplus n}=\lim_{\lambda}%
\sum_{\lambda}\left\langle \left\langle \zeta,\overline{f}\right\rangle
\right\rangle f^{\oplus n}+\left\langle \left\langle \zeta,\overline
{e}\right\rangle \right\rangle e^{\oplus n}\\
&  =\lim_{\lambda}\sum_{\lambda}\mu_{f}^{\ast}\left\vert d_{f}\right\vert
^{1/2}\left(  v_{f}f^{\oplus n}+e^{\oplus n}\right)  \left\vert d_{f}%
\right\vert ^{1/2}\mu_{f}+\left(  \left\langle \left\langle \zeta,\overline
{e}\right\rangle \right\rangle -\sum_{\lambda}\left\vert \left\langle
\left\langle \zeta,\overline{f}\right\rangle \right\rangle \right\vert
\right)  e^{\oplus n}\\
&  \in\left(  S\left(  \overline{\mathfrak{c}}\right)  ^{c}\right)
^{-}=S\left(  \overline{\mathfrak{c}}\right)  ^{\boxdot\boxdot}=\mathfrak{c}%
^{\boxdot\boxdot}=\max\mathfrak{c},
\end{align*}
where $\lambda$ is running over all finite subsets of $F\backslash\left\{
e\right\}  $. Whence $\zeta\in\max\mathfrak{c}$.
\end{proof}

\begin{corollary}
\label{corabsIn}If $S\left(  \mathfrak{c}\right)  ^{\circ}$ is the polar of
$S\left(  \mathfrak{c}\right)  $ with respect to the duality $\left(
H,\overline{H}\right)  $ then $S\left(  \mathfrak{c}\right)  ^{\circ}\cap
H_{h}=\operatorname{abc}\left\{  \left\{  e\right\}  \cup\operatorname{ball}%
\left(  H_{h}^{e}\right)  \right\}  $ in the real Hilbert space $H_{h}$, and
\[
\operatorname{abc}\left\{  \left\{  e\right\}  \cup\operatorname{ball}\left(
H^{e}\right)  \right\}  \subseteq S\left(  \mathfrak{c}\right)  ^{\circ
}\subseteq2\operatorname{abc}\left\{  \left\{  e\right\}  \cup
\operatorname{ball}\left(  H_{h}^{e}\right)  \right\}
\]
in the Hilbert space $H$, where $\operatorname{abc}$ indicates to the
absolutely convex hull of a given set.
\end{corollary}

\begin{proof}
First note that $\left\langle e,S\left(  \mathfrak{c}\right)  \right\rangle
=\left\langle e,\operatorname{ball}\left(  \overline{H_{h}^{e}}\right)
+\overline{e}\right\rangle =\left\langle e,\overline{e}\right\rangle =1$ and
\[
\sup\left\vert \left\langle \operatorname{ball}\left(  H^{e}\right)  ,S\left(
\mathfrak{c}\right)  \right\rangle \right\vert =\sup\left\vert \left\langle
\operatorname{ball}\left(  H^{e}\right)  ,\operatorname{ball}\left(
\overline{H_{h}^{e}}\right)  \right\rangle \right\vert \leq1,
\]
therefore $\operatorname{abc}\left\{  \left\{  e\right\}  \cup
\operatorname{ball}\left(  H^{e}\right)  \right\}  \subseteq S\left(
\mathfrak{c}\right)  ^{\circ}$. Note also that $\operatorname{abc}\left\{
\left\{  e\right\}  \cup\operatorname{ball}\left(  H^{e}\right)  \right\}  $
is closed. Indeed, if $\zeta=\lim_{n}\left(  \lambda_{n}\zeta_{n}+\mu
_{n}e\right)  $ for $\left(  \zeta_{n}\right)  _{n}\subseteq
\operatorname{ball}\left(  H^{e}\right)  $, $\lambda_{n},\mu_{n}\in\mathbb{C}$
with $\left\vert \lambda_{n}\right\vert +\left\vert \mu_{n}\right\vert \leq1$,
then $\mu=\lim_{n}\mu_{n}$ and $\zeta-\mu e=\lim_{n}\lambda_{n}\zeta_{n}=\eta$
and $\left\Vert \eta\right\Vert =\lim_{n}\left\vert \lambda_{n}\right\vert
\left\Vert \zeta_{n}\right\Vert \leq1$, Thus $\eta\in\operatorname{ball}%
\left(  H^{e}\right)  $ and $\zeta=\eta+\mu e=\left\Vert \eta\right\Vert
\left(  \left\Vert \eta\right\Vert ^{-1}\eta\right)  +\mu e$ with $\left\Vert
\eta\right\Vert +\left\vert \mu\right\vert =\lim_{n}\left(  \left\vert
\lambda_{n}\right\vert \left\Vert \zeta_{n}\right\Vert +\left\vert \mu
_{n}\right\vert \right)  \leq\lim_{n}\left(  \left\vert \lambda_{n}\right\vert
+\left\vert \mu_{n}\right\vert \right)  \leq1$, that is, $\zeta\in
\operatorname{abc}\left\{  \left\{  e\right\}  \cup\operatorname{ball}\left(
H^{e}\right)  \right\}  $.

Take $\zeta\in S\left(  \mathfrak{c}\right)  ^{\circ}\cap H_{h}$ with
$\zeta=\zeta_{0}+\left(  \zeta,e\right)  e$, $\zeta_{0}\in H_{h}^{e}$ and
$\left(  \zeta,e\right)  \in\mathbb{R}$. Then $s=\sup\left\vert \left\langle
\zeta,S\left(  \mathfrak{c}\right)  \right\rangle \right\vert \leq1$. If
$\zeta_{0}=0$ then $\zeta=\lambda e$ with $\left\vert \lambda\right\vert
=\left\vert \left\langle \zeta,\overline{e}\right\rangle \right\vert \leq s$,
therefore $\zeta\in\operatorname{abc}\left\{  \left\{  e\right\}
\cup\operatorname{ball}\left(  H^{e}\right)  \right\}  $. Assume that
$\zeta_{0}\neq0$. Then $\left\vert \left(  \zeta,e\right)  -\left(  \zeta
_{0},\eta\right)  \right\vert =\left\vert \left\langle \zeta,-\overline{\eta
}+\overline{e}\right\rangle \right\vert \leq s$ for all $\eta\in
\operatorname{ball}\left(  H_{h}^{e}\right)  $. In particular, $\left\vert
\left(  \zeta,e\right)  -r\left\Vert \zeta_{0}\right\Vert \right\vert
=\left\vert \left(  \zeta,e\right)  -\left(  \zeta_{0},r\left\Vert \zeta
_{0}\right\Vert ^{-1}\zeta_{0}\right)  \right\vert \leq s$ for all
$r\in\mathbb{R}$, $\left\vert r\right\vert \leq1$. It follows that $\left\vert
\left(  \zeta,e\right)  \right\vert \leq s$ (for $r=0$) and $\left\Vert
\zeta_{0}\right\Vert \leq s-\left\vert \left(  \zeta,e\right)  \right\vert $
(for $r=\pm1$). Hence $\zeta=\left\Vert \zeta_{0}\right\Vert \left(
\left\Vert \zeta_{0}\right\Vert ^{-1}\zeta_{0}\right)  +\left(  \zeta
,e\right)  e$ with $\left\Vert \zeta_{0}\right\Vert +\left\vert \left(
\zeta,e\right)  \right\vert \leq s\leq1$, that is, $\zeta\in\operatorname{abc}%
\left\{  \left\{  e\right\}  \cup\operatorname{ball}\left(  H_{h}^{e}\right)
\right\}  $ in $H_{h}$. Hence $S\left(  \mathfrak{c}\right)  ^{\circ}\cap
H_{h}=\operatorname{abc}\left\{  \left\{  e\right\}  \cup\operatorname{ball}%
\left(  H_{h}^{e}\right)  \right\}  $.

In the case of a non-hermitian $\zeta\in S\left(  \mathfrak{c}\right)
^{\circ}$ we have $\sup\left\vert \left\langle \operatorname{Re}\zeta,S\left(
\mathfrak{c}\right)  \right\rangle \right\vert =\sup\left\vert
\operatorname{Re}\left\langle \zeta,S\left(  \mathfrak{c}\right)
\right\rangle \right\vert \leq\sup\left\vert \left\langle \zeta,S\left(
\mathfrak{c}\right)  \right\rangle \right\vert \leq1$, which means that
$\operatorname{Re}\zeta\in S\left(  \mathfrak{c}\right)  ^{\circ}\cap H_{h}$.
Similarly, $\operatorname{Im}\zeta\in S\left(  \mathfrak{c}\right)  ^{\circ
}\cap H_{h}$. Finally, $\zeta=\left\Vert \zeta_{0}\right\Vert \left(
\left\Vert \zeta_{0}\right\Vert ^{-1}\zeta_{0}\right)  +\left(  \zeta
,e\right)  e$ and $\left\Vert \zeta_{0}\right\Vert +\left\vert \left(
\zeta,e\right)  \right\vert \leq\left\Vert \operatorname{Re}\zeta
_{0}\right\Vert +\left\Vert \operatorname{Im}\zeta_{0}\right\Vert +\left\vert
\left(  \operatorname{Re}\zeta,e\right)  \right\vert +\left\vert \left(
\operatorname{Im}\zeta,e\right)  \right\vert \leq2$, that is, $\zeta
\in2\operatorname{abc}\left\{  \left\{  e\right\}  \cup\operatorname{ball}%
\left(  H^{e}\right)  \right\}  $. Actually, $\zeta=\operatorname{Re}%
\zeta+i\operatorname{Im}\zeta\in2\operatorname{abc}\left\{  \operatorname{Re}%
\zeta,\operatorname{Im}\zeta\right\}  \in2\operatorname{abc}\left\{  \left\{
e\right\}  \cup\operatorname{ball}\left(  H_{h}^{e}\right)  \right\}  $ in the
Hilbert space $H$.
\end{proof}

Note that $S\left(  \mathfrak{c}\right)  ^{\circ}$ is the unit ball of the
norm $\left\Vert \cdot\right\Vert _{e}$. Based on Corollary \ref{corabsIn}, we
conclude that $\left\Vert \cdot\right\Vert _{e}$ is equivalent to the
Minkowski functional of the closed set $\operatorname{abc}\left\{  \left\{
e\right\}  \cup\operatorname{ball}\left(  H^{e}\right)  \right\}  $.

\subsection{The min and max quantizations of $\mathfrak{c}$}

As above let $H$ be a Hilbert $\ast$-space $H$ with a unit $e$ and related
unital cone $\mathfrak{c}$. Take $\zeta\in M_{n}\left(  H\right)  _{h}$, which
is identified with the bounded $\ast$-linear mapping $\zeta:\overline
{H}\rightarrow M_{n}$ such that $\zeta\left(  \overline{\eta}\right)
=\left\langle \left\langle \zeta,\overline{\eta}\right\rangle \right\rangle $
for all $\eta\in H$. Put $b=\zeta\left(  \overline{e}\right)  \in M_{n}$,
which is a hermitian matrix. We say that $\overline{e}$ is \textit{a dominant
point for} $\zeta$ if $b\geq0$ and $\zeta\left(  \overline{\eta}\right)
=\alpha\left(  \eta_{0}\right)  b\alpha\left(  \eta_{0}\right)  ^{\ast}$,
$\overline{\eta}=\overline{\eta_{0}}+\overline{e}\in S\left(  \mathfrak{c}%
\right)  $ for a certain continuous mapping $\alpha:\operatorname{ball}\left(
\overline{H_{h}^{e}}\right)  \rightarrow\operatorname{ball}M_{n}$ such that
$\alpha\left(  0\right)  =I_{n}$.

\begin{proposition}
\label{pmin11}For each $n$ we have
\[
\left(  \min\mathfrak{c}\right)  \cap M_{n}\left(  H\right)  _{h}=\left\{
\zeta\in M_{n}\left(  H\right)  _{h}:\left\Vert a^{\ast}\zeta a\right\Vert
\leq\sqrt{2}\left(  a^{\ast}\zeta a,e\right)  ,a\in M_{n,1}\right\}  .
\]
Moreover, for $\zeta\in M_{n}\left(  H\right)  _{h}$ we have $\zeta\in
\min\mathfrak{c}$ iff $\overline{e}$ is a dominant point for $\zeta$.
\end{proposition}

\begin{proof}
Take $\zeta\in M_{n}\left(  H\right)  _{h}$. By Proposition \ref{propOKE21},
$\min\mathfrak{c=}S\left(  \mathfrak{c}\right)  ^{\boxdot}$, therefore
$\zeta\in\min\mathfrak{c}$ iff $\left\langle \left\langle \zeta,\overline
{\eta_{0}}+\overline{e}\right\rangle \right\rangle \geq0$ in $M_{n}$ for all
$\eta_{0}\in\operatorname{ball}\left(  H_{h}^{e}\right)  $. Since $\left(
\left\langle \left\langle \zeta,\overline{\eta_{0}}+\overline{e}\right\rangle
\right\rangle a,a\right)  =\left\langle a^{\ast}\zeta a,\overline{\eta_{0}%
}+\overline{e}\right\rangle $, $a\in M_{n,1}$, we derive that $\zeta\in
\min\mathfrak{c}$ iff $a^{\ast}\zeta a\in S\left(  \mathfrak{c}\right)
^{\boxdot}\cap H$ for all $a\in M_{n,1}$. But $S\left(  \mathfrak{c}\right)
^{\boxdot}\cap H=\mathfrak{c}$, which is a classical version of the Unital
Bipolar Theorem \ref{tUBP} (see also Remark \ref{remDPC}).

Now take $\zeta\in\left(  \min\mathfrak{c}\right)  \cap M_{n}\left(  H\right)
_{h}$ with its canonical expansion $\zeta=\zeta_{0}+be^{\oplus n}$ with
$b=\left\langle \left\langle \zeta,\overline{e}\right\rangle \right\rangle $
and $\zeta_{0}\in M_{n}\left(  H\right)  _{h}$, and $\left\langle \left\langle
\zeta_{0},\overline{e}\right\rangle \right\rangle =0$. Note that $\left(
ba,a\right)  =\left(  \left\langle \left\langle \zeta,\overline{e}%
\right\rangle \right\rangle a,a\right)  =\left\langle a^{\ast}\zeta
a,\overline{e}\right\rangle =\left(  a^{\ast}\zeta a,e\right)  \geq0$ for all
$a\in M_{n,1}$, which means that $b\geq0$. Moreover, $\zeta_{0}$ defines a
$\ast$-linear mapping $\zeta_{0}:\overline{H^{e}}\rightarrow M_{n}$ with
$\zeta_{0}\left(  \overline{\eta_{0}}\right)  =\left\langle \left\langle
\zeta_{0},\overline{\eta_{0}}\right\rangle \right\rangle $ for all $\eta
_{0}\in H^{e}$. Since $a^{\ast}\zeta_{0}a=\left(  a^{\ast}\zeta a\right)
_{0}$ and $a^{\ast}\zeta a\in\mathfrak{c}$, it follows that
\[
\left\vert \left(  \zeta_{0}\left(  \overline{\eta_{0}}\right)  a,a\right)
\right\vert =\left\vert \left\langle a^{\ast}\zeta_{0}a,\overline{\eta_{0}%
}\right\rangle \right\vert \leq\left\Vert a^{\ast}\zeta_{0}a\right\Vert
\left\Vert \overline{\eta}_{0}\right\Vert \leq\left(  a^{\ast}\zeta
a,e\right)  \left\Vert \overline{\eta_{0}}\right\Vert =\left(  \left\Vert
\overline{\eta_{0}}\right\Vert ba,a\right)
\]
for all $a$, which in turn implies that $-b\leq\zeta_{0}\left(  \overline
{\eta_{0}}\right)  \leq b$ whenever $\overline{\eta_{0}}\in\operatorname{ball}%
\overline{H_{h}^{e}}$. In particular, $0\leq\zeta_{0}\left(  \overline
{\eta_{0}}\right)  +b\leq2b$. Hence $\zeta_{0}\left(  \overline{\eta_{0}%
}\right)  +b=\alpha\left(  \eta_{0}\right)  b\alpha\left(  \eta_{0}\right)
^{\ast}$ for a unique $\alpha\left(  \eta_{0}\right)  \in M_{n}$ such that
\[
\left\vert \alpha\left(  \eta_{0}\right)  \right\vert ^{2}=\lim_{k}\left(
b+k^{-1}\right)  ^{-1/2}\left(  \zeta_{0}\left(  \overline{\eta_{0}}\right)
+b\right)  \left(  b+k^{-1}\right)  ^{-1/2}%
\]
(see \cite[1.6. Lemma 2]{Dik}). Note that $\alpha\left(  \eta_{0}\right)
:b\left(  \mathbb{C}^{n}\right)  \rightarrow\mathbb{C}^{n}$ is a well defined
linear mapping such that $\alpha\left(  \eta_{0}\right)  b^{1/2}=\left(
\zeta_{0}\left(  \overline{\eta_{0}}\right)  +b\right)  ^{1/2}$, $\left\Vert
\alpha\left(  \eta_{0}\right)  \right\Vert \leq\sqrt{2}$ and $\alpha\left(
\eta_{0}\right)  \left(  \ker\left(  b\right)  \right)  =\left\{  0\right\}
$. Notice that $b^{1/2}\left(  \mathbb{C}^{n}\right)  ^{\perp}=b\left(
\mathbb{C}^{n}\right)  ^{\perp}=\ker\left(  b\right)  $. For the fixed
$a=a_{1}+b^{1/2}\left(  a_{2}\right)  $ with $a_{1}\in\ker\left(  b\right)  $
and $b^{1/2}\left(  a_{2}\right)  \in\operatorname{im}\left(  b^{1/2}\right)
$, we obtain that
\[
\left\Vert \alpha\left(  \eta_{0}\right)  a\right\Vert =\left\Vert
\alpha\left(  \eta_{0}\right)  \left(  b^{1/2}\left(  a_{2}\right)  \right)
\right\Vert =\left\Vert \left(  \zeta_{0}\left(  \overline{\eta_{0}}\right)
+b\right)  ^{1/2}a_{2}\right\Vert \text{,}%
\]
which means (see below Remark \ref{remICR}) that $\alpha:\operatorname{ball}%
\left(  \overline{H_{h}^{e}}\right)  \rightarrow\operatorname{ball}M_{n}$ is a
continuous mapping (one can equip $\operatorname{ball}M_{n}$ with the strong
operator topology SOT). Since $\alpha\left(  \eta_{0}\right)  $ is uniquely
defined and $\zeta_{0}$ is linear, we derive that $\alpha\left(  0\right)
=I_{n}$. Moreover,
\[
\zeta_{0}\left(  \overline{\eta_{0}}\right)  +b=\zeta_{0}\left(
\overline{\eta}_{0}\right)  +\zeta\left(  \overline{e}\right)  =\left\langle
\left\langle \zeta_{0}+be^{\oplus n},\overline{\eta_{0}}+\overline
{e}\right\rangle \right\rangle =\left\langle \left\langle \zeta,\overline
{\eta_{0}}+\overline{e}\right\rangle \right\rangle =\zeta\left(
\overline{\eta_{0}}+\overline{e}\right)
\]
for all $\overline{\eta_{0}}\in\operatorname{ball}\overline{H_{h}^{e}}$. Thus
$\zeta\left(  \overline{\eta}\right)  =\alpha\left(  \eta_{0}\right)
b\alpha\left(  \eta_{0}\right)  ^{\ast}$ for all $\overline{\eta}%
=\overline{\eta_{0}}+\overline{e}\in S\left(  \mathfrak{c}\right)  $, which
means that $\overline{e}$ is a dominant point for $\zeta$.

Conversely, suppose $\zeta\left(  \overline{\eta}\right)  =\alpha\left(
\eta_{0}\right)  b\alpha\left(  \eta_{0}\right)  ^{\ast}$, $\overline{\eta
}=\overline{\eta_{0}}+\overline{e}\in S\left(  \mathfrak{c}\right)  $ for a
certain continuous mapping $\alpha:\operatorname{ball}\left(  \overline
{H_{h}^{e}}\right)  \rightarrow\operatorname{ball}M_{n}$ such that
$\alpha\left(  0\right)  =I_{n}$ and $b\geq0$. Then
\[
\left(  \left\langle \left\langle \zeta,\overline{\eta}\right\rangle
\right\rangle a,a\right)  =\left(  \alpha\left(  \eta_{0}\right)
b\alpha\left(  \eta_{0}\right)  ^{\ast}a,a\right)  =\left(  b\alpha\left(
\eta_{0}\right)  ^{\ast}a,\alpha\left(  \eta_{0}\right)  ^{\ast}a\right)
\geq0
\]
for all $\overline{\eta}=\overline{\eta_{0}}+\overline{e}\in S\left(
\mathfrak{c}\right)  $ and $a\in M_{n,1}$. It follows that $\left\langle
\left\langle \zeta,\overline{\eta}\right\rangle \right\rangle \geq0$ for all
$\overline{\eta}\in S\left(  \mathfrak{c}\right)  $, which means that
$\zeta\in\min\mathfrak{c}$.
\end{proof}

\begin{remark}
\label{remICR}Let $\left(  a_{\gamma}\right)  _{\gamma}$ be a net of positive
operators from $\mathcal{B}\left(  H\right)  $ such that $r1\leq a_{\gamma
}\leq s1$ for some $r,s>0$. If $a=\lim_{\gamma}a_{\gamma}$ in $\mathcal{B}%
\left(  H\right)  $ then $a^{1/2}=\lim_{\gamma}a_{\gamma}^{1/2}$ in
$\mathcal{B}\left(  H\right)  $. Indeed, let us surround the interval $\left[
r,s\right]  $ by a circle $C$ in $\operatorname{Re}>0$, and put $d$ to be the
distance from $C$ to $\left[  r,s\right]  $. The resolvent functions
$\mathcal{R}_{\gamma}\left(  z\right)  =\left(  z-a_{\gamma}\right)  ^{-1}$
and $\mathcal{R}\left(  z\right)  =\left(  z-a\right)  ^{-1}$ are holomorphic
on $\mathbb{C}\backslash\left[  r,s\right]  $ for all $\gamma$. Since
$\mathcal{R}_{\gamma}\left(  z\right)  $ and $\mathcal{R}\left(  z\right)  $
are normal operators, it follows that $\left\Vert \mathcal{R}_{\gamma}\left(
z\right)  \right\Vert \leq\sup\left\{  \left\vert z-t\right\vert ^{-1}:r\leq
t\leq s\right\}  \leq d^{-1}$ for all $z\in C$ and $\gamma$. Similarly,
$\left\Vert \mathcal{R}\left(  z\right)  \right\Vert \leq d^{-1}$, $z\in C$.
If $\sqrt{z}$ is the principal branch of the root function then
\[
\left\Vert \sqrt{z}\mathcal{R}\left(  z\right)  -\sqrt{z}\mathcal{R}_{\gamma
}\left(  z\right)  \right\Vert =\left\vert \sqrt{z}\right\vert \left\Vert
\mathcal{R}\left(  z\right)  \left(  a-a_{\gamma}\right)  \mathcal{R}_{\gamma
}\left(  z\right)  \right\Vert \leq\sup\left\vert \sqrt{C}\right\vert
d^{-2}\left\Vert a-a_{\gamma}\right\Vert
\]
for all $z\in C$, that is, $\lim_{\gamma}\sqrt{z}\mathcal{R}_{\gamma}\left(
z\right)  =\sqrt{z}\mathcal{R}\left(  z\right)  $ uniformly on $C$. Using the
holomorphic functional calculus (see \cite[2.2.15]{HelBA}) on the interior of
$C$, we conclude that
\[
\lim_{\gamma}a_{\gamma}^{1/2}=\lim_{\gamma}\int_{C}\sqrt{z}\mathcal{R}%
_{\gamma}\left(  z\right)  dz=\int_{C}\sqrt{z}\mathcal{R}\left(  z\right)
dz=a^{1/2},
\]
that is, $a^{1/2}=\lim_{\gamma}a_{\gamma}^{1/2}$ in $\mathcal{B}\left(
H\right)  $. The assertion just proven is valid still in the case of $r=0$
\cite{CH} (see also \cite{Win}).
\end{remark}

Now let us prove a duality result for $\min$ and $\max$ quantizations of the
cone $\mathfrak{c}$.

\begin{theorem}
\label{tmax32}The equalities $\left(  \max\overline{\mathfrak{c}}\right)
^{\boxdot}=\min\mathfrak{c}$ and $\left(  \min\overline{\mathfrak{c}}\right)
^{\boxdot}=\max\mathfrak{c}$ hold.
\end{theorem}

\begin{proof}
By Proposition \ref{propOKE21}, $\min\mathfrak{c=}S\left(  \mathfrak{c}%
\right)  ^{\boxdot}$ is a closed, quantum cone on $H$. Using the Bipolar
Theorem \ref{t1}, we derive that $\min\mathfrak{c=}\left(  \min\mathfrak{c}%
\right)  ^{\boxdot\boxdot}=S\left(  \mathfrak{c}\right)  ^{\boxdot
\boxdot\boxdot}$. By Corollary \ref{cormincbar}, we have that $\left(
\max\overline{\mathfrak{c}}\right)  ^{\boxdot}=S\left(  \mathfrak{c}\right)
^{\boxdot\boxdot\boxdot}=\min\mathfrak{c}$. Hence $\left(  \max\overline
{\mathfrak{c}}\right)  ^{\boxdot}=\min\mathfrak{c}$. By symmetry, we also have
$\left(  \max\mathfrak{c}\right)  ^{\boxdot}=\min\overline{\mathfrak{c}}$,
which in turn implies that $\left(  \min\overline{\mathfrak{c}}\right)
^{\boxdot}=\left(  \max\mathfrak{c}\right)  ^{\boxdot\boxdot}=\max
\mathfrak{c}$ by the Bipolar Theorem \ref{t1}.
\end{proof}

\subsection{The unital quantum cones on $\left(  H,e\right)  $%
\label{subsecUCHE1}}

Now we introduce new quantizations of the separated, unital cone
$\mathfrak{c}$ in a Hilbert $\ast$-space $H$. Since the functional
$\overline{e}:H\rightarrow\mathbb{C}$ is a contraction, it turns out to be a
matrix contraction on the operator Hilbert space $H_{o}$. The projection
$\phi_{e}:H_{o}\rightarrow H_{o}$, $\phi_{e}\zeta=\left\langle \zeta
,\overline{e}\right\rangle e$ is a matrix contraction as well, for $\left\Vert
\phi_{e}^{\left(  n\right)  }\left(  \zeta\right)  \right\Vert _{o}=\left\Vert
\left\langle \left\langle \zeta,\overline{e}\right\rangle \right\rangle
e^{\oplus n}\right\Vert _{o}\leq\left\Vert \left\langle \left\langle
\zeta,\overline{e}\right\rangle \right\rangle \right\Vert \left\Vert e^{\oplus
n}\right\Vert _{o}\leq\left\Vert \zeta\right\Vert _{o}$, $\zeta\in
M_{n}\left(  H_{o}\right)  $, $n\in\mathbb{N}$. In particular, $\varphi
_{e}:H_{o}\rightarrow H_{o}$, $\varphi_{e}\left(  \zeta\right)  =\zeta
-\phi_{e}\left(  \zeta\right)  =\zeta_{0}$ is a matrix bounded mapping and
$\left\Vert \varphi_{e}^{\left(  n\right)  }\left(  \zeta\right)  \right\Vert
_{o}=\left\Vert \zeta-\phi_{e}^{\left(  n\right)  }\left(  \zeta\right)
\right\Vert _{o}\leq2\left\Vert \zeta\right\Vert _{o}$, $\zeta\in M_{n}\left(
H_{o}\right)  $, $n\in\mathbb{N}$. Thus $\zeta_{0}=\varphi_{e}^{\left(
n\right)  }\left(  \zeta\right)  =\sum_{f\neq e}\left\langle \left\langle
\zeta,\overline{f}\right\rangle \right\rangle f^{\oplus n}$, $\left\Vert
\zeta_{0}\right\Vert _{o}\leq2\left\Vert \zeta\right\Vert _{o}$ and
$\left\Vert \zeta_{0}\right\Vert _{2}\leq\left\Vert \zeta\right\Vert _{2}$
(see Lemma \ref{lHS11}) whenever $\zeta=\sum_{f\in F}\left\langle \left\langle
\zeta,\overline{f}\right\rangle \right\rangle f^{\oplus n}$. Moreover,
$a^{\ast}\zeta a=\sum_{f\in F}\left\langle \left\langle a^{\ast}\zeta
a,\overline{f}\right\rangle \right\rangle f^{\oplus m}$ for all $a\in M_{n,m}%
$, which in turn implies that $\left(  a^{\ast}\zeta a\right)  _{0}%
=\varphi_{e}^{\left(  m\right)  }\left(  a^{\ast}\zeta a\right)  =a^{\ast
}\varphi_{e}^{\left(  n\right)  }\left(  \zeta\right)  a=a^{\ast}\zeta_{0}a$.
On the unital space $\left(  H,e\right)  $ consider the following quantum
cones $\mathfrak{C}_{l}$, $\mathfrak{C}_{o}$ and $\mathfrak{C}_{u}$ whose
slices given by
\begin{align*}
\mathfrak{C}_{l}\mathfrak{\cap}M_{n}\left(  H\right)   &  =\left\{  \zeta\in
M_{n}\left(  H\right)  _{h}:\left\Vert a^{\ast}\zeta_{0}a\right\Vert _{2}\leq
m^{-1/2}\tau\left(  \left\langle \left\langle a^{\ast}\zeta a,\overline
{e}\right\rangle \right\rangle \right)  ,a\in M_{n,m},m\in\mathbb{N}\right\}
,\\
\mathfrak{C}_{o}\mathfrak{\cap}M_{n}\left(  H\right)   &  \mathfrak{=}\left\{
\zeta\in M_{n}\left(  H\right)  _{h}:\left\Vert a^{\ast}\zeta_{0}a\right\Vert
_{so}\leq\tau\left(  \left\langle \left\langle a^{\ast}\zeta a,\overline
{e}\right\rangle \right\rangle \right)  ,a\in M_{n,m},m\in\mathbb{N}\right\}
,\\
\mathfrak{C}_{u}\mathfrak{\cap}M_{n}\left(  H\right)   &  \mathfrak{=}\left\{
\zeta\in M_{n}\left(  H\right)  _{h}:\left\Vert a^{\ast}\zeta_{0}a\right\Vert
_{2}\leq\tau\left(  \left\langle \left\langle a^{\ast}\zeta a,\overline
{e}\right\rangle \right\rangle \right)  ,a\in M_{n,m},m\in\mathbb{N}\right\}
.
\end{align*}
The fact that these quantum cones are unital will be verified below. Note that
for every $\zeta$ from each of these cones we have $\tau\left(  \left\langle
\left\langle a^{\ast}\zeta a,\overline{e}\right\rangle \right\rangle \right)
\geq0$ for all $a\in M$. Taking into account that $\left(  \left\langle
\left\langle \zeta,\overline{e}^{\oplus n}\right\rangle \right\rangle
a,a\right)  =\left\langle a^{\ast}\zeta a,\overline{e}^{\oplus n}\right\rangle
=\tau\left(  \left\langle \left\langle a^{\ast}\zeta a,\overline
{e}\right\rangle \right\rangle \right)  $, we conclude that $\left\langle
\left\langle \zeta,\overline{e}^{\oplus n}\right\rangle \right\rangle \geq0$.
Further, note that
\[
\mathfrak{C}_{l}\mathfrak{\cap}H=\mathfrak{C}_{o}\mathfrak{\cap}%
H=\mathfrak{C}_{u}\mathfrak{\cap}H=\left\{  \zeta\in H_{h}:\left\Vert
\zeta_{0}\right\Vert \leq\left(  \zeta,e\right)  \right\}  =\mathfrak{c.}%
\]
As above for each $\zeta\in M_{n}\left(  H\right)  _{h}$ there corresponds a
unique expansion $\zeta=\sum_{f\in F}\left\langle \left\langle \zeta
,\overline{f}\right\rangle \right\rangle f^{\oplus n}$ with hermitian
$\left\langle \left\langle \zeta,\overline{f}\right\rangle \right\rangle \in
M_{n}$. Put $M_{n}\left(  H\right)  ^{e}=\left\{  \zeta\in M_{n}\left(
H\right)  :\left\langle \left\langle \zeta,\overline{e}\right\rangle
\right\rangle =0\right\}  $, $M_{n}\left(  H\right)  _{h}^{e}=M_{n}\left(
H\right)  _{h}\cap M_{n}\left(  H\right)  ^{e}$, and $\mathfrak{H}_{h}%
^{e}=\mathfrak{H}_{h}\cap M\left(  H\right)  _{h}^{e}$, which is the unit ball
of $M_{n}\left(  H\right)  _{h}^{e}$ relative to the norm $\left\Vert
\cdot\right\Vert _{2}$. Similarly, $\mathfrak{B}_{h}^{e}=\mathfrak{B}_{h}\cap
M\left(  H\right)  _{h}^{e}$ is the unit ball of $M_{n}\left(  H\right)
_{h}^{e}$ relative to the matrix norm $\left\Vert \cdot\right\Vert _{o}$. We
also put $\mathfrak{K}_{h}^{e}=\left(  n^{1/2}\mathfrak{H}_{h}^{e}\right)
_{n}$, which is a convex quantum set. Thus $\mathfrak{K}_{h}^{e}%
+\mathfrak{e}=\left(  n^{1/2}\mathfrak{H}_{h}^{e}+e^{\oplus n}\right)  _{n}$,
$\mathfrak{B}_{h}^{e}+\mathfrak{e}=\left(  \mathfrak{B}_{h}^{e}+e^{\oplus
n}\right)  _{n}$ and $\mathfrak{H}_{h}^{e}+\mathfrak{e}=\left(  \mathfrak{H}%
_{h}^{e}+e^{\oplus n}\right)  _{n}$ are quantum sets on $H$. Similarly, we
have the quantum sets $\overline{\mathfrak{K}_{h}^{e}}+\overline{\mathfrak{e}%
}$, $\overline{\mathfrak{B}_{h}^{e}}+\overline{\mathfrak{e}}$ and
$\overline{\mathfrak{H}_{h}^{e}}+\overline{\mathfrak{e}}$ on $\overline{H}$.

\begin{proposition}
\label{lHS21}The following equalities $\mathfrak{C}_{l}=\left(  \overline
{\mathfrak{K}_{h}^{e}}+\overline{\mathfrak{e}}\right)  ^{\boxdot}$,
$\mathfrak{C}_{o}=\left(  \overline{\mathfrak{B}_{h}^{e}}+\overline
{\mathfrak{e}}\right)  ^{\boxdot}$ and $\mathfrak{C}_{u}=\left(
\overline{\mathfrak{H}_{h}^{e}}+\overline{\mathfrak{e}}\right)  ^{\boxdot}$
hold with respect to the dual $\ast$-pair $\left(  H,\overline{H}\right)  $.
In particular, $\mathfrak{C}_{l}\subseteq\mathfrak{C}_{o}\subseteq
\mathfrak{C}_{u}$ are the inclusions of the separated, closed, unital, quantum
cones on $H$, which are quantizations of $\mathfrak{c}$.
\end{proposition}

\begin{proof}
First take $\zeta\in\mathfrak{C}_{l}\mathfrak{\cap}M_{n}\left(  H\right)  $.
Then $\zeta\in M_{n}\left(  H\right)  _{h}$ and $\left\Vert \zeta
_{0}\right\Vert _{2}\leq n^{-1/2}\tau\left(  \left\langle \left\langle
\zeta,\overline{e}\right\rangle \right\rangle \right)  $, where $\zeta
_{0}=\sum_{f\neq e}\left\langle \left\langle \zeta,\overline{f}\right\rangle
\right\rangle f^{\oplus n}=\varphi_{e}^{\left(  n\right)  }\left(
\zeta\right)  \in M_{n}\left(  H\right)  _{h}^{e}$. If $\eta\in\mathfrak{H}%
_{h}^{e}$ then $\eta=\sum_{f\neq e}\left\langle \left\langle \eta,\overline
{f}\right\rangle \right\rangle f^{\oplus n}$, $\left\Vert \eta\right\Vert
_{2}\leq1$ and
\[
\left\langle \zeta,\overline{\eta}\right\rangle =\sum_{f\neq e}\left\langle
\zeta,\left\langle \left\langle \eta,\overline{f}\right\rangle \right\rangle
^{\operatorname{t}}\overline{f}^{\oplus n}\right\rangle =\sum_{f\neq e}%
\tau\left(  \left\langle \left\langle \zeta,\overline{f}\right\rangle
\right\rangle \left\langle \left\langle \eta,\overline{f}\right\rangle
\right\rangle \right)  =\left\langle \zeta_{0},\overline{\eta}\right\rangle
\]
(see (\ref{faeb}) and (\ref{fND})). Note that $\left\langle \zeta
,\overline{\eta}\right\rangle =\left\langle \zeta^{\ast},\overline{\eta
}\right\rangle =\sum_{i,j}\left\langle \zeta_{ji}^{\ast},\overline{\eta_{ij}%
}\right\rangle =\sum_{i,j}\left\langle \zeta_{ji},\overline{\eta_{ij}^{\ast}%
}\right\rangle ^{\ast}=\left\langle \zeta,\overline{\eta^{\ast}}\right\rangle
^{\ast}=\left\langle \zeta,\overline{\eta}\right\rangle ^{\ast}$, which means
that $\left\langle \zeta,\overline{\eta}\right\rangle \in\mathbb{R}$. Note
also that all matrices $\left\langle \left\langle \zeta,\overline
{f}\right\rangle \right\rangle $ and $\left\langle \left\langle \eta
,\overline{f}\right\rangle \right\rangle $, $f\neq e$ are hermitians,
therefore $\tau\left(  \left\langle \left\langle \zeta,\overline
{f}\right\rangle \right\rangle \left\langle \left\langle \eta,\overline
{f}\right\rangle \right\rangle \right)  \in\mathbb{R}$ as well. Since
$\left\langle \zeta,\overline{\eta}\right\rangle \leq\left\Vert \zeta
_{0}\right\Vert _{2}\left\Vert \eta\right\Vert _{2}$ (see Remark
\ref{remlHS}), we deduce that $\sup\left\vert \left\langle \zeta
,\overline{\mathfrak{H}_{h}^{e}}\right\rangle \right\vert \leq\left\Vert
\zeta_{0}\right\Vert _{2}\leq n^{-1/2}\tau\left(  \left\langle \left\langle
\zeta,\overline{e}\right\rangle \right\rangle \right)  =n^{-1/2}\left\langle
\zeta,\overline{e}^{\oplus n}\right\rangle $ or $-\left\langle \zeta
,\overline{e}^{\oplus n}\right\rangle \leq\left\langle \zeta,n^{1/2}%
\overline{\mathfrak{H}_{h}^{e}}\right\rangle \leq\left\langle \zeta
,\overline{e}^{\oplus n}\right\rangle $. Hence $\left\langle \zeta
,n^{1/2}\overline{\mathfrak{H}_{h}^{e}}+\overline{e}^{\oplus n}\right\rangle
\geq0$. Conversely, suppose the latter holds for $\zeta\in M_{n}\left(
H\right)  _{h}$. Since $-\mathfrak{H}_{h}^{e}=\mathfrak{H}_{h}^{e}$, it
follows that $\sup\left\vert \left\langle \zeta,\overline{\mathfrak{H}_{h}%
^{e}}\right\rangle \right\vert \leq n^{-1/2}\left\langle \zeta,\overline
{e}^{\oplus n}\right\rangle $. But
\begin{align*}
\left\Vert \zeta_{0}\right\Vert _{2}  &  =\sup\left\vert \left\langle
\zeta_{0},\overline{\mathfrak{H}^{e}}\right\rangle \right\vert =\sup\left\vert
\operatorname{Re}\left\langle \zeta_{0},\overline{\mathfrak{H}^{e}%
}\right\rangle \right\vert =\sup\left\vert \left\langle \zeta_{0}%
,\operatorname{Re}\overline{\mathfrak{H}^{e}}\right\rangle \right\vert
=\sup\left\vert \left\langle \zeta_{0},\overline{\mathfrak{H}_{h}^{e}%
}\right\rangle \right\vert \\
&  =\sup\left\vert \left\langle \zeta,\overline{\mathfrak{H}_{h}^{e}%
}\right\rangle \right\vert ,
\end{align*}
therefore $\left\Vert \zeta_{0}\right\Vert _{2}\leq n^{-1/2}\left\langle
\zeta,\overline{e}^{\oplus n}\right\rangle $. Consequently,
\[
\mathfrak{C}_{l}\mathfrak{\cap}M_{n}\left(  H\right)  =\left\{  \zeta\in
M_{n}\left(  H\right)  _{h}:\left\langle a^{\ast}\zeta a,m^{1/2}%
\overline{\mathfrak{H}_{h}^{e}}+\overline{e}^{\oplus m}\right\rangle
\geq0,a\in M_{n,m},m\in\mathbb{N}\right\}  .
\]
Taking into account that $\left\langle a^{\ast}\zeta a,m^{1/2}\overline
{\mathfrak{H}_{h}^{e}}+\overline{e}^{\oplus m}\right\rangle =\left(
\left\langle \left\langle \zeta,m^{1/2}\overline{\mathfrak{H}_{h}^{e}%
}+\overline{e}^{\oplus m}\right\rangle \right\rangle a,a\right)  $, we derive
that $\zeta\in\mathfrak{C}_{l}$ iff $\left\langle \left\langle \zeta
,m^{1/2}\overline{\mathfrak{H}_{h}^{e}}+\overline{e}^{\oplus m}\right\rangle
\right\rangle \geq0$ or $\zeta\in\left(  \overline{\mathfrak{K}_{h}^{e}%
}+\overline{\mathfrak{e}}\right)  ^{\boxdot}$. Similarly, $\zeta
\in\mathfrak{C}_{u}$ iff $\left\langle \left\langle \zeta,\overline
{\mathfrak{H}_{h}^{e}}+\overline{e}^{\oplus n}\right\rangle \right\rangle
\geq0$ or $\zeta\in\left(  \overline{\mathfrak{H}_{h}^{e}}+\overline
{\mathfrak{e}}\right)  ^{\boxdot}$.

Further, $\left\Vert \zeta_{0}\right\Vert _{so}\leq\tau\left(  \left\langle
\left\langle \zeta,\overline{e}\right\rangle \right\rangle \right)  $ means
that $\sup\left\vert \left\langle \zeta,\overline{\mathfrak{B}_{h}^{e}%
}\right\rangle \right\vert \leq\tau\left(  \left\langle \left\langle
\zeta,\overline{e}\right\rangle \right\rangle \right)  =\left\langle
\zeta,\overline{e}^{\oplus n}\right\rangle $. The latter in turn is equivalent
to $\left\langle \zeta,\overline{\mathfrak{B}_{h}^{e}}+\overline{e}^{\oplus
n}\right\rangle \geq0$. Thus
\[
\mathfrak{C}_{o}\mathfrak{\cap}M_{n}\left(  H\right)  \mathfrak{=}\left\{
\zeta\in M_{n}\left(  H\right)  _{h}:\left\langle a^{\ast}\zeta a,\overline
{\mathfrak{B}_{h}^{e}}+\overline{e}^{\oplus m}\right\rangle \geq0,a\in
M_{n,m},m\in\mathbb{N}\right\}  .
\]
As above, we derive that $\zeta\in\mathfrak{C}_{o}$ iff $\left\langle
\left\langle \zeta,\overline{\mathfrak{B}_{h}^{e}}+\overline{e}^{\oplus
n}\right\rangle \right\rangle \geq0$, that is, $\mathfrak{C}_{o}=\left(
\overline{\mathfrak{B}_{h}^{e}}+\overline{\mathfrak{e}}\right)  ^{\boxdot}$.
Based on Lemma \ref{lHS11}, we obtain that $\mathfrak{H}_{h}^{e}+e^{\oplus
n}\subseteq\mathfrak{B}_{h}^{e}+e^{\oplus n}\subseteq\sqrt{n}\mathfrak{H}%
_{h}^{e}+e^{\oplus n}$ in $M_{n}\left(  H\right)  _{h}$, or $\overline
{\mathfrak{H}_{h}^{e}}+\overline{\mathfrak{e}}\subseteq\overline
{\mathfrak{B}_{h}^{e}}+\overline{\mathfrak{e}}\subseteq\overline
{\mathfrak{K}_{h}^{e}}+\overline{\mathfrak{e}}$ are inclusions of the quantum
sets on $\overline{H}$. Therefore $\mathfrak{C}_{l}=\left(  \overline
{\mathfrak{K}_{h}^{e}}+\overline{\mathfrak{e}}\right)  ^{\boxdot}%
\subseteq\mathfrak{C}_{o}=\left(  \overline{\mathfrak{B}_{h}^{e}}%
+\overline{\mathfrak{e}}\right)  ^{\boxdot}\subseteq\mathfrak{C}_{u}=\left(
\overline{\mathfrak{H}_{h}^{e}}+\overline{\mathfrak{e}}\right)  ^{\boxdot}$.

Further, prove that all these quantum cones are separated. Take $\zeta\in
M_{n}\left(  H\right)  _{h}$. Suppose $\zeta\in\mathfrak{C}_{u}\cap
-\mathfrak{C}_{u}$ with $\zeta=\zeta_{0}+\left\langle \left\langle
\zeta,\overline{e}\right\rangle \right\rangle e^{\oplus n}$. Since
$\left\langle \left\langle \pm\zeta,\overline{\mathfrak{H}_{h}^{e}}%
+\overline{e}^{\oplus n}\right\rangle \right\rangle \geq0$, it follows that
$\left\langle \left\langle \zeta,\overline{e}\right\rangle \right\rangle =0$
and $\left\langle \left\langle \zeta_{0},\overline{\mathfrak{H}_{h}^{e}%
}\right\rangle \right\rangle =\left\{  0\right\}  $. In particular,
$\left\langle \left\langle \zeta_{0},\overline{\mathfrak{B}_{h}^{e}%
}\right\rangle \right\rangle =\left\{  0\right\}  $. Since $\left\Vert
\eta^{\ast}\right\Vert _{o}=\left\Vert \eta\right\Vert _{o}$ for all $\eta\in
M\left(  H_{o}\right)  $, it follows that $\operatorname{Re}\eta$,
$\operatorname{Im}\eta\in\mathfrak{B}$ whenever $\eta\in\mathfrak{B}$. Hence
$\left\langle \left\langle \zeta_{0},\overline{\mathfrak{B}}\right\rangle
\right\rangle =\left\{  0\right\}  $ and $\left\Vert \zeta_{0}\right\Vert
=\sup\left\Vert \left\langle \left\langle \zeta_{0},\overline{\mathfrak{B}%
}\right\rangle \right\rangle \right\Vert =0$, that is, $\zeta=0$. Thus all
quantum cones are separated.

Finally prove that $\mathfrak{C}_{l}$ is unital. Since $\mathfrak{C}_{l}$ is a
topologically closed quantization of the unital cone $\mathfrak{c}$, it
follows that $\max\mathfrak{c=c}^{\boxdot\boxdot}=\left(  \mathfrak{c}%
^{c}\right)  ^{-}\subseteq\mathfrak{C}_{l}$ thanks to Proposition
\ref{propOKE21}. But $\max\mathfrak{c}$ is unital (see Lemma \ref{lemPTT}),
therefore so is $\mathfrak{C}_{l}$. In particular, so are both $\mathfrak{C}%
_{o}$ and $\mathfrak{C}_{u}$.
\end{proof}

\begin{remark}
The fact that $\mathfrak{C}_{o}$ (in turn $\mathfrak{C}_{u}$) is unital also
follows from the following argument. Take $\zeta\in M_{n}\left(  H\right)
_{h}$. Prove that $\left\langle \left\langle \zeta+re^{\oplus n}%
,\overline{\mathfrak{B}_{h}^{e}}+\overline{e}^{\oplus m}\right\rangle
\right\rangle \geq0$ for large positive $r$. But $\left\langle \left\langle
\zeta,\overline{\mathfrak{B}_{h}^{e}}+\overline{\mathfrak{e}}\right\rangle
\right\rangle $ is a bounded set of hermitian matrices in $M$. Then
$-rI_{nm}\leq\left\langle \left\langle \zeta,\overline{\mathfrak{B}_{h}^{e}%
}+\overline{e}^{\oplus m}\right\rangle \right\rangle \leq rI_{nm}$ for all
$m$, which in turn implies that
\begin{align*}
\left\langle \left\langle \zeta+re^{\oplus n},\overline{\mathfrak{B}_{h}^{e}%
}+\overline{e}^{\oplus m}\right\rangle \right\rangle  &  \subseteq\left\langle
\left\langle \zeta,\overline{\mathfrak{B}_{h}^{e}}+\overline{e}^{\oplus
m}\right\rangle \right\rangle +\left\langle \left\langle re^{\oplus
n},\overline{\mathfrak{B}_{h}^{e}}+\overline{e}^{\oplus m}\right\rangle
\right\rangle \\
&  =\left\langle \left\langle \zeta,\overline{\mathfrak{B}_{h}^{e}}%
+\overline{e}^{\oplus m}\right\rangle \right\rangle +rI_{nm}\geq0,
\end{align*}
that is, $\mathfrak{C}_{o}$ is unital. In particular, so is $\mathfrak{C}_{u}$.
\end{remark}

\subsection{The quantum polars}

Now consider the quantum set $\overline{\mathfrak{B}_{h}^{e}}+\overline
{\mathfrak{e}}=\left(  \overline{\mathfrak{B}_{h}^{e}}+\overline{e}^{\oplus
n}\right)  _{n}$ on $\overline{H}$. If $\eta\in\mathfrak{B}_{h}^{e}$ then
$\left\langle \left\langle e,\overline{\eta}\right\rangle \right\rangle
=\left\langle \left\langle e,\overline{\eta}\right\rangle \right\rangle
^{\ast}=\left[  \left\langle e,\overline{\eta_{ij}}\right\rangle \right]
_{i,j}^{\ast}=\left[  \left\langle e,\overline{\eta_{ji}}\right\rangle ^{\ast
}\right]  _{i,j}=\left[  \left(  e,\eta_{ji}\right)  ^{\ast}\right]
_{i,j}=\left[  \left(  \eta_{ji},e\right)  \right]  _{i,j}=\left\langle
\left\langle \eta,\overline{e}\right\rangle \right\rangle ^{\operatorname{t}%
}=0$. Thus $\overline{\mathfrak{B}_{h}^{e}}+\overline{\mathfrak{e}}\subseteq
M\left(  \overline{H}\right)  _{h}\cap M\left(  \overline{H}\right)
_{e}=M\left(  \overline{H}\right)  _{he}$ (see Subsection \ref{subsecUQC}),
and put $\mathfrak{B}_{e}=\left(  \overline{\mathfrak{B}_{h}^{e}}%
+\overline{\mathfrak{e}}\right)  ^{\odot}$, which is a closed, absolutely
matrix convex subset on $H$. Note that $\overline{\mathfrak{B}_{h}^{e}%
}+\overline{\mathfrak{e}}$ is a matrix convex subset of $M\left(  \overline
{H}\right)  _{h}$.

\begin{lemma}
\label{lBP1}The equality $\mathfrak{B}_{e}^{\odot}\cap M\left(  \overline
{H}\right)  _{he}=\overline{\mathfrak{B}_{h}^{e}}+\overline{\mathfrak{e}}$ holds.
\end{lemma}

\begin{proof}
Take $\overline{z}\in\operatorname{amc}\left(  \overline{\mathfrak{B}_{h}^{e}%
}+\overline{\mathfrak{e}}\right)  ^{-}\cap M_{n}\left(  \overline{H}\right)
_{he}$, where $\operatorname{amc}\left(  \overline{\mathfrak{B}_{h}^{e}%
}+\overline{\mathfrak{e}}\right)  ^{-}$ is the closed absolutely matrix convex
hull of $\overline{\mathfrak{B}_{h}^{e}}+\overline{\mathfrak{e}}$ in $M\left(
\overline{H}\right)  $. Then $\overline{z}\in M_{n}\left(  \overline
{H}\right)  _{h}$, $\left\langle \left\langle e,\overline{z}\right\rangle
\right\rangle =I_{n}$ and $\overline{z}=\lim_{k}a_{k}\left(  \overline
{\eta_{k}}+\overline{e}^{\oplus n_{k}}\right)  b_{k}$ with $a_{k},b_{k}%
\in\operatorname{ball}M$ and $\eta_{k}\in\mathfrak{B}_{h}^{e}$. It follows
that
\[
I_{n}=\left\langle \left\langle e,\overline{z}\right\rangle \right\rangle
=\lim_{k}\left\langle \left\langle e,a_{k}\left(  \overline{\eta_{k}%
}+\overline{e}^{\oplus n_{k}}\right)  b_{k}\right\rangle \right\rangle
=\lim_{k}a_{k}\left\langle \left\langle e,\overline{\eta_{k}}\right\rangle
\right\rangle b_{k}+a_{k}\left\langle \left\langle e,\overline{e}^{\oplus
n_{k}}\right\rangle \right\rangle b_{k}=\lim_{k}a_{k}b_{k},
\]
which in turn implies that $\lim_{k}a_{k}\overline{e}^{\oplus n_{k}}b_{k}%
=\lim_{k}a_{k}b_{k}\overline{e}^{\oplus n}=\overline{e}^{\oplus n}$. In
particular, we have the limit $\overline{\eta}=\lim_{k}a_{k}\overline{\eta
_{k}}b_{k}=\overline{z}-\overline{e}^{\oplus n}\in M_{n}\left(  \overline
{H}\right)  _{h}$ and $\left\Vert \overline{\eta}\right\Vert _{o}\leq
\limsup_{k}\left\Vert a_{k}\right\Vert \left\Vert \overline{\eta_{k}%
}\right\Vert _{o}\left\Vert b_{k}\right\Vert \leq1$, that is, $\overline{\eta
}\in\overline{\mathfrak{B}_{h}^{e}}$. Hence $\overline{z}=\overline{\eta
}+\overline{e}^{\oplus n}\in\overline{\mathfrak{B}_{h}^{e}}+\overline
{\mathfrak{e}}$. Thus $\operatorname{amc}\left(  \overline{\mathfrak{B}%
_{h}^{e}}+\overline{\mathfrak{e}}\right)  ^{-}\cap M\left(  \overline
{H}\right)  _{he}=\overline{\mathfrak{B}_{h}^{e}}+\overline{\mathfrak{e}}$.

Finally, using the Bipolar Theorem \ref{BiT}, we deduce that $\mathfrak{B}%
_{e}^{\odot}\cap M\left(  \overline{H}\right)  _{he}=\left(  \overline
{\mathfrak{B}_{h}^{e}}+\overline{\mathfrak{e}}\right)  ^{\odot\odot}\cap
M\left(  \overline{H}\right)  _{he}=\operatorname{amc}\left(  \overline
{\mathfrak{B}_{h}^{e}}+\overline{\mathfrak{e}}\right)  ^{-}\cap M\left(
\overline{H}\right)  _{he}=\overline{\mathfrak{B}_{h}^{e}}+\overline
{\mathfrak{e}}$.
\end{proof}

\begin{lemma}
\label{lBP2}The following inclusions $10^{-1}\mathfrak{B}_{e}\subseteq
\mathfrak{B\subseteq}2\mathfrak{B}_{e}$ of quantum balls on $H$ hold. In
particular, the Minkowski functional $\left\Vert \cdot\right\Vert _{e}$ of
$\mathfrak{B}_{e}$ is a matrix norm which is equivalent to $\left\Vert
\cdot\right\Vert _{o}$.
\end{lemma}

\begin{proof}
First note that $\overline{\mathfrak{B}_{h}^{e}}+\overline{\mathfrak{e}%
}\subseteq\overline{\mathfrak{B}_{h}}+\overline{\mathfrak{B}_{h}}%
\subseteq2\overline{\mathfrak{B}}$. Therefore $2^{-1}\mathfrak{B=}\left(
2\overline{\mathfrak{B}}\right)  ^{\odot}\subseteq\left(  \overline
{\mathfrak{B}_{h}^{e}}+\overline{\mathfrak{e}}\right)  ^{\odot}=\mathfrak{B}%
_{e}$, that is, the second inclusion follows. To prove the first one, take
$z\in2^{-1}\mathfrak{B}_{h}\cap M_{n}\left(  H\right)  $ with its expansion
$z=w+ae^{\oplus n}$, where $w=\sum_{f\neq e}\left\langle \left\langle
z,\overline{f}\right\rangle \right\rangle f^{\oplus n}\in M_{n}\left(
H\right)  _{h}^{e}$ and $a=\left\langle \left\langle z,\overline
{e}\right\rangle \right\rangle $ is hermitian. Moreover, $\left\Vert
w\right\Vert _{o}=\left\Vert \varphi_{e}^{\left(  n\right)  }\left(  z\right)
\right\Vert _{o}\leq2\left\Vert z\right\Vert _{o}\leq1$ (that is,
$w\in\mathfrak{B}_{h}^{e}$) and $\left\Vert a\right\Vert \leq\left\Vert
z\right\Vert _{o}\leq2^{-1}$. Thus $w+e^{\oplus n}\in\mathfrak{B}_{h}%
^{e}+\mathfrak{e}$ and $\left(  a-I_{n}\right)  e^{\oplus n}\in\left(
3/2\right)  \operatorname{amc}\left(  \mathfrak{e}\right)  \subseteq\left(
3/2\right)  \operatorname{amc}\left(  \mathfrak{B}_{h}^{e}+\mathfrak{e}%
\right)  $, which in turn implies that $z=w+e^{\oplus n}+\left(
a-I_{n}\right)  e^{\oplus n}\in\mathfrak{B}_{h}^{e}+\mathfrak{e+}\left(
3/2\right)  \operatorname{amc}\left(  \mathfrak{B}_{h}^{e}+\mathfrak{e}%
\right)  \subseteq\left(  5/2\right)  \operatorname{amc}\left(  \mathfrak{B}%
_{h}^{e}+\mathfrak{e}\right)  $. Hence $2^{-1}\overline{\mathfrak{B}_{h}%
}\subseteq\left(  5/2\right)  \operatorname{amc}\left(  \overline
{\mathfrak{B}_{h}^{e}}+\overline{\mathfrak{e}}\right)  \subseteq\left(
5/2\right)  \left(  \overline{\mathfrak{B}_{h}^{e}}+\overline{\mathfrak{e}%
}\right)  ^{\odot\odot}=\left(  5/2\right)  \mathfrak{B}_{e}^{\odot}$, and
$\overline{\mathfrak{B}}\subseteq2\operatorname{amc}\left(  \overline
{\mathfrak{B}_{h}}\right)  \subseteq10\mathfrak{B}_{e}^{\odot}$. By passing to
the quantum polars, we obtain that $10^{-1}\mathfrak{B}_{e}\subseteq
\overline{\mathfrak{B}}^{\odot}=\mathfrak{B}$.

Finally, for the matrix norm $\left\Vert \zeta\right\Vert _{e}=\sup\left\Vert
\left\langle \left\langle \zeta,\overline{\mathfrak{B}_{h}^{e}}+\overline
{\mathfrak{e}}\right\rangle \right\rangle \right\Vert $ defined by means of
$\mathfrak{B}_{e}$, we obtain that $2^{-1}\left\Vert \zeta\right\Vert _{e}%
\leq\left\Vert \zeta\right\Vert \leq10\left\Vert \zeta\right\Vert _{e}$ for
all $\zeta\in M\left(  H\right)  $, which means that $\left\Vert
\cdot\right\Vert _{e}$ and $\left\Vert \cdot\right\Vert _{o}$ are equivalent
matrix norms.
\end{proof}

\begin{theorem}
\label{tMm1}Let $H$ be a Hilbert space. The operator Hilbert space $H_{o}$ is
an operator system whose unital quantum cone of positive elements is given by
$\mathfrak{C}_{o}$ with $\mathcal{S}\left(  \mathfrak{C}_{o}\right)
=\overline{\mathfrak{B}_{h}^{e}}+\overline{\mathfrak{e}}$. Moreover,
$\mathfrak{C}_{o}^{\boxdot}=\overline{\mathfrak{C}}_{o}$, where $\overline
{\mathfrak{C}}_{o}$ is the related quantum cone on $\overline{H}_{o}$ with
$\mathcal{S}\left(  \overline{\mathfrak{C}}_{o}\right)  =\mathfrak{B}_{h}%
^{e}+\mathfrak{e}$. Thus $H_{o}$ is a self-dual operator system.
\end{theorem}

\begin{proof}
As above $\mathfrak{B}$ denotes the unit ball of the matrix norm $\left\Vert
\cdot\right\Vert _{o}$. By Lemma \ref{lBP2}, $\mathfrak{B}_{e}$ is an
absorbent, $\mathfrak{s}\left(  H,\overline{H}\right)  $-closed, absolutely
matrix convex set on $H$. As in \cite[Lemma 4.3]{Dhjm} (see also \cite{DSbM})
consider the Paulsen's power $\mathcal{P}_{H}$ of $H$ and related
$\mathfrak{s}\left(  \mathcal{P}_{H},\mathcal{P}_{\overline{H}}\right)
$-closed (see Theorem \ref{tdos}), unital, quantum cone $\mathfrak{C}%
_{\mathfrak{B}_{e}}$ on $\mathcal{P}_{H}$ obtained by means of $\mathfrak{B}%
_{e}$. For brevity we write $\mathfrak{C}\left(  \mathfrak{B}_{e}\right)  $
instead of $\mathfrak{C}_{\mathfrak{B}_{e}}$. Notice that $\mathfrak{C}\left(
\mathfrak{B}_{e}\right)  $ is a cone on $H$. The $\mathfrak{s}\left(
H,\overline{H}\right)  $-closed, quantum cone on $H$ generated by
$\mathfrak{C}\left(  \mathfrak{B}_{e}\right)  $ is denoted by $\mathfrak{C}%
_{e}$. Actually, $\mathfrak{C}_{e}=\mathfrak{C}\left(  \mathfrak{B}%
_{e}\right)  ^{\boxdot\boxdot}$, where $\mathfrak{C}\left(  \mathfrak{B}%
_{e}\right)  ^{\boxdot}$ is the quantum polar of the cone $\mathfrak{C}\left(
\mathfrak{B}_{e}\right)  $ with respect to the dual $\ast$-pair $\left(
H,\overline{H}\right)  $. The quantum cone $\mathfrak{C}_{e}$ is unital and
$\mathcal{S}\left(  \mathfrak{C}_{e}\right)  =\mathfrak{C}\left(
\mathfrak{B}_{e}\right)  ^{\boxdot}\cap M\left(  \overline{H}\right)  _{e}$.
Since $\mathfrak{B}_{e}$ is an absorbent, $\mathfrak{s}\left(  H,\overline
{H}\right)  $-closed, absolutely matrix convex set on $H$, we derive that%
\[
\mathcal{S}\left(  \mathfrak{C}_{e}\right)  =\mathfrak{B}_{e}^{\odot}\cap
M\left(  \overline{H}\right)  _{he}=\overline{\mathfrak{B}_{h}^{e}}%
+\overline{\mathfrak{e}}%
\]
by virtue of Lemma \ref{lBP1}. Using Proposition \ref{lHS21}, we deduce that
\[
\mathfrak{C}_{e}=\left(  \overline{\mathfrak{B}_{h}^{e}}+\overline
{\mathfrak{e}}\right)  ^{\boxdot}=\mathfrak{C}_{o}.
\]
The matrix normed topology on $H$ of the unital quantum cone $\mathfrak{C}%
_{o}$ is given by the absolutely matrix convex set $\widehat{\mathfrak{C}_{o}%
}=h_{H}^{-1}\left(  \mathfrak{C}_{o}-\mathfrak{e}\right)  $ on $H$ (see
\cite[Corollary 5.1]{Dhjm}). Namely, if $p_{o}$ is the Minkowski functional of
$\widehat{\mathfrak{C}_{o}}$ then $p_{o}$ is a matrix norm by Proposition
\ref{lHS21}. Moreover,
\begin{align*}
p_{o}\left(  \zeta\right)   &  =\sup\left\Vert \left\langle \left\langle
\zeta,\mathfrak{C}_{o}^{\boxdot}\cap M\left(  \overline{H}\right)
_{e}\right\rangle \right\rangle \right\Vert =\sup\left\Vert \left\langle
\left\langle \zeta,\mathfrak{C}_{e}^{\boxdot}\cap M\left(  \overline
{H}\right)  _{e}\right\rangle \right\rangle \right\Vert \\
&  =\sup\left\Vert \left\langle \left\langle \zeta,\mathfrak{C}\left(
\mathfrak{B}_{e}\right)  ^{\boxdot}\cap M\left(  \overline{H}\right)
_{e}\right\rangle \right\rangle \right\Vert =\sup\left\Vert \left\langle
\left\langle \zeta,\mathfrak{B}_{e}^{\odot}\cap M\left(  \overline{H}\right)
_{he}\right\rangle \right\rangle \right\Vert \\
&  =\sup\left\Vert \left\langle \left\langle \zeta,\overline{\mathfrak{B}%
_{h}^{e}}+\overline{\mathfrak{e}}\right\rangle \right\rangle \right\Vert
=\left\Vert \zeta\right\Vert _{e}%
\end{align*}
for all $\zeta\in M\left(  H\right)  $, that is, $\widehat{\mathfrak{C}_{o}%
}=\widehat{\mathfrak{C}_{e}}=\mathcal{S}\left(  \mathfrak{C}_{e}\right)
^{\odot}=\left(  \overline{\mathfrak{B}_{h}^{e}}+\overline{\mathfrak{e}%
}\right)  ^{\odot}=\mathfrak{B}_{e}$. Thus $p_{o}=\left\Vert \cdot\right\Vert
_{e}$, which is equivalent to the matrix norm $\left\Vert \cdot\right\Vert
_{o}$ thanks to Lemma \ref{lBP2}. Consequently, $H_{o}$ is an operator system
with the related separated, closed, unital, quantum cone $\mathfrak{C}_{o}$.

By symmetry we have a separated, closed, unital, quantum cone $\overline
{\mathfrak{C}}_{o}$ on $\overline{H}$ with $\mathcal{S}\left(  \overline
{\mathfrak{C}}_{o}\right)  =\mathfrak{B}_{h}^{e}+\mathfrak{e}$. By Proposition
\ref{lHS21}, $\overline{\mathfrak{C}}_{o}$ is a quantization of the unital
cone $\overline{\mathfrak{c}}$ in $H$. Therefore $\max\overline{\mathfrak{c}%
}\subseteq\overline{\mathfrak{C}}_{o}\subseteq\min\overline{\mathfrak{c}}$. By
passing to the quantum polars and using Theorem \ref{tmax32}, we deduce that
$\max\mathfrak{c\subseteq}\overline{\mathfrak{C}}_{o}^{\boxdot}\subseteq
\min\mathfrak{c}$. In particular, $\overline{\mathfrak{C}}_{o}^{\boxdot}$ is a
unital quantum cone. By Unital Bipolar Theorem \ref{tUBP} and Bipolar Theorem
\ref{t1}, we derive that
\[
\overline{\mathfrak{C}}_{o}^{\boxdot}=\left(  \overline{\mathfrak{C}}%
_{o}^{\boxdot\boxdot}\cap M\left(  \overline{H}\right)  _{he}\right)
^{\boxdot}=\left(  \left(  \mathfrak{B}_{h}^{e}+\mathfrak{e}\right)
^{\boxdot\boxdot\boxdot}\cap M\left(  \overline{H}\right)  _{he}\right)
^{\boxdot}=\left(  \left(  \mathfrak{B}_{h}^{e}+\mathfrak{e}\right)
^{\boxdot}\cap M\left(  \overline{H}\right)  _{he}\right)  ^{\boxdot}.
\]
If $\overline{\eta}\in\left(  \mathfrak{B}_{h}^{e}+\mathfrak{e}\right)
^{\boxdot}\cap M_{n}\left(  \overline{H}\right)  _{he}$ then $\overline{\eta
}=\overline{\eta_{0}}+\overline{e}^{\oplus n}$ and $\left\langle \left\langle
\mathfrak{B}_{h}^{e},\overline{\eta_{0}}\right\rangle \right\rangle +I\geq0$,
which in turn implies that $\sup\left\Vert \left\langle \left\langle
\mathfrak{B}_{h}^{e},\overline{\eta_{0}}\right\rangle \right\rangle
\right\Vert \leq1$. In particular, $\left\Vert \eta_{0}\right\Vert
_{o}=\left\Vert \left\langle \left\langle \left\Vert \eta_{0}\right\Vert
_{o}^{-1}\eta_{0},\overline{\eta_{0}}\right\rangle \right\rangle \right\Vert
\leq1$ or $\overline{\eta_{0}}\in\overline{\mathfrak{B}_{h}^{e}}$. Thus
$\left(  \mathfrak{B}_{h}^{e}+\mathfrak{e}\right)  ^{\boxdot}\cap M\left(
\overline{H}\right)  _{he}\subseteq\overline{\mathfrak{B}_{h}^{e}}%
+\overline{\mathfrak{e}}$. Conversely, take $\overline{\eta}=\overline
{\eta_{0}}+\overline{e}^{\oplus n}\in\overline{\mathfrak{B}_{h}^{e}}%
+\overline{\mathfrak{e}}$. Using the matrix Schwarz inequality \cite[3.5.1]%
{ER}, we obtain that $\sup\left\Vert \left\langle \left\langle \mathfrak{B}%
_{h}^{e},\overline{\eta_{0}}\right\rangle \right\rangle \right\Vert \leq
\sup\left\Vert \mathfrak{B}_{h}^{e}\right\Vert _{o}\left\Vert \eta
_{0}\right\Vert _{o}\leq1$, which in turn implies that $\left\langle
\left\langle \mathfrak{B}_{h}^{e}+\mathfrak{e},\overline{\eta}\right\rangle
\right\rangle =\left\langle \left\langle \mathfrak{B}_{h}^{e},\overline
{\eta_{0}}\right\rangle \right\rangle +I\geq0$. The latter means that
$\overline{\eta}\in\left(  \mathfrak{B}_{h}^{e}+\mathfrak{e}\right)
^{\boxdot}\cap M\left(  \overline{H}\right)  _{he}$. Hence $\left(
\mathfrak{B}_{h}^{e}+\mathfrak{e}\right)  ^{\boxdot}\cap M\left(  \overline
{H}\right)  _{he}=\overline{\mathfrak{B}_{h}^{e}}+\overline{\mathfrak{e}}$ and
$\overline{\mathfrak{C}}_{o}^{\boxdot}=\left(  \overline{\mathfrak{B}_{h}^{e}%
}+\overline{\mathfrak{e}}\right)  ^{\boxdot}=\mathfrak{C}_{o}$ by Proposition
\ref{lHS21}.
\end{proof}

\begin{remark}
\label{remBPL1}The unital, quantum cone $\mathfrak{C}_{o}$ on $H$ in Theorem
\ref{tMm1} can be replaced by $\mathcal{S}^{\boxdot}$ for $\mathcal{S=}\left(
2\overline{\mathfrak{B}}\right)  \cap M\left(  \overline{H}\right)  _{he}$.
Namely, note that $\overline{\mathfrak{B}_{h}^{e}}+\overline{\mathfrak{e}%
}\subseteq\left(  2\overline{\mathfrak{B}}\right)  \cap M\left(  \overline
{H}\right)  _{he}\subseteq5\operatorname{abc}\left(  \overline{\mathfrak{B}%
_{h}^{e}}+\overline{\mathfrak{e}}\right)  $. The first inclusion is immediate.
Further, take $\overline{\eta}=\overline{\eta_{0}}+\overline{e}^{\oplus n}%
\in\left(  2\overline{\mathfrak{B}}\right)  \cap M_{n}\left(  \overline
{H}\right)  _{he}$. Then $\left\Vert \overline{\eta_{0}}\right\Vert _{o}%
\leq\left\Vert \overline{\eta}\right\Vert _{o}+\left\Vert \overline{e}^{\oplus
n}\right\Vert _{o}\leq3$, $\overline{\theta}=\left\Vert \overline{\eta_{0}%
}\right\Vert _{o}^{-1}\overline{\eta_{0}}+\overline{e}^{\oplus n}\in
\overline{\mathfrak{B}_{h}^{e}}+\overline{\mathfrak{e}}$ and $\overline{\eta
}=\left\Vert \overline{\eta_{0}}\right\Vert _{o}\overline{\theta}+\left(
1-\left\Vert \overline{\eta_{0}}\right\Vert _{o}\right)  \overline{e}^{\oplus
n}\in5\operatorname{abc}\left(  \overline{\mathfrak{B}_{h}^{e}}+\overline
{\mathfrak{e}}\right)  $. In particular, $5^{-1}\mathfrak{B}_{e}%
\subseteq\mathcal{S}^{\odot}\subseteq\mathfrak{B}_{e}$. As above
$\mathcal{S}^{\odot}$ responds to a unique closed, unital, separated, quantum
cone $\mathfrak{C}$ on $H$ such that $\widehat{\mathfrak{C}}=\mathcal{S}%
^{\odot}$ and $\mathcal{S}\left(  \mathfrak{C}\right)  =\mathcal{S}%
^{\odot\odot}\cap M\left(  \overline{H}\right)  _{he}$. Since
$\mathcal{S\subseteq S}\left(  \mathfrak{C}\right)  $ and $\mathcal{S}^{\odot
}=\widehat{\mathfrak{C}}$, it follows that $\mathfrak{C=}\mathcal{S}^{\boxdot
}$, that is, $\mathcal{S}$ is a prematricial state space of $\mathfrak{C}$
\cite{DSbM}, and the related normed quantum topology coincides with the
original one of $H_{o}$.
\end{remark}

\begin{remark}
\label{remBPL2}The matricial state space $\overline{\mathfrak{B}_{h}^{e}%
}+\overline{\mathfrak{e}}$ can not be replaced by $\mathfrak{B}^{\odot}\cap
M\left(  \overline{H}\right)  _{he}$. Indeed, first note that $\mathfrak{B}%
^{\odot}\cap M\left(  \overline{H}\right)  _{he}=\overline{\mathfrak{B}}\cap
M\left(  \overline{H}\right)  _{he}\subseteq\overline{\mathfrak{B}_{h}^{e}%
}+\overline{\mathfrak{e}}$. Take $\overline{\eta}=\overline{\eta_{0}%
}+\overline{e}^{\oplus n}\in\overline{\mathfrak{B}}\cap M_{n}\left(
\overline{H}\right)  _{he}$. Since $\overline{\mathfrak{B}_{h}^{e}}%
+\overline{\mathfrak{e}}$ is a matrix convex set, it follows that $\eta
_{i,i}=\eta_{0,i,i}+e\in\overline{\mathfrak{B}}\cap\overline{H}_{he}=\left(
\operatorname{ball}\overline{H}\right)  \cap\overline{H}_{he}$ for all $i$.
Taking into account that $\left\Vert \eta_{i,i}\right\Vert ^{2}=\left\Vert
\eta_{0,i,i}\right\Vert ^{2}+1$, we conclude that $\eta_{0,i,i}=0$ for all
$i$, that is, the diagonal of $\overline{\eta_{0}}$ consists of zeros. In
particular, every diagonal entry of $\left\langle \left\langle \eta
_{0},\overline{\eta_{0}}\right\rangle \right\rangle =\left[  \left\langle
\eta_{0,i,k},\overline{\eta_{0,j,l}}\right\rangle \right]  _{\left(
i,j\right)  ,\left(  k,l\right)  }$ is zero. Hence $\left\langle \left\langle
\eta,\overline{\eta}\right\rangle \right\rangle =I+\left\langle \left\langle
\eta_{0},\overline{\eta_{0}}\right\rangle \right\rangle $ in $M_{n^{2}}$ and
the hermitian matrix $\left\langle \left\langle \eta_{0},\overline{\eta_{0}%
}\right\rangle \right\rangle $ admits a positive eigenvalue $\lambda$. It
follows that $1+\lambda$ is an eigenvalue of $\left\langle \left\langle
\eta,\overline{\eta}\right\rangle \right\rangle $. But $\left\Vert
\left\langle \left\langle \eta,\overline{\eta}\right\rangle \right\rangle
\right\Vert =\left\Vert \eta\right\Vert _{o}^{2}\leq1$, therefore $\eta_{0}%
=0$. Consequently, $\mathfrak{B}^{\odot}\cap M\left(  \overline{H}\right)
_{he}=\overline{\mathfrak{e}}$ and $M\left(  H\right)  ^{e}\subseteq
\overline{\mathfrak{e}}^{\odot}=\left(  \mathfrak{B}^{\odot}\cap M\left(
\overline{H}\right)  _{he}\right)  ^{\odot}$, that is, $\left(  \mathfrak{B}%
^{\odot}\cap M\left(  \overline{H}\right)  _{he}\right)  ^{\odot}$ is an
unbounded quantum set, which can not generate the original normed quantum
topology of $H_{o}$.
\end{remark}

\begin{remark}
In the case of a finite dimensional Hilbert space $H$ of dimension $n$ the
quantum cone $\mathfrak{C}_{o}$ is reduced to one from \cite{NP}, that is,
$\left(  H,\mathfrak{C}_{o}\right)  =\operatorname{SOH}\left(  n\right)  $.
Namely, let us prove that if $\mathfrak{C}$ is the quantum cone of the
operator system $\operatorname{SOH}\left(  n\right)  $ then $\mathcal{S}%
\left(  \mathfrak{C}\right)  =\overline{\mathfrak{B}_{h}^{e}}+\overline
{\mathfrak{e}}$. First notice that if $\zeta=\zeta_{0}+ae^{\oplus m}%
\in\mathfrak{C\cap}M_{m}\left(  H\right)  $ then $a\geq0$, $\zeta_{0}^{\ast
}=\zeta_{0}$ and $-ae^{\oplus m}\leq\zeta_{0}\leq ae^{\oplus m}$ in the
operator system $M_{m}\left(  \operatorname{SOH}\left(  n\right)  \right)  $
(see \cite[Proposition 3.3]{NP}). If $a=I$ the latter is equivalent to
$\left\Vert \zeta_{0}\right\Vert _{o}\leq1$. Thus $\mathfrak{B}_{h}%
^{e}+\mathfrak{e\subseteq C}$. Further, take $\overline{\eta}=\overline
{\eta_{0}}+c\overline{e}^{\oplus k}\in\mathcal{S}\left(  \mathfrak{C}\right)
$ then $c=\left\langle \left\langle e,\overline{\eta}\right\rangle
\right\rangle =I$ and $\left\langle \left\langle \mathfrak{B}_{h}%
^{e}+\mathfrak{e,}\overline{\eta}\right\rangle \right\rangle \geq0$. In
particular, $\left\langle \left\langle \mathfrak{B}_{h}^{e}\mathfrak{,}%
\overline{\eta_{0}}\right\rangle \right\rangle +I\geq0$, which means that
$\sup\left\Vert \left\langle \left\langle \mathfrak{B}_{h}^{e}\mathfrak{,}%
\overline{\eta_{0}}\right\rangle \right\rangle \right\Vert \leq1$. Taking into
account that $\eta_{0}/\left\Vert \eta_{0}\right\Vert _{o}\in\mathfrak{B}%
_{h}^{e}$, we derive that $\left\Vert \eta_{0}\right\Vert _{o}=\left\Vert
\left\langle \left\langle \eta_{0}/\left\Vert \eta_{0}\right\Vert
_{o}\mathfrak{,}\overline{\eta_{0}}\right\rangle \right\rangle \right\Vert
\leq\sup\left\Vert \left\langle \left\langle \mathfrak{B}_{h}^{e}%
\mathfrak{,}\overline{\eta_{0}}\right\rangle \right\rangle \right\Vert \leq1$,
that is, $\overline{\eta_{0}}\in\overline{\mathfrak{B}_{h}^{e}}$. Thus
$\mathcal{S}\left(  \mathfrak{C}\right)  \subseteq\overline{\mathfrak{B}%
_{h}^{e}}+\overline{\mathfrak{e}}$. Conversely, $\operatorname{SOH}\left(
n\right)  $ is a self-dual operator system \cite[Theorem 3.4]{NP}, therefore
$\overline{\mathfrak{B}_{h}^{e}}+\overline{\mathfrak{e}}\subseteq
\overline{\mathfrak{C}}=\mathfrak{C}^{\boxdot}$, which in turn implies that
$\overline{\mathfrak{B}_{h}^{e}}+\overline{\mathfrak{e}}\subseteq
\mathfrak{C}^{\boxdot}\cap M\left(  \overline{H}\right)  _{e}=\mathcal{S}%
\left(  \mathfrak{C}\right)  $. Consequently, $\mathfrak{C=}\mathcal{S}\left(
\mathfrak{C}\right)  ^{\boxdot}=\left(  \overline{\mathfrak{B}_{h}^{e}%
}+\overline{\mathfrak{e}}\right)  ^{\boxdot}=\mathfrak{C}_{o}$ thanks to the
Unital Bipolar Theorem \ref{tUBP} and Proposition \ref{lHS21}.
\end{remark}

\section{The positive maps of operator Hilbert systems\label{sectionPMOH}}

In this section we analyze the positive maps between ordered Hilbert spaces.
Everywhere below $X$ denotes a Hausdorff compact topological space, $C\left(
X\right)  $ the abelian $C^{\ast}$-algebra of all complex continuous functions
on $X$ with the norm $\left\Vert v\right\Vert _{\infty}=\sup\left\vert
v\left(  X\right)  \right\vert $, $v\in C\left(  X\right)  $, and the unital
quantum cone $M\left(  C\left(  X\right)  \right)  _{+}$ of all positive
matrix valued functions on $X$, which is a quantization of the cone $C\left(
X\right)  _{+}$.

\subsection{Positive maps between unital Hilbert spaces\label{subsecPM11}}

Now let $\left(  K,u\right)  $ and $\left(  H,e\right)  $ be unital Hilbert
spaces with the related unital cones $\mathfrak{c}_{u}$ and $\mathfrak{c}_{e}%
$, respectively, and let $F$ be a hermitian basis for $H$, which contains $e$.
A bounded family $k=\left\{  k_{f}:f\in F\right\}  \subseteq K_{h}$ is said to
be \textit{an }$H$\textit{-support in} $K$ if
\[
k_{e}\in\mathfrak{c}_{u}\quad\text{and\quad}\sum_{f\neq e}\left(  \eta
,k_{f}\right)  ^{2}\leq\left(  \eta,k_{e}\right)  ^{2}\quad\text{for all\quad
}\eta\in\mathfrak{c}_{u}.
\]
In this case, $\left(  \eta,k_{e}\right)  \geq0$ for all $\eta\in
\mathfrak{c}_{u}$. Indeed, put $k_{f}=k_{f}^{u}+r_{f}u$ with $k_{f}^{u}\in
K_{h}^{u}$, $r_{f}\in\mathbb{R}$. Note that $r_{e}\geq0$ and $\left\Vert
k_{e}^{u}\right\Vert \leq r_{e}$, for $k_{e}=k_{e}^{u}+r_{e}u\in
\mathfrak{c}_{u}$. It follows that $\left(  \eta,k_{e}\right)  =\left(
\eta_{0},k_{e}^{u}\right)  +r_{e}\left(  \eta,u\right)  $ and $\left\vert
\left(  \eta_{0},k_{e}^{u}\right)  \right\vert \leq\left\Vert \eta
_{0}\right\Vert \left\Vert k_{e}^{u}\right\Vert \leq r_{e}\left\Vert \eta
_{0}\right\Vert $. In particular, if $\eta\in\mathfrak{c}_{u}$ then
$\left\Vert \eta_{0}\right\Vert \leq\left(  \eta,u\right)  $, and $\left\vert
\left(  \eta,k_{e}^{u}\right)  \right\vert =\left\vert \left(  \eta_{0}%
,k_{e}^{u}\right)  \right\vert \leq r_{e}\left(  \eta,u\right)  $, which in
turn implies that $\left(  \eta,k_{e}\right)  =\left(  \eta_{0},k_{e}%
^{u}\right)  +r_{e}\left(  \eta,u\right)  \geq0$. Note also that $\sum_{f\neq
e}r_{f}^{2}=\sum_{f\neq e}\left(  u,k_{f}\right)  ^{2}\leq\left(
u,k_{e}\right)  ^{2}=r_{e}^{2}$, for $u\in\mathfrak{c}_{u}$. If $r_{e}=1$ and
$r_{f}=0$ for all $f\neq e$ then we say that $k$ is \textit{a unital }%
$H$-\textit{support.}

\begin{remark}
\label{remUHS}Let $k=\left\{  k_{f}:f\in F\right\}  \subseteq K_{h}$ be a
bounded family with $k_{e}\in\operatorname{ball}K_{h}^{u}+u$ and $k_{f}\perp
u$, $f\neq e$. Then $k$ is a unital $H$-support iff $\sum_{f\neq e}\left(
\eta_{0},k_{f}\right)  ^{2}\leq\left(  \left(  \eta_{0},k_{e}\right)
+\left\Vert \eta_{0}\right\Vert \right)  ^{2}$ for all $\eta_{0}\in K_{h}^{u}%
$. Indeed, since $\eta=\eta_{0}+\left\Vert \eta_{0}\right\Vert u\in
\mathfrak{c}_{u}$ for every $\eta_{0}\in K_{h}^{u}$, it follows that
$\sum_{f\neq e}\left(  \eta_{0},k_{f}\right)  ^{2}=\sum_{f\neq e}\left(
\eta,k_{f}\right)  ^{2}\leq\left(  \eta,k_{e}\right)  ^{2}=\left(  \left(
\eta_{0},k_{e}\right)  +\left\Vert \eta_{0}\right\Vert \right)  ^{2}$. Note
also that $\left(  \eta_{0},k_{e}\right)  =0$ whenever $k_{e}=u$ (see below
Subsection \ref{subsecCHe}).
\end{remark}

If additionally $\sum_{f}\left\Vert k_{f}\right\Vert ^{p}<\infty$ we say that
$k$ is \textit{of type }$p$, where $p=1,2$. An $H$-support $k$ in $K$ defines
a linear operator
\[
T_{k}:K\rightarrow H,\quad T_{k}\eta=\sum_{f}\left(  \eta,k_{f}\right)  f,
\]
which is positive in the sense of $T_{k}\left(  \mathfrak{c}_{u}\right)
\subseteq\mathfrak{c}_{e}$. In particular, $T_{k}$ is a $\ast$-linear mapping.
Note that $\left\vert \left(  T_{k}\eta,e\right)  \right\vert =\left\vert
\left(  \eta,k_{e}\right)  \right\vert \leq\sup\left\Vert k\right\Vert $ for
all $\eta\in\operatorname{ball}K$. It follows that
\begin{align*}
\left\Vert T_{k}\eta\right\Vert  &  \leq\left\Vert T_{k}\eta_{0}\right\Vert
+\left\vert \left(  \eta,u\right)  \right\vert \left\Vert T_{k}u\right\Vert
\leq\left\Vert T_{k}\eta_{0}\right\Vert +\left\Vert T_{k}u\right\Vert
\leq\left\Vert T_{k}\left(  \eta_{0}+u\right)  \right\Vert +2\left\Vert
T_{k}u\right\Vert \\
&  \leq\sqrt{2}\left(  T_{k}\left(  \eta_{0}+3u\right)  ,e\right)  \leq
4\sqrt{2}\sup\left\Vert k\right\Vert
\end{align*}
for all $\eta=\eta_{0}+\left(  \eta,u\right)  u\in\operatorname{ball}K_{h}$.
For $\eta\in\operatorname{ball}K$ we derive that $\left\Vert T_{k}%
\eta\right\Vert \leq\left\Vert T_{k}\operatorname{Re}\eta\right\Vert
+\left\Vert T_{k}\operatorname{Im}\eta\right\Vert \leq8\sqrt{2}\sup\left\Vert
k\right\Vert $, that is, $T_{k}\in\mathcal{B}\left(  K,H\right)  $ with
$\left\Vert T_{k}\right\Vert \leq8\sqrt{2}\sup\left\Vert k\right\Vert $.

\begin{remark}
\label{remHSop}Let $k$ be an $H$-support in $K$. Then $T_{k}\in\mathcal{B}%
^{2}\left(  K,H\right)  $ iff $k$ is of type $2$. Indeed, take a Hilbert basis
$\left(  \eta_{i}\right)  _{i\in I}$ for $K$. Since $T_{k}\eta=\sum_{f}\left(
\eta,k_{f}\right)  f$, $\eta\in K$, we deduce that $\sum_{f}\left\Vert
k_{f}\right\Vert _{2}^{2}=\sum_{f}\sum_{i\in I}\left\vert \left(  k_{f}%
,\eta_{i}\right)  \right\vert ^{2}=\sum_{i\in I}\sum_{f}\left\vert \left(
\eta_{i},k_{f}\right)  \right\vert ^{2}=\sum_{i\in I}\left\Vert T\eta
_{i}\right\Vert ^{2}=\left\Vert T\right\Vert _{2}^{2}$. If $k$ is of type $1$
then $T_{k}\in\mathcal{B}^{1}\left(  K,H\right)  $, $\left\Vert T\right\Vert
_{1}\leq\sum_{f}\left\Vert f\right\Vert \left\Vert k_{f}\right\Vert \leq
\sum_{f}\left\Vert k_{f}\right\Vert <\infty$.
\end{remark}

\begin{proposition}
\label{propKom1}If $T:\left(  K,u\right)  \rightarrow\left(  H,e\right)  $ is
a (untal) positive mapping then $T=T_{k}$ for a certain (unital) $H$-support
$k$ in $K$.
\end{proposition}

\begin{proof}
First note that $Tu=\zeta_{0}+re\in\mathfrak{c}_{e}$, that is, $\zeta_{0}%
=\sum_{f\neq e}r_{f}f\in H_{h}^{e}$, $r\geq0$ and $\sum_{f\neq e}r_{f}%
^{2}=\left\Vert \zeta_{0}\right\Vert ^{2}\leq r^{2}$. For $\eta\in K$ we have
\[
T\eta=T\eta_{0}+\left(  \eta,u\right)  Tu=\sum_{f}\left(  T\eta_{0},f\right)
f+\left(  \eta,u\right)  \left(  \zeta_{0}+re\right)  =S\eta_{0}+\left(
\eta,u\right)  \zeta_{0}+\left(  \gamma\left(  \eta_{0}\right)  +\left(
\eta,ru\right)  \right)  e,
\]
where $S:K^{u}\rightarrow H^{e}$, $S\eta_{0}=\sum_{f\neq e}\left(  T\eta
_{0},f\right)  f$, and $\gamma:K^{u}\rightarrow\mathbb{C}$, $\gamma\left(
\eta_{0}\right)  =\left(  T\eta_{0},e\right)  $. Take $\eta_{0}\in K_{h}^{u}$.
Then $\eta=\eta_{0}+\left\Vert \eta_{0}\right\Vert u\in\mathfrak{c}_{u}$ and
$S\eta_{0}+\left\Vert \eta_{0}\right\Vert \zeta_{0}+\left(  \gamma\left(
\eta_{0}\right)  +r\left\Vert \eta_{0}\right\Vert \right)  e=T\eta
\in\mathfrak{c}_{e}$. It follows that $S\eta_{0}\in H_{h}^{e}$, $\gamma\left(
\eta_{0}\right)  \geq-r\left\Vert \eta_{0}\right\Vert $ and $\left\Vert
S\eta_{0}+\left\Vert \eta_{0}\right\Vert \zeta_{0}\right\Vert \leq
\gamma\left(  \eta_{0}\right)  +r\left\Vert \eta_{0}\right\Vert $. In
particular, $S$ and $\gamma$ are $\ast$-linear maps, $\left\vert \gamma\left(
\eta_{0}\right)  \right\vert \leq r\left\Vert \eta_{0}\right\Vert $ and
\[
\left\Vert S\eta_{0}\right\Vert \leq\left\Vert S\eta_{0}+\left\Vert \eta
_{0}\right\Vert \zeta_{0}\right\Vert +\left\Vert \left\Vert \eta
_{0}\right\Vert \zeta_{0}\right\Vert \leq\left\vert \gamma\left(  \eta
_{0}\right)  \right\vert +2r\left\Vert \eta_{0}\right\Vert \leq3r\left\Vert
\eta_{0}\right\Vert
\]
for all $\eta_{0}\in K_{h}^{u}$. Thus $\gamma\in\left(  K^{u}\right)  ^{\ast}%
$, $\left\vert \gamma\left(  \eta_{0}\right)  \right\vert \leq\left\vert
\gamma\left(  \operatorname{Re}\eta_{0}\right)  \right\vert +\left\vert
\gamma\left(  \operatorname{Im}\eta_{0}\right)  \right\vert \leq2r\left\Vert
\eta_{0}\right\Vert $ and $\left\Vert S\eta_{0}\right\Vert \leq\left\Vert
S\operatorname{Re}\eta_{0}\right\Vert +\left\Vert S\operatorname{Im}\eta
_{0}\right\Vert \leq6r\left\Vert \eta_{0}\right\Vert $ for all $\eta_{0}\in
K^{u}$. But $K^{u}$ is a Hilbert space, therefore $\gamma\left(  \eta
_{0}\right)  =\left(  \eta_{0},\gamma_{0}\right)  $ for a certain $\gamma
_{0}\in K^{u}$. Since $\left(  \eta_{0},\gamma_{0}^{\ast}\right)  =\left(
\eta_{0}^{\ast},\gamma_{0}\right)  ^{\ast}=\gamma\left(  \eta_{0}^{\ast
}\right)  ^{\ast}=\gamma\left(  \eta_{0}\right)  =\left(  \eta_{0},\gamma
_{0}\right)  $, $\eta_{0}\in K^{u}$, it follows that $\gamma_{0}\in K_{h}^{u}$
and $\left\Vert \gamma_{0}\right\Vert ^{2}=\left(  \gamma_{0},\gamma
_{0}\right)  =\gamma\left(  \gamma_{0}\right)  \leq r\left\Vert \gamma
_{0}\right\Vert $, that is, $\gamma_{0}\in r\operatorname{ball}K_{h}^{u}$. Put
$k_{e}=\gamma_{0}+ru\in r\operatorname{ball}K_{h}^{u}+ru\subseteq
\mathfrak{c}_{u}$. Thus
\[
T\eta=S\eta_{0}+\left(  \eta,u\right)  \zeta_{0}+\left(  \left(  \eta
,\gamma_{0}\right)  +\left(  \eta,ru\right)  \right)  e=S\eta_{0}+\left(
\eta,u\right)  \zeta_{0}+\left(  \eta,k_{e}\right)  e
\]
and both $S$ and $T$ are bounded $\ast$-linear operators. It follows that
$S\eta_{0}=\sum_{f\neq e}\left(  S\eta_{0},f\right)  f=\sum_{f\neq e}\left(
\eta_{0},k_{f}^{u}\right)  f$ for uniquely defined $k_{f}^{u}=S^{\ast}f\in
K_{h}^{u}$, $\left\Vert k_{f}^{u}\right\Vert \leq\left\Vert S^{\ast
}\right\Vert \leq6r$, $f\neq e$. Put $k=\left\{  k_{f}:f\in F\right\}  $ with
$k_{f}=k_{f}^{u}+r_{f}u$. Note that $\left\Vert k_{f}\right\Vert
^{2}=\left\Vert k_{f}^{u}\right\Vert ^{2}+r_{f}^{2}\leq37r^{2}$ for all $f\neq
e$, and $\left\Vert k_{e}\right\Vert ^{2}=\left\Vert \gamma_{0}\right\Vert
^{2}+r^{2}\leq2r^{2}$, that is, $\sup\left\Vert k\right\Vert \leq7r$.
Moreover,
\[
T\eta=\sum_{f\neq e}\left(  \eta_{0},k_{f}^{u}\right)  f+\sum_{f\neq e}\left(
\eta,r_{f}u\right)  f+\left(  \eta,k_{e}\right)  e=\sum_{f\neq e}\left(
\eta,k_{f}\right)  f+\left(  \eta,k_{e}\right)  e=\sum_{f}\left(  \eta
,k_{f}\right)  f.
\]
Finally, for $\eta\in\mathfrak{c}_{u}$ we have $T\eta=S\eta_{0}+\left(
\eta,u\right)  \zeta_{0}+\left(  \gamma\left(  \eta_{0}\right)  +\left(
\eta,ru\right)  \right)  e\in\mathfrak{c}_{e}$ and
\[
\sum_{f\neq e}\left(  \eta,k_{f}\right)  ^{2}=\left\Vert S\eta_{0}+\left(
\eta,u\right)  \zeta_{0}\right\Vert ^{2}\leq\left(  \gamma\left(  \eta
_{0}\right)  +\left(  \eta,ru\right)  \right)  ^{2}=\left(  \left(
\eta,\gamma_{0}\right)  +\left(  \eta,ru\right)  \right)  ^{2}=\left(
\eta,k_{e}\right)  ^{2},
\]
which means that $k$ is an $H$-support in $K$ and $T\eta=T_{k}\eta$ for all
$\eta\in K$. If $T$ is a unital positive mapping then $\zeta_{0}=0$, that is,
$r_{f}=0$ for all $f\neq e$, and $r=1$. The latter means that $k$ is a unital
$H$-support (see Remark \ref{remUHS}).
\end{proof}

\subsection{The unital cone $L^{2}\left(  X,\mu\right)  _{+}$%
\label{subsecL1XM}}

The matrix algebra $M_{n}\left(  C\left(  X\right)  \right)  $ is identified
with the algebra $C\left(  X,M_{n}\right)  $ of all $M_{n}$-valued continuous
functions on $X$. The following result is known (see \cite[Theorem
3.2]{PaulTT}). For the sake of a reader we provide its detailed proof within
the duality context, which is a bit different than its original one.

\begin{proposition}
\label{propPTT1}The equality holds $M\left(  C\left(  X\right)  \right)
_{+}=\min C\left(  X\right)  _{+}$.
\end{proposition}

\begin{proof}
By its very definition $S\left(  C\left(  X\right)  _{+}\right)
=\mathcal{P}\left(  X\right)  $ is the space of all probability measures on
$X$. Note that $\mathcal{P}\left(  X\right)  $ is a $w^{\ast}$-compact subset
of the space $\mathcal{M}\left(  X\right)  =C\left(  X\right)  ^{\ast}$ of all
finite Radon charges on $X$. Based on Krein-Milman theorem, we conclude that
$\mathcal{P}\left(  X\right)  $ is the $w^{\ast}$-closure of the convex hull
of its extremal boundary $\partial\mathcal{P}\left(  X\right)  $ which
consists of Dirac measures $\delta_{t}$, $t\in X$. For every $v\in C\left(
X\right)  $ we have $\left\Vert v\right\Vert _{\infty}=\sup\left\{  \left\vert
v\left(  t\right)  \right\vert :t\in X\right\}  =\sup\left\{  \left\vert
\left\langle v,\delta_{t}\right\rangle \right\vert :t\in X\right\}
=\sup\left\vert \left\langle v,\partial\mathcal{P}\left(  X\right)
\right\rangle \right\vert \leq\sup\left\vert \left\langle v,\mathcal{P}\left(
X\right)  \right\rangle \right\vert =\left\Vert v\right\Vert _{e}\leq
\sup\left\vert \left\langle v,\operatorname{ball}\mathcal{M}\left(  X\right)
\right\rangle \right\vert =\left\Vert v\right\Vert _{\infty}$, that is,
$\left\Vert v\right\Vert _{\infty}=\left\Vert v\right\Vert _{e}$ (see
Subsection \ref{subsecUQC}). It follows that $\min C\left(  X\right)
_{+}=\mathcal{P}\left(  X\right)  ^{\boxdot}=\left(  \partial\mathcal{P}%
\left(  X\right)  \right)  ^{\boxdot}$. Take $v\in M_{n}\left(  C\left(
X\right)  \right)  $. Then $v\in\min C\left(  X\right)  _{+}$ iff
$\left\langle \left\langle v,\partial\mathcal{P}\left(  X\right)
\right\rangle \right\rangle \geq0$. The latter means (see Proposition
\ref{pmin11}) that $\left(  v\left(  t\right)  a,a\right)  =a^{\ast}v\left(
t\right)  a=\left(  a^{\ast}va\right)  \left(  t\right)  =\left\langle
a^{\ast}va,\delta_{t}\right\rangle =\left(  \left\langle \left\langle
v,\delta_{t}\right\rangle \right\rangle a,a\right)  \geq0$ for all $t\in X$
and $a\in M_{n,1}$, that is, $v\in M_{n}\left(  C\left(  X\right)  \right)
_{+}$. Whence $M_{n}\left(  C\left(  X\right)  \right)  \cap\min C\left(
X\right)  _{+}=M_{n}\left(  C\left(  X\right)  \right)  _{+}$ for all $n$.
\end{proof}

Now fix $\mu\in\mathcal{M}\left(  X\right)  _{+}$ and consider the Hilbert
$\ast$-space $H=L^{2}\left(  X,\mu\right)  $ with the canonical representation
mapping $\iota:C\left(  X\right)  \rightarrow L^{2}\left(  X,\mu\right)  $.
Put $\iota\left(  1\right)  =u$. Note that $\iota$ a $\ast$-linear mapping and
$\mu\left(  X\right)  ^{1/2}=\left(  \int1\right)  ^{1/2}=\left\Vert
u\right\Vert _{2}=\left\Vert \iota\left(  1\right)  \right\Vert _{2}%
\leq\left\Vert \iota\right\Vert \left\Vert 1\right\Vert _{\infty}=\left\Vert
\iota\right\Vert =\sup\left\{  \left\Vert \iota\left(  \operatorname{ball}%
C\left(  X\right)  \right)  \right\Vert _{2}\right\}  \leq\sup\left\{
\left\Vert \operatorname{ball}C\left(  X\right)  \right\Vert _{\infty}\left(
\int1\right)  ^{1/2}\right\}  \leq\mu\left(  X\right)  ^{1/2}$, that is,
$\left\Vert \iota\right\Vert =\mu\left(  X\right)  ^{1/2}$. Recall that each
element $\eta^{\sim}\in L^{2}\left(  X,\mu\right)  $ being an equivalence
class has a Borel representative $\eta$. We use the same notation $\eta$ for
the class $\eta^{\sim}$ either. If $\mu\in\mathcal{P}\left(  X\right)  $ then
$u$ takes place the role of a unit in $L^{2}\left(  X,\mu\right)  $, and the
related cone $\mathfrak{c}$ consists of those real-valued Borel functions
$\eta$ on $X$ such that $\eta=\eta_{0}+ru$ with $\eta_{0}\perp u$, $r=\left(
\eta,u\right)  =\int\eta d\mu\geq0$ and $\int\eta_{0}\left(  t\right)
^{2}d\mu\leq r^{2}$. We use the notation $L^{2}\left(  X,\mu\right)  _{+}$
instead of $\mathfrak{c}$. Recall that $L^{2}\left(  X,\mu\right)  $ possesses
another conventional cone lifted from the cone $C\left(  X\right)  _{+}$. Thus
a hermitian class $\eta^{\sim}\in L^{2}\left(  X,\mu\right)  $ is positive iff
$\eta\left(  t\right)  \geq0$ for $\mu$-almost all $t\in X$. These cones are
essentially distinct. A real-valued Borel representative of a class from the
cone $L^{2}\left(  X,\mu\right)  _{+}$ could take an highly negative values
being far to be positive in the ordinary sense.

\begin{example}
\label{exm11}Let us equip the compact interval $X=\left[  -1,1\right]
\subseteq\mathbb{R}$ with Lebesgue's measure $2^{-1}dt$. Put $\eta
=\chi_{\left[  -1,-1/n\right]  }+\left(  1-\sqrt{n}\right)  \chi_{\left[
-1/n,0\right]  }+\left(  1+\sqrt{n}\right)  \chi_{\left[  0,1/n\right]  }%
+\chi_{\left[  1/n,1\right]  }$, where $\chi_{M}$ indicates to the
characteristic function of a subset $M$ from $X$. Note that $\eta=\eta_{0}+1$
with $\eta_{0}=-\sqrt{n}\chi_{\left[  -1/n,0\right]  }+\sqrt{n}\chi_{\left[
0,1/n\right]  }$. Since $\int\eta_{0}=2^{-1}\left(  -\sqrt{n}/n+\sqrt
{n}/n\right)  =0$, we conclude that $\eta=\eta_{0}+1$ is an orthogonal
expansion in $L^{2}\left(  X,\mu\right)  _{+}$ and $\int\eta_{0}\left(
t\right)  ^{2}d\mu=1$, that is, $\eta\in L^{2}\left(  X,\mu\right)  _{+}$. But
$\eta\left(  \left[  -1/n,0\right]  \right)  =1-\sqrt{n}<0$ for $n>1$. A very
similar example can be constructed with a continuous function (or
representative) $\eta$.
\end{example}

\begin{corollary}
\label{cormin12}Let $\eta\in M_{n}\left(  L^{2}\left(  X,\mu\right)  \right)
_{h}$ with its expansion $\eta=\eta_{0}+au^{\oplus n}$, $a=\left\langle
\left\langle \eta,u\right\rangle \right\rangle \in M_{n}$. Then $\eta\in\min
L^{2}\left(  X,\mu\right)  _{+}$ iff $a\geq0$ and $\int\left(  \eta_{0}\left(
t\right)  \beta,\beta\right)  ^{2}d\mu\leq\left(  a\beta,\beta\right)  ^{2}$
for all $\beta\in M_{n,1}$. In the case of an atomic measure $\mu$
concentrated on a countable subset $S\subseteq X$ we have $-a\leq\mu\left(
s\right)  ^{1/2}\eta_{0}\left(  s\right)  \leq a$ in $M_{n}$, $s\in S$
whenever $\eta\in\min L^{2}\left(  X,\mu\right)  _{+}$.
\end{corollary}

\begin{proof}
Using Proposition \ref{pmin11}, we deduce that $\eta\in\min L^{2}\left(
X,\mu\right)  _{+}$ iff $\beta^{\ast}\eta\beta\in L^{2}\left(  X,\mu\right)
_{+}$ for all $\beta\in M_{n,1}$. Since $\beta^{\ast}\eta\beta=\beta^{\ast
}\eta_{0}\beta+\beta^{\ast}a\beta u$, it follows that $\beta^{\ast}a\beta
\geq0$ and $\left\Vert \beta^{\ast}\eta_{0}\beta\right\Vert _{2}\leq
\beta^{\ast}a\beta$. But $\eta_{0}$ is identified with a Borel function
$\eta_{0}:X\rightarrow M_{n}$ and $\left(  \beta^{\ast}\eta_{0}\beta\right)
\left(  t\right)  =\beta^{\ast}\eta_{0}\left(  t\right)  \beta=\left(
\eta_{0}\left(  t\right)  \beta,\beta\right)  $ for all $\beta\in M_{n,1}$.
Similarly, $\beta^{\ast}a\beta=\left(  a\beta,\beta\right)  \geq0$, which
means that $a\geq0$. Thus $\eta\in\min L^{2}\left(  X,\mu\right)  _{+}$ iff
$a\geq0$ and $\int\left(  \eta_{0}\left(  t\right)  \beta,\beta\right)
^{2}d\mu=\int\left(  \beta^{\ast}\eta_{0}\beta\right)  \left(  t\right)
^{2}d\mu=\left\Vert \beta^{\ast}\eta_{0}\beta\right\Vert _{2}^{2}\leq\left(
a\beta,\beta\right)  ^{2}$.

Finally, assume that $\mu$ is an atomic measure concentrated on $S$, and
$\eta\in\min L^{2}\left(  X,\mu\right)  _{+}$. For every $s\in S$ we have
$\mu\left(  s\right)  \left(  \eta_{0}\left(  s\right)  \beta,\beta\right)
^{2}\leq\int\left(  \eta_{0}\left(  t\right)  \beta,\beta\right)  ^{2}d\mu
\leq\left(  a\beta,\beta\right)  ^{2}$, which in turn implies that $\mu\left(
s\right)  ^{1/2}\left\vert \left(  \eta_{0}\left(  s\right)  \beta
,\beta\right)  \right\vert \leq\left(  a\beta,\beta\right)  $ for all
$\beta\in M_{n,1}$, that is, $-a\leq\mu\left(  s\right)  ^{1/2}\eta_{0}\left(
s\right)  \leq a$ in $M_{n}$ for all $s\in S$.
\end{proof}

\begin{remark}
If $v\in C\left(  X\right)  _{+}$ with its orthogonal expansion $v=v_{0}+ru$
in $L^{2}\left(  X,\mu\right)  $ satisfies an extra positivity
condition$-v_{0}+ru\geq0$ in $C\left(  X\right)  $ then $v\in L^{2}\left(
X,\mu\right)  _{+}$. Indeed, since $v\geq0$, it follows that $r=\left(
v,u\right)  =\int v\geq0$. Moreover, $-ru\leq v_{0}\leq ru$ or $v_{0}^{2}\leq
r^{2}u$, which in turn implies that $\left\Vert v_{0}\right\Vert _{2}=\left(
\int v_{0}^{2}\right)  ^{1/2}\leq r\left(  \int u\right)  =r$, that is, $v\in
L^{2}\left(  X,\mu\right)  _{+}$. Conversely, if $v\in L^{2}\left(
X,\mu\right)  _{+}$ for some $v\in C\left(  X\right)  $, and $\mu$ is atomic
measure concentrated on $S$, then using Corollary \ref{cormin12}, we derive
that $-r\mu\left(  s\right)  ^{-1/2}\leq v_{0}\left(  s\right)  \leq
r\mu\left(  s\right)  ^{-1/2}$ for all $s\in S$. Thus $\pm v_{0}+r\mu
^{-1/2}\geq0$.
\end{remark}

Thus the canonical, unital $\ast$-linear mapping $\iota:C\left(  X\right)
\rightarrow L^{2}\left(  X,\mu\right)  $ is not positive is the sense of
Subsection \ref{subsecPM11}.

\begin{proposition}
Let $A\subseteq X$ be a $\mu$-measurable subset with $\mu\left(  A\right)
>0$. Then $\chi_{A}\in L^{2}\left(  X,\mu\right)  _{+}$ iff $\mu\left(
A\right)  \geq1/2$.
\end{proposition}

\begin{proof}
First notice that $\left(  \chi_{A},u\right)  =\mu\left(  A\right)  $ and
$\chi_{A}-\mu\left(  A\right)  u\in L^{2}\left(  X,\mu\right)  _{h}^{u}$. Thus
$\chi_{A}=\left(  \chi_{A}-\mu\left(  A\right)  u\right)  +\mu\left(
A\right)  u$ is the orthogonal decomposition of $\chi_{A}$ in $L^{2}\left(
X,\mu\right)  $. It follows that $\chi_{A}\in L^{2}\left(  X,\mu\right)  _{+}$
iff $\left\Vert \chi_{A}-\mu\left(  A\right)  u\right\Vert _{2}\leq\mu\left(
A\right)  $. But
\begin{align*}
\left\Vert \chi_{A}-\mu\left(  A\right)  u\right\Vert _{2}^{2}  &
=\int\left(  \chi_{A}\left(  t\right)  -\mu\left(  A\right)  \right)  ^{2}%
d\mu=\int_{A}\left(  \chi_{A}\left(  t\right)  -\mu\left(  A\right)  \right)
^{2}d\mu+\int_{X\backslash A}\left(  \chi_{A}\left(  t\right)  -\mu\left(
A\right)  \right)  ^{2}d\mu\\
&  =\left(  1-\mu\left(  A\right)  \right)  ^{2}\mu\left(  A\right)
+\mu\left(  A\right)  ^{2}\mu\left(  X\backslash A\right)  =\left(
1-\mu\left(  A\right)  \right)  \mu\left(  A\right)  .
\end{align*}
Thus $\left(  1-\mu\left(  A\right)  \right)  \mu\left(  A\right)  \leq
\mu\left(  A\right)  ^{2}$ iff $\mu\left(  A\right)  \geq1/2$.
\end{proof}

Finally, suppose that $\mu\sim\mu^{\prime}$ in $\mathcal{P}\left(  X\right)
$, that is, $I_{\mu}\left(  X\right)  =I_{\mu^{\prime}}\left(  X\right)  $. By
Lebesgue-Nikodym Theorem, $\mu^{\prime}=k\mu$ for some $k\in L^{1}\left(
X,\mu\right)  $ such that $k\left(  t\right)  >0$ for $\mu$-almost all $t\in
X$ and $\int k\left(  t\right)  d\mu=1$. In this case, $k^{-1}\in L^{1}\left(
X,\mu^{\prime}\right)  $ or $k^{-1/2}\in L^{2}\left(  X,\mu^{\prime}\right)
$. Moreover, $L^{2}\left(  X,\mu\right)  $ is identified with $L^{2}\left(
X,\mu^{\prime}\right)  $ along with the $\ast$-linear unitary $U:L^{2}\left(
X,\mu\right)  \rightarrow L^{2}\left(  X,\mu^{\prime}\right)  $, $U\left(
\eta\right)  =\eta/\sqrt{k}$. Namely,%
\[
\left(  U\eta_{1},U\eta_{2}\right)  ^{\prime}=\int\eta_{1}\left(  t\right)
\eta_{2}^{\ast}\left(  t\right)  k\left(  t\right)  ^{-1}d\mu^{\prime}%
=\int\eta_{1}\left(  t\right)  \eta_{2}^{\ast}\left(  t\right)  d\mu=\left(
\eta_{1},\eta_{2}\right)
\]
for all $\eta_{i}\in L^{2}\left(  X,\mu\right)  $. Note that $u^{\prime
}=u/\sqrt{k}$ is a unit vector in $L^{2}\left(  X,\mu^{\prime}\right)  $, and
we have the related unital cone $L^{2}\left(  X,\mu^{\prime}\right)  _{+}$. If
$\eta\in L^{2}\left(  X,\mu\right)  _{+}$ then $U\left(  \eta\right)  \in
L^{2}\left(  X,\mu^{\prime}\right)  _{h}$ and
\begin{align*}
\left\Vert U\left(  \eta\right)  \right\Vert _{2}^{\prime}  &  =\left\Vert
\eta\right\Vert _{2}\leq\sqrt{2}\left(  \eta,u\right)  =\sqrt{2}\int%
\eta\left(  t\right)  d\mu=\sqrt{2}\int\left(  \eta/\sqrt{k_{e}}\right)
\left(  t\right)  \left(  1/\sqrt{k_{e}}\right)  \left(  t\right)
d\mu^{\prime}\\
&  =\sqrt{2}\left(  U\left(  \eta\right)  ,u^{\prime}\right)  ^{\prime},
\end{align*}
which means that $U\left(  \eta\right)  \in L^{2}\left(  X,\mu^{\prime
}\right)  $. Thus $UL^{2}\left(  X,\mu\right)  _{+}=L^{2}\left(  X,\mu
^{\prime}\right)  _{+}$ or $U$ is an order isomorphism of the related unital
Hilbert spaces. In this case, $U\iota:C\left(  X\right)  \rightarrow
L^{2}\left(  X,\mu^{\prime}\right)  $, $\left(  U\iota\right)  \left(
1\right)  =1/\sqrt{k}$ is not the canonical mapping that responds to
$\mu^{\prime}$.

\subsection{A unital positive mapping from $C\left(  X\right)  $ to $\left(
H,e\right)  $\label{subsecCHe}}

For brevity we focus on unital positive maps instead of positive maps. As
above we fix a Hilbert space $H$ with its hermitian basis $F$, the unital cone
$\mathfrak{c}$, and fix also a probability measure $\mu$ (or integral $\int$)
on a compact Hausdorff topological space $X$. A family of real valued Borel
functions $k=\left\{  k_{f}:f\in F\right\}  \subseteq\operatorname{ball}%
L^{\infty}\left(  X,\mu\right)  _{h}$ with $k_{e}=u$ is said to be \textit{an
}$H$-\textit{support on }$X$\textit{ }if%
\[
k_{f}\perp k_{e},f\neq e\quad\text{and\quad}\sum_{f\neq e}\left(
v,k_{f}\right)  ^{2}\leq\left(  v,k_{e}\right)  ^{2}\quad\text{in\quad}%
L^{2}\left(  X,\mu\right)  \quad\text{for all }v\in C\left(  X\right)  _{+}.
\]
Note that $\left(  v,k_{e}\right)  =\int v\geq0$ whenever $v\geq0$. If
additionally, $\sum_{f\neq e}\left\Vert k_{f}\right\Vert _{2}^{p}<\infty$ in
$L^{2}\left(  X,\mu\right)  $ for $p=1,2$, we say that $k$ is \textit{an }%
$H$-\textit{support on }$X$\textit{ of type }$p$. But if $\sum_{f\neq e}%
k_{f}^{2}\leq k_{e}^{2}$ in $L^{\infty}\left(  X,\mu\right)  $ we say that $k$
is \textit{a maximal }$H$-\textit{support on }$X$. Note that a maximal support
if of type $2$ automatically. Indeed, $\sum_{f\in\lambda}k_{f}^{2}\leq u$ in
$L^{\infty}\left(  X,\mu\right)  $ implies that $\sum_{f\in\lambda}\int
k_{f}^{2}\leq1$ for every finite subset $\lambda\subseteq F\backslash\left\{
e\right\}  $, therefore $\sum_{f\neq e}\left\Vert k_{f}\right\Vert _{2}%
^{2}=\sum_{f\neq e}\int k_{f}^{2}\leq1$.

\begin{lemma}
\label{lemCRT0}If $k$ is an $H$-support on $X$ then $T:C\left(  X\right)
\rightarrow\left(  H,e\right)  $, $Tv=\sum_{f}\left(  v,k_{f}\right)  f$ is a
unital positive mapping, that is, $T\left(  1\right)  =e$ and $T\left(
C\left(  X\right)  _{+}\right)  \subseteq\mathfrak{c}$. Moreover, if $k$ is of
type $p$ then $T$ admits a unique bounded linear extension $T_{k}:L^{2}\left(
X,\mu\right)  \rightarrow\left(  H,e\right)  $, $T_{k}=\sum_{f}f\odot
\overline{k_{f}}$, which is a nuclear operator if $p=1$ and Hilbert-Schmidt
operator if $p=2$.
\end{lemma}

\begin{proof}
If $v\in C\left(  X\right)  _{h}$ with $-1\leq v\leq1$, then $v\pm k_{e}\geq
0$, $\left\vert \left(  v,k_{e}\right)  \right\vert \leq\int\left\vert
v\right\vert \leq\int1=1$ and $\sum_{f\neq e}\left(  v,k_{f}\right)  ^{2}%
=\sum_{f\neq e}\left(  v\pm k_{e},k_{f}\right)  ^{2}\leq\left(  v\pm
k_{e},k_{e}\right)  ^{2}=\left(  \left(  v,k_{e}\right)  \pm1\right)  ^{2}$.
In particular, $\sum_{f\neq e}\left(  v,k_{f}\right)  ^{2}\leq\left(
1-\left\vert \left(  v,k_{e}\right)  \right\vert \right)  ^{2}$, which in turn
implies that
\[
\left\Vert Tv\right\Vert =\left(  \sum_{f}\left(  v,k_{f}\right)  ^{2}\right)
^{1/2}\leq\left\vert \left(  v,k_{e}\right)  \right\vert +\left(  \sum_{f\neq
e}\left(  v,k_{f}\right)  ^{2}\right)  ^{1/2}\leq1.
\]
Hence $\left\Vert T|\operatorname{ball}C\left(  X\right)  _{h}\right\Vert
\leq1$. In the case of any $v\in\operatorname{ball}C\left(  X\right)  $, we
have $\operatorname{Re}v,\operatorname{Im}v\in\operatorname{ball}C\left(
X\right)  _{h}$ and $\left\Vert Tv\right\Vert \leq\left\Vert
T\operatorname{Re}v\right\Vert +\left\Vert T\operatorname{Im}v\right\Vert
\leq2$, that is, $T$ is a well defined bounded linear mapping. Further, take
$v\in C\left(  X\right)  _{+}$. Taking into account that $k$ is an $H$-support
on $X$, we deduce that $\left\Vert Tv\right\Vert ^{2}=\sum_{f}\left(
v,k_{f}\right)  ^{2}\leq2\left(  v,k_{e}\right)  ^{2}=2\left(  Tv,e\right)
^{2}$ or $\left\Vert Tv\right\Vert \leq\sqrt{2}\left(  Tv,e\right)  $, that
is, $T\left(  C\left(  X\right)  _{+}\right)  \subseteq\mathfrak{c}$.
Moreover, $Tu=\sum_{f}\left(  k_{e},k_{f}\right)  f=\left(  k_{e}%
,k_{e}\right)  e=\left(  \int1\right)  e=e$. Thus $T$ is a unital positive mapping.

Finally, assume that $k$ is of type $2$. For every $v\in C\left(  X\right)  $
we have $\left\Vert Tv\right\Vert ^{2}=\sum_{f}\left\vert \left(
v,k_{f}\right)  \right\vert ^{2}\leq\left\Vert v\right\Vert _{2}^{2}\sum
_{f}\left\Vert k_{f}\right\Vert _{2}^{2}$. By continuity argument $T$ admits a
unique extension $T_{k}:L^{2}\left(  X,\mu\right)  \rightarrow\left(
H,e\right)  $, $T_{k}\iota=T$ such that $T_{k}=\sum_{f}f\odot\overline{k_{f}}$
and $\left\Vert T_{k}\right\Vert _{2}^{2}=\sum_{f}\left\Vert T_{k}^{\ast
}f\right\Vert ^{2}=\sum_{f}\left\Vert k_{f}\right\Vert _{2}^{2}<\infty$. Hence
$T_{k}$ is a Hilbert-Schmidt operator. If $k$ is of type $1$ then $\left\Vert
T_{k}\right\Vert _{1}\leq\sum_{f}\left\Vert f\right\Vert \left\Vert
k_{f}\right\Vert _{2}=\sum_{f}\left\Vert k_{f}\right\Vert _{2}<\infty$, which
means that $T_{k}$ is a nuclear operator.
\end{proof}

Below in Theorem \ref{corCHe1}, we prove that the bounded linear extension
$T_{k}:L^{2}\left(  X,\mu\right)  \rightarrow\left(  H,e\right)  $ exists for
every $H$-support $k$ on $X$.

\begin{proposition}
\label{lemCRT1}Let $T:C\left(  X\right)  \rightarrow\left(  H,e\right)  $ be a
unital positive mapping. There is a unique probability measure $\mu$ on $X$
and an $H$-support $k\subseteq\operatorname{ball}L^{\infty}\left(
X,\mu\right)  _{h}$ on $X$ such that $Tv=\sum_{f}\left(  v,k_{f}\right)  f$,
$v\in C\left(  X\right)  $. The functions $k_{f}$, $f\neq e$ are uniquely
determined modulo $\mu$-null functions, and
\[
Tv=\lim_{\lambda}\int v\left(  t\right)  \left(  \sum_{f\in\lambda}%
k_{f}\left(  t\right)  f+e\right)  d\mu,
\]
where $\lambda$ is running over all finite subsets in $F\backslash\left\{
e\right\}  $, and we used the related Radon integral for $H$-valued measurable
functions on $X$. Thus there is a one to one correspondence between unital
positive maps $C\left(  X\right)  \rightarrow\left(  H,e\right)  $ and
$H$-supports on $X$.
\end{proposition}

\begin{proof}
If $v\in C\left(  X\right)  _{+}$ then $Tv\in\mathfrak{c}$. In particular,
$\left(  Tv,e\right)  \geq0$, which means that $v\mapsto\left(  T\left(
v\right)  ,e\right)  $ is a positive Radon integral, that is, $\left(
Tv,e\right)  =\left\langle v,\mu\right\rangle $ for a certain $\mu
\in\mathcal{M}\left(  X\right)  _{+}$. Note that $\int1d\mu=\left(
T1,e\right)  =\left\Vert e\right\Vert ^{2}=1$, that is, $\mu\in\mathcal{P}%
\left(  X\right)  $. Moreover, $\sum_{f\neq e}\left(  Tv,f\right)  ^{2}%
\leq\left(  Tv,e\right)  ^{2}$ for all $v\in C\left(  X\right)  _{+}$. Since
$\overline{f}+\overline{e}\in S\left(  \mathfrak{c}\right)  $ (see Lemma
\ref{lHSc11}), it follows that $\left(  Tv,f+e\right)  =\left\langle v,\mu
_{f}\right\rangle $ for some $\mu_{f}\in\mathcal{M}\left(  X\right)  _{+}$.
But $\left(  Tv,f+e\right)  =\left(  Tv,f\right)  +\left(  Tv,e\right)
\leq2\left(  Tv,e\right)  $ for all $v\in C\left(  X\right)  _{+}$, which
means that $\mu_{f}\leq2\mu$ in $\mathcal{M}\left(  X\right)  _{h}$ for all
$f\neq e$. Thus $\left\{  \mu_{f}\right\}  \subseteq I_{\mu}\left(  X\right)
$, where $I_{\mu}\left(  X\right)  $ is the closed (lattice) ideal of the
complete lattice $\mathcal{M}\left(  X\right)  _{h}$ generated by $\mu$ (see
Subsection \ref{subsecARM}). Using Lebesgue-Nikodym Theorem, we deduce that
$\mu_{f}=m_{f}\mu$ for some (real) Borel function $m_{f}\in L^{1}\left(
X,\mu\right)  _{h}$ such that $0\leq m_{f}\leq2$. The functions $\left\{
m_{f}:f\neq e\right\}  $ are uniquely determined modulo $\mu$-null functions.
It follows that
\[
\left(  Tv,f\right)  =\left(  Tv,f+e\right)  -\left(  Tv,e\right)
=\left\langle v,m_{f}\mu\right\rangle -\left\langle v,\mu\right\rangle
=\left\langle v,k_{f}\mu\right\rangle
\]
for all $v\in C\left(  X\right)  $, where $k_{f}=m_{f}-1$ is a bounded Borel
function from $L^{1}\left(  X,\mu\right)  _{h}$. Since $T1=e$, we obtain that
$\left\langle 1,k_{f}\mu\right\rangle =\left(  T1,f\right)  =0$, that is,
$k_{f}\perp u$ in $L^{2}\left(  X,\mu\right)  $ for all $f\neq e$.

Thus $Tv=\sum_{f}\left(  v,k_{f}\right)  f=\sum_{f\neq e}\left(  \int v\left(
t\right)  k_{f}\left(  t\right)  d\mu\right)  f+\left(  \int v\left(
t\right)  d\mu\right)  e$. In particular,%
\[
Tv=\lim_{\lambda}\sum_{f\in\lambda}\left(  \int v\left(  t\right)
k_{f}\left(  t\right)  d\mu\right)  f+\left(  \int v\left(  t\right)
d\mu\right)  e=\lim_{\lambda}\int\left(  \sum_{f\in\lambda}v\left(  t\right)
k_{f}\left(  t\right)  f+v\left(  t\right)  e\right)  d\mu,
\]
where $\lambda$ is running over all finite subsets in $F\backslash\left\{
e\right\}  $. Notice that we used the canonical extension of the Radon
integral to $H$-valued functions on $X$ (see below Remark \ref{remCRT1}).

Finally, prove that $k=\left\{  k_{f}\right\}  \subseteq\operatorname{ball}%
L^{\infty}\left(  X,\mu\right)  _{h}$. Since $\sum_{f\neq e}\left\langle
v,k_{f}\mu\right\rangle ^{2}\leq\left\langle v,\mu\right\rangle ^{2}\ $for all
$v\in C\left(  X\right)  _{+}$, we conclude that $\left\vert \left\langle
v,k_{f}\mu\right\rangle \right\vert \leq\left\langle v,\mu\right\rangle $,
$v\in C\left(  X\right)  _{+}$, which means that $-\mu\leq k_{f}\mu\leq\mu$ in
$\mathcal{M}\left(  X\right)  _{h}$. It follows that $\left\vert
k_{f}\right\vert \mu=\left\vert k_{f}\mu\right\vert =\left(  k_{f}\mu\right)
\vee\left(  -k_{f}\mu\right)  \leq\mu$ (see \cite[Ch. V, 5.4]{BourInt}), that
is, $\left\vert k_{f}\right\vert \leq1$ for $\mu$-almost everywhere on $X$.
Thus $k\subseteq\operatorname{ball}L^{\infty}\left(  X,\mu\right)  _{h}$ and
it is an $H$-support on $X$. The rest follows from Lemma \ref{lemCRT0}.
\end{proof}

\begin{remark}
\label{remCRT1}Let $\mu$ be a Radon measure on a Hausdorff compact space $X$,
$H$ a Hilbert space and let $\mathbf{v}:X\rightarrow H$ be a weakly (or
weak$^{\ast}$) measurable mapping with $\mu$-integrable norm. Thus
$\left\langle \mathbf{v}\left(  \cdot\right)  ,\overline{\eta}\right\rangle $
is measurable for every $\eta\in H$, and $\int\left\Vert \mathbf{v}\left(
t\right)  \right\Vert d\mu<\infty$. There is a unique element $\int%
\mathbf{v}\left(  t\right)  d\mu\in H$ such that $\left\langle \int%
\mathbf{v}\left(  t\right)  d\mu,\overline{\eta}\right\rangle =\int%
\left\langle \mathbf{v}\left(  t\right)  ,\overline{\eta}\right\rangle d\mu$
for all $\eta\in H$ (see \cite[2.5.14]{Ped}). If $\mathbf{v}$ is continuous
then $\int\mathbf{v}\left(  t\right)  d\mu$ is a limit of Riemann sums
$\sum_{m=1}^{N}\mu\left(  E_{m}\right)  \mathbf{v}\left(  t_{m}\right)  $
taken over all partitions $\left\{  E_{m}\right\}  $ of $X$ into disjoint
Borel subsets (see \cite[E 2.5.8]{Ped}). In particular, if $\mathbf{v}\left(
X\right)  \subseteq\mathfrak{c}$ for a certain closed cone $\mathfrak{c}$ then
$\int\mathbf{v}\left(  t\right)  d\mu\in\mathfrak{c}$.
\end{remark}

Now we can prove that all unital positive maps $C\left(  X\right)
\rightarrow\left(  H,e\right)  $ admit unique extensions up to positive maps
between Hilbert spaces.

\begin{theorem}
\label{corCHe1}Let $T:C\left(  X\right)  \rightarrow\left(  H,e\right)  $ be a
unital positive mapping with its $H$-support $k\subseteq\operatorname{ball}%
L^{\infty}\left(  X,\mu\right)  $ on $X$. Then $T$ is an absolutely summable
mapping, $k$ is a unital $H$-support in $L^{2}\left(  X,\mu\right)  $, and $T$
admits a unique bounded linear extension $T_{k}:\left(  L^{2}\left(
X,\mu\right)  ,u\right)  \rightarrow\left(  H,e\right)  $, which is a unital
positive mapping of Hilbert spaces.
\end{theorem}

\begin{proof}
By Proposition \ref{lemCRT1}, there is a unique probability measure $\mu$ on
$X$ and an $H$-support $k\subseteq\operatorname{ball}L^{\infty}\left(
X,\mu\right)  _{h}$ on $X$ such that $Tv=\sum_{f}\left(  v,k_{f}\right)  f$,
$v\in C\left(  X\right)  $. The functions $k_{f}$ are uniquely determined
modulo $\mu$-null functions. Prove that $T:\left(  C\left(  X\right)
,\left\Vert \cdot\right\Vert _{2}\right)  \rightarrow H$ is bounded. If $v\in
C\left(  X\right)  _{h}$ then $v=v_{+}-v_{-}$ with $v_{+},v_{-}\in C\left(
X\right)  _{+}$ and $\left\vert v\right\vert =v_{+}\vee v_{-}=v_{+}+v_{-}$.
Moreover,
\begin{align*}
\sum_{f\neq e}\left(  v,k_{f}\right)  ^{2}  &  =\sum_{f\neq e}\left(
v_{+},k_{f}\right)  ^{2}+\sum_{f\neq e}\left(  v_{-},k_{f}\right)  ^{2}%
-2\sum_{f\neq e}\left(  v_{+},k_{f}\right)  \left(  v_{-},k_{f}\right) \\
&  \leq\left(  v_{+},k_{e}\right)  ^{2}+\left(  v_{-},k_{e}\right)  ^{2}%
+2\sum_{f\neq e}\left\vert \left(  v_{+},k_{f}\right)  \left(  v_{-}%
,k_{f}\right)  \right\vert \\
&  \leq\left(  v_{+},k_{e}\right)  ^{2}+\left(  v_{-},k_{e}\right)
^{2}+2\left(  \sum_{f\neq e}\left(  v_{+},k_{f}\right)  ^{2}\right)
^{1/2}\left(  \sum_{f\neq e}\left(  v_{-},k_{f}\right)  ^{2}\right)  ^{1/2}\\
&  \leq\left(  v_{+},k_{e}\right)  ^{2}+\left(  v_{-},k_{e}\right)
^{2}+2\left(  v_{+},k_{e}\right)  \left(  v_{-},k_{e}\right)  =\left(
v_{+}+v_{-},k_{e}\right)  ^{2}\\
&  =\left(  \left\vert v\right\vert ,k_{e}\right)  ^{2},
\end{align*}
which in turn implies that
\[
\left\Vert Tv\right\Vert ^{2}=\sum_{f\neq e}\left(  v,k_{f}\right)
^{2}+\left(  v,k_{e}\right)  ^{2}\leq\left(  \left\vert v\right\vert
,k_{e}\right)  ^{2}+\left(  v,k_{e}\right)  ^{2}\leq2\left(  \int\left\vert
v\right\vert d\mu\right)  ^{2},
\]
that is, $\left\Vert Tv\right\Vert \leq\sqrt{2}\int\left\vert v\right\vert
d\mu$. In the case of any $v\in C\left(  X\right)  $ we derive that
$\left\Vert Tv\right\Vert \leq\left\Vert T\operatorname{Re}v\right\Vert
+\left\Vert T\operatorname{Im}v\right\Vert \leq\sqrt{2}\int\left(  \left\vert
\operatorname{Re}v\right\vert +\left\vert \operatorname{Im}v\right\vert
\right)  d\mu\leq2\sqrt{2}\int\left\vert v\right\vert d\mu$. By the known
result of Pietsch \cite[2.3.3]{AlP}, we deduce that $T$ is an absolutely
summable mapping with $\left\Vert T\right\Vert \leq\pi\left(  T\right)
\leq2\sqrt{2}\mu\left(  X\right)  =2\sqrt{2}$. It follows that $T$ is
factorized throughout the Hilbert space $L^{2}\left(  X,\mu\right)  $
\cite[3.3.4]{AlP}. Namely, $\left\Vert Tv\right\Vert \leq2\sqrt{2}\left(
\int\left\vert v\right\vert ^{2}d\mu\right)  ^{1/2}\left(  \int1d\mu\right)
^{1/2}=2\sqrt{2}\left\Vert v\right\Vert _{2}$ for all $v\in C\left(  X\right)
$, and taking into account the density of $\iota\left(  C\left(  X\right)
\right)  $ in $L^{2}\left(  X,\mu\right)  $, we obtain a unique bounded linear
extension $T_{k}:L^{2}\left(  X,\mu\right)  \rightarrow H$, $T_{k}\iota=T$.
Moreover, $T_{k}\eta=\sum_{f}\left(  \eta,k_{f}\right)  f$ for all $\eta\in
L^{2}\left(  X,\mu\right)  $ due to the density of $\iota\left(  C\left(
X\right)  \right)  $ in $L^{2}\left(  X,\mu\right)  $.

It remains to prove that $k$ is a unital $H$-support in the unital Hilbert
space $\left(  L^{2}\left(  X,\mu\right)  ,u\right)  $. If $v_{0}\in
\iota\left(  C\left(  X\right)  \right)  \cap L^{2}\left(  X,\mu\right)
_{h}^{u}$ then as above we have $\sum_{f\neq e}\left(  v_{0},k_{f}\right)
^{2}\leq\left(  \left\vert v_{0}\right\vert ,k_{e}\right)  ^{2}\leq\left\Vert
\left\vert v_{0}\right\vert \right\Vert _{2}^{2}\left(  \int1d\mu\right)
^{2}=\left\Vert v_{0}\right\Vert _{2}^{2}=\left(  \left(  v_{0},k_{e}\right)
+\left\Vert v_{0}\right\Vert _{2}\right)  ^{2}$. Notice that $\left(
v_{0},k_{e}\right)  =\left(  v_{0},u\right)  =0$. Take $\eta_{0}\in
L^{2}\left(  X,\mu\right)  _{h}^{u}$. Then $\eta_{0}=\lim_{n}v_{0,n}$ in
$L^{2}\left(  X,\mu\right)  $ for a certain sequence $\left(  v_{0,n}\right)
_{n}$ from $\iota\left(  C\left(  X\right)  \right)  \cap L^{2}\left(
X,\mu\right)  _{h}^{u}$. For every finite subset $\lambda\subseteq
F\backslash\left\{  e\right\}  $ we have
\[
\sum_{f\in\lambda}\left(  \eta_{0},k_{f}\right)  ^{2}=\lim_{n}\sum
_{f\in\lambda}\left(  v_{0,n},k_{f}\right)  ^{2}\leq\lim_{n}\left\Vert
v_{0,n}\right\Vert _{2}^{2}=\left\Vert \eta_{0}\right\Vert _{2}^{2}=\left(
\left(  \eta_{0},k_{e}\right)  +\left\Vert \eta_{0}\right\Vert _{2}\right)
^{2},
\]
which in turn implies that $\sum_{f\neq e}\left(  \eta_{0},k_{f}\right)
^{2}\leq\left(  \left(  \eta_{0},k_{e}\right)  +\left\Vert \eta_{0}\right\Vert
_{2}\right)  ^{2}$. Consequently, $k$ is a unital $H$-support in $\left(
L^{2}\left(  X,\mu\right)  ,u\right)  $ (see Remark \ref{remUHS}), and
$T=T_{k}$ in the sense of Proposition \ref{propKom1}.
\end{proof}

Notice that $T\left(  C\left(  X\right)  _{+}\right)  \subseteq\mathfrak{c}$
implies that $T^{\ast}\left(  S\left(  \mathfrak{c}\right)  \right)
\subseteq\mathcal{P}\left(  X\right)  $. Using Lemma \ref{lDualM} and
Proposition \ref{propPTT1}, we obtain that $T^{\left(  \infty\right)  }\left(
M\left(  C\left(  X\right)  \right)  _{+}\right)  =T^{\left(  \infty\right)
}\left(  \min C\left(  X\right)  _{+}\right)  =T^{\left(  \infty\right)
}\left(  \mathcal{P}\left(  X\right)  ^{\boxdot}\right)  \subseteq S\left(
\mathfrak{c}\right)  ^{\boxdot}=\min\mathfrak{c}$.

\subsection{Separable and nuclear morphisms}

Recall that a positive mapping $\phi:\mathcal{V\rightarrow W}$ of operator
systems is called \textit{a separable} if $\phi=\sum_{l}p_{l}\odot q_{l}$ for
some positive functionals $q_{l}$ on $\mathcal{V}$ and positive elements
$p_{l}$ from $\mathcal{W}$, where $\left(  p_{l}\odot q_{l}\right)
v=q_{l}\left(  v\right)  p_{l}$ for all $v\in\mathcal{V}$. Thus $\phi\left(
v\right)  =\lim_{k}\sum_{l=1}^{k}q_{l}\left(  v\right)  p_{l}$ in
$\mathcal{W}$ for every $v\in\mathcal{V}$. Notice that a separable mapping
$\phi$ defines a matrix positive mapping $\phi:\mathcal{V\rightarrow}\left(
\mathcal{W},\max\mathcal{W}_{+}\right)  $ automatically. Indeed, take $v\in
M_{n}\left(  \mathcal{V}\right)  _{+}$. Since the positive functionals $q_{l}$
on $\mathcal{V}$ are matrix positive, we deduce that $\phi^{\left(  n\right)
}\left(  v\right)  =\lim_{k}\sum_{l=1}^{k}q_{l}^{\left(  n\right)  }\left(
v\right)  p_{l}^{\oplus n}=\lim_{k}\sum_{l=1}^{k}q_{l}^{\left(  n\right)
}\left(  v\right)  ^{1/2}p_{l}^{\oplus n}q_{l}^{\left(  n\right)  }\left(
v\right)  ^{1/2}$. But $q_{l}^{\left(  n\right)  }\left(  v\right)
^{1/2}p_{l}^{\oplus n}q_{l}^{\left(  n\right)  }\left(  v\right)  ^{1/2}%
\in\mathcal{W}_{+}^{c}$, therefore $\phi^{\left(  n\right)  }\left(  v\right)
\in\mathcal{W}_{+}^{\boxdot\boxdot}=\left(  \mathcal{W}_{+}^{c}\right)
^{-}=\max\mathcal{W}_{+}$.

Now let $T:C\left(  X\right)  \rightarrow\left(  H,e\right)  $ be a unital
positive mapping. By Proposition \ref{lemCRT1}, $T$ is given by an $H$-support
$k\subseteq\operatorname{ball}L^{\infty}\left(  X,\mu\right)  _{h}$ on $X$.
Suppose $T$ is a nuclear mapping, that is, $T=\sum_{l}\gamma_{l}\odot q_{l}$
for some $\left(  \gamma_{l}\right)  _{l}\subseteq H$ and $\left(
q_{l}\right)  _{l}\subseteq\mathcal{M}\left(  X\right)  $ such that $\sum
_{l}\left\Vert \gamma_{l}\right\Vert \left\Vert q_{l}\right\Vert <\infty$.
Taking into account that $T$ is a $\ast$-linear mapping and both $\left(
C\left(  X\right)  ,\mathcal{M}\left(  X\right)  \right)  $ and $\left(
H,\overline{H}\right)  $ are dual $\ast$-pairs, we can assume that $\left(
\gamma_{l}\right)  _{l}\subseteq\operatorname{ball}H_{h}$ and $\left(
q_{l}\right)  _{l}\subseteq\mathcal{M}\left(  X\right)  _{h}$ with $\sum
_{l}\left\Vert q_{l}\right\Vert <\infty$. We say that $T$ is a \textit{nuclear
morphism }if $T=\sum_{l}\gamma_{l}\odot q_{l}$ for some $\left(  \gamma
_{l}\right)  _{l}\subseteq\operatorname{ball}H_{h}$ and $\left(  q_{l}\right)
_{l}\subseteq I_{\mu}\left(  X\right)  $ with $\sum_{l}\left\Vert
q_{l}\right\Vert <\infty$.

\begin{lemma}
\label{lemNSO1}Let $T:C\left(  X\right)  \rightarrow\left(  H,e\right)  $ be a
unital positive mapping given by an $H$-support $k\subseteq\operatorname{ball}%
L^{\infty}\left(  X,\mu\right)  _{h}$ on $X$. Then $T$ is a nuclear morphism
if and only if $T+e\odot q$ is separable for a certain $q\in I_{\mu}\left(
X\right)  $. In this case, one can assume that $q\in I_{\mu}\left(  X\right)
_{+}$.
\end{lemma}

\begin{proof}
First assume that $T+e\odot q$ is separable for a certain $q\in I_{\mu}\left(
X\right)  $. Then $T=\sum_{l}p_{l}\odot q_{l}-e\odot q$ for some $\left(
p_{l}\right)  _{l}\subseteq\mathfrak{c}$ and $\left(  q_{l}\right)
_{l}\subseteq\mathcal{M}\left(  X\right)  _{+}$. We have $p_{l}=\eta_{l}%
+r_{l}e$, $\eta_{l}\in H_{h}^{e}$, $\left\Vert \eta_{l}\right\Vert \leq r_{l}%
$. Put $\zeta_{l}=r_{l}^{-1}\eta_{l}\in\operatorname{ball}H_{h}^{e}$, and
$\mu_{l}=r_{l}q_{l}$. Then%
\begin{align*}
T\left(  v\right)   &  =\sum_{l}\left\langle v,q_{l}\right\rangle \left(
\eta_{l}+r_{l}e\right)  -\left\langle v,q\right\rangle e=\sum_{l}\left\langle
v,\mu_{l}\right\rangle \left(  \zeta_{l}+e\right)  -\left\langle
v,q\right\rangle e\\
&  =\sum_{l}\left\langle v,\mu_{l}\right\rangle \zeta_{l}+\left(  \sum
_{l}\left\langle v,\mu_{l}\right\rangle -\left\langle v,q\right\rangle
\right)  e\in H^{e}\oplus\mathbb{C}e=H
\end{align*}
for all $v\in C\left(  X\right)  $. In particular, $\sum_{l}\left\langle
1,\mu_{l}\right\rangle \zeta_{l}=0$ and $\sum_{l}\left\langle 1,\mu
_{l}\right\rangle =1+\left\langle 1,q\right\rangle $. The latter means that
$\tau=\sum_{l}\mu_{l}\in\mathcal{M}\left(  X\right)  _{+}$ with $\sum
_{l}\left\Vert \mu_{l}\right\Vert =\sum_{l}\left\langle 1,\mu_{l}\right\rangle
=1+\left\Vert q\right\Vert <\infty$. By Proposition \ref{lemCRT1}, we obtain
the equality $\mu=\tau-q$. But $q\in I_{\mu}\left(  X\right)  $, therefore
$\tau=\mu+q\in I_{\mu}\left(  X\right)  _{+}$. Since $\left\{  \mu
_{l}\right\}  \leq\tau$ and $I_{\mu}\left(  X\right)  $ is a lattice ideal, it
follows that $\left\{  \mu_{l}\right\}  \subseteq I_{\mu}\left(  X\right)
_{+}$. Moreover, $\sum_{l}\left\Vert \zeta_{l}\right\Vert \left\Vert \mu
_{l}\right\Vert \leq\sum_{l}\left\Vert \mu_{l}\right\Vert =\left\Vert
\tau\right\Vert =1+\left\Vert q\right\Vert $, which means that $T$ is a
nuclear morphism given by $T=\sum_{l}\zeta_{l}\odot\mu_{l}+e\odot\mu$,
$\left\{  \mu_{l}\right\}  \subseteq I_{\mu}\left(  X\right)  $, and
$\left\Vert T\right\Vert _{1}\leq2+\left\Vert q\right\Vert $.

Conversely, suppose that $T$ is a nuclear morphism. Then $T=\sum_{l}\gamma
_{l}\odot q_{l}$ for some $\left(  \gamma_{l}\right)  _{l}\subseteq
\operatorname{ball}H_{h}$ and $\left(  q_{l}\right)  _{l}\subseteq I_{\mu
}\left(  X\right)  $ with $\sum_{l}\left\Vert q_{l}\right\Vert <\infty$. Thus
$\gamma_{l}=\zeta_{l}+r_{l}e$ with $\zeta_{l}\in\operatorname{ball}H_{h}^{e}$,
$r_{l}\in\mathbb{R}$ and $\left\Vert \zeta_{l}\right\Vert ^{2}+r_{l}^{2}\leq
1$. If $v\in C\left(  X\right)  $ then
\[
Tv=\sum_{l}\left\langle v,q_{l}\right\rangle \left(  \zeta_{l}+r_{l}e\right)
=\sum_{l}\left\langle v,q_{l}\right\rangle \zeta_{l}+\sum_{l}\left\langle
v,r_{l}q_{l}\right\rangle e=\sum_{l}\left\langle v,q_{l}\right\rangle
\zeta_{l}+\left\langle v,\mu\right\rangle e,
\]
where $\mu=\sum_{l}r_{l}q_{l}$, $\sum_{l}\left\Vert r_{l}q_{l}\right\Vert
=\sum_{l}\left\vert r_{l}\right\vert \left\Vert q_{l}\right\Vert \leq\sum
_{l}\left\Vert q_{l}\right\Vert <\infty$. Thus $T=\sum_{l}\zeta_{l}\odot
q_{l}+e\odot\mu$ with $\sum_{l}\left\Vert \zeta_{l}\right\Vert \left\Vert
q_{l}\right\Vert \leq\sum_{l}\left\Vert q_{l}\right\Vert <\infty$. Using the
Jordan decompositions $q_{l}=q_{l,+}-q_{l,-}$ with $q_{l,+}$, $q_{l,-}%
\in\mathcal{M}\left(  X\right)  _{+}$ and $\left\Vert q_{l}\right\Vert
=\left\Vert q_{l,+}\right\Vert +\left\Vert q_{l,-}\right\Vert $ \cite[Ch. 3,
2.6]{BourInt}, we obtain that $T=\sum_{l}\zeta_{l}\odot q_{l,+}+\sum
_{l}\left(  -\zeta_{l}\right)  \odot q_{l,-}+e\odot\mu$ and $\sum
_{l}\left\Vert \zeta_{l}\right\Vert \left\Vert q_{l,+}\right\Vert +\sum
_{l}\left\Vert -\zeta_{l}\right\Vert \left\Vert q_{l,-}\right\Vert \leq
\sum_{l}\left\Vert \zeta_{l}\right\Vert \left\Vert q_{l}\right\Vert <\infty$.
Taking into account that $\left\{  q_{l}\right\}  \subseteq I_{\mu}\left(
X\right)  $, we deduce that $\left\{  q_{l,+},q_{l,-}\right\}  \subseteq
I_{\mu}\left(  X\right)  $ either. Thus we can assume that $T=\sum_{l}%
\zeta_{l}\odot\mu_{l}+e\odot\mu$ with $\zeta_{l}\in\operatorname{ball}%
H_{h}^{e}$, $\mu_{l}\in I_{\mu}\left(  X\right)  _{+}$ and $\sum_{l}\left\Vert
\mu_{l}\right\Vert <\infty$. It follows that%
\[
T=\sum_{l}\left(  \zeta_{l}+e\right)  \odot\mu_{l}+e\odot\mu-e\odot\tau
=\sum_{l}\eta_{l}\odot\mu_{l}+e\odot\mu-e\odot\tau,
\]
where $\eta_{l}=\zeta_{l}+e\in\mathfrak{c}$ and $\tau=\sum_{l}\mu_{l}\in
I_{\mu}\left(  X\right)  _{+}$. Consequently, we can assume that $T=\sum
_{l}\eta_{l}\odot\mu_{l}-e\odot\tau$ for some $\left(  \eta_{l}\right)
_{l}\subseteq\mathfrak{c}$, $\left(  \mu_{l}\right)  _{l}\subseteq
\mathcal{M}\left(  X\right)  _{+}$, $\tau\in I_{\mu}\left(  X\right)  _{+}$
such that $\sum_{l}\left\Vert \eta_{l}\right\Vert \left\Vert \mu
_{l}\right\Vert <\infty$, which means that $T+e\odot\tau$ is separable.
\end{proof}

\begin{corollary}
If $T:C\left(  X\right)  \rightarrow\left(  H,e\right)  $ is a separable
morphism then $T$ is a nuclear morphism automatically.
\end{corollary}

\begin{proof}
One needs to use Lemma \ref{lemNSO1} with $q=0$.
\end{proof}

\begin{theorem}
\label{propSepNuc1}Let $T:C\left(  X\right)  \rightarrow\left(  H,e\right)  $
be a unital positive mapping with its $H$-support $k\subseteq
\operatorname{ball}L^{\infty}\left(  X,\mu\right)  $ on $X$. If $T$ is a
nuclear-morphism then its bounded linear extension $T_{k}:L^{2}\left(
X,\mu\right)  \rightarrow\left(  H,e\right)  $ is a Hilbert-Schmidt operator.
In this case, the $H$-support $k$ on $X$ is maximal whenever $T$ is separable.
\end{theorem}

\begin{proof}
Assume that $T$ is a nuclear morphism. By Lemma \ref{lemNSO1}, $T+e\odot q$ is
separable for a certain $q\in I_{\mu}\left(  X\right)  _{+}$. Thus $T+e\odot
q=\sum_{l}\left(  \zeta_{l}+e\right)  \odot\mu_{l}$ with $\left(  \zeta
_{l}\right)  _{l}\subseteq\operatorname{ball}H_{h}^{e}$ and $\left(  \mu
_{l}\right)  _{l}\subseteq\mathcal{M}\left(  X\right)  _{+}$. Put $\tau
=\sum_{l}\mu_{l}\in\mathcal{M}\left(  X\right)  _{+}$. Notice that $\left(
1+\left\Vert q\right\Vert \right)  e=T\left(  1\right)  +\left\langle
1,q\right\rangle e=\sum_{l}\left\langle 1,\mu_{l}\right\rangle \left(
\zeta_{l}+e\right)  =\sum_{l}\left\langle 1,\mu_{l}\right\rangle e$ and
$\left\langle 1,\tau\right\rangle =\sum_{l}\left\langle 1,\mu_{l}\right\rangle
=\sum_{l}\left\Vert \mu_{l}\right\Vert =1+\left\Vert q\right\Vert <\infty$.
Then $T=\sum_{l}\zeta_{l}\odot\mu_{l}+e\odot\left(  \tau-q\right)  $, which in
turn implies that $\mu=\tau-q$. Since $q\in I_{\mu}\left(  X\right)  _{+}$, we
obtain that $\tau=\mu+q\in I_{\mu}\left(  X\right)  _{+}$ and $\mu\leq\tau$.
Hence $I_{\mu}\left(  X\right)  =I_{\tau}\left(  X\right)  $. By
Lebesgue-Nikodym Theorem, $\mu=m\tau$ for a Borel function $m$ such that
$0<m\left(  t\right)  \leq1$ for $\mu$-almost all $t\in X$ (see \cite[Ch. V,
5.6, Proposition 10]{BourInt}). Since $\left\{  \mu_{l}\right\}  \leq\tau$, we
deduce also that $\left\{  \mu_{l}\right\}  \subseteq I_{\tau}\left(
X\right)  $ and there are (unique) positive bounded Borel function $\left\{
n_{l}\right\}  $ on $X$ such that $\mu_{l}=n_{l}\tau$ for all $l$. Notice
that
\[
\tau=\sum_{l}\mu_{l}=\sum_{l}n_{l}\tau=\vee\left\{  \sum_{l=1}^{k}n_{l}%
\tau\right\}  =\vee\left\{  \left(  \sum_{l=1}^{k}n_{l}\right)  \tau\right\}
=\left(  \vee\sum_{l=1}^{k}n_{l}\right)  \tau=\left(  \sum_{l}n_{l}\right)
\tau
\]
thanks to \cite[Ch. V, 5.4, Proposition 6]{BourInt}. Hence $\sum_{l}n_{l}=1$
for $\tau$-almost (or $\mu$-almost) everywhere on $X$. Put $m_{l}=\dfrac
{n_{l}}{m}$ for all $l$. Thus $m_{l}$ are $\mu$-almost everywhere finite Borel
functions on $X$, and $\mu_{l}=n_{l}\tau=m_{l}m\tau=m_{l}\mu$. Moreover,
\begin{equation}
\sum_{l}\left\Vert \zeta_{l}\right\Vert \left\Vert m_{l}\right\Vert _{1}%
\leq\sum_{l}\left\Vert m_{l}\right\Vert _{1}\leq\sum_{l}\left\langle 1,\mu
_{l}\right\rangle =\left\langle 1,\tau\right\rangle =1+\left\Vert q\right\Vert
, \label{est40}%
\end{equation}
thereby $m^{-1}=\sum_{l}n_{l}m^{-1}=\sum_{l}m_{l}\in L^{1}\left(
X,\mu\right)  $ being an absolutely summable series in $L^{1}\left(
X,\mu\right)  $, and $\tau=\sum_{l}m_{l}\mu=m^{-1}\mu$. Actually, $m^{-1}\in
L^{2}\left(  X,\mu\right)  $. Indeed,
\[
\left\vert \left(  v,m^{-1}\right)  \right\vert \leq\sum_{l}\int\left\vert
v\right\vert m_{l}d\mu=\sum_{l}\int\left\vert v\right\vert d\mu_{l}\leq
\int\left\vert v\right\vert d\tau\leq\left(  \int\left\vert v\right\vert
^{2}d\tau\right)  ^{1/2}\left(  \int1d\tau\right)  ^{1/2}%
\]
for all $v\in C\left(  X\right)  $. Take a sequence $\left(  v_{r}\right)
_{r}\subseteq C\left(  X\right)  $ with $\lim_{r}\iota\left(  v_{r}\right)
=0$ in $L^{2}\left(  X,\mu\right)  $. Then\ $\lim_{r}\int\left\vert
v\right\vert ^{2}d\tau=0$ by Lebesgue-Nikodym Theorem, and $\lim_{r}\left(
v_{r},m^{-1}\right)  =0$, which means that $\left(  \cdot,m^{-1}\right)  $ is
a bounded linear functional on $L^{2}\left(  X,\mu\right)  $, or $m^{-1}\in
L^{2}\left(  X,\mu\right)  $. In particular, $\left\{  m_{l}\right\}
\subseteq L^{2}\left(  X,\mu\right)  $. If $\overline{m_{l}}:L^{2}\left(
X,\mu\right)  \rightarrow\mathbb{C}$ is the related bounded linear functional
then%
\[
\left\langle v,\overline{m_{l}}\right\rangle =\left(  v,m_{l}\right)  =\int
v\left(  t\right)  m_{l}\left(  t\right)  d\mu=\int v\left(  t\right)
m_{l}\left(  t\right)  m\left(  t\right)  d\tau=\int v\left(  t\right)
d\mu_{l}=\left\langle v,\mu_{l}\right\rangle
\]
for all $v\in C\left(  X\right)  $, that is, $\mu_{l}=\left(  \cdot
,m_{l}\right)  $ for all $l$.

By Theorem \ref{corCHe1}, $T$ admits a unique bounded linear extension
$T_{k}:L^{2}\left(  X,\mu\right)  \rightarrow\left(  H,e\right)  $ with the
related unital $H$-support $k=\left\{  k_{f}:f\in F\right\}  $ (see
Proposition \ref{propKom1}). Note that $\left(  T_{k}\iota\right)  \left(
v\right)  =T\left(  v\right)  =\sum_{l}\left(  v,m_{l}\right)  \zeta
_{l}+\left(  v,1\right)  e$ and
\begin{equation}
\sum_{l}\left\Vert \left(  v,m_{l}\right)  \zeta_{l}\right\Vert \leq\sum
_{l}\int\left\vert v\right\vert m_{l}d\mu\left\Vert \zeta_{l}\right\Vert
=\sum_{l}\int\left\vert v\right\vert d\mu_{l}\left\Vert \zeta_{l}\right\Vert
\leq\int\left\vert v\right\vert d\tau, \label{est41}%
\end{equation}
that is, the series $\sum_{l}\left(  v,m_{l}\right)  \zeta_{l}$ is absolutely
summable in $H$ for every $v\in C\left(  X\right)  $. Take a Borel function
$\eta\in L^{2}\left(  X,\mu\right)  $. Then $\eta=\lim_{r}v_{r}$ in
$L^{2}\left(  X,\mu\right)  $ for some sequence $\left(  v_{r}\right)
_{r}\subseteq C\left(  X\right)  $. In particular, $\left\Vert v_{r}%
-v_{s}\right\Vert _{2}\rightarrow0$ for large $r$, $s$. Since $\tau\sim\mu$,
we have $\int\left\vert v_{r}-v_{s}\right\vert d\tau\leq\left(  \int\left\vert
v_{r}-v_{s}\right\vert ^{2}d\tau\right)  ^{1/2}\left(  \int1d\tau\right)
^{1/2}\rightarrow0$ for large $r$ and $s$. Using (\ref{est41}), we obtain
that
\[
\left\vert \sum_{l}\left\Vert \left(  v_{r},m_{l}\right)  \zeta_{l}\right\Vert
-\sum_{l}\left\Vert \left(  v_{s},m_{l}\right)  \zeta_{l}\right\Vert
\right\vert \leq\sum_{l}\left\Vert \left(  v_{r}-v_{s},m_{l}\right)  \zeta
_{l}\right\Vert \leq\int\left\vert v_{r}-v_{s}\right\vert d\tau\rightarrow0
\]
for large $r$ and $s$. Hence there is a limit $\lim_{r}\sum_{l}\left\Vert
\left(  v_{r},m_{l}\right)  \zeta_{l}\right\Vert $. Using the lower
semicontinuity property, we obtain that%
\[
\sum_{l}\left\Vert \left(  \eta,m_{l}\right)  \zeta_{l}\right\Vert =\sum
_{l}\left\Vert \left(  \lim_{r}v_{r},m_{l}\right)  \zeta_{l}\right\Vert
\leq\lim\inf_{r}\sum_{l}\left\Vert \left(  v_{r},m_{l}\right)  \zeta
_{l}\right\Vert =\lim_{r}\sum_{l}\left\Vert \left(  v_{r},m_{l}\right)
\zeta_{l}\right\Vert <\infty,
\]
that is, $\zeta=\sum_{l}\left(  \eta,m_{l}\right)  \zeta_{l}\in H$ being the
sum of an absolutely summable series in $H$. Actually, $\zeta=\lim_{r}\sum
_{l}\left(  v_{r},m_{l}\right)  \zeta_{l}$. Indeed, for $\varepsilon>0$ one
can find $r_{0}$ such that $\sum_{l}\left\Vert \left(  v_{r}-v_{s}%
,m_{l}\right)  \zeta_{l}\right\Vert \leq\varepsilon$ for all $r,s\geq r_{0}$.
Then
\begin{align*}
\left\Vert \zeta-\sum_{l}\left(  v_{r},m_{l}\right)  \zeta_{l}\right\Vert  &
\leq\sum_{l}\left\Vert \left(  \eta-v_{r},m_{l}\right)  \zeta_{l}\right\Vert
=\sum_{l}\left\Vert \lim_{s}\left(  v_{s}-v_{r},m_{l}\right)  \zeta
_{l}\right\Vert \\
&  \leq\lim\inf_{s}\sum_{l}\left\Vert \left(  v_{s}-v_{r},m_{l}\right)
\zeta_{l}\right\Vert \leq\varepsilon
\end{align*}
for all $r\geq r_{0}$. Thus $T_{k}\eta=\lim_{r}\left(  T_{k}\iota\right)
\left(  v_{r}\right)  =\lim_{r}\sum_{l}\left(  v_{r},m_{l}\right)  \zeta
_{l}+\left(  v_{r},1\right)  e=\sum_{l}\left(  \eta,m_{l}\right)  \zeta
_{l}+\left(  \eta,1\right)  e$. Hence
\begin{equation}
T_{k}=\sum_{l}\zeta_{l}\odot\overline{m_{l}}+e\odot\overline{1}\qquad
\text{on}\qquad L^{2}\left(  X,\mu\right)  . \label{TKL}%
\end{equation}
Finally take expansions $\zeta_{l}=\sum_{f\neq e}\zeta_{l,f}f$ in $F$ with
real $\zeta_{l,f}$, $\left\vert \zeta_{l,f}\right\vert \leq1$, and put
$k_{f}^{\prime}=\sum_{l}\zeta_{l,f}m_{l}$ for all $f\neq e$. Since $\left\vert
k_{f}^{\prime}\right\vert \leq\sum_{l}m_{l}=m^{-1}$, it follows that
$k_{f}^{\prime}\in L^{2}\left(  X,\mu\right)  $. Based on (\ref{TKL}), we
deduce that $\left(  T_{k}\eta,f\right)  =\left(  \sum_{l}\left(  \eta
,m_{l}\right)  \zeta_{l},f\right)  =\sum_{l}\left(  \eta,m_{l}\right)
\zeta_{l,f}=\left(  \eta,k_{f}^{\prime}\right)  $ for all $\eta\in
L^{2}\left(  X,\mu\right)  $ and $f\neq e$, therefore $k_{f}=k_{f}^{\prime}$
for all $f\neq e$, and $k_{e}=u$. For a finite subset $\lambda\subseteq
F\backslash\left\{  e\right\}  $ we have
\begin{align*}
\sum_{f\in\lambda}k_{f}^{2}  &  =\sum_{f\in\lambda}\sum_{l,t}\zeta_{l,f}%
\zeta_{t,f}m_{l}m_{t}\leq\sum_{l,t}\left(  \sum_{f\in\lambda}\left\vert
\zeta_{l,f}\right\vert \left\vert \zeta_{t,f}\right\vert \right)  m_{l}m_{t}\\
&  \leq\sum_{l,t}\left(  \sum_{f\in\lambda}\left\vert \zeta_{l,f}\right\vert
^{2}\right)  ^{1/2}\left(  \sum_{f\in\lambda}\left\vert \zeta_{t,f}\right\vert
^{2}\right)  ^{1/2}m_{l}m_{t}\\
&  \leq\sum_{l,t}\left\Vert \zeta_{l}\right\Vert \left\Vert \zeta
_{t}\right\Vert m_{l}m_{t}\leq\sum_{l,t}m_{l}m_{t}=\left(  \sum_{l}%
m_{l}\right)  ^{2}=m^{-2},
\end{align*}
that is, $\sum_{f\neq e}k_{f}^{2}\leq m^{-2}$. Consequently, $\left\Vert
T_{k}\right\Vert _{2}^{2}=\sum_{f}\left\Vert k_{f}\right\Vert _{2}^{2}%
=1+\sum_{f\neq e}\int k_{f}^{2}d\mu\leq1+\int m^{-2}d\mu<\infty$, which means
that $T_{k}$ is a Hilbert-Schmidt operator. In particular, the $H$-support $k$
is maximal if $m=1$ or $q=0$. The latter is the case of a separable $T$.
\end{proof}

\begin{remark}
As follows from the proof of Theorem \ref{propSepNuc1}, the Radon-Nikodym
derivative $\dfrac{d\left(  \mu+q\right)  }{d\mu}$ belongs to $L^{2}\left(
X,\mu\right)  \ $and $\sum_{f\neq e}k_{f}^{2}\leq\left(  \dfrac{d\left(
\mu+q\right)  }{d\mu}\right)  ^{2}$ if $T+e\odot q$ is separable for a certain
$q\in I_{\mu}\left(  X\right)  _{+}$. In particular, $k$ is a maximal
$H$-support on $X$ if $q=0$ (or $T$ is separable).
\end{remark}

\subsection{The maximal and Hilbert-Schmidt supports in $L^{2}\left(
X,\mu\right)  $}

As above fix $\mu\in\mathcal{P}\left(  X\right)  $ on a Hausdorff compact
topological space $X$, and let $T:\left(  L^{2}\left(  X,\mu\right)
,u\right)  \rightarrow\left(  H,e\right)  $ be a unital positive mapping. By
Proposition \ref{propKom1}, $T=T_{k}$ for a unital $H$-support $k$ in
$L^{2}\left(  X,\mu\right)  $. Thus $k=\left\{  k_{f}:f\in F\right\}  $ is a
bounded family in $L^{2}\left(  X,\mu\right)  _{h}$ such that $k_{f}\perp u$
for all $f\neq e$, $k_{e}=k_{e}^{u}+u\in\operatorname{ball}L^{2}\left(
X,\mu\right)  _{h}^{u}+u$, and $\sum_{f\neq e}\left(  \eta_{0},k_{f}\right)
^{2}\leq\left(  \left(  \eta_{0},k_{e}^{u}\right)  +\left\Vert \eta
_{0}\right\Vert \right)  ^{2}$ for all $\eta_{0}\in L^{2}\left(  X,\mu\right)
_{h}^{u}$ (see Remark \ref{remUHS}). Certainly we can assume that $k$ consists
of real-valued Borel functions on $X$. We say that $k$ is \textit{a maximal
}$H$-\textit{support in }$L^{2}\left(  X,\mu\right)  $ if $k_{e}\geq0$ and
$\sum_{f\neq e}k_{f}^{2}\leq k_{e}^{2}$. The latter means that $\sum
_{f\in\lambda}k_{f}^{2}\leq k_{e}^{2}$ as the Borel functions for every finite
subset $\lambda\subseteq F\backslash\left\{  e\right\}  $. Since $\sum
_{f\in\lambda}\int k_{f}^{2}d\mu\leq\int k_{e}^{2}d\mu$, it follows that $k$
is an $H$-support in $L^{2}\left(  X,\mu\right)  $ of type $2$ automatically.

\begin{proposition}
\label{tCRT1}Let $T_{k}:\left(  L^{2}\left(  X,\mu\right)  ,u\right)
\rightarrow\left(  H,e\right)  $ be a unital positive mapping that responds to
a maximal $H$-support $k$ in $L^{2}\left(  X,\mu\right)  $, and let
$T=T_{k}\iota:C\left(  X\right)  \rightarrow\left(  H,e\right)  $ be the
related unital $\ast$-linear mapping. Then $T^{\left(  \infty\right)  }\left(
\min C\left(  X\right)  _{+}\right)  \subseteq\max\mathfrak{c}$, which means
that $T:\left(  C\left(  X\right)  ,M\left(  C\left(  X\right)  \right)
_{+}\right)  \rightarrow\left(  H,\max\mathfrak{c}\right)  $ is a morphism of
the relevant operator systems, whose support $k^{\prime}\subseteq
\operatorname{ball}L^{\infty}\left(  X,\mu^{\prime}\right)  $ on $X$ is given
by the family $k_{f}^{\prime}=\dfrac{k_{f}}{k_{e}}$, $f\neq e$ and
$k_{e}^{\prime}=u$, where $\mu^{\prime}=k_{e}\mu\in I_{\mu}\left(  X\right)
_{+}$. Moreover, in this case $T$ is a nuclear operator.
\end{proposition}

\begin{proof}
Take $v\in M_{n}\left(  C\left(  X\right)  \right)  _{+}$. Then $v\left(
t\right)  \in M_{n}^{+}$ for all $t\in X$. Note that
\begin{align*}
T^{\left(  n\right)  }v  &  =\sum_{f}\left\langle \left\langle v,\overline
{k_{f}}\right\rangle \right\rangle f^{\oplus n}=\lim_{\lambda}\sum
_{f\in\lambda}\left\langle \left\langle v,\overline{k_{f}}\right\rangle
\right\rangle f^{\oplus n}+\left\langle \left\langle v,\overline{k_{e}%
}\right\rangle \right\rangle e^{\oplus n}\\
&  =\lim_{\lambda}\sum_{f\in\lambda}\int v\left(  t\right)  k_{f}\left(
t\right)  ^{\oplus n}f^{\oplus n}d\mu+\int v\left(  t\right)  k_{e}\left(
t\right)  ^{\oplus n}e^{\oplus n}d\mu\\
&  =\lim_{\lambda}\int v\left(  t\right)  ^{1/2}\left(  \sum_{f\in\lambda
}k_{f}\left(  t\right)  f+k_{e}\left(  t\right)  e\right)  ^{\oplus n}v\left(
t\right)  ^{1/2}d\mu\\
&  =\lim_{\lambda}\int\mathbf{v}_{\lambda}\left(  t\right)  d\mu,
\end{align*}
where $\mathbf{v}_{\lambda}\left(  t\right)  =v\left(  t\right)  ^{1/2}\left(
\sum_{f\in\lambda}k_{f}\left(  t\right)  f+k_{e}\left(  t\right)  e\right)
^{\oplus n}v\left(  t\right)  ^{1/2}$ and $\lambda$ is running over all finite
subsets in $F\backslash\left\{  e\right\}  $. Notice that we used the
canonical extension of the Radon integral to $M_{n}\left(  H\right)  $-valued
functions on $X$ (see Remark \ref{remCRT1}). Fix a finite subset
$\lambda\subseteq F\backslash\left\{  e\right\}  $. By assumption
$k_{e}\left(  t\right)  \geq0$ and $\sum_{f\in\lambda}k_{f}\left(  t\right)
^{2}\leq k_{e}\left(  t\right)  ^{2}$, which means that $\sum_{f\in\lambda
}k_{f}\left(  t\right)  f+k_{e}\left(  t\right)  e\in\mathfrak{c}$, therefore
$\mathbf{v}_{\lambda}\left(  t\right)  \in\mathfrak{c}^{c}$. In the case of
continuous $k_{f}$, $f\in\lambda$, and $k_{e}$, we derive that $\int%
\mathbf{v}_{\lambda}\left(  t\right)  d\mu\in\left(  \mathfrak{c}^{c}\right)
^{-}=\mathfrak{c}^{\boxdot\boxdot}=\max\mathfrak{c}$ (see Remark
\ref{remCRT1}). In the general case, $k_{f}\left(  t\right)  =\lim_{m}%
k_{f,m}\left(  t\right)  $ is a sequential limit of continuous functions
$\left\{  k_{f,m}\right\}  \subseteq C\left(  X\right)  $, and $k_{e}\left(
t\right)  =\lim_{m}k_{e,m}\left(  t\right)  $ for an increasing sequence
$\left\{  k_{e,m}\right\}  \subseteq C\left(  X\right)  _{+}$ for $\mu$-almost
all $t\in X$. We can assume that $\sum_{f\in\lambda}k_{f,m}^{2}\leq
k_{e,m}^{2}$ (just replace $k_{e,m}$ by $k_{e,m}\vee\left(  \sum_{f\in\lambda
}k_{f,m}^{2}\right)  ^{1/2}$) for all $m$, and put $\mathbf{v}_{\lambda
,m}\left(  t\right)  =v\left(  t\right)  ^{1/2}\left(  \sum_{f\in\lambda
}k_{f,m}\left(  t\right)  f+k_{e,m}\left(  t\right)  e\right)  ^{\oplus
n}v\left(  t\right)  ^{1/2}$. As above $\mathbf{v}_{\lambda,m}\left(
t\right)  \in\mathfrak{c}^{c}$ and $z_{\lambda,m}=\int\mathbf{v}_{\lambda
,m}\left(  t\right)  d\mu\in\max\mathfrak{c}$ for all $m$. Then
\begin{align*}
\int\mathbf{v}_{\lambda}\left(  t\right)  d\mu &  =\lim_{m}\int v\left(
t\right)  ^{1/2}\left(  \sum_{f\in\lambda}k_{f,m}\left(  t\right)
f+k_{e,m}\left(  t\right)  e\right)  ^{\oplus n}v\left(  t\right)  ^{1/2}%
d\mu\\
&  =\lim_{m}\int\mathbf{v}_{\lambda,m}\left(  t\right)  d\mu=\lim
_{m}z_{\lambda,m}\in\max\mathfrak{c},
\end{align*}
which in turn implies that $T^{\left(  n\right)  }v=\lim_{\lambda}%
\int\mathbf{v}_{\lambda}\left(  t\right)  d\mu\left(  t\right)  \in
\max\mathfrak{c}$. In particular, $T:C\left(  X\right)  \rightarrow H$ is a
unital positive mapping. By Proposition \ref{lemCRT1}, $T$ is given by an
$H$-support $k^{\prime}\subseteq\operatorname{ball}L^{\infty}\left(
X,\mu^{\prime}\right)  $ on $X$, where $\left\langle v,\mu^{\prime
}\right\rangle =\left(  T\iota\left(  v\right)  ,e\right)  =\left(
\iota\left(  v\right)  ,k_{e}\right)  =\int v\left(  t\right)  k_{e}\left(
t\right)  d\mu=\left\langle v,k_{e}\mu\right\rangle $ for all $v\in C\left(
X\right)  $, that is, $\mu^{\prime}=k_{e}\mu\in I_{\mu}\left(  X\right)  _{+}%
$. Similarly,
\[
\left(  T\iota\left(  v\right)  ,f\right)  =\left(  \iota\left(  v\right)
,k_{f}\right)  =\int v\left(  t\right)  k_{f}\left(  t\right)  d\mu=\int
v\left(  t\right)  k_{f}\left(  t\right)  k_{e}\left(  t\right)  ^{-1}%
d\mu^{\prime}=\left\langle v,k_{f}k_{e}^{-1}\mu^{\prime}\right\rangle
\]
for all $v\in C\left(  X\right)  $, which means that $k_{f}^{\prime}%
=\dfrac{k_{f}}{k_{e}}$ for all $f\neq e$. Notice that $\mu^{\prime}\left\{
k_{e}=0\right\}  =0$. Finally, taking into account that $T_{k}:\left(
L^{2}\left(  X,\mu\right)  ,u\right)  \rightarrow\left(  H,e\right)  $ is a
Hilbert-Schmidt operator (see Remark \ref{remHSop}), we deduce that
$T:C\left(  X\right)  \rightarrow\left(  H,e\right)  $ is a nuclear operator
\cite[3.3.3]{AlP}.
\end{proof}

We say that $k$ is \textit{a Hilbert-Schmidt }$H$-\textit{support in }%
$L^{2}\left(  X,\mu\right)  $ if $k_{e}=u$ and $\sum_{f\in\lambda}\int
k_{f}^{2}d\mu<\infty$.

\begin{theorem}
\label{tCRT23}Let $T_{k}:\left(  L^{2}\left(  X,\mu\right)  ,u\right)
\rightarrow\left(  H,e\right)  $ be a unital positive mapping that responds to
a Hilbert-Schmidt support $k$ in $L^{2}\left(  X,\mu\right)  $ and let
$T=T_{k}\iota:C\left(  X\right)  \rightarrow\left(  H,e\right)  $ be the
related unital $\ast$-linear mapping. Then $T+e\odot q$ is a separable
morphism for a certain $q\in I_{\mu}\left(  X\right)  _{+}$.
\end{theorem}

\begin{proof}
By assumption $T_{k}$ is a Hilbert-Schmidt operator given by $T_{k}\eta
=\sum_{f}\left(  \eta,k_{f}\right)  f$, $\eta\in L^{2}\left(  X,\mu\right)  $,
and $\left\Vert T_{k}\right\Vert _{2}^{2}=\sum_{f}\left\Vert k_{f}\right\Vert
_{2}^{2}=\sum_{f}\int k_{f}^{2}d\mu<\infty$. For every $n$ choose a finite
subset $\lambda_{n}\subseteq F\backslash\left\{  e\right\}  $ such that
$\sum_{F\backslash\lambda_{n}}\left\Vert k_{f}\right\Vert _{2}^{2}<\dfrac
{1}{2n^{2}}$. Take $f\in\lambda_{n}$ and a real-valued $\mu$-step function
$h_{f,n}$ on $X$ such that $\left\vert h_{f,n}\right\vert \leq\left\vert
k_{f}\right\vert $ and $\left\Vert k_{f}-h_{f,n}\right\Vert _{2}^{2}\leq
\dfrac{1}{2\left\vert \lambda_{n}\right\vert n^{2}}$, where $\left\vert
\lambda_{n}\right\vert $ indicates to the cardinality of $\lambda_{n}$.
Namely, since $k_{f}=k_{f,+}-k_{f,-}$ with $k_{f,+}\geq0$, $k_{f,-}\geq0$ and
$\left\vert k_{f}\right\vert =k_{f,+}+k_{f,-}$, one can choose increasing
sequences $h_{f,n}^{\left(  1\right)  }$ and $h_{f,n}^{\left(  2\right)  }$ of
positive $\mu$-step functions such that $h_{f,n}^{\left(  1\right)  }\uparrow
k_{f,+}$ and $h_{f,n}^{\left(  2\right)  }\uparrow k_{f,-}$. If $h_{f,n}%
=h_{f,n}^{\left(  1\right)  }-h_{f,n}^{\left(  2\right)  }$ then $k_{f}%
=\lim_{n}h_{f,n}$. Note that $\left\vert h_{f,n}\left(  t\right)  \right\vert
=h_{f,n}^{\left(  1\right)  }\left(  t\right)  -h_{f,n}^{\left(  2\right)
}\left(  t\right)  \leq h_{f,n}^{\left(  1\right)  }\left(  t\right)  \leq
k_{f,+}\left(  t\right)  \leq\left\vert k_{f}\left(  t\right)  \right\vert $
if $h_{f,n}^{\left(  1\right)  }\left(  t\right)  \geq h_{f,n}^{\left(
2\right)  }\left(  t\right)  $, and $\left\vert h_{f,n}\left(  t\right)
\right\vert =-h_{f,n}^{\left(  1\right)  }\left(  t\right)  +h_{f,n}^{\left(
2\right)  }\left(  t\right)  \leq h_{f,n}^{\left(  2\right)  }\left(
t\right)  \leq k_{f,-}\left(  t\right)  \leq\left\vert k_{f}\left(  t\right)
\right\vert $ if $h_{f,n}^{\left(  1\right)  }\left(  t\right)  \leq
h_{f,n}^{\left(  2\right)  }\left(  t\right)  $. Define $T_{n}:L^{2}\left(
X,\mu\right)  \rightarrow H$, $T_{n}=e\odot\overline{k_{e}}+\sum_{f\in
\lambda_{n}}f\odot\overline{h_{f,n}}$, which is a finite rank operator such
that $T_{n}^{\ast}f=h_{f,n}$, $f\in\lambda_{n}$, $T_{n}^{\ast}e=k_{e}=u$ and
$T_{n}^{\ast}f=0$, $f\notin\left\{  e\right\}  \cup\lambda_{n}$. Moreover,
\begin{align*}
\left\Vert T_{k}-T_{n}\right\Vert _{2}^{2}  &  =\sum_{f}\left\Vert T_{k}%
^{\ast}f-T_{n}^{\ast}f\right\Vert _{2}^{2}=\sum_{f\in\left\{  e\right\}
\cup\lambda_{n}}\left\Vert T_{k}^{\ast}f-T_{n}^{\ast}f\right\Vert _{2}%
^{2}+\sum_{f\notin\left\{  e\right\}  \cup\lambda_{n}}\left\Vert T_{k}^{\ast
}f\right\Vert _{2}^{2}\\
&  =\sum_{f\in\lambda_{n}}\left\Vert k_{f}-h_{f,n}\right\Vert _{2}^{2}%
+\sum_{f\notin\left\{  e\right\}  \cup\lambda_{n}}\left\Vert k_{f}\right\Vert
_{2}^{2}\leq\dfrac{1}{n^{2}},
\end{align*}
that is, $T_{k}=\lim_{n}T_{n}$ in $\mathcal{B}^{2}\left(  L^{2}\left(
X,\mu\right)  ,H\right)  $. Further, for every $n$ there is a partition
$X=X_{n1}\cup\ldots\cup X_{nm_{n}}$ of $X$ into $\mu$-measurable subsets
$X_{nr}$ such that $h_{f,n}=\sum_{r=1}^{m_{n}}\alpha_{f,n,r}\chi_{nr}$, where
$\chi_{nr}$ is the characteristic function of $X_{nr}$. Then $T_{n}%
=e\odot\overline{k_{e}}+\sum_{f\in\lambda_{n}}\sum_{r=1}^{m_{n}}\alpha
_{f,n,r}f\odot\overline{\chi_{nr}}=e\odot\overline{k_{e}}+\sum_{r=1}^{m_{n}%
}\zeta_{nr}\odot\overline{\chi_{nr}}$ with $\zeta_{nr}=\sum_{f\in\lambda_{n}%
}\alpha_{f,n,r}f\in H_{h}^{e}$. For every $t\in X$ we have
\begin{align*}
\left\Vert \zeta_{nr}\right\Vert ^{2}\chi_{nr}\left(  t\right)   &
=\sum_{f\in\lambda_{n}}\alpha_{f,n,r}^{2}\chi_{nr}\left(  t\right)
=\sum_{f\in\lambda_{n}}h_{f,n}^{2}\left(  t\right)  \chi_{nr}\left(  t\right)
=\left(  \sum_{f\in\lambda_{n}}h_{f,n}^{2}\chi_{nr}\right)  \left(  t\right)
\\
&  \leq\left(  \sum_{f\in\lambda_{n}}k_{f}^{2}\chi_{nr}\right)  \left(
t\right)  =\left(  \sum_{f\in\lambda_{n}}k_{f}^{2}\right)  \left(  t\right)
\chi_{nr}\left(  t\right)  ,
\end{align*}
that is, $\left\Vert \zeta_{nr}\right\Vert ^{2}\chi_{nr}\leq\left(  \sum
_{f\in\lambda_{n}}k_{f}^{2}\right)  \chi_{nr}$. Taking into account that
$\sum_{r=1}^{m_{n}}\chi_{nr}=1$ for all $n$, we obtain that
\begin{equation}
\sum_{r=1}^{m_{n}}\left\Vert \zeta_{nr}\right\Vert ^{2}\chi_{nr}\leq\left(
\sum_{f\in\lambda_{n}}k_{f}^{2}\right)  \sum_{r=1}^{m_{n}}\chi_{nr}\leq
\sum_{f\in\lambda_{n}}k_{f}^{2} \label{KeyP}%
\end{equation}
for all $n$. In particular,
\begin{equation}
\sum_{r=1}^{m_{n}}\left\Vert \zeta_{nr}\right\Vert ^{2}\mu\left(
X_{nr}\right)  =\sum_{r=1}^{m_{n}}\int\left\Vert \zeta_{nr}\right\Vert
^{2}\chi_{nr}d\mu\leq\int\sum_{f\in\lambda_{n}}k_{f}^{2}d\mu=\sum_{f\in
\lambda_{n}}\left\Vert k_{f}\right\Vert _{2}^{2}. \label{KeyY}%
\end{equation}
Put $\mu_{nr}=\chi_{nr}\mu\in I_{\mu}\left(  X\right)  _{+}$. It follows that
$T_{n}\iota=e\odot k_{e}\mu+\sum_{r=1}^{m_{n}}\zeta_{nr}\odot\mu_{nr}$ and
$\left\Vert T\left(  v\right)  -T_{n}\iota\left(  v\right)  \right\Vert
\leq\left\Vert T_{k}-T_{n}\right\Vert \left\Vert v\right\Vert _{2}%
\leq\left\Vert T_{k}-T_{n}\right\Vert _{2}\left\Vert v\right\Vert _{2}$ for
all $v\in C\left(  X\right)  $, that is, $T\left(  v\right)  =\lim T_{n}%
\iota\left(  v\right)  $, $v\in C\left(  X\right)  $. Hence
\[
T=\lim_{n}T_{n}\iota=\lim_{n}\sum_{r=1}^{m_{n}}\zeta_{nr}\odot\mu_{nr}+e\odot
k_{e}\mu=\sum_{n,r}\zeta_{nr}\odot\mu_{nr}+e\odot\mu
\]
with $\left\{  \zeta_{nr}\right\}  \subseteq H_{h}^{e}$ and $\left\{  \mu
_{nr}\right\}  \subseteq I_{\mu}\left(  X\right)  _{+}$. Put $q=\sum
_{n,r}\left\Vert \zeta_{nr}\right\Vert \mu_{nr}$. Using (\ref{KeyY}), we
derive that
\begin{align*}
\left\Vert q\right\Vert  &  \leq\lim_{n}\sum_{r=1}^{m_{n}}\left\Vert
\zeta_{nr}\right\Vert \left\Vert \mu_{nr}\right\Vert =\lim_{n}\sum
_{r=1}^{m_{n}}\left\Vert \zeta_{nr}\right\Vert \mu\left(  X_{nr}\right)
^{1/2}\mu\left(  X_{nr}\right)  ^{1/2}\\
&  \leq\lim_{n}\left(  \sum_{r=1}^{m_{n}}\left\Vert \zeta_{nr}\right\Vert
^{2}\mu\left(  X_{nr}\right)  \right)  ^{1/2}\left(  \sum_{r=1}^{m_{n}}%
\mu\left(  X_{nr}\right)  \right)  ^{1/2}\leq\lim_{n}\left(  \sum_{f\in
\lambda_{n}}\left\Vert k_{f}\right\Vert _{2}^{2}\right)  ^{1/2}\\
&  \leq\left(  \sum_{f}\left\Vert k_{f}\right\Vert _{2}^{2}\right)
^{1/2}=\left\Vert T_{k}\right\Vert _{2}^{2}<\infty,
\end{align*}
that is, $q\in I_{\mu}\left(  X\right)  _{+}$. It follows that
\[
T=\sum_{n,r}\left(  \zeta_{nr}+\left\Vert \zeta_{nr}\right\Vert e\right)
\odot\mu_{nr}+e\odot k_{e}\mu-e\odot q.
\]
But $\left(  \zeta_{nr}+\left\Vert \zeta_{nr}\right\Vert e\right)
_{n,r}\subseteq\operatorname{ball}H_{h}^{e}+e=S\left(  \overline{\mathfrak{c}%
}\right)  \subseteq\mathfrak{c}$ and $\sum_{n,r}\left(  \zeta_{nr}+\left\Vert
\zeta_{nr}\right\Vert e\right)  \odot\mu_{nr}+e\odot k_{e}\mu$ is a separable
morphism. Whence $T+e\odot q$ is separable for some $q\in I_{\mu}\left(
X\right)  _{+}$.
\end{proof}

\begin{theorem}
\label{thMX1}Let $T:C\left(  X\right)  \rightarrow\left(  H,e\right)  $ be a
unital positive mapping with its $H$-support $k$ on $X$. Then $T$ is a nuclear
morphism iff $k$ is of type $2$. Moreover, $T$ is separable iff $k$ is a
maximal $H$-support on $X$. Thus there is a natural bijection between nuclear
morphisms $T:C\left(  X\right)  \rightarrow\left(  H,e\right)  $ and
$H$-supports $k$ on $X$ of type $2$. In this case, separable morphisms
correspond to the maximal supports.
\end{theorem}

\begin{proof}
Let $T:C\left(  X\right)  \rightarrow\left(  H,e\right)  $ be a unital
positive mapping with its $H$-support $k\subseteq\operatorname{ball}L^{\infty
}\left(  X,\mu\right)  $ on $X$. If $T$ is a nuclear morphism then its bounded
linear extension $T_{k}:L^{2}\left(  X,\mu\right)  \rightarrow\left(
H,e\right)  $ is a Hilbert-Schmidt operator by virtue of Theorem
\ref{propSepNuc1}. In particular, $k$ is of type $2$.

Conversely, suppose $k$ is an $H$-support on $X$ of type $2$. By Theorem
\ref{corCHe1}, $T_{k}:L^{2}\left(  X,\mu\right)  \rightarrow\left(
H,e\right)  $ is a unital positive mapping of the Hilbert spaces. Moreover,
$k$ turns out to be a Hilbert-Schmidt $H$-support in\textit{ }$L^{2}\left(
X,\mu\right)  $. Notice that $k_{e}=u$ automatically. By Theorem \ref{tCRT23},
$T_{k}\iota+e\odot q$ is separable for some $q\in I_{\mu}\left(  X\right)
_{+}$. But $T_{k}\iota=T$ is a unital positive mapping by assumption. By Lemma
\ref{lemNSO1}, $T$ is a nuclear morphism.

Further, the $H$-support $k$ on $X$ is maximal whenever $T$ is separable
thanks to Theorem \ref{propSepNuc1}. Conversely, suppose $\sum_{f\neq e}%
k_{f}^{2}\leq k_{e}^{2}$. Using (\ref{KeyP}) from the proof of Theorem
\ref{tCRT23}, we have $\sum_{r=1}^{m_{n}}\left\Vert \zeta_{nr}\right\Vert
^{2}\chi_{nr}\leq k_{e}^{2}=1$, which in turn implies that $\sum_{r=1}^{m_{n}%
}\left\Vert \zeta_{nr}\right\Vert \mu_{nr}\leq k_{e}\mu$ for all $n$. Then
$q=\vee\left\{  \sum_{r=1}^{m_{n}}\left\Vert \zeta_{nr}\right\Vert \mu
_{nr}\right\}  \leq k_{e}\mu$ or $k_{e}\mu-q\geq0$. Hence $T=\sum_{n,r}\left(
\zeta_{nr}+\left\Vert \zeta_{nr}\right\Vert e\right)  \odot\mu_{nr}%
+e\odot\left(  k_{e}\mu-q\right)  $ turns out to be a separable morphism.
\end{proof}

\begin{example}
Consider the Hilbert space $H=\ell^{2}$ with its canonical (hermitian) basis
$F=\left\{  f_{n}:n\geq1\right\}  $ and put $e=f_{1}$. The cone $\mathfrak{c}$
consists of those hermitians $\zeta\in\ell^{2}$ such that $\left\Vert
\zeta\right\Vert \leq\sqrt{2}\left(  \zeta,e\right)  $. As in Example
\ref{exm11}, we equip the compact interval $X=\left[  -1,1\right]
\subseteq\mathbb{R}$ with Lebesgue's measure $2^{-1}dt$. Put $k_{n}=k_{f_{n}%
}=n^{-1}\sin\left(  n\pi t\right)  $, $n\geq2$, and $k_{1}=k_{e}=1$. The
family $k=\left\{  k_{n}\right\}  $ is an $\ell^{2}$-support on $\left[
-1,1\right]  $. Indeed, we know that $k_{n}\perp k_{1}$, $n\geq2$, and
\[
\sum_{n>1}\left(  v,k_{n}\right)  ^{2}=\sum_{n>1}n^{-2}\left(  \int v\left(
t\right)  \sin\left(  n\pi t\right)  2^{-1}dt\right)  ^{2}\leq\left(
\sum_{n>1}n^{-2}\right)  \left(  \int v\left(  t\right)  2^{-1}dt\right)
^{2}\leq\left(  v,k_{1}\right)  ^{2}%
\]
in $L^{2}\left[  -1,1\right]  $ for all $v\in C\left[  -1,1\right]  _{+}$.
Thus $T:C\left[  -1,1\right]  \rightarrow\ell^{2}$, $Tv=\sum_{n\geq1}\left(
v,k_{n}\right)  f_{n}$ is a unital positive mapping. Actually, it is a
separable morphism. Indeed, based on Theorem \ref{tCRT23}, it suffices to
prove that the support $k$ is maximal, which can easily be detected
\[
\sum_{n>1}k_{n}^{2}=\sum_{n>1}n^{-2}\sin^{2}\left(  n\pi t\right)  \leq
\sum_{n>1}n^{-2}\leq1=k_{e},
\]
In particular, $T:L^{2}\left[  -1,1\right]  \rightarrow\ell^{2}$,
$T=\sum_{n\geq1}f_{n}\odot\overline{k_{n}}$ is a Hilbert-Schmidt operator and
$\left\Vert T\right\Vert _{2}^{2}=\sum_{n\geq1}\left\Vert k_{n}\right\Vert
_{2}^{2}=\sum_{n\geq1}n^{-2}\leq\pi^{2}/6$.
\end{example}

\begin{corollary}
\label{corCHe2}Let $T:C\left(  X\right)  \rightarrow\left(  H,e\right)  $ be a
unital positive mapping with its $H$-support $k\subseteq\operatorname{ball}%
L^{\infty}\left(  X,\mu\right)  $ on $X$. If $\mu$ is an atomic measure on $X$
of finite support then $T$ is a separable morphism. In particular, a unital
positive mapping $T:\ell^{\infty}\left(  n\right)  \rightarrow\left(
H,e\right)  $ defines a morphism $T:\left(  \ell^{\infty}\left(  n\right)
,\min\ell^{\infty}\left(  n\right)  _{+}\right)  \rightarrow\left(
H,\max\mathfrak{c}\right)  $ of the relevant operator systems.
\end{corollary}

\begin{proof}
Let $S\subseteq X$ be a finite subset and let $\left\{  c_{t}:t\in S\right\}
$ be a family of positive real numbers with $\sum_{t\in S}c_{t}=1$. By
assumption, $\mu=\sum_{t\in S}c_{t}\delta_{t}\in\mathcal{P}\left(  X\right)  $
is an atomic measure with the support $\operatorname{supp}\left(  \mu\right)
=S$. Using Proposition \ref{tCRT1}, we deduce that $\left(  Tv,e\right)
=\left\langle v,\mu\right\rangle =\sum_{t\in S}v\left(  t\right)  c_{t}$ for
all $v\in C\left(  X\right)  $. Moreover, $\left(  Tv,f\right)  =\sum_{t\in
S}v\left(  t\right)  k_{f}\left(  t\right)  c_{t}=\left\langle v,k_{f}%
\mu\right\rangle $ for $k_{f}\in L^{\infty}\left(  X,\mu\right)  $,
$k_{f}\perp1$ in $L^{2}\left(  X,\mu\right)  $ (or $\sum_{t\in S}k_{f}\left(
t\right)  c_{t}=0$) for all $f\neq e$. Since $k$ is an $H$-support on $X$, we
obtain that
\[
\sum_{f\neq e}\left(  \sum_{t\in S}v\left(  t\right)  k_{f}\left(  t\right)
c_{t}\right)  ^{2}\leq\left(  \sum_{t\in S}v\left(  t\right)  c_{t}\right)
^{2}\text{ for all }v\in C\left(  X\right)  _{+}.
\]
Fix $s\in S$ and choose its neighborhood $U$ such that $U\cap S=\left\{
s\right\}  $. Take $v\in C\left(  X\right)  _{h}$, $0\leq v\leq1$ such that
$\operatorname{supp}\left(  v\right)  \subseteq U$ and $v\left(  s\right)
=1$. Then $\left\langle v,\mu\right\rangle =\sum_{t\in S\cap U}v\left(
t\right)  c_{t}=c_{s}$, $\left\langle v,k_{f}\mu\right\rangle =k_{f}\left(
s\right)  c_{s}$ and $\sum_{f\neq e}k_{f}\left(  s\right)  ^{2}c_{s}^{2}%
=\sum_{f\neq e}\left\langle v,k_{f}\mu\right\rangle ^{2}\leq\left\langle
v,\mu\right\rangle ^{2}=c_{s}^{2}$. Thus $\sum_{f\neq e}k_{f}^{2}\leq1$ in
$L^{\infty}\left(  X,\mu\right)  $, which means that $k$ is a maximal
$H$-support on $X$. Using Theorem \ref{tCRT23}, we conclude that $T$ is a
separable morphism.

Finally, if $X=\left\{  1,2,\ldots,n\right\}  $ is a finite set then
$\ell^{\infty}\left(  X\right)  =C\left(  X\right)  $ and $\mathcal{P}\left(
X\right)  $ consists of atomic measures with their finite supports. Therefore
the support of every unital positive mapping $T:\ell^{\infty}\left(  X\right)
\rightarrow\left(  H,e\right)  $ is maximal. It follows that $T$ is separable.
In particular, $T:\left(  \ell^{\infty}\left(  X\right)  ,\min\ell^{\infty
}\left(  X\right)  _{+}\right)  \rightarrow\left(  H,\max\mathfrak{c}\right)
$ is a morphism of the operator systems.
\end{proof}

\subsection{Paulsen-Todorov-Tomforde problem}

Fix two basis elements $u$ and $e$ from a hermitian basis $F$ for $H$, and
consider the related unital cones $\mathfrak{c}_{u}$ and $\mathfrak{c}_{e}$ in
$H$, respectively. Thus we have the unital spaces $\left(  H,u\right)  $ and
$\left(  H,e\right)  $, respectively. Since $F$ is a basis for $H$, the
correspondence $T\left(  u\right)  =e$, $T\left(  e\right)  =u$, $T\left(
f\right)  =f$, $f\neq e,u$ is uniquely extended up to a unitary operator
$T\in\mathcal{B}\left(  H\right)  $ such that $T\zeta=\left(  \zeta,e\right)
u+\left(  \zeta,u\right)  e+\sum_{f\neq u,e}\left(  \zeta,f\right)  f$. Note
that $T=T_{k}$ for the $H$-support $k=\left\{  k_{f}:f\in F\right\}  $ with
$k_{e}=u$, $k_{u}=e$ and $k_{f}=f$ for all $f\neq u,e$. Notice that for every
$\zeta_{0}\in H_{h}^{u}$ we have $\sum_{f\neq e}\left(  \zeta_{0}%
,k_{f}\right)  ^{2}=\sum_{f\neq u}\left(  \zeta_{0},f\right)  ^{2}%
\leq\left\Vert \zeta_{0}\right\Vert ^{2}=\left(  \left(  \zeta_{0}%
,k_{e}\right)  +\left\Vert \zeta_{0}\right\Vert \right)  ^{2}$, which means
that $k$ is a unital $H$-support in $\left(  H,u\right)  $ (see Remark
\ref{remUHS}). Moreover,%
\begin{align*}
\left(  T\zeta,\eta\right)   &  =\left(  \zeta,e\right)  \left(
\eta,u\right)  ^{\ast}+\left(  \zeta,u\right)  \left(  \eta,e\right)  ^{\ast
}+\sum_{f\neq u,e}\left(  \zeta,f\right)  \left(  \eta,f\right)  ^{\ast}\\
&  =\left(  \zeta,\left(  \eta,e\right)  u\right)  +\left(  \zeta,\left(
\eta,u\right)  e\right)  +\sum_{f\neq u,e}\left(  \zeta,\left(  \eta,f\right)
f\right)  =\left(  \zeta,T\eta\right)
\end{align*}
for all $\zeta,\eta\in H$, which means that $T^{\ast}=T=T^{-1}$. In
particular, $\left\langle T\zeta,\overline{\eta}\right\rangle =\left(
T\zeta,\eta\right)  =\left(  \zeta,T\eta\right)  =\left\langle \zeta
,\overline{T\eta}\right\rangle $, which means that $\overline{T}\in
\mathcal{B}\left(  \overline{H}\right)  $, $\overline{T}\left(  \overline
{\eta}\right)  =\overline{T\eta}$ is the dual mapping to $T$. Note also that
$T:\left(  H,u\right)  \rightarrow\left(  H,e\right)  $ is a unital $\ast
$-linear mapping of unital spaces. Indeed, $T\zeta^{\ast}=\left(  \zeta^{\ast
},e\right)  u+\left(  \zeta^{\ast},u\right)  e+\sum_{f\neq u,e}\left(
\zeta^{\ast},f\right)  f=\left(  \zeta,e\right)  ^{\ast}u+\left(
\zeta,u\right)  ^{\ast}e+\sum_{f\neq u,e}\left(  \zeta,f\right)  ^{\ast
}f=\left(  T\zeta\right)  ^{\ast}$. Notice that $F$ is a hermitian basis for
$H$.

\begin{lemma}
\label{lemTminmax}For the cones $\mathfrak{c}_{u}$ and $\mathfrak{c}_{e}$ we
have $T\left(  \mathfrak{c}_{u}\right)  =\mathfrak{c}_{e}$ and $\overline
{T}\left(  S\left(  \mathfrak{c}_{e}\right)  \right)  =S\left(  \mathfrak{c}%
_{u}\right)  $. Similarly, $\overline{T}\left(  \overline{\mathfrak{c}_{u}%
}\right)  =\overline{\mathfrak{c}_{e}}$ and $T\left(  S\left(  \overline
{\mathfrak{c}_{e}}\right)  \right)  =S\left(  \overline{\mathfrak{c}_{u}%
}\right)  $. In particular, $T^{\left(  \infty\right)  }\left(  \min
\mathfrak{c}_{u}\right)  =\min\mathfrak{c}_{e}$ and $T^{\left(  \infty\right)
}\left(  \max\mathfrak{c}_{u}\right)  =\max\mathfrak{c}_{e}$.
\end{lemma}

\begin{proof}
Take $\zeta\in\mathfrak{c}_{u}$ with $\zeta=\zeta_{0}+\left(  \zeta,u\right)
u$, $\left\Vert \zeta_{0}\right\Vert \leq\left(  \zeta,u\right)  $, where
$\zeta_{0}\in H_{h}^{u}$. But $\zeta_{0}=\left(  \zeta,e\right)  e+\zeta
_{0}^{\prime}$, $\zeta_{0}^{\prime}\in H_{h}^{u}\cap H_{h}^{e}$, therefore
$T\zeta=\zeta_{0}^{\prime}+\left(  \zeta,e\right)  u+\left(  \zeta,u\right)
e$ and $\left\Vert \zeta_{0}^{\prime}+\left(  \zeta,e\right)  u\right\Vert
^{2}=\left\Vert \zeta_{0}^{\prime}\right\Vert ^{2}+\left\vert \left(
\zeta,e\right)  \right\vert ^{2}=\left\Vert \zeta_{0}\right\Vert ^{2}%
\leq\left(  \zeta,u\right)  ^{2}=\left(  T\zeta,e\right)  ^{2}$. The latter
means that $T\zeta\in\mathfrak{c}_{e}$, that is, $T\left(  \mathfrak{c}%
_{u}\right)  \subseteq\mathfrak{c}_{e}$. If $\left(  \zeta,u\right)  =1$ then
$\left(  T\zeta,e\right)  =1$ as well, which means that $T\left(  S\left(
\overline{\mathfrak{c}_{u}}\right)  \right)  =T\left(  \operatorname{ball}%
H_{h}^{u}+u\right)  \subseteq\operatorname{ball}H_{h}^{e}+e=S\left(
\overline{\mathfrak{c}_{e}}\right)  $ (see Lemma \ref{lHSc11}). By symmetry,
$T\left(  \mathfrak{c}_{e}\right)  \subseteq\mathfrak{c}_{u}$ and $T\left(
S\left(  \overline{\mathfrak{c}_{e}}\right)  \right)  \subseteq S\left(
\overline{\mathfrak{c}_{u}}\right)  $, therefore $T\left(  \mathfrak{c}%
_{u}\right)  =\mathfrak{c}_{e}$ and $T\left(  S\left(  \overline
{\mathfrak{c}_{u}}\right)  \right)  =S\left(  \overline{\mathfrak{c}_{e}%
}\right)  $. Similarly, $\overline{T}\left(  \overline{\mathfrak{c}_{u}%
}\right)  =\overline{\mathfrak{c}_{e}}$, $T\left(  S\left(  \overline
{\mathfrak{c}_{e}}\right)  \right)  =S\left(  \overline{\mathfrak{c}_{u}%
}\right)  $, and $\overline{T}\left(  S\left(  \mathfrak{c}_{e}\right)
\right)  =S\left(  \mathfrak{c}_{u}\right)  $.

Finally, the equality $\overline{T}\left(  S\left(  \mathfrak{c}_{e}\right)
\right)  =S\left(  \mathfrak{c}_{u}\right)  $ implies that $T^{\left(
\infty\right)  }\left(  \min\mathfrak{c}_{u}\right)  =T^{\left(
\infty\right)  }\left(  S\left(  \mathfrak{c}_{u}\right)  ^{\boxdot}\right)
\subseteq S\left(  \mathfrak{c}_{e}\right)  ^{\boxdot}=\min\mathfrak{c}_{e}$
due to Lemma \ref{lDualM}. But $\overline{T}\left(  S\left(  \mathfrak{c}%
_{u}\right)  \right)  =S\left(  \mathfrak{c}_{e}\right)  $ as well, thereby
$T^{\left(  \infty\right)  }\left(  \min\mathfrak{c}_{e}\right)  \subseteq
\min\mathfrak{c}_{u}$. Hence $T^{\left(  \infty\right)  }\left(
\min\mathfrak{c}_{u}\right)  =\min\mathfrak{c}_{e}$. In particular,
$\overline{T}^{\left(  \infty\right)  }\left(  \min\overline{\mathfrak{c}_{e}%
}\right)  =\min\overline{\mathfrak{c}_{u}}$. Using again Lemma \ref{lDualM}
and Theorem \ref{tmax32}, we obtain that $T^{\left(  \infty\right)  }\left(
\max\mathfrak{c}_{u}\right)  =T^{\left(  \infty\right)  }\left(  \left(
\min\overline{\mathfrak{c}_{u}}\right)  ^{\boxdot}\right)  \subseteq\left(
\min\overline{\mathfrak{c}_{e}}\right)  ^{\boxdot}=\max\mathfrak{c}_{e}$. By
symmetry, $T^{\left(  \infty\right)  }\left(  \max\mathfrak{c}_{e}\right)
\subseteq\max\mathfrak{c}_{u}$. Whence $T^{\left(  \infty\right)  }\left(
\max\mathfrak{c}_{u}\right)  =\max\mathfrak{c}_{e}$.
\end{proof}

Thus $T:\left(  H,\max\mathfrak{c}_{u}\right)  \rightarrow\left(
H,\max\mathfrak{c}_{e}\right)  $ is a matrix positive mapping. Actually it is
an isomorphism of the operator systems.

\begin{theorem}
\label{tSEP1}Let $H$ be an infinite dimensional Hilbert space and let
$T\in\mathcal{B}\left(  H\right)  $ be a unitary given by $T=u\odot
\overline{e}+e\odot\overline{u}+\sum_{f\neq u,e}f\odot\overline{f}$. The
matrix positive mapping $T:\left(  H,\max\mathfrak{c}_{u}\right)
\rightarrow\left(  H,\max\mathfrak{c}_{e}\right)  $ given by $T$ is not separable.
\end{theorem}

\begin{proof}
Suppose that $T$ is separable, that is, $T=\sum_{l}p_{l}\odot q_{l}$ for some
$\mathfrak{c}_{u}$-positive functionals $q_{l}$ on $\left(  H,u\right)  $ and
$\mathfrak{c}_{e}$-positive elements $p_{l}$ from $\left(  H,e\right)  $. By
Corollary \ref{cormincbar}, $q_{l}=\overline{\eta_{l}}+s_{l}\overline{u}$,
$\eta_{l}\in H_{h}^{u}$, $\left\Vert \eta_{l}\right\Vert \leq s_{l}$, and
$p_{l}=\zeta_{l}+r_{l}e$, $\zeta_{l}\in H_{h}^{e}$, $\left\Vert \zeta
_{l}\right\Vert \leq r_{l}$. Then
\begin{align*}
T\zeta &  =\sum_{l}\left(  \left(  \zeta,\eta_{l}\right)  +s_{l}\left(
\zeta,u\right)  \right)  \left(  \zeta_{l}+r_{l}e\right) \\
&  =\sum_{l}\left(  \left(  \zeta,\eta_{l}\right)  +s_{l}\left(
\zeta,u\right)  \right)  \zeta_{l}+\sum_{l}r_{l}\left(  \left(  \zeta,\eta
_{l}\right)  +s_{l}\left(  \zeta,u\right)  \right)  e
\end{align*}
for all $\zeta\in H$. In particular, $Te=\sum_{l}\left(  \eta_{l},e\right)
\zeta_{l}+\sum_{l}r_{l}\left(  \eta_{l},e\right)  e=u$ and $\sum_{l}\left(
\eta_{l},e\right)  \zeta_{l}\in H_{h}^{e}$ imply that $\sum_{l}r_{l}\left(
\eta_{l},e\right)  =0$ and $\sum_{l}\left(  \eta_{l},e\right)  \zeta_{l}=u$.
Similarly, $Tu=\sum_{l}s_{l}\zeta_{l}+\sum_{l}r_{l}s_{l}e=e$ implies that
$\sum_{l}s_{l}\zeta_{l}=0$ and $\sum_{l}r_{l}s_{l}=1$. Put $\varphi\left(
\zeta\right)  =\sum_{l}r_{l}\left(  \zeta,\eta_{l}\right)  $, $\zeta\in H$.
Then $\left\vert \varphi\left(  \zeta\right)  \right\vert \leq\sum_{l}%
r_{l}\left\vert \left(  \zeta,\eta_{l}\right)  \right\vert \leq\left\Vert
\zeta\right\Vert \sum_{l}r_{l}s_{l}=\left\Vert \zeta\right\Vert $, that is,
$\varphi\in\operatorname{ball}H^{\ast}$ and
\[
T\zeta=G\left(  \zeta\right)  +\left(  \varphi\left(  \zeta\right)  +\left(
\zeta,u\right)  \right)  e,
\]
where $G\left(  \zeta\right)  =\sum_{l}\left(  \zeta,\eta_{l}\right)
\zeta_{l}\in H_{h}^{e}$ for all $\zeta\in H$. As we have seen above $G\left(
e\right)  =u$, $G\left(  u\right)  =0$ and $\varphi\left(  u\right)
=\varphi\left(  e\right)  =0$. Since $Tf=f$ for all $f\neq u,e$, we deduce
that $\varphi\left(  f\right)  =0$ and $G\left(  f\right)  =f$ for all $f\neq
u,e$. Thus $\varphi\left(  F\right)  =\left\{  0\right\}  $, which means that
$\varphi=0$. Consequently,
\[
T=e\odot\overline{u}+\sum_{l}\zeta_{l}\odot\overline{\eta_{l}}\text{ with
}\sum_{l}\left\Vert \zeta_{l}\right\Vert \left\Vert \overline{\eta_{l}%
}\right\Vert \leq\sum_{l}r_{l}s_{l}=1,
\]
which means that $T$ is a nuclear operator. In particular, $T$ is a compact
operator. But $\operatorname{im}T$ contains an infinite dimensional closed
subspace generated by $F\backslash\left\{  u,e\right\}  $, a contradiction.
\end{proof}

Thus a matrix positive mapping into $\max$-quantization may not be separable
(see \cite{PaulTT}).

\subsection{The operator Hilbert system $\ell^{2}\left(  2\right)  $}

In this subsection we analyze the $2$-dimensional case of $\ell^{2}\left(
2\right)  $. Suppose that $K=\ell^{2}\left(  \epsilon\right)  $ with an
hermitian basis $\epsilon=\left(  \epsilon_{1},\epsilon_{2}\right)  $ and unit
$u=\epsilon_{1}$. Thus $\mathfrak{c}_{u}$ consists of those $\eta=\eta
_{1}\epsilon_{1}+\eta_{2}\epsilon_{2}\in K_{h}$ such that $\eta_{1}\geq0$ and
$\left\vert \eta_{2}\right\vert \leq\eta_{1}$. Moreover, $S\left(
\mathfrak{c}_{u}\right)  =\operatorname{ball}\overline{K_{h}^{u}}+u=\left\{
\eta_{2}\overline{\epsilon_{2}}+\overline{\epsilon_{1}}:\left\vert \eta
_{2}\right\vert \leq1\right\}  $, and $\zeta\in\mathfrak{c}_{u}$ iff
$\left\langle \zeta,S\left(  \mathfrak{c}_{u}\right)  \right\rangle \geq0$. We
have the canonical $\ast$-representation $K\rightarrow C\left(  S\left(
\mathfrak{c}_{u}\right)  \right)  $, $\zeta\mapsto\widehat{\zeta}$,
$\widehat{\zeta}\left(  t\right)  =\left\langle \zeta,t\right\rangle $, $t\in
S\left(  \mathfrak{c}_{u}\right)  $ (see below Appendix Section \ref{secAPPX}%
). Notice that $\zeta\in\mathfrak{c}_{u}$ iff $\widehat{\zeta}\in C\left(
S\left(  \mathfrak{c}_{u}\right)  \right)  _{+}$. Consider the algebra
$\ell^{\infty}\left(  \theta\right)  $ with a basis $\theta=\left(  \theta
_{1},\theta_{2},\theta_{3}\right)  $, and the unital linear embedding
$\kappa:K\rightarrow\ell^{\infty}\left(  \theta\right)  $, $\kappa\left(
\eta\right)  =\kappa\left(  \eta_{1},\eta_{2}\right)  =\left(  \eta_{1}%
,\eta_{1}+\eta_{2},\eta_{1}-\eta_{2}\right)  $. Note that $\kappa\left(
u\right)  =\kappa\left(  \epsilon_{1}\right)  =\left(  1,1,1\right)  $,
$\kappa\left(  \epsilon_{2}\right)  =\left(  0,1,-1\right)  =\theta_{2}%
-\theta_{3}$.

\begin{lemma}
\label{lemFDC1}Let $\zeta\in K_{h}$. Then $\zeta\in\mathfrak{c}_{u}$ iff
$\widehat{\zeta}\left(  \overline{\epsilon_{1}}\right)  \geq0$ and
$\widehat{\zeta}\left(  \overline{\epsilon_{1}}\pm\overline{\epsilon_{2}%
}\right)  \geq0$. In particular, $\kappa\left(  \mathfrak{c}_{u}\right)
=\ell^{\infty}\left(  \theta\right)  _{+}\cap\kappa\left(  K\right)  $, and
every unital positive mapping $T:K\rightarrow\mathcal{V}$ from $K$ to an
operator system $\mathcal{V}$ admits a unital positive extension
$\widetilde{T}:\ell^{\infty}\left(  \theta\right)  \rightarrow\mathcal{V}$,
$\widetilde{T}\cdot\kappa=T$.
\end{lemma}

\begin{proof}
If $\zeta\in\mathfrak{c}_{u}$ then $\widehat{\zeta}\left(  \overline
{\epsilon_{1}}\right)  \geq0$ and $\widehat{\zeta}\left(  \overline
{\epsilon_{1}}\pm\overline{\epsilon_{2}}\right)  \geq0$, for $\left\{
\overline{\epsilon_{1}},\overline{\epsilon_{1}}\pm\overline{\epsilon_{2}%
}\right\}  \subseteq S\left(  \mathfrak{c}_{u}\right)  $. Conversely, assume
that $\widehat{\zeta}\left(  \overline{\epsilon_{1}}\right)  \geq0$ and
$\widehat{\zeta}\left(  \overline{\epsilon_{1}}\pm\overline{\epsilon_{2}%
}\right)  \geq0$. Then $\left\vert \widehat{\zeta}\left(  \overline
{\epsilon_{2}}\right)  \right\vert \leq\widehat{\zeta}\left(  \overline
{\epsilon_{1}}\right)  $. Take $t=r\overline{\epsilon_{2}}+\overline
{\epsilon_{1}}\in S\left(  \mathfrak{c}_{u}\right)  $ with $\left\vert
r\right\vert \leq1$. Then $\left\vert \widehat{\zeta}\left(  r\overline
{\epsilon_{2}}\right)  \right\vert =\left\vert r\widehat{\zeta}\left(
\overline{\epsilon_{2}}\right)  \right\vert \leq\widehat{\zeta}\left(
\overline{\epsilon_{1}}\right)  $, which in turn implies that $\widehat{\zeta
}\left(  t\right)  =\widehat{\zeta}\left(  \overline{\epsilon_{1}}%
+r\overline{\epsilon_{2}}\right)  \geq0$. Hence $\widehat{\zeta}\in C\left(
S\left(  \mathfrak{c}_{u}\right)  \right)  _{+}$ or $\zeta\in\mathfrak{c}_{u}$.

In particular, $\zeta\in\mathfrak{c}_{u}$ iff $\kappa\left(  \zeta\right)
\left(  1\right)  =\zeta_{1}=\widehat{\zeta}\left(  \overline{\epsilon_{1}%
}\right)  \geq0$, $\kappa\left(  \zeta\right)  \left(  2\right)  =\zeta
_{1}+\zeta_{2}=\widehat{\zeta}\left(  \overline{\epsilon_{1}}+\overline
{\epsilon_{2}}\right)  \geq0$ and $\kappa\left(  \zeta\right)  \left(
3\right)  =\zeta_{1}-\zeta_{2}=\widehat{\zeta}\left(  \overline{\epsilon_{1}%
}-\overline{\epsilon_{2}}\right)  \geq0$, that is, $\kappa\left(
\zeta\right)  \in\ell^{\infty}\left(  \theta\right)  _{+}$.

Finally, suppose $T:K\rightarrow\mathcal{V}$ is a unital positive mapping into
an operator system $\mathcal{V}$. Since $-u\leq\epsilon_{2}\leq u$ in $K_{h}$,
it follows that $-e\leq T\left(  \epsilon_{2}\right)  \leq e$, where
$e=T\left(  u\right)  $ is the unit of $\mathcal{V}$. Thus $\left\Vert
T\left(  \epsilon_{2}\right)  \right\Vert \leq1$. Choose $v_{i}\in
\mathcal{V}_{+}$, $i=1,2$ such that $v_{1}+2v_{2}=e+T\left(  \epsilon
_{2}\right)  $, $T\left(  \epsilon_{2}\right)  \leq v_{2}$. For example, one
can choose $v_{1}=0$ and $v_{2}=2^{-1}\left(  e+T\left(  \epsilon_{2}\right)
\right)  $, for $T\left(  \epsilon_{2}\right)  \leq e$ implies that
$e+T\left(  \epsilon_{2}\right)  \geq2T\left(  \epsilon_{2}\right)  $ or
$v_{2}\geq T\left(  \epsilon_{2}\right)  $. Define $\widetilde{T}:\ell
^{\infty}\left(  \theta\right)  \rightarrow\mathcal{V}$ to be $\widetilde{T}%
\left(  \lambda_{1}\theta_{1}+\lambda_{2}\theta_{2}+\lambda_{3}\theta
_{3}\right)  =\lambda_{1}v_{1}+\left(  \lambda_{2}+\lambda_{3}\right)
v_{2}-T\left(  \lambda_{3}\epsilon_{2}\right)  $. Then
\begin{align*}
\widetilde{T}\kappa\left(  \eta\right)   &  =\widetilde{T}\left(  \eta
_{1}\theta_{1}+\left(  \eta_{1}+\eta_{2}\right)  \theta_{2}+\left(  \eta
_{1}-\eta_{2}\right)  \theta_{3}\right)  =\eta_{1}v_{1}+\left(  2\eta
_{1}\right)  v_{2}-T\left(  \left(  \eta_{1}-\eta_{2}\right)  \epsilon
_{2}\right) \\
&  =\eta_{1}\left(  v_{1}+2v_{2}\right)  -\left(  \eta_{1}-\eta_{2}\right)
T\left(  \epsilon_{2}\right)  =\eta_{1}\left(  e+T\left(  \epsilon_{2}\right)
\right)  -\left(  \eta_{1}-\eta_{2}\right)  T\left(  \epsilon_{2}\right) \\
&  =\eta_{1}e+\eta_{2}T\left(  \epsilon_{2}\right)  =\eta_{1}T\left(
\epsilon_{1}\right)  +\eta_{2}T\left(  \epsilon_{2}\right)  =T\eta,
\end{align*}
that is, $\widetilde{T}\kappa=T$. In particular, $\widetilde{T}\left(
1\right)  =\widetilde{T}\left(  \kappa\left(  1,0\right)  \right)  =T\left(
\epsilon_{1}\right)  =T\left(  u\right)  =e$. Note also that $\widetilde{T}%
\left(  \theta_{1}\right)  =v_{1}$, $\widetilde{T}\left(  \theta_{2}\right)
=v_{2}$ and $\widetilde{T}\left(  \theta_{3}\right)  =v_{2}-T\left(
\epsilon_{2}\right)  $, that is, $\widetilde{T}\left(  \theta_{i}\right)
\in\mathcal{V}_{+}$ for all $i$. Consequently, $\widetilde{T}\left(
\ell^{\infty}\left(  \theta\right)  _{+}\right)  \subseteq\mathcal{V}_{+}$,
which means that $\widetilde{T}$ is a unital positive extension of $T$.
\end{proof}

Now let $H$ be an operator Hilbert space with its Hermitian basis $F$ and a
unit $e\in F$. If $T:K\rightarrow H\mathcal{\ }$is a unital positive mapping
then it admits a unital positive extension $\widetilde{T}:\ell^{\infty}\left(
\theta\right)  \rightarrow H$, $\widetilde{T}\cdot\kappa=T$ thanks to Lemma
\ref{lemFDC1}.

\begin{proposition}
\label{propFDC1}Let $K=\ell^{2}\left(  \epsilon\right)  $ be the operator
Hilbert system with its hermitian basis $\epsilon=\left(  \epsilon
_{1},\epsilon_{2}\right)  $ and unit $u=\epsilon_{1}$, $\left(  H,e\right)  $
an operator Hilbert system, and let $T:\left(  K,u\right)  \rightarrow\left(
H,e\right)  $ be a unital positive mapping. Then $T$ is a separable morphism
automatically. In particular, $T:\left(  K,\min\mathfrak{c}_{u}\right)
\rightarrow\left(  H,\max\mathfrak{c}_{e}\right)  $ is a morphism of the
operator systems.
\end{proposition}

\begin{proof}
Based on Lemma \ref{lemFDC1}, there is a unital positive extension
$\widetilde{T}:\ell^{\infty}\left(  \theta\right)  \rightarrow H$ of $T$.
Using Corollary \ref{corCHe2}, we deduce that $\widetilde{T}$ is a separable
morphism. Since $\kappa:K\rightarrow\ell^{\infty}\left(  \theta\right)  $ is a
unital positive mapping (see Lemma \ref{lemFDC1}), it follows that
$T=\widetilde{T}\cdot\kappa$ is a separable morphism either.
\end{proof}

\begin{remark}
Optionally, one can use the following argument. Since $\widetilde{T}$ is
separable, we have $\widetilde{T}^{\left(  \infty\right)  }\left(  \min
\ell^{\infty}\left(  \theta\right)  _{+}\right)  \subseteq\mathfrak{c}_{e}%
^{c}\subseteq\max\mathfrak{c}_{e}$. By Lemma \ref{lemFDC1}, $\kappa\left(
\mathfrak{c}_{u}\right)  =\ell^{\infty}\left(  \theta\right)  _{+}\cap
\kappa\left(  K\right)  $, which in turn implies that $\kappa^{\ast}\left(
S\left(  \ell^{\infty}\left(  \theta\right)  _{+}\right)  \right)  \subseteq
S\left(  \mathfrak{c}_{u}\right)  $. Using Lemma \ref{lDualM} and Proposition
\ref{propPTT1}, we derive that $\kappa^{\left(  \infty\right)  }\left(
\min\mathfrak{c}_{u}\right)  =\kappa^{\left(  \infty\right)  }\left(  S\left(
\mathfrak{c}_{u}\right)  ^{\boxdot}\right)  \subseteq S\left(  \ell^{\infty
}\left(  \theta\right)  _{+}\right)  ^{\boxdot}=\min\ell^{\infty}\left(
\theta\right)  _{+}$. Consequently,%
\[
T^{\left(  \infty\right)  }\left(  \min\mathfrak{c}_{u}\right)  =\widetilde{T}%
^{\left(  \infty\right)  }\kappa^{\left(  \infty\right)  }\left(
\min\mathfrak{c}_{u}\right)  \subseteq\widetilde{T}^{\left(  \infty\right)
}\left(  \min\ell^{\infty}\left(  \theta\right)  _{+}\right)  \subseteq
\max\mathfrak{c}_{e},
\]
which means that $T:\left(  K,\min\mathfrak{c}_{u}\right)  \rightarrow\left(
H,\max\mathfrak{c}_{e}\right)  $ is matrix positive.
\end{remark}

\begin{corollary}
The operator Hilbert system $\ell^{2}\left(  2\right)  $ admits only one
quantization, that is, $\min\mathfrak{c}_{u}=\max\mathfrak{c}_{u}$.
\end{corollary}

\begin{proof}
Put $H=\ell^{2}\left(  2\right)  $ and $T=\operatorname{id}$. Using
Proposition \ref{propFDC1}, we derive that $\min\mathfrak{c}_{u}=T^{\left(
\infty\right)  }\left(  \min\mathfrak{c}_{u}\right)  \subseteq\max
\mathfrak{c}_{u}\subseteq\min\mathfrak{c}_{u}$, that is, $\min\mathfrak{c}%
_{u}=\max\mathfrak{c}_{u}$.
\end{proof}

\begin{corollary}
Every unital positive mapping $T:C\left(  X\right)  \rightarrow\ell^{2}\left(
2\right)  $ is a separable morphism automatically.
\end{corollary}

\begin{proof}
Based on Proposition \ref{propFDC1}, we conclude that the support $k$ of $T$
is given by Borel functions $k_{1}$ and $k_{2}$ from $\operatorname{ball}%
L^{\infty}\left(  X,\mu\right)  $ such that $k_{1}=1$ and $\left\vert
k_{2}\right\vert \leq1$ (or $k_{2}^{2}\leq1$), that is, $k$ is a maximal
support on $X$. By Theorem \ref{tCRT23}, $T$ is a separable morphism.
\end{proof}

\subsection{The operator Hilbert system $HS_{n}$}

Now consider the Hilbert space $HS_{n}$ of all Hilbert-Schmidt operators on
$\ell^{2}\left(  n\right)  $. Thus $HS_{n}=M_{n}$ equipped with the inner
product $\left(  x,y\right)  _{\tau}=\tau\left(  xy^{\ast}\right)  $, $x,y\in
M_{n}$, where $\tau$ is the normalized trace on $M_{n}$. In this case,
$\left\Vert x\right\Vert _{2}=\left(  x,x\right)  _{\tau}^{1/2}=\tau\left(
\left\vert x\right\vert ^{2}\right)  ^{1/2}$ is the Hilbert-Schmidt norm,
$\left\Vert x\right\Vert _{2}\leq\left\Vert x\right\Vert \leq\sqrt
{n}\left\Vert x\right\Vert _{2}$, $x\in M_{n}$, and $\tau\left(  e\right)
=1$. Moreover, $\left(  HS_{n}\right)  _{h}^{e}=\left\{  x_{0}\in\left(
M_{n}\right)  _{h}:\tau\left(  x_{0}\right)  =\left(  x_{0},e\right)  _{\tau
}=0\right\}  =\left(  HS_{n}\right)  _{h}\cap\ker\tau$ and every hermitian $x$
admits a unique orthogonal expansion $x=x_{0}+\tau\left(  x\right)  e$. In
particular,
\[
HS_{n}^{+}=\left\{  x\in\left(  HS_{n}\right)  _{h}:x=x_{0}+\tau\left(
x\right)  e,\left\Vert x_{0}\right\Vert _{2}\leq\tau\left(  x\right)
\right\}
\]
is the unital, separated cone of the operator Hilbert system $HS_{n}$ called
\textit{an operator Hilbert-Schmidt system.}

\begin{proposition}
The equality $M_{n}^{+}=HS_{n}^{+}$ holds for $n<3$.
\end{proposition}

\begin{proof}
Since the equality is trivial for $n=1$ we need just to look at the case of
$n=2$. Take $x\in M_{2}^{+}$. Then $\tau\left(  x\right)  \geq0$ and
$x=x_{0}+\tau\left(  x\right)  e$ with $x_{0}\in\left(  HS_{2}\right)
_{h}^{e}$. If $\lambda_{1}$ and $\lambda_{2}$ are (real) eigenvalues of
$x_{0}$, then $\lambda_{1}+\lambda_{2}=2\tau\left(  x_{0}\right)  =0$, that
is, $\lambda_{1}\geq0\geq\lambda_{2}=-\lambda_{1}$. But $x_{0}\geq-\tau\left(
x\right)  e$, thereby $\lambda_{2}+\tau\left(  x\right)  \geq0$ or
$\lambda_{1}\leq\tau\left(  x\right)  $. It follows that $-\tau\left(
x\right)  e\leq x_{0}\leq\tau\left(  x\right)  e$ and $\left\Vert
x_{0}\right\Vert _{2}\leq\left\Vert x_{0}\right\Vert \leq\tau\left(  x\right)
$. The latter means that $x\in HS_{2}^{+}$. Thus $M_{2}^{+}\subseteq
HS_{2}^{+}$.

Conversely, take $x\in HS_{2}^{+}$ with its expansion $x=x_{0}+\tau\left(
x\right)  e$, $x_{0}\in\left(  HS_{2}\right)  _{h}^{e}$. Since $\tau\left(
x_{0}\right)  =0$, we obtain that
\[
x_{0}=\left[
\begin{array}
[c]{cc}%
a & b\\
b^{\ast} & -a
\end{array}
\right]  \text{ with }a\in\mathbb{R}\text{ and }b\in\mathbb{C}\text{.}%
\]
Note that $\left\Vert x_{0}\right\Vert =\left\Vert x_{0}^{2}\right\Vert
^{1/2}=\sqrt{a^{2}+\left\vert b\right\vert ^{2}}$ and $\left\Vert
x_{0}\right\Vert _{2}=\tau\left(  x_{0}^{2}\right)  ^{1/2}=\left(
2^{-1}\left(  2a^{2}+2\left\vert b\right\vert ^{2}\right)  \right)
^{1/2}=\sqrt{a^{2}+\left\vert b\right\vert ^{2}}=\left\Vert x_{0}\right\Vert
$. It follows that $x_{0}\in\left(  M_{n}\right)  _{h}\ $and $\left\Vert
x_{0}\right\Vert \leq\tau\left(  x\right)  $, which means that $-\tau\left(
x\right)  e\leq x_{0}\leq\tau\left(  x\right)  e$. The latter in turn implies
that $x=x_{0}+\tau\left(  x\right)  e\geq0$, that is, $x\in M_{2}^{+}$. Whence
$M_{2}^{+}=HS_{2}^{+}$.
\end{proof}

\begin{remark}
The equality $M_{n}^{+}=HS_{n}^{+}$ fails to be true for $n\geq3$. For
example, take
\[
x_{0}=\left(  1/\sqrt{2}\right)  \left[
\begin{array}
[c]{ccc}%
1 & 0 & 1\\
0 & -1 & 1\\
1 & 1 & 0
\end{array}
\right]  \in\left(  HS_{3}\right)  _{h}^{e}.
\]
Then $-\sqrt{3/2}$, $0$ and $\sqrt{3/2}$ are eigenvalues of $x_{0}$ and
$\left\Vert x_{0}\right\Vert _{2}=1<\sqrt{3/2}=\left\Vert x_{0}\right\Vert $.
It follows that $x=x_{0}+e\in HS_{3}^{+}$, whereas $x\notin M_{3}^{+}$, for
$x$ admits a negative eigenvalue $1-\sqrt{3/2}$.
\end{remark}

\subsection{Operator Hilbert systems and entanglement breaking mappings}

Let $\mathcal{V}$ be an operator system and let $\mathfrak{C}$ be a
quantization of its cone $\mathcal{V}_{+}$ of positive elements. For every $n$
the quantum cone $\mathfrak{C}$ defines a unital, closed, separated cone
$\mathfrak{C\cap}M_{n}\left(  \mathcal{V}\right)  $ in $M_{n}\left(
\mathcal{V}\right)  $, whose state space in $M_{n}\left(  \mathcal{V}^{\ast
}\right)  $ is denoted by $S_{n}\left(  \mathfrak{C}\right)  $. Thus
$S_{n}\left(  \mathfrak{C}\right)  =S\left(  \mathfrak{C\cap}M_{n}\left(
\mathcal{V}\right)  \right)  $. These state spaces $S_{n}\left(
\mathfrak{C}\right)  $ in turn define the state space $S\left(  \mathfrak{C}%
\right)  $ of $\mathfrak{C}$ on $\mathcal{V}^{\ast}$.

Now let $\mathcal{W}$ be another operator system with its unit $e^{\prime}$
and a quantization $\mathfrak{K}$ of $\mathcal{W}_{+}$. Thus $\left(
\mathcal{V},\mathfrak{C}\right)  $ and $\left(  \mathcal{W},\mathfrak{K}%
\right)  $ are quantum systems. Consider a matrix (or completely) positive
mapping $\varphi:\left(  \mathcal{V},\mathfrak{C}\right)  \rightarrow\left(
\mathcal{W},\mathfrak{K}\right)  $, that is, $\varphi^{\left(  \infty\right)
}\left(  \mathfrak{C}\right)  \subseteq\mathfrak{K}$. In particular,
$\varphi\left(  \mathcal{V}_{+}\right)  =\varphi\left(  \mathfrak{C\cap
}\mathcal{V}\right)  \subseteq\mathfrak{K\cap}\mathcal{W=W}_{+}$. Hence
$\varphi$ is positive. It is well known \cite[5.1.1]{ER} that $\varphi$ is
completely bounded. Moreover, $\left(  \varphi^{\ast}\right)  ^{\left(
\infty\right)  }\left(  S\left(  \mathfrak{K}\right)  \right)  \subseteq
\mathbb{R}_{+}S\left(  \mathfrak{C}\right)  $, where $\mathbb{R}_{+}S\left(
\mathfrak{C}\right)  $ indicates to the quantum set of all positive
functionals on the matrix spaces. Indeed,%
\[
\left\langle \mathfrak{C\cap}M_{n}\left(  \mathcal{V}\right)  ,\left(
\varphi^{\ast}\right)  ^{\left(  n\right)  }\left(  s\right)  \right\rangle
=\left\langle \varphi^{\left(  n\right)  }\left(  \mathfrak{C\cap}M_{n}\left(
\mathcal{V}\right)  \right)  ,s\right\rangle \subseteq\left\langle
\mathfrak{K\cap}M_{n}\left(  \mathcal{W}\right)  ,s\right\rangle \geq0
\]
for all $s\in S_{n}\left(  \mathfrak{K}\right)  $. If $\varphi$ is unital
(that is, $\varphi\left(  e\right)  =e^{\prime}$) then $\left\langle e^{\oplus
n},\left(  \varphi^{\ast}\right)  ^{\left(  n\right)  }\left(  s\right)
\right\rangle =\left\langle \varphi^{\left(  n\right)  }\left(  e^{\oplus
n}\right)  ,s\right\rangle =\left\langle e^{\prime\oplus n},s\right\rangle
=1$, which means $\left(  \varphi^{\ast}\right)  ^{\left(  n\right)  }\left(
s\right)  \in S_{n}\left(  \mathfrak{C}\right)  $ for every $s\in S_{n}\left(
\mathfrak{K}\right)  $. Thus $\left(  \varphi^{\ast}\right)  ^{\left(
\infty\right)  }\left(  S\left(  \mathfrak{K}\right)  \right)  \subseteq
S\left(  \mathfrak{C}\right)  $ whenever $\varphi$ is a morphism.

A linear mapping $\varphi:\mathcal{V}\rightarrow\mathcal{W}$ of operator
systems is called \textit{an entanglement breaking }if $\left(  \varphi^{\ast
}\right)  ^{\left(  \infty\right)  }\left(  S\left(  M\left(  \mathcal{W}%
\right)  _{+}\right)  \right)  \subseteq\mathbb{R}_{+}S\left(  \mathfrak{C}%
\right)  $ for every quantization $\mathfrak{C}$ of $\mathcal{V}_{+}$, where
$\varphi^{\ast}$ indicates to the algebraic dual mapping to $\varphi$. An
entanglement breaking mapping $\varphi:\mathcal{V}\rightarrow\mathcal{W}$ is
bounded automatically. Moreover, if $\varphi:\mathcal{V}\rightarrow\left(
\mathcal{W},\max\mathcal{W}_{+}\right)  $ is an entanglement breaking mapping
then so is $\varphi:\mathcal{V\rightarrow W}$. Indeed, by its very definition,
$\left(  \varphi^{\ast}\right)  ^{\left(  \infty\right)  }\left(  S\left(
\max\mathcal{W}_{+}\right)  \right)  \subseteq\mathbb{R}_{+}S\left(
\mathfrak{C}\right)  $ for every quantization $\mathfrak{C}$ of $\mathcal{V}%
_{+}$. But $\max\mathcal{W}_{+}\subseteq M\left(  \mathcal{W}\right)  _{+}$,
therefore $S\left(  M\left(  \mathcal{W}\right)  _{+}\right)  \subseteq
S\left(  \max\mathcal{W}_{+}\right)  $. In particular, $\left(  \varphi^{\ast
}\right)  ^{\left(  \infty\right)  }\left(  S\left(  M\left(  \mathcal{W}%
\right)  _{+}\right)  \right)  \subseteq\mathbb{R}_{+}S\left(  \mathfrak{C}%
\right)  $ for every quantization $\mathfrak{C}$ of $\mathcal{V}_{+}$, which
means that $\varphi:\mathcal{V\rightarrow W}$ is an entanglement breaking mapping.

Now assume that $\mathcal{H}$ is an operator Hilbert system with the unit $e$
and the related unital cone $\mathcal{H}_{+}$ (the notation instead of
$\mathfrak{c}_{e}$). In this case, $\mathcal{H}^{\ast}=\overline{\mathcal{H}}$
is an operator system with the unit $\overline{e}$ and the cone $\overline
{\mathcal{H}}_{+}$.

\begin{proposition}
Let $\mathcal{M}$ be either a finite-dimensional von Neumann algebra or
another operator Hilbert system, and let $\varphi:\mathcal{H\rightarrow M}$ be
a linear mapping. Then $\varphi$ is an entanglement breaking mapping iff
$\varphi^{\ast}:\mathcal{M}^{\ast}\rightarrow\left(  \overline{\mathcal{H}%
},\max\overline{\mathcal{H}}_{+}\right)  $ is matrix positive. Similarly,
$\varphi^{\ast}:\overline{\mathcal{H}}\rightarrow\mathcal{M}^{\ast}$ is an
entanglement breaking mapping iff $\varphi:\mathcal{M}\rightarrow\left(
\mathcal{H},\max\mathcal{H}_{+}\right)  $ is matrix positive.
\end{proposition}

\begin{proof}
For brevity we assume that $\mathcal{M}$ is a finite-dimensional von Neumann
algebra. The case of an operator Hilbert system $\mathcal{M}$ can be proved in
a very similar way. It is known (see \cite{PaulTT}) that $\varphi$ is an
entanglement breaking mapping iff $\varphi:\left(  \mathcal{H},\min
\mathcal{H}_{+}\right)  \rightarrow\mathcal{M}$ is matrix positive, that is,
$\varphi^{\left(  \infty\right)  }\left(  \min\mathcal{H}_{+}\right)
\subseteq M\left(  \mathcal{M}\right)  _{+}$. Using Lemma \ref{lDualM} and
Theorem \ref{tmax32}, we have
\[
\left(  \varphi^{\ast}\right)  ^{\left(  \infty\right)  }\left(  T\left(
\mathcal{M}_{\ast}\right)  _{+}\right)  =\left(  \varphi^{\ast}\right)
^{\left(  \infty\right)  }\left(  M\left(  \mathcal{M}\right)  _{+}^{\boxdot
}\right)  \subseteq\left(  \min\mathcal{H}_{+}\right)  ^{\boxdot}%
=\max\overline{\mathcal{H}}_{+},
\]
that is, $\varphi^{\ast}:\mathcal{M}^{\ast}\rightarrow\left(  \overline
{\mathcal{H}},\max\overline{\mathcal{H}}_{+}\right)  $ is matrix positive.
Conversely, if the latter mapping is matrix positive then $\varphi^{\left(
\infty\right)  }\left(  \left(  \max\overline{\mathcal{H}}_{+}\right)
^{\boxdot}\right)  \subseteq M\left(  \mathcal{M}^{\ast}\right)  _{+}%
^{\boxdot}$ thanks to Lemma \ref{lDualM}. But $\left(  \max\overline
{\mathcal{H}}_{+}\right)  ^{\boxdot}=\min\mathcal{H}_{+}$ and $M\left(
\mathcal{M}^{\ast}\right)  _{+}^{\boxdot}=M\left(  \mathcal{M}\right)  _{+}$
again by Theorem \ref{tmax32}. Hence $\varphi^{\left(  \infty\right)  }\left(
\min\mathcal{H}_{+}\right)  \subseteq M\left(  \mathcal{M}\right)  _{+}$,
which means that $\varphi:\left(  \mathcal{H},\min\mathcal{H}_{+}\right)
\rightarrow\mathcal{M}$ is matrix positive.

Finally, $\varphi^{\ast}:\left(  \overline{\mathcal{H}},\min\overline
{\mathcal{H}}_{+}\right)  \rightarrow\mathcal{M}^{\ast}$ is matrix positive
iff $\varphi:\mathcal{M}\rightarrow\left(  \mathcal{H},\max\mathcal{H}%
_{+}\right)  $ is matrix positive thanks to Lemma \ref{lDualM} and Theorem
\ref{tmax32}.
\end{proof}

\section{Appnedix: The unital measures on the state space\label{secAPPX}}

In this section we analyse so called measured state space of an ordered
Hilbert space $H$ to generate positive $L_{2}$-representations of $H$.

\subsection{The canonical $\ast$-representation $H\rightarrow C\left(
S\left(  \mathfrak{c}\right)  \right)  $\label{subsecHCS}}

Put $X=S\left(  \mathfrak{c}\right)  $ equipped with the weak$^{\ast}$
topology $\sigma\left(  \overline{H},H\right)  $. Thus $X$ is a compact
Hausdorff topological space and there is a canonical $\ast$-representation
$\kappa:H\rightarrow C\left(  X\right)  $, $\kappa\left(  \zeta\right)
=\zeta\left(  \cdot\right)  $, where $\zeta\left(  t\right)  =\left\langle
\zeta,t\right\rangle $, $t\in X$. If $\zeta\left(  \cdot\right)  =0$ then
$\left\langle \zeta,\overline{\mathfrak{c}}\right\rangle =\left\langle
\zeta,\mathbb{R}_{+}S\left(  \mathfrak{c}\right)  \right\rangle =\mathbb{R}%
_{+}\left\langle \zeta,X\right\rangle =\left\{  0\right\}  $ by Corollary
\ref{cormincbar}, and $\left\langle \zeta,\overline{H}_{h}\right\rangle
=\left\langle \zeta,\overline{\mathfrak{c}}\right\rangle -\left\langle
\zeta,\overline{\mathfrak{c}}\right\rangle =\left\{  0\right\}  $, which in
turn implies that $\left\langle \zeta,\overline{H}\right\rangle =\left\langle
\zeta,\overline{H}_{h}\right\rangle +i\left\langle \zeta,\overline{H}%
_{h}\right\rangle =\left\{  0\right\}  $, that is, $\zeta=0$. Moreover,
$\zeta\in\mathfrak{c}$ iff $\zeta\left(  \cdot\right)  \in C\left(  X\right)
_{+}$ (see Remark \ref{remDPC}), and $\kappa\left(  e\right)  =e\left(
\cdot\right)  =1$. Thus the unital $\ast$-representation $\kappa:H\rightarrow
C\left(  X\right)  $ is an order isomorphism onto its range, which means that
$H$ is realized as an operator system in $C\left(  X\right)  $, and $M\left(
H\right)  \cap M\left(  C\left(  X\right)  \right)  _{+}=M\left(  H\right)
\cap\min C\left(  X\right)  _{+}=\min\mathfrak{c}$ (see Proposition
\ref{propPTT1}, and \cite[Theorem 3.2]{PaulTT}). Further, $\left\Vert
\zeta\left(  \cdot\right)  \right\Vert _{\infty}=\sup\left\vert \zeta\left(
X\right)  \right\vert =\sup\left\vert \left\langle \zeta,S\left(
\mathfrak{c}\right)  \right\rangle \right\vert =\left\Vert \zeta\right\Vert
_{e}$ for all $\zeta\in H$. By Proposition \ref{propOKE21}, $H$ turns out to
be a complete subspace with respect to the uniform norm of $C\left(  X\right)
$, or $H$ is a norm-closed operator system in $C\left(  X\right)  $. Therefore
$\kappa^{\ast}$ is an exact quotient mapping
\[
\kappa^{\ast}:\mathcal{M}\left(  X\right)  \rightarrow\overline{H},\quad
\kappa^{\ast}\left(  \mu\right)  =\mu\cdot\kappa,\quad\operatorname{ball}%
\overline{H}=\kappa^{\ast}\left(  \operatorname{ball}\mathcal{M}\left(
X\right)  \right)  \text{,\quad}\ker\left(  \kappa^{\ast}\right)  =H^{\perp},
\]
where $H^{\perp}$ is the polar of $H$ in $C\left(  X\right)  ^{\ast}$. Note
that
\[
\left\langle \zeta,\kappa^{\ast}\left(  \mu^{\ast}\right)  \right\rangle
=\left\langle \kappa\left(  \zeta\right)  ,\mu^{\ast}\right\rangle
=\left\langle \zeta\left(  \cdot\right)  ^{\ast},\mu\right\rangle ^{\ast
}=\left\langle \zeta^{\ast}\left(  \cdot\right)  ,\mu\right\rangle ^{\ast
}=\left\langle \kappa\left(  \zeta^{\ast}\right)  ,\mu\right\rangle ^{\ast
}=\left\langle \zeta^{\ast},\kappa^{\ast}\left(  \mu\right)  \right\rangle
^{\ast}=\left\langle \zeta,\kappa^{\ast}\left(  \mu\right)  ^{\ast
}\right\rangle
\]
for all $\zeta\in H$, which means that $\kappa^{\ast}$ is $\ast$-linear. If
$\mu\in\mathcal{M}\left(  X\right)  _{+}$ then $\left\langle \zeta
,\kappa^{\ast}\left(  \mu\right)  \right\rangle =\left\langle \zeta\left(
\cdot\right)  ,\mu\right\rangle \geq0$ for all $\zeta\in\mathfrak{c}$, which
in turn implies that $\kappa^{\ast}\left(  \mu\right)  \in\overline
{\mathfrak{c}}$ thanks to Corollary \ref{cormincbar}. Thus $\kappa^{\ast}$ is
a positive mapping in the sense of $\kappa^{\ast}\left(  \mathcal{M}\left(
X\right)  _{+}\right)  \subseteq\overline{\mathfrak{c}}$. If $\mu
\in\mathcal{P}\left(  X\right)  $ then $\kappa^{\ast}\left(  \mu\right)  \in
S\left(  \mathfrak{c}\right)  $, that is, $\kappa^{\ast}\left(  \mu\right)
=s$ for some $s\in X$. We skip the bars in $\overline{s}$ for the elements of
$X$ for brevity, and we write $s=s_{0}+e\in X$ (instead of $\overline
{s}=\overline{s_{0}}+\overline{e}$) with uniquely defined $s_{0}%
\in\operatorname{ball}H_{h}^{e}$.

The closed subspace in $\mathcal{M}\left(  X\right)  $ of all atomic measures
on $X$ is denoted by $\ell^{1}\left(  X\right)  $. Take $\mu=\sum_{t\in
S}c_{t}\delta_{t}\in\ell^{1}\left(  X\right)  $ with $\sum_{t\in S}\left\vert
c_{t}\right\vert =\left\Vert \mu\right\Vert <\infty$ and a subset $S\subseteq
X$. Since $\sum_{t\in S}\left\Vert c_{t}t\right\Vert \leq\sqrt{2}\left\Vert
\mu\right\Vert <\infty$, it follows that $\overline{\eta}=\sum_{t\in S}c_{t}t$
defines an element of $\overline{H}$ with $\left\Vert \overline{\eta
}\right\Vert \leq\sqrt{2}\left\Vert \mu\right\Vert $ and
\[
\left\langle \zeta,\kappa^{\ast}\left(  \mu\right)  \right\rangle
=\left\langle \zeta\left(  \cdot\right)  ,\mu\right\rangle =\sum_{t\in S}%
c_{t}\left\langle \zeta\left(  \cdot\right)  ,\delta_{t}\right\rangle
=\sum_{t\in S}c_{t}\zeta\left(  t\right)  =\sum_{t\in S}c_{t}\left\langle
\zeta,t\right\rangle =\left\langle \zeta,\overline{\eta}\right\rangle
\]
for all $\zeta\in H$. Hence
\begin{equation}
\kappa^{\ast}\left(  \sum_{t\in S}c_{t}\delta_{t}\right)  =\sum_{t\in S}%
c_{t}t, \label{Atom11}%
\end{equation}
which means that $\kappa^{\ast}\left(  \ell^{1}\left(  X\right)  \right)
=\overline{H}$.

\begin{lemma}
\label{lemDL11}For each $\mu\in\mathcal{M}\left(  X\right)  $ there are points
$s,t\in X$, $c_{s},c_{t},c_{e}\in\mathbb{C}$ and $\nu\in H^{\perp}$ such that
\[
\mu=c_{s}\delta_{s}+c_{t}\delta_{t}+c_{e}\delta_{e}+\nu.
\]
If $\mu\in\mathcal{M}\left(  X\right)  _{h}$ then $\mu=c_{s}\delta_{s}%
+c_{e}\delta_{e}+\nu$ for some $s\in X$, $c_{s},c_{e}\in\mathbb{R}$ and
$\nu\in H^{\perp}\cap\mathcal{M}\left(  X\right)  _{h}$. If $\mu\in
\mathcal{P}\left(  X\right)  $ then $\mu=\delta_{s}+\nu$ for some $s\in X$ and
$\nu\in H^{\perp}\cap\mathcal{M}\left(  X\right)  _{h}$.
\end{lemma}

\begin{proof}
Put $\overline{\eta}=\kappa^{\ast}\left(  \mu\right)  $. If $\overline{\eta
}\in\overline{H}_{h}=\mathbb{R}_{+}\overline{\mathfrak{c}}-\mathbb{R}%
_{+}\overline{\mathfrak{c}}$ then $\overline{\eta}=c_{s}s-c_{t}t$ for some
$c_{s},c_{t}\in\mathbb{R}_{+}$ and $s,t\in S\left(  \mathfrak{c}\right)  $. It
follows that $\overline{\eta}=\kappa^{\ast}\left(  c_{s}\delta_{s}-c_{t}%
\delta_{t}\right)  $ thanks to (\ref{Atom11}). Actually,
\[
\overline{\eta}=\overline{\eta_{0}}+r\overline{e}=\left\Vert \eta
_{0}\right\Vert \left(  \left\Vert \eta_{0}\right\Vert ^{-1}\overline{\eta
_{0}}+\overline{e}\right)  +\left(  r-\left\Vert \eta_{0}\right\Vert \right)
\overline{e}=\kappa^{\ast}\left(  \left\Vert \eta_{0}\right\Vert \delta
_{s}+\left(  r-\left\Vert \eta_{0}\right\Vert \right)  \delta_{e}\right)  ,
\]
where $s$ represents the point $\left\Vert \eta_{0}\right\Vert ^{-1}%
\overline{\eta_{0}}+\overline{e}$ from $X$. In the case of any $\overline
{\eta}\in\overline{H}$ we have
\[
\overline{\eta}=\operatorname{Re}\overline{\eta}+i\operatorname{Im}%
\overline{\eta}=\kappa^{\ast}\left(  c_{s}\delta_{s}+c_{t}\delta_{t}%
+c_{e}\delta_{e}\right)
\]
for some $s,t\in X$ and $c_{s},c_{t},c_{e}\in\mathbb{C}$. Thus $\mu
=c_{s}\delta_{s}+c_{t}\delta_{t}+c_{e}\delta_{e}+\nu$ for some $\nu\in
H^{\perp}$. In the case of $\mu\in\mathcal{P}\left(  X\right)  $ we have
$\kappa^{\ast}\left(  \mu\right)  =s\in X$ and $\mu=\delta_{s}+\nu$ with
$\nu\in H^{\perp}$. But $\mu,\delta_{s}\in\mathcal{M}\left(  X\right)  _{h}$,
therefore $\nu\in H^{\perp}\cap\mathcal{M}\left(  X\right)  _{h}$.
\end{proof}

We say that $\mu$ is \textit{a unital measure on} $X$ if $\mu\in
\mathcal{P}\left(  X\right)  $ and $\kappa^{\ast}\left(  \mu\right)
=\overline{e}$. By Lemma \ref{lemDL11}, $\mu$ is unital iff $\mu=\delta
_{e}+\nu$ for some $\nu\in H^{\perp}\cap\mathcal{M}\left(  X\right)  _{h}$. In
particular, $\delta_{e}$ is a unital measure. An atomic probability measure
$\mu=\sum_{t\in S}c_{t}\delta_{t}\in\mathcal{P}\left(  X\right)  $ with
$c_{t}\geq0$ and $\sum_{t}c_{t}=1$ is unital iff $\sum_{t\in S}c_{t}t_{0}=0$
thanks to (\ref{Atom11}), where $t_{0}\in\operatorname{ball}H_{h}^{e}$ with
$t=t_{0}+e\in X$.

\begin{example}
\label{exmpUM32}If $F_{0}\subseteq\operatorname{ball}H_{h}^{e}$ is a finite
subset whose convex hull contains the origin, then $S=F_{0}+e\subseteq X$ is a
finite subset and $\mu=\sum_{t\in S}c_{t}\delta_{t}$ is a unital measure on
$X$ with the finite support, where $c_{t}\geq0$, $\sum_{t\in S}c_{t}=1$ and
$\sum_{t_{0}\in F_{0}}c_{t}t_{0}=0$ in $H_{h}^{e}$.
\end{example}

Notice that $X$ is identified with the subset $\delta_{X}=\left\{  \delta
_{t}:t\in X\right\}  \subseteq\ell^{1}\left(  X\right)  $ along with the
weak$^{\ast}$ continuous mapping $X\rightarrow\mathcal{M}\left(  X\right)  $,
$t\mapsto\delta_{t}$. Thus $\delta_{X}$ is a $w^{\ast}$-compact subset of
$\ell^{1}\left(  X\right)  $ being a homeomorphic copy of $X$. Further, the
mapping $\kappa^{\ast}:\mathcal{M}\left(  X\right)  \rightarrow\overline{H}$
implements a bijection of $\delta_{X}$ onto $X$, for the equality $\delta
_{s}=\delta_{t}$ over $H$ implies that $s$ and $t$ are the same states of the
cone $\mathfrak{c}$, that is, $s=t$ in $X$. Since $\left\vert \left\langle
\zeta,t\right\rangle \right\vert =\left\vert \left\langle \zeta\left(
\cdot\right)  ,\delta_{t}\right\rangle \right\vert $ for all $\zeta\in H$ and
$t\in X$, it follows that $\kappa^{\ast}|\delta_{X}:\delta_{X}\rightarrow X$
is a weak$^{\ast}$ continuous mapping of compact spaces. Thus $\kappa^{\ast
}|\delta_{X}$ is a homeomorphic inverse of the mapping $X\rightarrow\delta
_{X}$, $t\mapsto\delta_{t}$. Put $\widetilde{X}=\left(  \kappa^{\ast}\right)
^{-1}\left(  X\right)  $ to be a $w^{\ast}$-closed subset of $\mathcal{M}%
\left(  X\right)  $, which contains $\delta_{X}$. We say that $\widetilde{X}$
is \textit{the measured state space of the cone} $\mathfrak{c}$, and we also
use the notation $\widetilde{S}\left(  \mathfrak{c}\right)  $ instead of
$\widetilde{X}$. Taking into account that $\kappa^{\ast}$ is a $\ast$-linear
mapping, we conclude that $\widetilde{X}$ is a self-adjoint subset of
$\mathcal{M}\left(  X\right)  $ in the sense of $\widetilde{X}^{\ast
}=\widetilde{X}$, and $\widetilde{X}\cap H^{\perp}=\varnothing$. Thus
$\widetilde{S}\left(  \mathfrak{c}\right)  $ is a $w^{\ast}$-closed, convex,
$\ast$-subset of $\mathcal{M}\left(  X\right)  $ disjoint with $H^{\perp}$.

\begin{corollary}
\label{corDL1}The measured state space $\widetilde{S}\left(  \mathfrak{c}%
\right)  $ of the unital cone $\mathfrak{c}$ is the disjoint union of all
$\delta_{s}+H^{\perp}$, that is, $\widetilde{S}\left(  \mathfrak{c}\right)
=\bigvee\left\{  \delta_{s}+H^{\perp}:s\in S\left(  \mathfrak{c}\right)
\right\}  $. In particular, $\mathcal{P}\left(  X\right)  =\widetilde{X}%
\cap\mathcal{M}\left(  X\right)  _{+}$.
\end{corollary}

\begin{proof}
If $\delta_{s}-\delta_{t}\in H^{\perp}$ for some $s,t\in X$ then $s=t$ as we
have just confirmed above. Therefore the union $\cup\left\{  \delta
_{s}+H^{\perp}:s\in S\left(  \mathfrak{c}\right)  \right\}  $ is a disjoint
union which is $\widetilde{S}\left(  \mathfrak{c}\right)  $. Moreover,
$\mathcal{P}\left(  X\right)  \subseteq\widetilde{X}\cap\mathcal{M}\left(
X\right)  _{+}$ thanks to Lemma \ref{lemDL11}. Conversely, if $\mu=\delta
_{s}+\nu\in\mathcal{M}\left(  X\right)  _{+}$ with $\nu\in H^{\perp}$, then
$\int d\mu=\left\langle e\left(  \cdot\right)  ,\mu\right\rangle =\left\langle
e\left(  \cdot\right)  ,\delta_{s}\right\rangle +\left\langle e\left(
\cdot\right)  ,\nu\right\rangle =e\left(  s\right)  =1$, for $e\left(
\cdot\right)  \in H$ and $\left\langle e\left(  \cdot\right)  ,\nu
\right\rangle =0$. Whence $\mu\in\mathcal{P}\left(  X\right)  $.
\end{proof}

Finally, notice that $\operatorname{Re}\mu\in\widetilde{X}$ whenever $\mu
\in\widetilde{X}$. Indeed, by Corollary \ref{corDL1}, we have $\mu=\delta
_{s}+\nu$ for some $\nu\in H^{\perp}$. Then $\operatorname{Re}\mu=\delta
_{s}+\operatorname{Re}\nu$. But $H^{\perp}$ is a $\ast$-subspace of
$\mathcal{M}\left(  X\right)  $, therefore $\operatorname{Re}\nu\in H^{\perp}%
$, and $\operatorname{Re}\mu\in\widetilde{X}$. Note also that
$\operatorname{Im}\mu=\operatorname{Im}\nu$ being an element of $H^{\perp}$
stays out of $\widetilde{X}$. Similarly, in the general case the positive part
$\mu_{+}$ of a hermitian $\mu\in\widetilde{X}$ may stay out of $\widetilde{X}%
$. The set of all unital measures on $X$ is denoted by $\mathcal{U}\left(
X\right)  $, that is, $\mathcal{U}\left(  X\right)  =\mathcal{P}\left(
X\right)  \cap\left(  \delta_{e}+H^{\perp}\right)  $ is a convex subset of
$\mathcal{P}\left(  X\right)  $. Notice that $\mathcal{U}\left(  X\right)
=\mathcal{M}\left(  X\right)  _{+}\cap\left(  \delta_{e}+H^{\perp}\right)
=\mathcal{M}\left(  X\right)  _{+}\cap\left(  \delta_{e}+H^{\perp}%
\cap\mathcal{M}\left(  X\right)  _{h}\right)  $ thanks to Corollary
\ref{corDL1}.

\subsection{The unital measures on $X$}

Fix $\mu\in\mathcal{U}\left(  X\right)  $. Put $u=e\left(  \cdot\right)  $,
which is a unit of $L^{1}\left(  X,\mu\right)  $ and it represents $\mu$ in
$\mathcal{M}\left(  X\right)  $. Consider the related Hilbert space
$L^{2}\left(  X,\mu\right)  $ (which is a subspace of $L^{1}\left(
X,\mu\right)  $ out of compactness of $X$) with its norm $\left\Vert
\cdot\right\Vert _{2}$, the unital cone $L^{2}\left(  X,\mu\right)  _{+}$ with
the unit $u$, and the unital $\ast$-linear mapping $\iota:C\left(  X\right)
\rightarrow L^{2}\left(  X,\mu\right)  $. The latter in turn defines the
following bounded, unital $\ast$-linear mapping $\iota\kappa:H\rightarrow
L^{2}\left(  X,\mu\right)  $ of Hilbert spaces. If $\eta\in L^{2}\left(
X,\mu\right)  $ then $\iota^{\ast}\left(  \overline{\eta}\right)  \in C\left(
X\right)  ^{\ast}=\mathcal{M}\left(  X\right)  $, $\eta\mu\in\mathcal{M}%
\left(  X\right)  $ and%
\[
\left\langle h,\iota^{\ast}\left(  \overline{\eta}\right)  \right\rangle
=\left\langle \iota\left(  h\right)  ,\overline{\eta}\right\rangle =\left(
\iota\left(  h\right)  ,\eta\right)  =\int h\left(  t\right)  \eta^{\ast
}\left(  t\right)  d\mu=\left\langle h,\eta^{\ast}\mu\right\rangle
\]
for all $h\in C\left(  X\right)  $, where the inner product is taken in
$L^{2}\left(  X,\mu\right)  $. Thus $\iota^{\ast}:L^{2}\left(  X,\mu\right)
^{\ast}\rightarrow\mathcal{M}\left(  X\right)  $ is reduced to the canonical
identification $\iota^{\ast}\left(  \overline{\eta}\right)  =\eta^{\ast}\mu$.
In particular, $\left(  \iota\kappa\right)  ^{\ast}\left(  \overline{\eta
}\right)  =\kappa^{\ast}\iota^{\ast}\left(  \overline{\eta}\right)
=\kappa^{\ast}\left(  \eta^{\ast}\mu\right)  $ for all $\eta\in L^{2}\left(
X,\mu\right)  $, which justifies to use a brief notation $\iota:H\rightarrow
L^{2}\left(  X,\mu\right)  $ instead of $\iota\kappa$. In this case,
$\iota\left(  e\right)  =u$, and its dual is reduced to $\kappa^{\ast}%
|L^{2}\left(  X,\mu\right)  $ for the exact quotient mapping $\kappa^{\ast
}:\mathcal{M}\left(  X\right)  \rightarrow\overline{H}$ considered above in
Subsection \ref{subsecHCS}.

Now let $\iota\left(  H\right)  ^{-}$ be the closure of the subspace
$\iota\left(  H\right)  $ in the Hilbert space $L^{2}\left(  X,\mu\right)  $,
and let $P\in\mathcal{B}\left(  L^{2}\left(  X,\mu\right)  \right)  $ be the
orthogonal projection onto $\iota\left(  H\right)  ^{-}$. Since $L^{2}\left(
X,\mu\right)  =\iota\left(  H\right)  ^{-}\oplus\iota\left(  H\right)
^{\perp}$ and $\iota$ is a $\ast$-linear mapping, it follows that both
$\iota\left(  H\right)  ^{-}$ and $\iota\left(  H\right)  ^{\perp}$ are $\ast
$-subspaces, and $u\in\iota\left(  H\right)  \subseteq\operatorname{im}\left(
P\right)  $. In particular, $P$ is a unital $\ast$-linear mapping. If $\eta
\in\iota\left(  H\right)  ^{\perp}$ for $\eta\in L^{2}\left(  X,\mu\right)  $,
then $\eta\mu\in\mathcal{M}\left(  X\right)  $ and $\left\langle \zeta\left(
\cdot\right)  ,\eta\mu\right\rangle =\left(  \zeta\left(  \cdot\right)
,\eta\right)  =\left(  \iota\left(  \zeta\right)  ,\eta\right)  =0$ for all
$\zeta\in H$, which means that $\eta\mu\in H^{\perp}$. Hence $L^{2}\left(
X,\mu\right)  \cap H^{\perp}=\iota\left(  H\right)  ^{\perp}$ (up to the
canonical identification). The orthogonal to $u$ subspace of
$\operatorname{im}\left(  P\right)  $ is denoted by $\operatorname{im}\left(
P\right)  ^{u}$ whereas $\operatorname{im}\left(  P\right)  _{h}^{u}$ denotes
the closed real subspace $\operatorname{im}\left(  P\right)  ^{u}\cap
L^{2}\left(  X,\mu\right)  _{h}$.

\begin{lemma}
\label{lemUTM11}The unital $\ast$-linear mapping $\iota:H\rightarrow
L^{2}\left(  X,\mu\right)  $ is a Hilbert-Schmidt operator with $\left\Vert
\iota\right\Vert _{2}\leq\sqrt{2}$ and $\iota\left(  \operatorname{ball}%
H^{e}\right)  \subseteq\operatorname{ball}\operatorname{im}\left(  P\right)
^{u}$, whose dual $\iota^{\ast}:L^{2}\left(  X,\mu\right)  \rightarrow H$ is a
unital $\ast$-linear mapping given by the following $\overline{H}$-valued
integral
\[
\iota^{\ast}\left(  \eta\right)  =\int\eta\left(  t\right)  td\mu,\quad\eta\in
L^{2}\left(  X,\mu\right)  .
\]
In particular, $\iota\left(  \operatorname{ball}H_{h}^{e}\right)
\subseteq\operatorname{ball}\operatorname{im}\left(  P\right)  _{h}^{u}$,
$\operatorname{im}\left(  P\right)  _{h}^{u}=\iota\left(  H_{h}^{e}\right)
^{-}$, and $\int td\mu=e$.
\end{lemma}

\begin{proof}
The fact that $\iota$ is a Hilbert-Schmidt operator follows from Remark
\ref{remIT11}. In the present case, if $F$ is a (hermitian) Hilbert basis for
$H$ containing $e$, then $\iota=\sum_{f}\iota\left(  f\right)  \odot
\overline{f}$ with
\begin{align*}
\left\Vert \iota\right\Vert _{2}^{2}  &  =\sum_{f}\left\Vert \iota\left(
f\right)  \right\Vert _{2}^{2}=\sum_{f\neq e}\int\left\vert \left(
f,t\right)  \right\vert ^{2}d\mu+\int\left\vert \left(  e,t\right)
\right\vert ^{2}d\mu=\int\sum_{f\neq e}\left\vert \left(  f,t\right)
\right\vert ^{2}d\mu+1\\
&  \leq\int\left\Vert t_{0}\right\Vert ^{2}d\mu+1\leq2,
\end{align*}
which means that $\left\Vert \iota\right\Vert _{2}\leq\sqrt{2}$. Note that
$\iota^{\ast}\left(  u\right)  =\kappa^{\ast}\left(  u^{\ast}\mu\right)
=\kappa^{\ast}\left(  \mu\right)  =e$, that is, $\iota^{\ast}$ is a unital
mapping. Further,
\begin{align*}
\left\langle \zeta,\iota^{\ast}\left(  \eta\right)  \right\rangle  &
=\left\langle \zeta,\kappa^{\ast}\left(  \eta^{\ast}\mu\right)  \right\rangle
=\left\langle \zeta\left(  \cdot\right)  ,\eta^{\ast}\mu\right\rangle =\left(
\zeta\left(  \cdot\right)  ,\eta\right)  =\int\left(  \zeta,t\right)
\eta^{\ast}\left(  t\right)  d\mu=\int\left(  \zeta,\eta\left(  t\right)
t\right)  d\mu\\
&  =\left(  \zeta,\int\eta\left(  t\right)  td\mu\right)
\end{align*}
for all $\zeta\in H$. It follows that $\iota^{\ast}\left(  \eta\right)
=\int\eta\left(  t\right)  td\mu$ for all $\eta\in L^{2}\left(  X,\mu\right)
$. In particular,
\[
\iota^{\ast}\left(  \eta^{\ast}\right)  =\int\eta^{\ast}\left(  t\right)
td\mu=\left(  \int\eta\left(  t\right)  td\mu\right)  ^{\ast}=\iota^{\ast
}\left(  \eta\right)  ^{\ast}%
\]
for all $\eta\in L^{2}\left(  X,\mu\right)  $, which means that $\iota^{\ast}$
is a unital $\ast$-linear mapping, and $\int td\mu=\iota^{\ast}\left(
u\right)  =e$.

Now prove that $\iota\left(  \operatorname{ball}H^{e}\right)  \subseteq
\operatorname{ball}\operatorname{im}\left(  P\right)  ^{u}$. Take $\zeta
_{0}\in H^{e}$ and a (hermitian) Hilbert basis $F$ for $H$ containing $e$.
Since $\mu$ is unital, we conclude that%
\[
\left(  \iota\left(  \zeta_{0}\right)  ,u\right)  =\int\zeta_{0}\left(
t\right)  d\mu=\left\langle \zeta_{0}\left(  \cdot\right)  ,\mu\right\rangle
=\left\langle \zeta_{0}\left(  \cdot\right)  ,\delta_{e}\right\rangle =\left(
\zeta_{0},e\right)  =0,
\]
that is, $\iota\left(  \zeta_{0}\right)  \perp u$. Thus $\iota\left(
H^{e}\right)  \subseteq\operatorname{im}\left(  P\right)  ^{u}$ and
$\iota\left(  F\backslash\left\{  e\right\}  \right)  \subseteq
\operatorname{im}\left(  P\right)  _{h}^{u}$. If $\zeta_{0}\in
\operatorname{ball}H^{e}$ then $\zeta_{0}=\sum_{f\neq e}\left(  \zeta
_{0},f\right)  f$ and
\begin{align*}
\left\Vert \iota\left(  \zeta_{0}\right)  \right\Vert _{2}  &  \leq\sum_{f\neq
e}\left\vert \left(  \zeta_{0},f\right)  \right\vert \left\Vert \iota\left(
f\right)  \right\Vert _{2}\leq\left(  \sum_{f\neq e}\left\vert \left(
\zeta_{0},f\right)  \right\vert ^{2}\right)  ^{1/2}\left(  \sum_{f\neq
e}\left\Vert \iota\left(  f\right)  \right\Vert _{2}^{2}\right)  ^{1/2}\\
&  \leq\left(  \sum_{f\neq e}\left\vert \left(  \zeta_{0},f\right)
\right\vert ^{2}\right)  ^{1/2}\left(  \int\left\Vert t_{0}\right\Vert
^{2}d\mu\right)  ^{1/2}\leq\left\Vert \zeta_{0}\right\Vert \leq1,
\end{align*}
that is, $\iota\left(  \zeta_{0}\right)  \in\operatorname{ball}%
\operatorname{im}\left(  P\right)  ^{u}$. Since $\iota$ is $\ast$-linear, we
also deduce that $\iota\left(  \operatorname{ball}H_{h}^{e}\right)
\subseteq\operatorname{ball}\operatorname{im}\left(  P\right)  _{h}^{u}$. In
particular, $\iota\left(  H_{h}^{e}\right)  ^{-}\subseteq\operatorname{im}%
\left(  P\right)  _{h}^{u}$.

Finally, prove that $\operatorname{im}\left(  P\right)  _{h}^{u}=\iota\left(
H_{h}^{e}\right)  ^{-}$. Take $\theta\in\operatorname{im}\left(  P\right)
_{h}^{u}$. Since $\iota\left(  H\right)  $ is dense in $\operatorname{im}%
\left(  P\right)  $ and $\iota$ is a $\ast$-linear mapping, it follows that
$\theta=\lim_{n}\zeta_{n}\left(  \cdot\right)  $ for a certain sequence
$\left(  \zeta_{n}\right)  _{n}\subseteq H_{h}$. But $\zeta_{n}\left(
\cdot\right)  =\zeta_{n,0}\left(  \cdot\right)  +r_{n}u$ with $\zeta_{n,0}\in
H_{h}^{e}$, $r_{n}\in\mathbb{R}$, and $\lim_{n}r_{n}=\lim_{n}\left(
\zeta_{n,0}\left(  \cdot\right)  ,u\right)  +r_{n}\left(  u,u\right)
=\lim_{n}\left(  \zeta_{n}\left(  \cdot\right)  ,u\right)  =\left(
\theta,u\right)  =0$. Thereby $\theta=\lim_{n}\zeta_{n,0}\left(  \cdot\right)
\in\iota\left(  H_{h}^{e}\right)  ^{-}$.
\end{proof}

\begin{remark}
Notice also that $\left\Vert \iota^{\ast}\left(  \eta\right)  \right\Vert
\leq\int\left\vert \eta\right\vert \left\Vert t\right\Vert d\mu\leq\sqrt
{2}\left\Vert \eta\right\Vert _{1}\leq\sqrt{2}\left\Vert \eta\right\Vert _{2}$
for all $\eta\in L^{2}\left(  X,\mu\right)  $.
\end{remark}

Put $\mathfrak{c}_{\mu}=\left\{  \zeta\in H_{h}:\left\Vert \zeta_{0}\left(
\cdot\right)  \right\Vert _{2}\leq\left(  \zeta,e\right)  \right\}  $, which
is a cone in $H$.

\begin{lemma}
\label{lemUTM2}The cone $\mathfrak{c}_{\mu}$ in $H$ is a unital cone
containing $\mathfrak{c}$, $\iota\left(  \mathfrak{c}_{\mu}\right)  \subseteq
L^{2}\left(  X,\mu\right)  _{+}$ and $\operatorname{im}\left(  P\right)  \cap
L^{2}\left(  X,\mu\right)  _{+}=\iota\left(  \mathfrak{c}_{\mu}\right)  ^{-}$.
In particular, $\iota:H\rightarrow L^{2}\left(  X,\mu\right)  $ is a unital
positive mapping in the sense of $\iota\left(  e\right)  =u$ and $\iota\left(
\mathfrak{c}\right)  \subseteq L^{2}\left(  X,\mu\right)  _{+}$, and $P$ is a
unital positive projection (a conditional expectation).
\end{lemma}

\begin{proof}
Take $\zeta=\zeta_{0}+re\in\mathfrak{c}$ with $\zeta_{0}\in H_{h}^{e}$ and
$\left\Vert \zeta_{0}\right\Vert \leq r$. Note that $\iota\left(
\zeta\right)  =\iota\left(  \zeta_{0}\right)  +ru$, $\iota\left(  \zeta
_{0}\right)  =\zeta_{0}\left(  \cdot\right)  \in\operatorname{im}\left(
P\right)  _{h}^{u}$ and $\left\Vert \zeta_{0}\left(  \cdot\right)  \right\Vert
_{2}\leq\left\Vert \zeta_{0}\right\Vert \leq r$ by virtue of Lemma
\ref{lemUTM11}, that is, $\zeta\in\mathfrak{c}_{\mu}$. Thus
$\mathfrak{c\subseteq c}_{\mu}$ and $\mathfrak{c}_{\mu}$ is a unital cone in
$H$. Prove that $\iota\left(  \mathfrak{c}_{\mu}\right)  \subseteq
L^{2}\left(  X,\mu\right)  _{+}$. Take $\zeta=\zeta_{0}+re\in\mathfrak{c}%
_{\mu}$. By Lemma \ref{lemUTM11}, $\iota\left(  \zeta_{0}\right)  \perp u$ and
$\left\Vert \zeta_{0}\left(  \cdot\right)  \right\Vert _{2}\leq\left(
\zeta,e\right)  =\left\langle \zeta\left(  \cdot\right)  ,\delta
_{e}\right\rangle =\left\langle \zeta\left(  \cdot\right)  ,\mu\right\rangle
=\left(  \zeta\left(  \cdot\right)  ,u\right)  $, which means that
$\iota\left(  \zeta\right)  \in L^{2}\left(  X,\mu\right)  _{+}$. Thus
$\iota\left(  \mathfrak{c}\right)  \subseteq\iota\left(  \mathfrak{c}_{\mu
}\right)  \subseteq L^{2}\left(  X,\mu\right)  _{+}$ and $\iota$ is a unital
positive mapping. If $\eta=\eta_{0}+ru\in L^{2}\left(  X,\mu\right)  _{+}$
then $P\left(  u\right)  =u$, $P\left(  \eta\right)  =P\left(  \eta
_{0}\right)  +ru\in L^{2}\left(  X,\mu\right)  _{h}$, $\left(  P\left(
\eta_{0}\right)  ,u\right)  =\left(  \eta_{0},u\right)  =0$ in $L^{2}\left(
X,\mu\right)  $, and $\left\Vert P\left(  \eta_{0}\right)  \right\Vert
_{2}\leq\left\Vert \eta_{0}\right\Vert _{2}\leq r$, which means that $P$ is a
unital positive projection.

Finally prove that $\operatorname{im}\left(  P\right)  \cap L^{2}\left(
X,\mu\right)  _{+}=\iota\left(  \mathfrak{c}_{\mu}\right)  ^{-}$. We saw above
$\iota\left(  \mathfrak{c}_{\mu}\right)  \subseteq L^{2}\left(  X,\mu\right)
_{+}\cap\operatorname{im}\left(  P\right)  $, which results in $\iota\left(
\mathfrak{c}_{\mu}\right)  ^{-}\subseteq L^{2}\left(  X,\mu\right)  _{+}%
\cap\operatorname{im}\left(  P\right)  $. Take $\eta=\eta_{0}+ru\in
\operatorname{im}\left(  P\right)  \cap L^{2}\left(  X,\mu\right)  _{+}$ with
$\left\Vert \eta_{0}\right\Vert _{2}\leq r$. Prove that $\eta\in\iota\left(
\mathfrak{c}_{\mu}\right)  ^{-}$. Since $\eta=\lim_{n}\left(  1-n^{-1}\right)
\eta_{0}+ru$, we can assume that $\left\Vert \eta_{0}\right\Vert _{2}<r$. By
Lemma \ref{lemUTM11}, $\eta_{0}\in\operatorname{im}\left(  P\right)  _{h}%
^{u}=\iota\left(  H_{h}^{e}\right)  ^{-}$, therefore $\eta=\lim_{n}\zeta
_{0,n}\left(  \cdot\right)  +ru$ for some $\left(  \zeta_{0,n}\right)
_{n}\subseteq H_{h}^{e}$. But $\lim_{n}\left\Vert \zeta_{0,n}\left(
\cdot\right)  \right\Vert _{2}=\left\Vert \eta_{0}\right\Vert _{2}<r$,
therefore we can assume that $\left\Vert \zeta_{0,n}\left(  \cdot\right)
\right\Vert _{2}<r$ for all $n$. Thus $\zeta_{n}=\zeta_{0,n}+re\in
\mathfrak{c}_{\mu}$ and $\eta=\lim_{n}\zeta_{n}\left(  \cdot\right)  =\lim
_{n}\iota\left(  \zeta_{n}\right)  \in\iota\left(  \mathfrak{c}_{\mu}\right)
^{-}$.
\end{proof}

The description of the state space of the cone $L^{2}\left(  X,\mu\right)
_{+}$ in terms of the measured state space $\widetilde{S}\left(
\mathfrak{c}\right)  $ considered above in Subsection \ref{subsecHCS} is
provided in the following assertion.

\begin{theorem}
\label{thUTM1}If $\mu\in\mathcal{U}\left(  X\right)  $ then $S\left(
L^{2}\left(  X,\mu\right)  _{+}\right)  =\widetilde{X}\cap\sqrt{2}%
\operatorname{ball}L^{2}\left(  X,\mu\right)  _{h}$. In particular,
$\iota^{\ast}\left(  L^{2}\left(  X,\mu\right)  _{+}\right)  \subseteq
\mathfrak{c}$, which means that $\iota^{\ast}:L^{2}\left(  X,\mu\right)
\rightarrow H$ is a unital positive mapping.
\end{theorem}

\begin{proof}
By Lemma \ref{lemUTM2}, $\iota:H\rightarrow L^{2}\left(  X,\mu\right)  $ is a
unital positive mapping. This fact in turn implies that $\iota^{\ast}\left(
S\left(  L^{2}\left(  X,\mu\right)  _{+}\right)  \right)  \subseteq X$. Let us
show the details of this inclusion. Take a state $\overline{\eta}%
=\overline{\eta_{0}}+\overline{u}\in S\left(  L^{2}\left(  X,\mu\right)
_{+}\right)  $ of the unital cone $L^{2}\left(  X,\mu\right)  _{+}$, where
$\eta_{0}\in\operatorname{ball}$ $L^{2}\left(  X,\mu\right)  _{h}^{u}$. Notice
that $\left\Vert \eta\right\Vert _{2}=\left(  \left\Vert \eta_{0}\right\Vert
_{2}^{2}+\left\Vert u\right\Vert _{2}^{2}\right)  ^{1/2}\leq\sqrt{2}$.
Further, $\iota^{\ast}\left(  \overline{\eta}\right)  =\kappa^{\ast}\left(
\eta\mu\right)  =\left(  \left(  \eta_{0}+u\right)  \mu\right)  |H$, which
means that
\[
\left\langle \zeta,\iota^{\ast}\left(  \eta\right)  \right\rangle
=\left\langle \zeta\left(  \cdot\right)  ,\left(  \eta_{0}+u\right)
\mu\right\rangle =\left\langle \zeta\left(  \cdot\right)  ,\eta_{0}%
\mu\right\rangle +\left\langle \zeta\left(  \cdot\right)  ,\mu\right\rangle
=\left\langle \zeta\left(  \cdot\right)  ,\eta_{0}\mu\right\rangle +\left(
\zeta,e\right)
\]
for all $\zeta\in H$. But $\eta_{0}\mu|H$ is a hermitian linear functional on
$H$ such that
\[
\left\Vert \eta_{0}\mu|H\right\Vert \leq\left\Vert \eta_{0}\mu\right\Vert
=\left\Vert \eta_{0}\right\Vert _{1}=\int\left\vert \eta_{0}\right\vert
d\mu\leq\left(  \int d\mu\right)  ^{1/2}\left\Vert \eta_{0}\right\Vert
_{2}\leq1.
\]
It follows that the functional $\eta_{0}\mu|H$ is given by a hermitian vector
$s_{0}\in H_{h}$ such that $\left\langle \zeta\left(  \cdot\right)  ,\eta
_{0}\mu\right\rangle =\left(  \zeta,s_{0}\right)  $. But $\left(
e,s_{0}\right)  =\left\langle e\left(  \cdot\right)  ,\eta_{0}\mu\right\rangle
=\int\eta_{0}d\mu=\left(  u,\eta_{0}\right)  =0$, that is, $s_{0}%
\in\operatorname{ball}H_{h}^{e}$. Consequently, $s_{\eta}=\overline{s_{0}%
}+\overline{e}\in X$ and $\left\langle \zeta,\iota^{\ast}\left(  \eta\right)
\right\rangle =\left\langle \zeta,s_{\eta}\right\rangle $ for all $\zeta\in
H$, which means that $s_{\eta}=\iota^{\ast}\left(  \overline{\eta}\right)  \in
X$. In particular, $S\left(  L^{2}\left(  X,\mu\right)  _{+}\right)
\subseteq\widetilde{X}\cap\sqrt{2}\operatorname{ball}L^{2}\left(
X,\mu\right)  _{h}$.

Conversely, suppose $\iota^{\ast}\left(  \overline{\eta}\right)  =s\in X$ for
some $\eta\in L^{2}\left(  X,\mu\right)  _{h}$ with $\left\Vert \eta
\right\Vert _{2}\leq\sqrt{2}$. Prove that $\overline{\eta}\in S\left(
L^{2}\left(  X,\mu\right)  _{+}\right)  $. Taking into account that
$\ker\left(  \iota^{\ast}\right)  =\iota\left(  H\right)  ^{\perp}$, we can
also assume that $\eta\in\iota\left(  H\right)  _{h}^{-}$. Then $\eta=\lim
_{n}\iota\left(  \zeta_{0,n}\right)  +r_{n}u$ for some $\zeta_{0,n}\in
H_{h}^{e}$ and $r_{n}\in\mathbb{R}$. By Lemma \ref{lemUTM11}, $\left\{
\iota\left(  \zeta_{0,n}\right)  \right\}  \subseteq\iota\left(  H_{h}%
^{e}\right)  \subseteq\operatorname{im}\left(  P\right)  _{h}^{u}$, and
\begin{align*}
\lim_{n}r_{n}  &  =\lim_{n}\left(  \iota\left(  \zeta_{0,n}\right)
+r_{n}u,u\right)  =\left(  \eta,u\right)  =\int\eta d\mu=\left\langle e\left(
\cdot\right)  ,\eta\mu\right\rangle \\
&  =\left\langle \iota\left(  e\right)  ,\eta\mu\right\rangle =\left\langle
e,\kappa^{\ast}\left(  \eta\mu\right)  \right\rangle =\left\langle
e,\iota^{\ast}\left(  \overline{\eta}\right)  \right\rangle =\left\langle
e,s\right\rangle =\left(  e,s\right)  =1.
\end{align*}
It follows that $\eta=\eta_{0}+u$ with $\eta_{0}=\lim_{n}\iota\left(
\zeta_{0,n}\right)  \in\iota\left(  H_{h}^{e}\right)  ^{-}=\operatorname{im}%
\left(  P\right)  _{h}^{u}$ (see Lemma \ref{lemUTM11}), and $\left\Vert
\eta_{0}\right\Vert _{2}^{2}+1=\left\Vert \eta\right\Vert _{2}^{2}\leq2$,
which means that $\eta_{0}\in\operatorname{ball}$ $L^{2}\left(  X,\mu\right)
_{h}^{u}$. Whence $\overline{\eta}=\overline{\eta_{0}}+\overline{u}\in
S\left(  L^{2}\left(  X,\mu\right)  _{+}\right)  $ and $s=s_{\eta}$. Hence
$S\left(  L^{2}\left(  X,\mu\right)  _{+}\right)  =\widetilde{X}\cap\sqrt
{2}\operatorname{ball}L^{2}\left(  X,\mu\right)  _{h}$.

Finally, prove that $\iota^{\ast}:L^{2}\left(  X,\mu\right)  \rightarrow H$
(or $\iota^{\ast}:\overline{L^{2}\left(  X,\mu\right)  }\rightarrow
\overline{H}$) is positive either. Since $\iota^{\ast}$ is unital and
$S\left(  L^{2}\left(  X,\mu\right)  _{+}\right)  \subseteq\widetilde{X}%
=\left(  \iota^{\ast}\right)  ^{-1}\left(  X\right)  $, it follows that
\[
\iota^{\ast}\left(  \overline{L^{2}\left(  X,\mu\right)  }_{+}\right)
=\iota^{\ast}\left(  \mathbb{R}_{+}S\left(  L^{2}\left(  X,\mu\right)
_{+}\right)  \right)  \subseteq\mathbb{R}_{+}\iota^{\ast}\left(  S\left(
L^{2}\left(  X,\mu\right)  _{+}\right)  \right)  \subseteq\mathbb{R}%
_{+}X=\overline{\mathfrak{c}}%
\]
or equivalently we have $\iota^{\ast}\left(  L^{2}\left(  X,\mu\right)
_{+}\right)  \subseteq\mathfrak{c}$, which means that $\iota^{\ast}$ is positive.
\end{proof}

\begin{remark}
\label{remUTM1}As follows from the proof of Theorem \ref{thUTM1}, if
$\iota^{\ast}\left(  \eta\right)  =s\in X$ for some $\eta\in L^{2}\left(
X,\mu\right)  _{h}$, then $\eta=\eta_{0}+u$ with $\eta_{0}\in\operatorname{im}%
\left(  P\right)  _{h}^{u}$. In particular, $\iota^{\ast}\left(  \eta
_{0}\right)  =s_{0}$ for $s=s_{0}+e$ with $s_{0}\in\operatorname{ball}%
H_{h}^{e}$.
\end{remark}

Recall that a point $s\in X$ is called a $\mu$-mass if $\mu\left(  s\right)
>0$ (see Subsection \ref{subsecARM}).

\begin{corollary}
\label{corUTM1}Let $\mu\in\mathcal{U}\left(  X\right)  $ and let $s$ be a
$\mu$-mass in $X$ with $\mu\left(  s\right)  \geq1/2$. Then $s$ is given by a
certain state $\overline{\eta}$ of $L^{2}\left(  X,\mu\right)  _{+}$, that is,
$s=\iota^{\ast}\left(  \overline{\eta}\right)  $ for $\overline{\eta}\in
S\left(  L^{2}\left(  X,\mu\right)  _{+}\right)  $.
\end{corollary}

\begin{proof}
By Lemma \ref{lematom1}, $\delta_{s}=s^{\prime}\mu$ for $s^{\prime}=\mu\left(
s\right)  ^{-1}\left[  s\right]  \in L^{2}\left(  X,\mu\right)  _{h}$. Then
$s=\kappa^{\ast}\left(  s^{\prime}\mu\right)  $, which means that $s^{\prime
}\mu\in\widetilde{X}$. But $\left\Vert s^{\prime}\right\Vert _{2}=\mu\left(
s\right)  ^{-1/2}\leq\sqrt{2}$, that is, $s^{\prime}\mu\in\widetilde{X}%
\cap\sqrt{2}\operatorname{ball}L^{2}\left(  X,\mu\right)  _{h}$. By Theorem
\ref{thUTM1}, $s^{\prime}\mu\in S\left(  L^{2}\left(  X,\mu\right)
_{+}\right)  $, and the result follows.
\end{proof}

In the case of any $\mu$-mass $s$ in $X$ we have $s^{\prime}=\mu\left(
s\right)  ^{-1}\left[  s\right]  \in L^{2}\left(  X,\mu\right)  _{h}$,
$\left(  s^{\prime},u\right)  =\int s^{\prime}d\mu=1$, and $s^{\prime
}=Ps^{\prime}+\left(  1-P\right)  s^{\prime}=s_{0}^{\prime}+u+\left(
1-P\right)  s^{\prime}$ with $s_{0}^{\prime}\in\operatorname{im}\left(
P\right)  _{h}^{u}$. It follows that $s=\kappa^{\ast}\left(  s^{\prime}%
\mu\right)  =\iota^{\ast}\left(  s^{\prime}\right)  =\iota^{\ast}\left(
s_{0}^{\prime}\right)  +\iota^{\ast}\left(  u\right)  =\iota^{\ast}\left(
s_{0}^{\prime}\right)  +e$ (see Remark \ref{remUTM1}), which in turn implies
that $\iota^{\ast}\left(  s_{0}^{\prime}\right)  =s_{0}\in\operatorname{ball}%
H_{h}^{e}$.

\begin{corollary}
\label{corUTM21}Let $\mu\in\mathcal{U}\left(  X\right)  $ and let $s$ be a
$\mu$-mass in $X$. Then
\[
\mu\left(  s\right)  ^{1/2}\left\Vert s_{0}\right\Vert ^{2}\leq\left\Vert
s_{0}\left(  \cdot\right)  \right\Vert _{2}\leq\left\Vert s_{0}\right\Vert
\leq\left\Vert s_{0}^{\prime}\right\Vert _{2}\leq\left(  \mu\left(  s\right)
^{-1}-1\right)  ^{1/2}.
\]
In particular, $\mu\left(  s\right)  \leq\left(  1+\left\Vert s_{0}\right\Vert
^{2}\right)  ^{-1}$, and $\mu\left(  s\right)  \leq1/2$ whenever $\left\Vert
s_{0}\right\Vert =1$.
\end{corollary}

\begin{proof}
As in the proof of Corollary \ref{corUTM1}, we have $\mu\left(  s\right)
^{-1}=\left\Vert s_{0}^{\prime}\right\Vert _{2}^{2}+1+\left\Vert \left(
1-P\right)  s^{\prime}\right\Vert _{2}^{2}\geq\left\Vert s_{0}^{\prime
}\right\Vert _{2}^{2}+1$, that is, $\left\Vert s_{0}^{\prime}\right\Vert
_{2}\leq\left(  \mu\left(  s\right)  ^{-1}-1\right)  ^{1/2}$. Further, notice
that $s_{0}^{\prime}\in\operatorname{im}\left(  P\right)  _{h}^{u}$,
$s_{0}=\iota^{\ast}\left(  s_{0}^{\prime}\right)  \in H_{h}^{e}$, and
$\iota\left(  \operatorname{ball}H_{h}^{e}\right)  \subseteq
\operatorname{ball}\operatorname{im}\left(  P\right)  _{h}^{u}$ thanks to
Lemma \ref{lemUTM11}. It follows that
\[
\left\Vert s_{0}\right\Vert =\sup\left\vert \left(  \operatorname{ball}%
H_{h}^{e},\iota^{\ast}\left(  s_{0}^{\prime}\right)  \right)  \right\vert
=\sup\left\vert \left(  \iota\left(  \operatorname{ball}H_{h}^{e}\right)
,s_{0}^{\prime}\right)  \right\vert \leq\sup\left\vert \left(
\operatorname{ball}\operatorname{im}\left(  P\right)  _{h}^{u},s_{0}^{\prime
}\right)  \right\vert =\left\Vert s_{0}^{\prime}\right\Vert _{2},
\]
that is, $\left\Vert s_{0}\right\Vert \leq\left\Vert s_{0}^{\prime}\right\Vert
_{2}$. Using again Lemma \ref{lemUTM11}, we deduce that
\[
\mu\left(  s\right)  ^{1/2}\left\Vert s_{0}\right\Vert ^{2}=\left(  \mu\left(
s\right)  \left(  s_{0},s_{0}\right)  ^{2}\right)  ^{1/2}\leq\left(
\int\left\vert \left(  s_{0},t\right)  \right\vert ^{2}d\mu\right)
^{1/2}=\left\Vert s_{0}\left(  \cdot\right)  \right\Vert _{2}\leq\left\Vert
s_{0}\right\Vert .
\]
Hence $\mu\left(  s\right)  ^{1/2}\left\Vert s_{0}\right\Vert ^{2}%
\leq\left\Vert s_{0}\left(  \cdot\right)  \right\Vert _{2}\leq\left\Vert
s_{0}\right\Vert \leq\left\Vert s_{0}^{\prime}\right\Vert _{2}\leq\left(
\mu\left(  s\right)  ^{-1}-1\right)  ^{1/2}$. In particular, $\left\Vert
s_{0}\right\Vert ^{2}+1\leq\mu\left(  s\right)  ^{-1}$ or $\mu\left(
s\right)  \leq\left(  1+\left\Vert s_{0}\right\Vert ^{2}\right)  ^{-1}$.
\end{proof}

\begin{corollary}
\label{corUTM2}For every $\lambda\in\widetilde{X}$ there corresponds $\mu
\in\mathcal{U}\left(  X\right)  $ such that $\lambda\in\sqrt{2}%
\operatorname{ball}L^{2}\left(  X,\mu\right)  _{h}$ modulo $H^{\perp}$.
\end{corollary}

\begin{proof}
The assertion is trivial for $\dim\left(  H\right)  \leq1$. Suppose that
$\dim\left(  H\right)  \geq2$. Take $\lambda\in\widetilde{X}$ with
$s=\kappa^{\ast}\left(  \lambda\right)  \in X$. Notice that $s=s_{0}+e$ for
$s_{0}\in\operatorname{ball}\overline{H}_{h}^{e}$. Put $s-=-s_{0}+e$, which is
another point of $X$. Then $\mu=2^{-1}\delta_{s}+2^{-1}\delta_{s-}$ is a
unital measure on $X$ (see Example \ref{exmpUM32}). Moreover, $\mu\left(
s\right)  =1/2$. By Corollary \ref{corUTM1}, $s=\iota^{\ast}\left(
\overline{\eta}\right)  $ for a certain $\overline{\eta}\in S\left(
L^{2}\left(  X,\mu\right)  _{+}\right)  $. Then $\eta\mu\in\widetilde{X}%
\cap\sqrt{2}\operatorname{ball}L^{2}\left(  X,\mu\right)  _{h}$ thanks to
Theorem \ref{thUTM1}, and $\kappa^{\ast}\left(  \eta\mu\right)  =s$. Whence
$\kappa^{\ast}\left(  \lambda-\eta\mu\right)  =0$ or $\lambda-\eta\mu\in
H^{\perp}$.
\end{proof}

\subsection{$\mu$-concentration sets}

As above for every $s\in X$ we use the notation $s=\overline{s_{0}}%
+\overline{e}$ with $s_{0}\in\operatorname{ball}\overline{H}_{h}^{e}$. Let
$S\subseteq X$ be a subset. By \textit{probabilistic mass on }$S$ we mean a
summable function $m:S\rightarrow\mathbb{R}_{+}$ such that $\sum_{s\in
S}m\left(  s\right)  \leq1$ and
\[
\left\Vert \sum_{s\in S}m\left(  s\right)  s_{0}\right\Vert \leq1-\sum_{s\in
S}m\left(  s\right)  .
\]
In this case, we say that $S$ is \textit{a concentration set} \textit{with a
mass }$m$. Note that $\sum_{s\in S}m\left(  s\right)  s_{0}$ converges
absolutely, for $\sum_{s\in S}\left\Vert m\left(  s\right)  s_{0}\right\Vert
\leq\sum_{s\in S}m\left(  s\right)  \leq1$. The function $m=0$ is a mass on
each subset $S$ automatically. Basically, we deal with a nontrivial mass $m$
on $S$, in this case, we say that $m$ is a positive mass on $S$.

\begin{lemma}
\label{lemConcM1}A subset $S\subseteq X$ is a concentration set with a
positive mass $m$ iff there is a unital measure $\mu$ on $X$ such that
$\mu:S\rightarrow\mathbb{R}_{+}$ is a nozero function. In this case, $\mu$
defines a mass on $S$ and $\mu\geq m$.
\end{lemma}

\begin{proof}
First assume that there is $\mu\in\mathcal{P}\left(  X\right)  $ such that
$\mu:S\rightarrow\mathbb{R}_{+}$, $s\mapsto\mu\left(  s\right)  $ is a
nontrivial function. Then $\left\{  \mu\left(  s\right)  \delta_{s}:s\in
S\right\}  $ is a summable family of measures on $X$ and $\mu=\sum_{s\in S}%
\mu\left(  s\right)  \delta_{s}+\nu$ for some $\nu\in\mathcal{M}\left(
X\right)  _{+}$ \cite[Ch. 5, 5.10, Proposition 15]{BourInt}. Certainly,
$\sum_{s\in S}\mu\left(  s\right)  =\sup\left\{  \mu\left(  F\right)
:F\subseteq S\right\}  \leq\mu\left(  X\right)  =1$, where $F$ is running over
all finite subsets of $S$. But
\begin{align*}
\left(  \zeta,e\right)   &  =\left\langle \zeta,\overline{e}\right\rangle
=\left\langle \zeta\left(  \cdot\right)  ,\mu\right\rangle =\left\langle
\zeta\left(  \cdot\right)  ,\sum_{s\in S}\mu\left(  s\right)  \delta
_{s}\right\rangle +\left\langle \zeta\left(  \cdot\right)  ,\nu\right\rangle
=\sum_{s\in S}\mu\left(  s\right)  \zeta\left(  s\right)  +\left\langle
\zeta\left(  \cdot\right)  ,\nu\right\rangle \\
&  =\sum_{s\in S}\mu\left(  s\right)  \left(  \zeta,s_{0}\right)  +\sum_{s\in
S}\mu\left(  s\right)  \left(  \zeta,e\right)  +\left\langle \zeta\left(
\cdot\right)  ,\nu\right\rangle ,
\end{align*}
which in turn implies that $\left\langle \zeta,\overline{\eta}\right\rangle
=\left(  \zeta,\eta\right)  =\left\langle \zeta\left(  \cdot\right)
,\nu\right\rangle $ for $\eta=-\sum_{s\in S}\mu\left(  s\right)  s_{0}+\left(
1-\sum_{s\in S}\mu\left(  s\right)  \right)  e\in H$. It follows that
$\iota^{\ast}\left(  \nu\right)  =\overline{\eta}$. But $\nu\geq0$, therefore
$\overline{\eta}\in\overline{\mathfrak{c}}$ or $\eta\in\mathfrak{c}$. Since
$-\sum_{s\in S}\mu\left(  s\right)  s_{0}\in H_{h}^{e}$, we deduce that
$\left\Vert \sum_{s\in S}\mu\left(  s\right)  s_{0}\right\Vert \leq
1-\sum_{s\in S}\mu\left(  s\right)  $. The latter means that $\mu$ is a mass
on $S$.

Conversely, suppose that $m$ is a nonzero mass on $S$. Then
\[
\eta=-\sum_{s\in S}m\left(  s\right)  s_{0}+\left(  1-\sum_{s\in S}m\left(
s\right)  \right)  e\in\mathfrak{c}%
\]
and $\overline{\eta}$ defines a positive functional, which in turn admits an
extension up to a positive measure $\nu$ on $X$. Put $\mu=\sum_{s\in
S}m\left(  s\right)  \delta_{s}+\nu\in\mathcal{M}\left(  X\right)  _{+}$.
Then
\[
\left\langle \zeta\left(  \cdot\right)  ,\mu\right\rangle =\sum_{s\in
S}m\left(  s\right)  \left(  \zeta,s\right)  +\left(  \zeta,\eta\right)
=\left(  \zeta,e\right)  =\left\langle \zeta\left(  \cdot\right)
,\overline{e}\right\rangle
\]
for all $\zeta\in H$. In particular, $\mu\left(  X\right)  =\left\langle
e\left(  \cdot\right)  ,\mu\right\rangle =\left(  e,e\right)  =1$, which means
that $\mu$ is a unital measure on $X$. Finally, $\mu\left(  s\right)
=m\left(  s\right)  +\nu\left(  s\right)  \geq m\left(  s\right)  $ for all
$s\in S$ (see \cite[Ch. 5, 3.5, Corollary 1]{BourInt}).
\end{proof}

Notice that if $m=0$ the assertion of Lemma \ref{lemConcM1} follows with any
unital measure $\mu$ on $X$.

\subsection{The exact and finite $H$-measures on $X$}

A unital measure $\mu$ on the state space $X$ is said to be \textit{a finite
}$H$\textit{-measure} if $\dim\iota\left(  H\right)  <\infty$. Notice that the
latter is equivalent to the fact that $\iota\left(  H\right)  ^{u}%
=\operatorname{im}\left(  P\right)  ^{u}$ and $\dim\iota\left(  H\right)
^{u}<\infty$. For example, if $\mu$ is a unital atomic measure with its finite
support then it is a finite $H$-measure on $X$. By Lemma \ref{lemUTM11},
$\iota\left(  \operatorname{ball}H_{h}^{e}\right)  ^{-}\subseteq
\operatorname{ball}\operatorname{im}\left(  P\right)  _{h}^{u}$ for every
unital measure $\mu$ on $X$. If the inclusion $\operatorname{ball}%
\operatorname{im}\left(  P\right)  _{h}^{u}\subseteq\iota\left(
\operatorname{ball}H_{h}^{e}\right)  ^{-}$ (or the equality $\iota\left(
\operatorname{ball}H_{h}^{e}\right)  ^{-}=\operatorname{ball}\operatorname{im}%
\left(  P\right)  _{h}^{u}$) holds for $\mu$, we say that $\mu$ is \textit{an
exact }$H$-\textit{measure on }$X$.

\begin{proposition}
\label{propEXM1}An exact $H$-measure $\mu$ on $X$ is a finite $H$-measure
automatically. In this case, we have $\operatorname{ball}\operatorname{im}%
\left(  P\right)  _{h}=\iota\left(  \operatorname{ball}H_{h}\right)  ^{-}$,
$\operatorname{im}\left(  P\right)  \cap L^{2}\left(  X,\mu\right)  _{+}%
=\iota\left(  \mathfrak{c}\right)  ^{-}$ and $\left\Vert \iota^{\ast}\left(
\eta_{0}\right)  \right\Vert =\left\Vert \eta_{0}\right\Vert _{2}$ for every
$\eta_{0}\in\operatorname{im}\left(  P\right)  _{h}^{u}$. In particular,
$\left\Vert s_{0}\right\Vert =\left\Vert s_{0}^{\prime}\right\Vert _{2}$
whenever $s=s_{0}+e$ is a $\mu$-mass in $X$.
\end{proposition}

\begin{proof}
Suppose that $\operatorname{ball}\operatorname{im}\left(  P\right)  _{h}%
^{u}\subseteq\iota\left(  \operatorname{ball}H_{h}^{e}\right)  ^{-}$. Then
$\iota:H_{h}^{e}\rightarrow\operatorname{im}\left(  P\right)  _{h}^{u}$ turns
out to be an open mapping thanks to the Open Mapping Theorem. But it is a
compact operator by virtue of Lemma \ref{lemUTM11}. It follows that
$\dim\operatorname{im}\left(  P\right)  _{h}^{u}<\infty$, which in turn
implies that $\dim\operatorname{im}\left(  P\right)  <\infty$. Thus
$\operatorname{im}\left(  P\right)  =\iota\left(  H\right)  ^{-}=\iota\left(
H\right)  $, which means that $\mu$ is a finite $H$-measure.

Now prove that $\operatorname{ball}\operatorname{im}\left(  P\right)
_{h}\subseteq\iota\left(  \operatorname{ball}H_{h}\right)  ^{-}$ whenever
$\mu$ is an exact $H$-measure. Take $\theta\in\operatorname{im}\left(
P\right)  _{h}$ with $\left\Vert \theta\right\Vert _{2}<1$. Then $\theta
=\zeta\left(  \cdot\right)  $ for some $\zeta\in H_{h}$. But $\zeta=\zeta
_{0}+re$, $\zeta_{0}\in H_{h}^{e}$. Then $\theta=\zeta_{0}\left(
\cdot\right)  +ru$, $\zeta_{0}\left(  \cdot\right)  \in\iota\left(  H_{h}%
^{e}\right)  \subseteq\operatorname{im}\left(  P\right)  _{h}^{u}$ (see Lemma
\ref{lemUTM11}) and $\left\Vert \zeta_{0}\left(  \cdot\right)  \right\Vert
_{2}^{2}+r^{2}=\left\Vert \theta\right\Vert _{2}^{2}<1$, that is, $\zeta
_{0}\left(  \cdot\right)  \in\sqrt{1-r^{2}}\operatorname{ball}%
\operatorname{im}\left(  P\right)  _{h}^{u}=\sqrt{1-r^{2}}\iota\left(
\operatorname{ball}H_{h}^{e}\right)  ^{-}$. In particular, $\zeta_{0}\left(
\cdot\right)  =\lim_{n}\zeta_{0,n}\left(  \cdot\right)  $ for some sequence
$\left(  \zeta_{0,n}\right)  _{n}\subseteq\sqrt{1-r^{2}}\operatorname{ball}%
H_{h}^{e}$. It follows that $\theta=\lim_{n}\zeta_{n}\left(  \cdot\right)  $
with $\zeta_{n}\left(  \cdot\right)  =\zeta_{0,n}\left(  \cdot\right)
+ru\in\iota\left(  \operatorname{ball}H_{h}\right)  $, that is, $\theta
\in\iota\left(  \operatorname{ball}H_{h}\right)  ^{-}$.

Further prove that $\operatorname{im}\left(  P\right)  \cap L^{2}\left(
X,\mu\right)  _{+}=\iota\left(  \mathfrak{c}\right)  ^{-}$. Based on Lemma
\ref{lemUTM2}, it suffices to prove that $\zeta\left(  \cdot\right)  \in
\iota\left(  \mathfrak{c}\right)  ^{-}$ for every $\zeta\in\mathfrak{c}_{\mu}%
$. Take $\zeta=\zeta_{0}+re\in\mathfrak{c}_{\mu}$ with $\zeta_{0}\in H_{h}%
^{e}$, $\left\Vert \zeta_{0}\left(  \cdot\right)  \right\Vert _{2}\leq r$.
Since $r^{-1}\zeta_{0}\left(  \cdot\right)  \in\operatorname{ball}\iota\left(
H_{h}^{e}\right)  \subseteq\iota\left(  \operatorname{ball}H_{h}^{e}\right)
^{-}$, it follows that $r^{-1}\zeta_{0}\left(  \cdot\right)  =\lim_{n}%
\zeta_{0,n}\left(  \cdot\right)  $ for a sequence $\left(  \zeta_{0,n}\right)
_{n}\subseteq\operatorname{ball}H_{h}^{e}$. Then $\zeta_{n}=r\zeta_{0,n}%
+re\in\mathfrak{c}$ and $\zeta\left(  \cdot\right)  =\lim_{n}\zeta_{n}\left(
\cdot\right)  \in\iota\left(  \mathfrak{c}\right)  ^{-}$ in $L^{2}\left(
X,\mu\right)  $.

Finally, take $\eta_{0}\in\operatorname{im}\left(  P\right)  _{h}^{u}$. Taking
into account that $\operatorname{ball}\operatorname{im}\left(  P\right)
_{h}^{u}=\iota\left(  \operatorname{ball}H_{h}^{e}\right)  ^{-}$ and
$\iota^{\ast}\left(  \eta_{0}\right)  \in H_{h}^{e}$, we deduce that
\[
\left\Vert \eta_{0}\right\Vert _{2}=\sup\left\vert \left(  \iota\left(
\operatorname{ball}H_{h}^{e}\right)  ,\eta_{0}\right)  \right\vert
=\sup\left\vert \left(  \operatorname{ball}H_{h}^{e},\iota^{\ast}\left(
\eta_{0}\right)  \right)  \right\vert =\left\Vert \iota^{\ast}\left(  \eta
_{0}\right)  \right\Vert ,
\]
that is, $\left\Vert \iota^{\ast}\left(  \eta_{0}\right)  \right\Vert
=\left\Vert \eta_{0}\right\Vert _{2}$. If $s=s_{0}+e$ is a $\mu$-mass in $X$
then $s_{0}^{\prime}\in\operatorname{im}\left(  P\right)  _{h}^{u}$ and
$\iota^{\ast}\left(  s_{0}^{\prime}\right)  =s_{0}$. Whence $\left\Vert
s_{0}\right\Vert =\left\Vert s_{0}^{\prime}\right\Vert _{2}$.
\end{proof}

Now take $s\in X$ such that $s_{0}\neq0$. As in the proof of Corollary
\ref{corUTM2}, we use the notation $s-=-s_{0}+e$ for the symmetric opposite of
$s$ in $X$. Notice that the unital measure $\mu_{s}=2^{-1}\delta_{s}%
+2^{-1}\delta_{s-}$ on $X$ is a finite $H$-measure, and the mapping
$U_{s}:L^{2}\left(  X,\mu_{s}\right)  \rightarrow\mathbb{C}^{2}$, $U_{s}%
\eta=\left(  \eta\left(  s\right)  /\sqrt{2},\eta\left(  s-\right)  /\sqrt
{2}\right)  $ implements a unitary equivalence of the Hilbert spaces. Indeed,
\[
\left\Vert U_{s}\eta\right\Vert =\left(  \left\vert \eta\left(  s\right)
\right\vert ^{2}/2+\left\vert \eta\left(  s-\right)  \right\vert
^{2}/2\right)  ^{1/2}=\left(  \int\left\vert \eta\right\vert ^{2}d\mu
_{s}\right)  ^{1/2}=\left\Vert \eta\right\Vert _{2}%
\]
for all $\eta\in L^{2}\left(  X,\mu_{s}\right)  $. Moreover, $U_{s}u=\left(
1/\sqrt{2},1/\sqrt{2}\right)  =e_{s}$ and $L^{2}\left(  X,\mu_{s}\right)
^{u}=\mathbb{C}U_{s}^{\ast}f_{s}$ for $f_{s}=\left(  1/\sqrt{2},-1/\sqrt
{2}\right)  \in\mathbb{C}^{2}$. Thus $\left(  \mathbb{C}^{2},e_{s}\right)
=\ell_{2}\left(  2\right)  $ is a unital Hilbert space equipped with the
unital cone%
\[
\ell_{2}\left(  2\right)  _{+}=\left\{  \lambda f_{s}+re_{s}:\lambda
,r\in\mathbb{R},\left\vert \lambda\right\vert \leq r\right\}  .
\]
In particular, $\iota:H\rightarrow L^{2}\left(  X,\mu_{s}\right)  $ is reduced
to the following mapping $\iota_{s}:H\rightarrow\ell_{2}\left(  2\right)  $,
$\iota_{s}\left(  \zeta\right)  =\left(  \left(  \zeta,s\right)  /\sqrt
{2},\left(  \zeta,s-\right)  /\sqrt{2}\right)  $. If $\zeta_{0}\in H_{h}^{e}$
then $\iota_{s}\left(  \zeta_{0}\right)  =\left(  \left(  \zeta_{0}%
,s_{0}\right)  /\sqrt{2},-\left(  \zeta_{0},s_{0}\right)  /\sqrt{2}\right)
=\left(  \zeta_{0},s_{0}\right)  f_{s}$. In particular, $\iota_{s}\left(
s_{0}\right)  =f_{s}$, which in turn implies that $P$ is the identity
projection and%
\[
\iota_{s}\left(  \operatorname{ball}H_{h}^{e}\right)  =\left\{  \left(
\zeta_{0},s_{0}\right)  f_{s}:\left\Vert \zeta_{0}\right\Vert \leq1\right\}
=\left\{  \left(  rs_{0},s_{0}\right)  f_{s}:\left\vert r\right\vert
\leq1\right\}  =\left\{  rf_{s}:\left\vert r\right\vert \leq1\right\}
=\operatorname{ball}\left(  \mathbb{C}^{2}\right)  _{h}^{e_{s}},
\]
which means that $\mu_{s}$ is an exact $H$-measure on $X$. Notice that $s$ is
a $\mu$-mass with $\mu\left(  s\right)  =1/2$, $s^{\prime}=2\left[  s\right]
$ and $U_{s}s^{\prime}=\left(  2/\sqrt{2},0\right)  =e_{s}+f_{s}$, which in
turn implies that $\iota_{s}\left(  s_{0}\right)  =f_{s}=s_{0}^{\prime}$ (see
Proposition \ref{propEXM1}). If $\zeta=\zeta_{0}+re\in\mathfrak{c}$ then
$\iota\left(  \zeta\right)  =\left(  \zeta_{0},s_{0}\right)  f_{s}+re_{s}$ and
$\left\vert \left(  \zeta_{0},s_{0}\right)  \right\vert \leq\left\Vert
\zeta_{0}\right\Vert \leq r$, that is, $\iota_{s}\left(  \zeta\right)  \in
\ell_{2}\left(  2\right)  _{+}$. Conversely, if $\eta=\lambda f_{s}+re_{s}%
\in\ell_{2}\left(  2\right)  _{+}$, $\lambda,r\in\mathbb{R}$, $\left\vert
\lambda\right\vert \leq r$ then $\zeta=\lambda s_{0}+re\in\mathfrak{c}$ and
$\iota_{s}\left(  \zeta\right)  =\left(  \lambda s_{0},s_{0}\right)
f_{s}+re_{s}=\zeta$, that is, $\iota_{s}\left(  \mathfrak{c}\right)  =\ell
_{2}\left(  2\right)  _{+}$.

\subsection{The factorization problem}

As above consider the canonical $\ast$-representation $H\rightarrow C\left(
X\right)  $, $\zeta\mapsto\zeta\left(  \cdot\right)  $ from Subsection
\ref{subsecHCS}, where $X=S\left(  \mathfrak{c}_{e}\right)  $, and fix $\mu
\in\mathcal{U}\left(  X\right)  $. By Lemma \ref{lemUTM2}, $\iota:H\rightarrow
L^{2}\left(  X,\mu\right)  $ is a unital positive mapping, that is,
$\iota\left(  e\right)  =u$ and $\iota\left(  \mathfrak{c}_{e}\right)
\subseteq L^{2}\left(  X,\mu\right)  _{+}$. Based on Proposition
\ref{propKom1}, we conclude that $\iota=\iota_{\sigma}$ for a certain unital
$L^{2}\left(  X,\mu\right)  $-support $\sigma=\left\{  \sigma_{\chi}:\chi\in
B\right\}  \subseteq H_{h}$, where $B$ is a (hermitian) Hilbert basis for
$L^{2}\left(  X,\mu\right)  $ containing $u$, $\sigma_{u}=e$, $\sigma_{\chi
}\perp e$ for all $\chi\neq u$, and $\sum_{\chi\neq u}\left(  \zeta_{0}%
,\sigma_{\chi}\right)  ^{2}\leq\left\Vert \zeta_{0}\right\Vert ^{2}$ for all
$\zeta_{0}\in H_{h}^{e}$ (see Remark \ref{remUHS}). Thus $\left\{
\sigma_{\chi}:\chi\neq u\right\}  \subseteq H_{h}^{e}$, and
\[
\left(  \zeta,\sigma_{\chi}\right)  =\left(  \zeta\left(  \cdot\right)
,\chi\right)  =\int\zeta\left(  t\right)  \chi\left(  t\right)  d\mu=\left(
\zeta\left(  \cdot\right)  ,P\chi\right)
\]
for all $\zeta\in H$ and $\chi\neq u$. Thus we can assume that $B$ is a basis
for $\operatorname{im}P$, and $\left(  \zeta,\sigma_{\chi}\right)  =\left(
\zeta\left(  \cdot\right)  ,\chi\right)  $ for all $\zeta\in H$ and $\chi\neq
u$. Taking into account that $\iota\in\mathcal{B}^{2}\left(  H,L^{2}\left(
X,\mu\right)  \right)  $ (see Lemma \ref{lemUTM11}), we conclude that $\sigma$
is of type $2$, that is, $\sum_{\chi}\left\Vert \sigma_{\chi}\right\Vert
^{2}=\sum_{\chi}\left\Vert \iota^{\ast}\chi\right\Vert ^{2}\leq\left\Vert
\iota^{\ast}\right\Vert _{2}\leq\sqrt{2}$ and $\left\Vert \sigma_{\chi
}\right\Vert =\left\Vert \iota^{\ast}\chi\right\Vert \leq\left\Vert
\chi\right\Vert \leq1$ for all $\chi\in B\backslash\left\{  u\right\}  $, that
is, $\left\{  \sigma_{\chi}:\chi\neq u\right\}  \subseteq\operatorname{ball}%
H_{h}^{e}$. Put $s_{\chi}=\sigma_{\chi}+e$, $\chi\neq u$ and $s_{u}=e$. Then
$S=\left\{  s_{\chi}:\chi\in B\right\}  $ is a subset of $X$ containing $e$,
and $\left(  \zeta,s_{\chi}\right)  =\left(  \zeta,\sigma_{\chi}\right)
+\left(  \zeta,e\right)  =\left(  \zeta\left(  \cdot\right)  ,\chi\right)
+\left(  \zeta\left(  \cdot\right)  ,u\right)  =\left(  \zeta\left(
\cdot\right)  ,\chi+u\right)  $ for all $\zeta\in H$. Thus $B+u=\left\{
\chi+u\right\}  \subseteq\widetilde{X}\cap\sqrt{2}\operatorname{ball}%
L^{2}\left(  X,\mu\right)  _{h}=S\left(  L^{2}\left(  X,\mu\right)
_{+}\right)  $ and $\iota^{\ast}\left(  B+u\right)  =S$ thanks to Theorem
\ref{thUTM1}, and $\sum_{s\neq e}\left\langle \zeta_{0}\left(  \cdot\right)
,\delta_{s}\right\rangle ^{2}\leq\left\Vert \zeta_{0}\right\Vert ^{2}$ for all
$\zeta_{0}\in H_{h}^{e}$.

\subsection{$L_{2}$-factorization}

Now let $\left(  K,u\right)  $ be a unital Hilbert space, $X=S\left(
\mathfrak{c}_{u}\right)  $ with the canonical $\ast$-representation
$K\rightarrow C\left(  X\right)  $, $\eta\mapsto\eta\left(  \cdot\right)  $,
and fix $\mu\in\mathcal{U}\left(  X\right)  $. As above there is a unital
$L^{2}\left(  X,\mu\right)  $-support $\sigma=\left\{  \sigma_{\chi}:\chi\in
B\right\}  \subseteq K_{h}$ such that $\iota=\iota_{\sigma}:K\rightarrow
L^{2}\left(  X,\mu\right)  $ is a unital positive mapping. Put $K_{\mu
}=\operatorname{im}\left(  \iota^{\ast}\right)  _{h}$, which is a real
subspace in $K_{h}$, and $\sigma\subseteq K_{\mu}$. Since $\iota^{\ast}%
=\iota^{\ast}P$, we conclude that $K_{\mu}=\iota^{\ast}\left(
\operatorname{im}\left(  P\right)  _{h}\right)  $ and $\iota^{\ast
}:\operatorname{im}\left(  P\right)  _{h}\rightarrow K_{\mu}$ is injective. In
particular, for every $\eta\in K_{\mu}$ there corresponds a unique
$\eta^{\prime}\in\operatorname{im}\left(  P\right)  _{h}$ such that
$\eta=\iota^{\ast}\left(  \eta^{\prime}\right)  $. Put $\left\Vert
\eta\right\Vert _{\mu}=\left\Vert \eta^{\prime}\right\Vert _{2}$, which
defines a norm on $K_{\mu}$ such that $\left\Vert \eta\right\Vert
\leq\left\Vert \iota^{\ast}\right\Vert \left\Vert \eta\right\Vert _{\mu}$ for
all $\eta\in K_{\mu}$. If $\mu$ is a finite $K$-measure on $X$ (see Subsection
\ref{subsecHCS}), then $B$ is a finite hermitian basis for $\operatorname{im}%
\left(  P\right)  $ and the $L^{2}\left(  X,\mu\right)  $-support $\sigma$ is
finite, which in turn implies that $K_{\mu}$ is a finite dimensional real
subspace of $K_{h}$.

\begin{lemma}
\label{lemL111}Take $\eta\in K_{h}$. Then $\eta\in K_{\mu}$ iff $\eta
=\sum_{\chi}\alpha_{\chi}\sigma_{\chi}$ is a sum of an absolutely convergent
series in $K$ for a certain $\alpha=\left(  \alpha_{\chi}\right)  _{\chi}%
\in\ell^{2}\left(  B\right)  _{h}$. In this case, $\eta^{\prime}=\sum_{\chi
}\alpha_{\chi}\chi$ and $\left\Vert \eta\right\Vert _{\mu}=\left\Vert
\alpha\right\Vert _{2}$.
\end{lemma}

\begin{proof}
Note that $\eta\in K_{\mu}$ iff $\eta=\iota^{\ast}\left(  \eta^{\prime
}\right)  $ for some $\eta^{\prime}\in\operatorname{im}\left(  P\right)  _{h}%
$. But $\eta^{\prime}=\sum_{\chi}\alpha_{\chi}\chi$ has a unique expansion
through the basis $B$ such that $\left\Vert \eta\right\Vert _{\mu}=\left\Vert
\eta^{\prime}\right\Vert _{2}=\left\Vert \alpha\right\Vert _{2}$, where
$\alpha=\left(  \alpha_{\chi}\right)  _{\chi}\in\ell^{2}\left(  B\right)
_{h}$. It follows that $\eta=\sum_{\chi}\alpha_{\chi}\sigma_{\chi}$ and%
\begin{align*}
\sum_{\chi}\left\Vert \alpha_{\chi}\sigma_{\chi}\right\Vert  &  =\sum_{\chi
}\left\vert \alpha_{\chi}\right\vert \left\Vert \sigma_{\chi}\right\Vert
\leq\left(  \sum_{\chi}\left\vert \alpha_{\chi}\right\vert ^{2}\right)
^{1/2}\left(  \sum_{\chi}\left\Vert \sigma_{\chi}\right\Vert ^{2}\right)
^{1/2}=\left\Vert \alpha\right\Vert _{2}\left(  \sum_{\chi}\left\Vert
\iota^{\ast}\chi\right\Vert ^{2}\right)  ^{1/2}\\
&  \leq\left\Vert \alpha\right\Vert _{2}\left\Vert \iota\right\Vert
_{2}<\infty,
\end{align*}
which means that $\eta$ is a sum of an absolutely convergent series in $K$.
\end{proof}

Let $T:\left(  K,u\right)  \rightarrow\left(  H,e\right)  $ be a unital
positive mapping given by a unital $H$-support $k\subseteq K_{h}$. We say that
$T$ is $L^{2}\left(  X,\mu\right)  $\textit{-factorable mapping} if $T=S\iota$
for a certain $S\in\mathcal{B}\left(  L^{2}\left(  X,\mu\right)  ,H\right)  $.
Taking into account $\iota\in\mathcal{B}^{2}\left(  K,L^{2}\left(
X,\mu\right)  \right)  $, we conclude that $T=S\iota\in\mathcal{B}^{2}\left(
K,H\right)  $. Based on Remark \ref{remHSop}, we obtain that $k$ is of type
$2$ automatically.

\begin{lemma}
\label{lemL12}If $T:\left(  K,u\right)  \rightarrow\left(  H,e\right)  $ is a
unital positive mapping given by a unital $H$-support $k\subseteq K_{h}$ and
$T=S\iota$ for a positve mapping $S:L^{2}\left(  X,\mu\right)  \rightarrow
H$\textit{, }then $k\subseteq K_{\mu}$ with $\sup\left\Vert k\right\Vert
_{\mu}<\infty$.
\end{lemma}

\begin{proof}
Since $T=SP\iota$ and $P$ is a positive projection (see Lemma \ref{lemUTM2}),
we can assume that $S=S_{m}$ for a certain $H$-support $m=\left\{  m_{f}:f\in
F\right\}  $ in $\operatorname{im}\left(  P\right)  $, where $m_{f}\perp u$,
$f\neq e$ and $m_{e}=m_{e}^{u}+u$, $m_{e}^{u}\in\operatorname{ball}%
\operatorname{im}\left(  P\right)  _{h}^{u}$ (see Proposition \ref{propKom1}).
Then
\[
k_{f}=\iota^{\ast}S^{\ast}f=\iota^{\ast}\left(  m_{f}\right)  =\iota^{\ast
}\left(  \sum_{\chi}\left(  m_{f},\chi\right)  \chi\right)  =\sum_{\chi
}\left(  m_{f},\chi\right)  \iota^{\ast}\left(  \chi\right)  =\sum_{\chi
}\left(  m_{f},\chi\right)  \sigma_{\chi}.
\]
By Lemma \ref{lemL111}, $k_{f}^{\prime}=m_{f}$ and $\left\Vert k_{f}%
\right\Vert _{\mu}^{2}=\sum_{\chi}\left(  m_{f},\chi\right)  ^{2}=\left\Vert
m_{f}\right\Vert _{2}^{2}<\infty$. If $f\neq e$ then $k_{f}=\sum_{\chi\neq
u}\left(  m_{f},\chi\right)  \sigma_{\chi}\in K_{\mu}\cap K_{h}^{u}$, whereas
$k_{e}=\sum_{\chi\neq u}\left(  m_{e}^{u},\chi\right)  \sigma_{\chi}+u$. Hence
$k\subseteq K_{\mu}$ and $\sup\left\Vert k\right\Vert _{\mu}=\sup\left\Vert
m\right\Vert _{2}<\infty$.
\end{proof}

Notice that $T$ is $L^{2}\left(  X,\mu\right)  $-factorable iff $\left\Vert
T\eta\right\Vert \leq C\left\Vert \eta\left(  \cdot\right)  \right\Vert _{2}$,
$\eta\in K$ for some positive constant $C$. In particular, $T$ transforms
$\left\Vert \cdot\right\Vert _{2}$-bounded sequences from $K$ to bounded
sequences in $H$.

\begin{lemma}
\label{lemL13}Let $T:\left(  K,u\right)  \rightarrow\left(  H,e\right)  $ be a
unital positive mapping given by a unital $H$-support $k\subseteq K_{h}$,
which transforms $\left\Vert \cdot\right\Vert _{2}$-bounded sequences from $K$
to bounded sequences in $H$ for some $\mu\in\mathcal{U}\left(  S\left(
\mathfrak{c}_{u}\right)  \right)  $. If $k\subseteq K_{\mu}$ with
$\sup\left\Vert k\right\Vert _{\mu}<\infty$, then $T=S\iota$ for a certain
unital $\ast$-linear mapping $S:L^{2}\left(  X,\mu\right)  \rightarrow H$.
\end{lemma}

\begin{proof}
As above we put $X=S\left(  \mathfrak{c}_{u}\right)  $. Since $k\subseteq
K_{\mu}$, it follows that $k_{f}=\sum_{\chi}\alpha_{f,\chi}\sigma_{\chi}$ with
$\alpha_{f}=\left(  \alpha_{f,\chi}\right)  _{\chi}\in\ell^{2}\left(
B\right)  _{h}$, $\left\Vert k_{f}\right\Vert _{\mu}=\left\Vert \alpha
_{f}\right\Vert _{2}$ thanks to Lemma \ref{lemL111}. Note that $\alpha
_{f,u}=0$, $f\neq e$, and $\alpha_{e,u}=1$. Put $m_{f}=\sum_{\chi}%
\alpha_{f,\chi}\chi\in\operatorname{im}\left(  P\right)  _{h}$, $f\in F$, and
$m=\left\{  m_{f}:f\in F\right\}  $. Then $m_{f}=k_{f}^{\prime}$ and
$\left\Vert k_{f}\right\Vert _{\mu}=\left\Vert m_{f}\right\Vert _{2}$ for all
$f$ (see Lemma \ref{lemL111}). Notice that $\left\{  m_{f}:f\neq e\right\}
\subseteq\operatorname{im}\left(  P\right)  _{h}^{u}$ and $m_{e}=m_{e}^{u}+u$
with $m_{e}^{u}=\sum_{\chi\neq u}\alpha_{e,\chi}\chi\in\operatorname{im}%
\left(  P\right)  _{h}^{u}$. Moreover, $\sup\left\Vert m\right\Vert _{2}%
=\sup\left\{  \left\Vert \alpha_{f}\right\Vert _{2}:f\in F\right\}
=\sup\left\Vert k\right\Vert _{\mu}<\infty$, that is, $m$ is a bounded family
in $L^{2}\left(  X,\mu\right)  _{h}$. Take $\theta\in\operatorname{im}\left(
P\right)  $. Then $\theta=\lim_{n}\eta_{n}\left(  \cdot\right)  $ for a
certain sequence $\left(  \eta_{n}\right)  _{n}\subseteq K$. Thus $\left(
\eta_{n}\right)  _{n}$ is a $\left\Vert \cdot\right\Vert _{2}$-bounded
sequence in $K$. By assumption, $\left(  T\eta_{n}\right)  _{n}$ is a bounded
sequence in $H$. Put $S\left(  \theta\right)  =\sum_{f}\left(  \theta
,m_{f}\right)  f$. Using again Lemma \ref{lemL111} and the lowersemicontinuity
argument, we deduce that%
\begin{align*}
\left\Vert S\left(  \theta\right)  \right\Vert ^{2}  &  =\sum_{f}\left\vert
\left(  \theta,m_{f}\right)  \right\vert ^{2}=\sum_{f}\lim_{n}\left\vert
\left(  \eta_{n}\left(  \cdot\right)  ,m_{f}\right)  \right\vert ^{2}\leq
\lim\inf_{n}\sum_{f}\left\vert \left(  \eta_{n}\left(  \cdot\right)
,m_{f}\right)  \right\vert ^{2}\\
&  =\lim\inf_{n}\sum_{f}\left\vert \left(  \eta_{n},k_{f}\right)  \right\vert
^{2}=\lim\inf_{n}\left\Vert T\eta_{n}\right\Vert ^{2}<\infty.
\end{align*}
By the Uniform Boundedness Principle, we obtain that $S\in\mathcal{B}\left(
L^{2}\left(  X,\mu\right)  ,H\right)  $ with $S=SP$. Moreover, $S\left(
\iota\left(  \eta\right)  \right)  =\sum_{f}\left(  \eta\left(  \cdot\right)
,m_{f}\right)  f=\sum_{f}\left(  \eta,k_{f}\right)  f=T\left(  \eta\right)  $,
$\eta\in K$ thanks to Lemma \ref{lemL111}.
\end{proof}

Notice that the unital $\ast$-linear mapping $S$ from Lemma \ref{lemL13} may
not be positive being just $\iota\left(  \mathfrak{c}_{u}\right)  ^{-}%
$-positive. But that is the case of an exact $K$-measure $\mu$.

\begin{proposition}
\label{propMXO1}Let $T:\left(  K,u\right)  \rightarrow\left(  H,e\right)  $ be
a unital positive mapping given by a unital $H$-support $k\subseteq K_{h}$,
and let $\mu$ be an exact $K$-measure on the state space $X=S\left(
\mathfrak{c}_{u}\right)  $. Then $T$ is factorized as $T=S\iota$ throughout
the canonical mapping $\iota:K\rightarrow L^{2}\left(  X,\mu\right)  $ and a
unital positive mapping $S:L^{2}\left(  X,\mu\right)  \rightarrow H$ iff
$k\subseteq K_{\mu}$. In this case, $T$ is of finite rank automatically.
\end{proposition}

\begin{proof}
If $T$ admits a positive $L^{2}\left(  X,\mu\right)  $-factorization then
$k\subseteq K_{\mu}$ thanks to Lemma \ref{lemL12}. Conversely, assume that the
latter inclusion holds. Using Lemma \ref{lemL111} and Proposition
\ref{propEXM1}, we deduce that $\left\Vert k_{f}\right\Vert _{\mu}=\left\Vert
k_{f}^{\prime}\right\Vert _{2}=\left\Vert k_{f}\right\Vert \leq\sup\left\Vert
k\right\Vert <\infty$ for all $f\neq e$. In particular, $\sup\left\Vert
k\right\Vert _{\mu}<\infty$. As in the proof of Lemma \ref{lemL13}, we have
the bounded family $m=\left\{  m_{f}:f\in F\right\}  =k^{\prime}$ in
$\operatorname{im}\left(  P\right)  $. Take $\eta\in\mathfrak{c}_{u}$. Using
Lemma \ref{lemL111}, we derive that
\begin{equation}
\sum_{f\neq e}\left(  \eta\left(  \cdot\right)  ,m_{f}\right)  ^{2}%
=\sum_{f\neq e}\left(  \eta,k_{f}\right)  ^{2}\leq\left(  \eta,k_{e}\right)
^{2}=\left(  \eta\left(  \cdot\right)  ,m_{e}\right)  ^{2} \label{est1}%
\end{equation}
in $L^{2}\left(  X,\mu\right)  $. Now take $\theta\in L^{2}\left(
X,\mu\right)  _{+}$. Then $P\left(  \theta\right)  \in\operatorname{im}\left(
P\right)  \cap L^{2}\left(  X,\mu\right)  _{+}=\iota\left(  \mathfrak{c}%
_{u}\right)  ^{-}$ thanks to Proposition \ref{propEXM1}, that is, $P\left(
\theta\right)  =\lim_{n}\eta_{n}\left(  \cdot\right)  $ in $L^{2}\left(
X,\mu\right)  $ for a certain sequence $\left(  \eta_{n}\right)  _{n}%
\subseteq\mathfrak{c}_{u}$. Using (\ref{est1}) and the lowersemicontinuity
argument, we deduce that
\begin{align*}
\sum_{f\neq e}\left(  \theta,m_{f}\right)  ^{2}  &  =\sum_{f\neq e}\left(
P\left(  \theta\right)  ,m_{f}\right)  ^{2}=\sum_{f\neq e}\lim_{n}\left(
\eta_{n}\left(  \cdot\right)  ,m_{f}\right)  ^{2}\leq\lim\inf_{n}\sum_{f\neq
e}\left(  \eta_{n}\left(  \cdot\right)  ,m_{f}\right)  ^{2}\\
&  \leq\lim\inf_{n}\left(  \eta_{n}\left(  \cdot\right)  ,m_{e}\right)
^{2}=\left(  \theta,m_{e}\right)  ^{2},
\end{align*}
which means that $m$ is a unital $H$-support in $\left(  L^{2}\left(
X,\mu\right)  ,u\right)  $. By Proposition \ref{propKom1}, $S:L^{2}\left(
X,\mu\right)  \rightarrow H$, $S\left(  \theta\right)  =\sum_{f}\left(
\theta,m_{f}\right)  f$ is a unital positive mapping that responds to $m$, and
$S\iota=T$.
\end{proof}

In particular, the assertion of Proposition \ref{propMXO1} is applicable to a
unital atomic measure $\mu$ on $X$ with its finite support$.$

\end{document}